\documentclass[12pt,a4paper]{article}
\usepackage[utf8]{inputenc}
\usepackage{amsmath}
\usepackage{amsfonts}
\usepackage{amssymb}
\usepackage{natbib}
\usepackage{url}
\usepackage[left=1.2cm, right=1.2cm]{geometry}
\usepackage{color}
\usepackage{hyperref}
\usepackage{authblk}
\usepackage{graphicx}
\usepackage{placeins}
\usepackage{morefloats}

\newtheorem{condition}{Condition}
\newtheorem{theorem}{Theorem}
\newtheorem{proposition}{Proposition}
\newtheorem{lemma}{Lemma}
\newtheorem{example}{Example}
\newtheorem{definition}{Definition}
\newtheorem{corollary}{Corollary}
\newtheorem{remark}{Remark}

\newenvironment{proof}{\paragraph{Proof:}}{\hfill$\square$}

\begin{document}

\title{A unifying approach for doubly-robust $\ell_1$ regularized estimation
of causal contrasts}
\date{}

\author[1]{Ezequiel Smucler \thanks{esmucler@utdt.edu}}
\affil[1]{Department of Mathematics and Statistics, Universidad Torcuato Di Tella}
\author[2]{Andrea Rotnitzky \thanks{arotnitzky@utdt.edu}}
\affil[2]{Department of Economics, Universidad Torcuato Di Tella and CONICET}
\author[3]{James M. Robins \thanks{robins@hsph.harvard.edu}}
\affil[3]{Department of Epidemiology and Biostatistics, Harvard T.H. Chan School of Public Health}

\maketitle

\begin{abstract}
We consider inference about a scalar parameter under a non-parametric model
based on a one-step estimator computed as a plug in estimator plus the
empirical mean of an estimator of the parameter's influence function. We
focus on a class of parameters that have influence function which depends on
two infinite dimensional nuisance functions and such that the bias of the
one-step estimator of the parameter of interest is the expectation of the
product of the estimation errors of the two nuisance functions. Our class
includes many important treatment effect contrasts of interest in causal
inference and econometrics, such as ATE, ATT, an integrated causal contrast
with a continuous treatment, and the mean of an outcome missing not at
random. We propose estimators of the target parameter that entertain
approximately sparse regression models for the nuisance functions allowing
for the number of potential confounders to be even larger than the sample
size. By employing sample splitting, cross-fitting and $\ell_1$-regularized
regression estimators of the nuisance functions based on objective functions
whose directional derivatives agree with those of the parameter's influence
function, we obtain estimators of the target parameter with two desirable
robustness properties: (1) they are \textit{rate doubly-robust} in that they
are root-n consistent and asymptotically normal when both nuisance functions
follow approximately sparse models, even if one function has a very
non-sparse regression coefficient, so long as the other has a sufficiently
sparse regression coefficient, and (2) they are \textit{model doubly-robust}
in that they are root-n consistent and asymptotically normal even if one of
the nuisance functions does not follow an approximately sparse model so long
as the other nuisance function follows an approximately sparse model with a
sufficiently sparse regression coefficient.
\end{abstract}

\affil[1]{Department of Mathematics and Statistics, Universidad Torcuato Di
Tella}

\affil[2]{Department of Economics, Universidad Torcuato Di Tella and CONICET}

\affil[3]{Department of Epidemiology and Biostatistics, Harvard T.H. Chan
School of Public Health}

\affil[1]{Department of Mathematics and Statistics, Universidad Torcuato Di
Tella}

\affil[2]{Department of Economics, Universidad Torcuato Di Tella and CONICET}

\affil[3]{Department of Epidemiology and Biostatistics, Harvard T.H. Chan
School of Public Health}

\section{Introduction}

\label{sec:intro}

This paper was motivated by a spate of recent papers (\cite{Farrell}, \cite%
{stijn}, \cite{Tan}, \cite{NeweyCherno19}, \cite{debiased} and \cite{Peng})\ on methods for
the estimation of the average treatment effect (ATE) from observational data
in the current `big data era' in which the collection of data on a
high-dimensional $p-$vector of potential confounding factors, of length
often greater than the sample size $n$, has become standard practice.
Assuming no confounding by unmeasured covariates, successful control of
confounding requires accurate estimation of the conditional mean of the
outcome of interest given data on potential confounders and the treatment
(referred to as the outcome regression) and/or the conditional expectation
of treatment given the confounders (referred to as the propensity score).
All of the existing papers have assumed that the outcome regression and the
propensity score functions were exactly or approximately sparse and
therefore proposed to estimate them using $\ell _{1}$ regularized methods.
Following \cite{BelloniCherno11} and \cite{BelloniChernoSparse}, by
approximately sparse we mean one of two things: either the function can be
well approximated by a linear combination of $s=o\left( n\right) $ of the $p$
covariates, possibly after a transformation by a non-linear link function,
or the function is smooth enough so that it can be well approximated by a
linear combination of $s=o\left( n\right) $ elements of a countable
dictionary of functions of the covariates $L.$ Note that, in particular,
approximate sparsity includes the exactly sparse case in which the function
depends solely on a small fraction, compared to the sample size, of the $p$
covariates.

The goal of this paper is to propose a unifying methodology that extends and
improves upon the various existing $\ell _{1}$ regularized methods. The
existing papers on the topic differ in the estimators proposed, the
assumptions made about the data generating process, and the theorems proved
or conjectures made about the statistical behavior of their estimators under
their assumptions. In particular, all earlier papers prove or conjecture
that their estimators are doubly robust; however, in the high dimensional
setting, there are two different natural definitions of double robustness:
model double robustness and rate double robustness, both rigorously defined
in Section \ref{sec:models} and loosely defined later in this introduction.
Each paper concentrates on one definition or the other. We propose an
estimator of ATE that improves upon previous estimators by being
simultaneously doubly robust in both senses.

In fact, our methodology is not restricted to estimation of just ATE.
Specifically, \cite{NeweyCherno19} showed that ATE is an instance of a much
larger class of functionals, which includes many parameters of interest in
causal inference, and which can be expressed as a continuos and linear
functional of the conditional mean of an outcome given covariates.
Parameters in \cite{NeweyCherno19} et. al. class have the property that rate
doubly robust estimators can be obtained by the estimation of two nuisance
functions of covariates, one of them being the conditional mean of the
outcome given covariates, even when the number $p$ of covariates exceeds the
sample size. In an earlier article, \cite{Robins08Higher} considered another
class of functionals, which also admit rate doubly robust estimators. The
classes of \cite{Robins08Higher} and \cite{NeweyCherno19} intersect but none
is included in the other. The class of \cite{Robins08Higher} includes
parameters not in the \cite{NeweyCherno19} class that are of interest to
statisticians, and that have been extensively studied in the low dimensional
setting with $p\ll n$ \citep{scharfstein}. In a companion paper (\cite%
{globalclass}) we show that there exists a strictly larger class of
functionals, i.e. a class which strictly includes the classes of \cite%
{NeweyCherno19} and of \cite{Robins08Higher}, for which it is possible to
construct estimators with the rate doubly robust property. In fact, the
unifying methodology presented in this paper is suitable for parameters in
this larger class and provides estimators that have both the model and rate
double robust property. Specifically, for parameters in this class we
construct estimators that are simultaneously rate and model doubly robust
and that use $\ell _{1}$ regularized estimation of the two nuisance
functions. As we will explain in Section \ref{sec:heuristic}, to achieve
model robustness our methodology strongly relies on specially chosen loss
functions for the $\ell _{1}$ regularized estimation of each of the two
nuisance functions. Each such loss function is derived as a consequence of a
result which establishes that for parameters in our class, the derivative of
their influence function in the direction of one of the nuisance functions
is an unbiased estimating function for the second nuisance function. This
key result extends to our general class of parameters, a similar one by \cite%
{Robins08Higher} for parameters in their class. In turn, both our result and 
\cite{Robins08Higher}, extend to the non-parametric setting a seminal
analogous result in the parametric setting by \cite{vermeulen}. To the best
of our knowledge, \cite{stijn} is the first article to have noticed that
using loss functions for the nuisance functions derived from this key
property, yields estimators with the model double robustness property under
working sparse models for the two nuisance functions, albeit for the special
case of estimation of ATE.

The existing papers on model double robustness restrict attention to
estimation of ATE (\cite{stijn}, \cite{Tan} and \cite{Peng}). These papers
begin by specifying two working $p$ dimensional parametric generalized
linear models with known link functions for the outcome regression and for
the propensity score, with $p$ often much greater than $n.$ An estimator of
ATE has the model double robustness property if it is consistent and
asymptotically normal (CAN) when at least one of the two working models is
correctly specified, without needing to know which of the two is correct.
Our estimators will be model doubly robust for working models that can be
not only exactly sparse, as assumed by the existing papers, but also
approximately sparse. However, as in all available aforementioned articles
about model double robustness, we will need to make strong assumptions about
the probability limit of the estimated, but missmodeled, nuisance function.
Specifically, we will need to assume that this limit exists and that it is
an approximately sparse function. This is a strong requirement because, even
if the true nuisance function depends only on a small subset of the $p$
covariates, if this nuisance function is estimated under a misspecified
model, it is not clear why the estimated function should converge to a limit
that also depends only on small subset of covariates. In Section \ref%
{sec:asym_lin} we provide a few stylized examples in which this requirement
holds. However, we are not aware of general sufficient conditions that
ensure the requirement. This appears to be an open question in the
literature. Having recognized this limitation we emphasize that our goal is
to provide a unifying methodology that encompasses and extends all available
proposals on model doubly robust estimation, and at the same time produces
estimators with the rate double robustness property.

Papers on rate double robustness of ATE \cite{Farrell}$\,$ or, more
generally, of parameters in the \cite{NeweyCherno19} class (see also \cite{cherno2}), with both
nuisance functions estimated with $\ell _{1}$ regularized methods, seek to
come up with estimators that are CAN if one succeeds in estimating both
nuisance functions at sufficiently fast rates, with the possibility of
trading off slower rates of convergence for estimators of one of the
nuisance functions for faster rates of convergence of the estimator of the
other nuisance function. In this paper we will show that, assuming that each
of the nuisance functions is approximately sparse, possibly on a non-linear
scale, with sparsities $s_{a}$ and $s_{b}$, then using $\ell _{1}-$%
penalization with the aforementioned specially chosen loss functions, one
obtains estimators of the nuisance functions that converge at rates $\sqrt{{%
s_{a}\log (p)}/{n}}$ and $\sqrt{{s_{b}\log (p)}/{n}}$ respectively. Our
estimators of the parameters in our class will have the rate double
robustness property in that they are CAN if ${s_{a}s_{b}\log (p)}^{2}=o(n).$
Thus, the rate double robustness property will imply the possibility of
obtaining $\sqrt{n}-$ consistent estimation of the target parameter, even
when one of the nuisance functions -regardless of which one - is quite
non-sparse, i.e. with a sparsity degree of order $n^{1-\delta }$ for small $%
\delta ,$ so long as the remaining nuisance function is sufficiently sparse,
i.e. with a sparsity degree of order $n^{\delta ^{\prime }},$ for any $%
\delta ^{\prime }<\delta .$ For parameters in the \cite{NeweyCherno19} class
for which, as indicated earlier, one of the nuisance functions is the
conditional mean of an outcome on covariates, our specially chosen loss
function yields the usual $\ell _{1}$ penalized least squares of this
nuisance function when the working model for it is a linear regression.
Furthermore, our estimators of the remaining nuisance function for the
special case in which the working model is a linear model, coincide with the
estimators proposed and studied in \cite{NeweyCherno19}. Lastly, when both
working models are linear and the two nuisance functions are estimated in
this fashion, our estimator of the target parameter coincides with that in 
\cite{NeweyCherno19}. As in \cite{NeweyCherno19}, to achieve the rate double
robustness property we require that our estimator of the target parameter\
use sample splitting and cross-fitting. By sample splitting we mean that the
data is randomly divided into two (or more) samples - the estimation sample
and the nuisance sample. The estimators of the nuisance parameters are
computed using the nuisance sample data. In turn, the estimator $\widetilde{%
\chi }$ of the target parameter is computed from the estimation sample data,
treating the estimates of the nuisance parameters as fixed functions. This
approach is employed to avoid imposing conditions on the complexity of the
nuisance functions. Without sample splitting $\sqrt{n}-$ consistency of the
estimator of the target parameter would not be guaranteed unless one makes
strong Donsker assumptions on the complexity of both nuisance functions.
However, such Donsker assumptions defeat the purpose of double robustness,
namely trading off the complexity of one function for the simplicity of the
other. The efficiency lost due to sample splitting can be recovered by
cross-fitting. The cross-fit estimator $\widehat{\chi }$ averages $%
\widetilde{\chi }$ with its `twin' obtained by exchanging the roles of the
estimation and training sample. In the semiparametric statistics literature,
the possibility of using sample-splitting with cross-fitting to avoid the
need for Donsker conditions has a long history (\cite{schick1986}, Chapter
25 of \cite{vaart-book98}), although the idea of explicitly combining
cross-fitting with machine learning was not emphasized until recently. \cite%
{ayyagariphd} Ph.D. thesis (subsequently published as \citet{robins2013new})
and \citet{zheng2011cross} are early examples of papers that emphasized the
theoretical and finite sample advantages of doubly robust machine learning
estimators.

The rest of this paper is organized as follows. In Section \ref{sec:setup}
we define the class of parameters that we consider estimation for. We list
several examples of parameters in this class that are of interest in causal
inference and econometrics. We also formally define the rate and model
double robustness properties. In Section \ref{sec:doubly_robust_algo} we
introduce the estimating algorithms that we propose. The approximately
sparse models that we consider for the nuisance functions are defined in
Section \ref{sec:models}. Moreover, in this section we provide some informal
heuristic arguments to explain why our estimators are simultaneously rate
and model doubly robust. Then, in Section \ref{sec:asym_results_chi} we
state and discuss our formal asymptotic results for the proposed estimators.
In Section \ref{sec:litrev} we review the related literature. Finally, in
Section \ref{sec:conclusion} we give some concluding remarks and discuss
possible future research directions. Due to space constraints, we report the
results of a simulation study in Appendix D.

Appendix A contains the proofs of all our results regarding the $\ell_{1}$
regularized estimators of the nuisance parameters that we propose. The
proofs of all the asymptotic results stated in Section \ref%
{sec:asym_results_chi} can be found in Appendix B. Appendix C contains
several technical results that are needed throughout.

\section{The setup}

\label{sec:setup}

Given a sample $\mathcal{D}_{n}$ of $n$ i.i.d. copies of $O$ with law $P$
assumed to belong to a model 
\begin{equation*}
\mathcal{M}=\left\{ P_{\eta }:\eta \in \mathbf{\Xi }\right\}
\end{equation*}%
where $\mathbf{\Xi }$ is a large, non-Euclidean, parameter space we consider
inference about a one dimensional regular parameter $\chi \left( \eta
\right) $. We allow the model, i.e. the sample sample space of $O$ and the
parameter space $\mathbf{\Xi ,}$ as well as the parameter of interest $\chi
\left( \eta \right) $ to depend on $n$ but we suppress $n$ from the
notation. We assume $O$ includes a vector $Z$ with sample space $\mathcal{Z}$
$\subset R^{d}$ where $d$ can depend on $n.$ Furthermore, we assume $\mathbf{%
\Xi =\Xi }_{1}\mathbf{\times \Xi }_{2}$ with parameters in $\mathbf{\Xi }%
_{1} $ indexing the law $P_{Z}$ of $Z$ and parameters in $\mathbf{\Xi }_{2}$
governing the law of \ $O|Z.$ We will consider inference under a
non-parametric model $\mathcal{M}$ in the sense that its maximal tangent
space at each $\eta $, i.e. the $L_{2}\left( P_{\eta }\right) $-closed
linear span of the collections of all scores for regular one dimensional
parametric submodels through $P_{\eta },$ is equal to $L_{2}\left( P_{\eta
}\right) .$

Let $\widehat{\eta }$ be some estimator of $\eta $ and $\chi \left( \widehat{%
\eta }\right) $ be the plug-in estimator of $\chi \left( \eta \right) .$ A
strategy for reducing the bias of $\chi \left( \widehat{\eta }\right) $ is
to subtract from it the estimate $-\mathbb{P}_{n}\chi _{\widehat{\eta }}^{1}$
of its first order bias (\cite{Robins17}, \cite{neweystep1}, \cite{neweystep2}), where $\chi _{\eta }^{1}$ is an
influence function of $\chi \left( \eta \right) $ in model $\mathcal{M}$,
yielding the one step estimator%
\begin{equation}
\widehat{\chi }=\chi \left( \widehat{\eta }\right) +\mathbb{P}_{n}\chi _{%
\widehat{\eta }}^{1}.  \label{eq:one_step}
\end{equation}%
An influence function $\chi _{\eta }^{1}$ of $\chi \left( \eta \right) $
under model $\mathcal{M}$ is any mean zero random variable with finite
variance under $P_{\eta }$ such that for every regular parametric submodel $%
t\rightarrow P_{\eta _{t}}$ of $\mathcal{M}$ through $P_{\eta }=P_{\eta
,t=0} $ with score $g$ at $t=0,$ satisfies $\left. \frac{d}{dt}\chi \left(
\eta _{t}\right) \right\vert _{t=0}=E_{\eta }\left( \chi _{\eta
}^{1}g\right) $ where throughout $E_{\eta }\left( \cdot \right) $ stands for
expectation under $P_{\eta }.$ Parameters for which an influence function
exists are called regular. Furthermore, when as we assume in this article,
the model $\mathcal{M}$ is non-parametric, regular parameters have a unique
influence function (Chapter 25, \cite{vaart-book98}).

We will consider one step estimation of parameters in the following class .

\begin{definition}[Class of bilinear influence function (BIF) functionals]
\label{def:quad_DR_class} The parameter $\chi \left( \eta \right) $ is in
the class of functionals with bilinear influence function if and only if $%
\chi \left( \eta \right) $ has an influence function of the form%
\begin{equation}
\chi _{\eta }^{1}\left( O\right) =S_{ab}a\left( Z\right) b\left( Z\right)
+m_{a}\left( O,a\right) +m_{b}\left( O,b\right) +S_{0}-\chi \left( \eta
\right) ,  \label{eq:mainIF}
\end{equation}%
where $a\left( Z\right) $ and $b\left( Z\right) $ are variation independent
elements of $L_{2}\left( P_{\eta ,Z}\right) $, $S_{ab}\equiv s_{ab}\left(
O\right) $ and $S_{0}\equiv s_{0}\left( O\right) $ are known functions of $O$
with $P\left( S_{ab}\geq 0\right) =1$ or $P\left( S_{ab}\leq 0\right) =1,\
E_{\eta }\left( S_{ab}|Z\right) a\left( Z\right) $ and $E_{\eta }\left(
S_{ab}|Z\right) b\left( Z\right) $ are in $L_{2}\left( P_{\eta ,Z}\right) ,$
and $m_{a}\left( \cdot ,\cdot \right) $ and $m_{b}\left( \cdot ,\cdot
\right) $ are known real valued functions such that for each $\eta $, the
maps 
\begin{equation*}
h\in L_{2}(P_{Z,\eta })\rightarrow m_{a}(O,h)\text{ and }h\in
L_{2}(P_{Z,\eta })\rightarrow m_{b}(O,h)\text{ are linear }a.s.(P_{Z,\eta })
\end{equation*}%
and the maps 
\begin{equation*}
h\in L_{2}(P_{Z,\eta })\rightarrow E_{\eta }\left[ m_{a}(O,h)\right] \text{
and }h\in L_{2}(P_{Z,\eta })\rightarrow E_{\eta }\left[ m_{b}(O,h)\right] 
\text{ are continuous}
\end{equation*}%
with Riesz representers $\mathcal{R}_{a}\left( Z\right) $ and $\mathcal{R}%
_{b}\left( Z\right) $ respectively.
\end{definition}

As we will illustrate, the class of parameters with bilinear influence
function includes many important examples of target parameters in causal
inference and econometrics. In fact, \cite{globalclass} showed that the
class strictly includes two classes of parameters studied earlier, one by 
\cite{Robins08Higher} and another by \cite{NeweyCherno19} (see also \cite{newey94} and \cite{hirshberg}). \ It is also
shown in \cite{globalclass} that a key feature of parameters with bilinear
influence functions is that they satisfy that for any $\eta ^{\prime }$ and
associated $a%
{\acute{}}%
$ and $b%
{\acute{}}%
$, 
\begin{equation}
\chi \left( \eta ^{\prime }\right) -\chi \left( \eta \right) +E_{\eta
}\left( \chi _{\eta ^{\prime }}^{1}\right) =E_{\eta }\left[ S_{ab}\left\{
a^{\prime }\left( Z\right) -a\left( Z\right) \right\} \left\{ b^{\prime
}\left( Z\right) -b\left( Z\right) \right\} \right]  \label{bias1}
\end{equation}%
We refer to this property as the mixed bias property.

For parameters in the BIF class, the one step estimator \eqref{eq:one_step}
depends on $\widehat{\eta }$ only through estimators $\widehat{a}$ and $%
\widehat{b}$. If $\widehat{a}$ and $\widehat{b}$ are estimated from a sample
independent of $\mathcal{D}_{n}$, then a key consequence of $\left( \ref%
{bias1}\right) $ is that the conditional bias of $\widehat{\chi },$ namely $%
E_{\eta }\left[ \left. \chi \left( \widehat{\eta }\right) +\chi _{\widehat{%
\eta }}^{1}\right\vert \widehat{a},\widehat{b}\right] -$ $\chi \left( \eta
\right) $, is of order $O\left( \gamma _{a,n}\gamma _{b,n}\right) $ if $%
\left\Vert \widehat{a}-a\right\Vert _{L_{2}\left( P_{Z}\right) }=O_{p}\left(
\gamma _{a,n}\right) $ and $\left\Vert \widehat{b}-b\right\Vert
_{L_{2}\left( P_{Z}\right) }=O_{p}\left( \gamma _{b,n}\right) $. Using
arguments as in Chapter 25 of \cite{vaart-book98}, it can be shown that this
in turn implies that, under regularity conditions, $\widehat{\chi }$ has the
following property.

\begin{definition}[Rate double robustness]
\label{def:rate_dr} An estimator $\widehat{\chi }$ that depends on estimates 
$\widehat{a},\widehat{b}$ of two nuisance parameters $a$ and $b$ has the
rate double robustness property if $\sqrt{n}\left( \widehat{\chi }-\chi
\left( \eta \right) \right) $ converges to a mean zero Normal distribution
whenever $\left\Vert \widehat{a}-a\right\Vert _{L_{2}\left( P_{Z}\right)
}=O_{p}\left( \gamma _{a,n}\right) ,$ $\left\Vert \widehat{b}-b\right\Vert
_{L_{2}\left( P_{Z}\right) }=O_{p}\left( \gamma _{b,n}\right) ,$ $\gamma
_{a,n}=o\left( 1\right) ,$ $\gamma _{b,n}=o\left( 1\right) $ and $\gamma
_{a,n}\gamma _{b,n}=o\left( n^{-1/2}\right)$.
\end{definition}

Because the rates of convergence $\gamma _{a,n}$ and $\gamma _{b,n}$ of
estimators $\widehat{a}$ and $\widehat{b}$ depend on the complexity of $a$
and $b,$ the rate double robustness property implies that $\widehat{\chi }$
is $\sqrt{n}-$ consistent and asymptotically normal even if one of the
functions $a$ or $b$ is very complex so long as the other is simple enough.

Now consider the following property.

\begin{definition}[Model double robustness]
\label{def:model_dr copy} An estimator $\widehat{\chi }$ that depends on
estimates $\widehat{a},\widehat{b}$ of two nuisance parameters $a$ and $b$
has the model double robustness property if $\sqrt{n}\left( \widehat{\chi }%
-\chi \left( \eta \right) \right) $ converges to a mean zero Normal
distribution whenever one of the following happens (i) $\left\Vert \widehat{a%
}-a\right\Vert _{L_{2}\left( P_{Z}\right) }=O_{p}\left( \gamma _{a,n}\right) 
$ and there exists $b^{0}\not=b$ such that $\left\Vert \widehat{b}%
-b^{0}\right\Vert _{L_{2}\left( P_{Z}\right) }=O_{p}\left( \gamma
_{b,n}\right) $ with $\gamma _{a,n}=o\left( 1\right) ,\gamma _{b,n}=o\left(
1\right) $ and $\gamma _{a,n}\gamma _{b,n}=o\left( n^{-1/2}\right) ,$ or
(ii) $\left\Vert \widehat{b}-b\right\Vert _{L_{2}\left( P_{Z}\right)
}=O_{p}\left( \gamma _{b,n}\right) $ and there exists $a^{0}\not=a$ such
that $\left\Vert \widehat{a}-a^{0}\right\Vert _{L_{2}\left( P_{Z}\right)
}=O_{p}\left( \gamma _{a,n}\right) ,$ with $\gamma _{a,n}=o\left( 1\right)
,\gamma _{b,n}=o\left( 1\right) $ and $\gamma _{a,n}\gamma _{b,n}=o\left(
n^{-1/2}\right) .$
\end{definition}

Convergence to the true nuisance function ordinarily requires that one
correctly posits a model for the unknown nuisance function. We thus call the
property model double robustness because it essentially establishes that $%
\sqrt{n}\left( \widehat{\chi }-\chi \left( \eta \right) \right) $ is mean
zero asymptotically normal so long as one posits a correct model for just
one of the two nuisance functions, regardless of which one.

Remarkably, in this paper we will show that for the parameters $\chi \left(
\eta \right) $ in the BIF class, with carefully chosen estimators $\widehat{a%
}$ and $\widehat{b}$, under regularity conditions, the one step estimator $%
\widehat{\chi }$ is simultaneously rate and model double robust.

The following proposition is part of Theorem 2 of \cite{globalclass}. As we
discuss in Section \ref{sec:doubly_robust_algo} this key property of
functionals in the BIF class allows us to construct loss functions for
estimating $a$ and $b$ using $\ell _{1}$ regularization.

\begin{proposition}
\label{prop:DR functional} If $\chi (\eta )$ is in the BIF class, Condition
R of the Appendix C holds and $\mathcal{M}$ is non-parametric then for any $%
h\left( Z\right) \in L_{2}\left( P_{Z,\eta }\right) ,$ 
\begin{equation*}
E_{\eta }\left[ S_{ab}a\left( Z\right) h\left( Z\right) +m_{b}\left(
O,h\right) \right] =0
\end{equation*}%
and%
\begin{equation*}
E_{\eta }\left[ S_{ab}b\left( Z\right) h\left( Z\right) +m_{a}\left(
O,h\right) \right] =0.
\end{equation*}
\end{proposition}

Next, we list several examples of target parameters of interest in causal
inference and econometrics which can be shown to belong to the BIF class.
See \cite{globalclass} for a discussion of which of these parameters belong
to the class of \cite{Robins08Higher} and/or to the class of \cite%
{NeweyCherno19}. In what follows we ignore regularity conditions, see \cite%
{globalclass} for the precise conditions.

\begin{example}[Mean of an outcome that is missing at random]
\label{ex:MAR} Suppose that $O=\left( DY,D,Z\right) $ where $D$ is binary, $%
Y $ is an outcome which is observed if and only if $D=1$ and $Z$ is a vector
of always observed covariates. If we assume that the density $p\left(
y|D=0,Z\right) $ is equal to the density $p\left( y|D=1,Z\right) $, that is,
that the outcome $Y$ is missing at random then, for $P=P_{\eta }$, the mean
of $Y$ is equal to 
\begin{equation*}
\chi \left( \eta \right) =E_{\eta }\left[ a\left( Z\right) \right]
\end{equation*}%
where $a\left( Z\right) \equiv E_{\eta }\left( DY|Z\right) /E_{\eta }\left(
D|Z\right) $. The parameter $\chi \left( \eta \right) $ is in the BIF class
with $a(Z)$ as defined, $b\left( Z\right) =1/E_{\eta }\left( D|Z\right) $, $%
S_{ab}=-D$, $m_{a}\left( O,h\right) \equiv h$, $m_{b}\left( O,h\right)
\equiv DYh$ and $S_{0}=0$.
\end{example}

\begin{example}[Mean of outcome missing at random in the non-respondents]
\label{ex:MAR-non-respondents} With the notation and assumptions of Example %
\ref{ex:MAR}, the parameter 
\begin{equation*}
\psi \left( \eta \right) \equiv \frac{E_{\eta }\left[ \left( 1-D\right)
a\left( Z\right) \right] }{E_{\eta }\left[ \left( 1-D\right) \right] }
\end{equation*}%
where again, $a\left( Z\right) \equiv E_{\eta }\left( DY|Z\right) /E_{\eta
}\left( D|Z\right) ,$ is equal to the mean of $Y$ among the non-respondents,
i.e. in the population with $D=0.$ The parameter $\nu \left( \eta \right)
\equiv $ $E_{\eta }\left[ \left( 1-D\right) a\left( Z\right) \right] $ is in
the BIF class with $a(Z)$ as defined, $b\left( Z\right) =E_{\eta }\left\{
\left( 1-D\right) |Z\right\} /E_{\eta }\left( D|Z\right)$, $S_{ab}=-D$, $%
m_{a}\left( O,h\right) \equiv \left( 1-D\right) h$, $m_{b}\left( O,h\right)
\equiv DY h$ and $S_0=0$.
\end{example}

\begin{example}[Population average treatment effect]
\label{ex:ATE} Suppose that $O=\left( Y,D,Z\right) $ where $D$ is a binary
treatment indicator, $Y$ is an outcome and $Z$ is a baseline covariate
vector. Under the assumption of unconfoundedness given $Z,$ the population
average treatment effect contrast is $ATE\left( \eta \right) \equiv \chi
_{1}\left( \eta \right) -\chi _{2}\left( \eta \right) $ where $\chi
_{1}\left( \eta \right) \equiv E_{\eta }\left[ a_{1}\left( Z\right) \right] $
and $\chi _{2}\left( \eta \right) \equiv E_{\eta }\left[ a_{2}\left(
Z\right) \right] $ with $a_{1}\left( Z\right) \equiv E_{\eta }\left(
DY|Z\right) /E_{\eta }\left( D|Z\right) $ and $a_{2}\left( Z\right) \equiv
E_{\eta }\left\{ \left( 1-D\right) Y|Z\right\} /E_{\eta }\left\{ \left(
1-D\right) |Z\right\} $. Regarding $1-D\,\ $as another missing data
indicator, example $\left( \ref{ex:MAR}\right) $ implies that $ATE\left(
\eta \right) $ is a difference of two parameters, $\chi _{1}\left( \eta
\right) $ and $\chi _{2}\left( \eta \right)$, each in the BIF class.
\end{example}

\begin{example}[Treatment effect on the treated]
\label{ex:ATT} With the notation and assumptions of Example \ref{ex:ATE},
the parameter $ATT\left( \eta \right) \equiv E\left( Y|D=1\right) -\rho
\left( \eta \right) /E_{\eta }\left( D\right) $ where $\rho \left( \eta
\right) \equiv E_{\eta }\left[ Da\left( Z\right) \right] $ and $a\left(
Z\right) $ defined as $E_{\eta }\left\{ \left( 1-D\right) Y|Z\right\}
/E_{\eta }\left\{ \left( 1-D\right) |Z\right\} $, the parameter $ATT\left(
\eta \right) $ is the average treatment effect on the treated. Once again,
regarding $1-D$ as another missing data indicator, Example \ref%
{ex:MAR-non-respondents} implies that $ATT\left( \eta \right) $ is a
continuous function of a parameter $\rho \left( \eta \right) $ in the BIF
class, and other parameters $E\left( Y|D=1\right) $ and $E_{\eta }\left(
D\right) $ whose estimation does not require the estimation of high
dimensional nuisance parameters.
\end{example}

\begin{example}[Expected conditional covariance]
\label{ex:CC} Let $O=\left( Y,D,Z\right) ,$ where $Y$ and $D$ are real
valued. Let $\chi \left( \eta \right) \equiv E_{\eta }\left[ cov_{\eta
}\left( D,Y|Z\right) \right] $ be the expected conditional covariance
between $D$ and $Y$. When $D$ is a binary treatment, $\chi \left( \eta
\right) $ is an important component of the variance weighted average
treatment effect \cite{Robins08Higher}. The parameter $\chi \left( \eta
\right) $ is in the BIF\ class with $S_{ab}=-1,a\left( Z\right) \equiv
E_{\eta }\left( Y|Z\right) ,b\left( Z\right) =-E_{\eta }\left( D|Z\right) ,$ 
$m_{a}\left( O,a\right) =-Da,$ $m_{b}\left( O,b\right) =Yb$ and $S_{0}=DY$.
\end{example}

\begin{example}[Mean of an outcome that is missing not at random]
\label{ex:MNAR}Suppose that $O=\left( DY,D,Z\right) $ where $D$ is binary, $%
Y $ is an outcome which is observed if and only if $D=1$ and $Z$ is a vector
of always observed covariates. If we assume that the density $p\left(
y|D=0,Z\right) $ is a known exponential tilt of the density $p\left(
y|D=1,Z\right) ,$ i.e. 
\begin{equation}
p\left( y|D=0,Z\right) =p\left( y|D=1,Z\right) \exp \left( \delta y\right)
/E \left[ \exp \left( \delta Y\right) |D=1,Z\right]  \label{tilt}
\end{equation}%
where $\delta $ is a given constant, then under $P=P_{\eta }$ the mean of $Y$
is 
\begin{equation*}
\chi \left( \eta \right) =E_{\eta }\left[ DY+\left( 1-D\right) a\left(
Z\right) \right]
\end{equation*}%
assuming $a\left( Z\right) \equiv E_{\eta }\left[ DY\exp \left( \delta
Y\right) |Z\right] /E_{\eta }\left[ D\exp \left( \delta Y\right) |Z\right] $
exists. Estimation of $\chi \left( \eta \right) $ under different fixed
values of $\delta $ has been proposed in the literature as a way of
conducting sensitivity analysis to departures from the missing at random
assumption \citep{scharfstein}. Under the sole restriction $\left( \ref{tilt}%
\right) $ the law $P$ of the observed data $O$ is unrestricted, and hence
the model for $P$ is non-parametric. The parameter $\chi \left( \eta \right) 
$ is in the BIF class with $a(Z)$ as defined, $b\left( Z\right) \equiv
-E_{\eta }\left( 1-D|Z\right) /E_{\eta }\left[ D\exp \left( \delta Y\right)
|Z\right] $, $S_{ab}=-D\exp \left( \delta Y\right) $, $m_{a}\left(
O,a\right) \equiv \left( 1-D\right) a$, $m_{b}(O,b)\equiv DY\exp \left(
\delta Y\right) b\left( Z\right) $ and $S_{0}=DY$.
\end{example}

\begin{example}
\label{ex:APE} In the next two examples, $O=\left( Y,D,L\right) ,$ $Z=\left(
D,L\right) ,Y$ and $D$ are real valued, $D$ is a treatment variable taking
any value in $\left[ 0,1\right] $ and $L$ is a covariate vector.
Furthermore, $Y_{d}$ \ denotes the counterfactual outcome under treatment $%
D=d$ and in the following we assume that $E_{\eta }\left( Y_{d}|L\right)
=E_{\eta }\left( Y|D=d,L\right) $.

\begin{itemize}
\item[a)] \textbf{Causal effect of a treatment taking values on a continuum}
The parameter $\chi \left( \eta \right) \equiv E_{\eta }\left[ m_{a}\left(
O,a\right) \right] $ with $a\left( D,L\right) \equiv E_{\eta }\left(
Y|D,L\right) ,m_{a}\left( O,a\right) \equiv \int_{0}^{1}a\left( u,L\right)
w\left( u\right) du$ where $w\left( \cdot \right) $ is a given scalar
function satisfying $\int_{0}^{1}w\left( u\right) du=0$ agrees with the
treatment effect contrast $\int_{0}^{1}E_{\eta }\left( Y_{u}\right) w\left(
u\right) du.$ The parameter $\chi \left( \eta \right) $ is in the BIF class
with $a(Z)$ as defined, $b\left( Z\right) ={w\left( D\right) }/{f\left(
D|L\right) }$, $S_{ab}=-1$, $m_{a}\left( O,a\right) \equiv
\int_{0}^{1}a\left( u,L\right) w\left( u\right) du,$ $m_{b}(O,b)=Yb$ and $%
S_{0}=0$.

\item[b)] \textbf{Average policy effect of a counterfactual change of
covariate values} The parameter $\chi \left( \eta \right) \equiv E_{\eta }%
\left[ a\left( t\left( D\right) ,L\right) \right] -E_{\eta }\left( Y\right) $
with $a\left( D,L\right) \equiv E_{\eta }\left( Y|D,L\right) $ is the
average policy effect of a counterfactual change $d\rightarrow t\left(
d\right) $ of treatment values \cite{stock1989}. This parameter $\chi \left(
\eta \right) $ is in the BIF class with $a\left( Z\right) $ as defined, $%
b\left( Z\right) ={f_{t}\left( D|L\right) }/{f\left( D|L\right) }$, $%
m_{a}\left( O,a\right) =a\left( t\left( D\right) ,L\right) ,m_{b}(O,b)=Yb,$ $%
S_{ab}=-1$ and $S_{0}=-Y$.
\end{itemize}
\end{example}

\begin{example}
\label{ex:artificial} The following parameter $\chi \left( \eta \right) $ is
in the BIF class but it is neither in the class of \cite{NeweyCherno19} nor
in that \cite{Robins08Higher} . Let $O=\left( Y_{1},Y_{2},Z\right) $ for $%
Y_{1}$ and $Y_{2}$ continuous random variables, $Y_{2}>0$ and $Z$ a scalar
vector taking any values in $\left[ 0,1\right] $. The parameter $\chi \left(
\eta \right) \equiv \int_{0}^{1}a\left( z\right) dz$ where $a\left( Z\right)
\equiv E_{\eta }\left( Y_{1}|Z\right) /E_{\eta }\left( Y_{2}|Z\right) $ is
in the BIF class with $a(Z)$ as defined, $b\left( Z\right) ={1}/{f\left(
Z\right) E_{\eta }\left( Y_{2}|Z\right) }$, $m_{a}\left( O,a\right) \equiv
\int_{0}^{1}a\left( z\right) dz,$ $m_{b}(O,b)=Y_{1}b$, $S_{ab}=-Y_{2}$ and $%
S_{0}=0$.
\end{example}

\section{The proposed estimators}

\label{sec:doubly_robust_algo}

To compute our proposed estimators of $\chi \left( \eta \right) $ in the BIF
class, the data analyst entertains working models $a\left( Z\right) =\varphi
_{a}\left( \left\langle \theta _{a},\phi ^{a}\left( Z\right) \right\rangle
\right) $ and $b\left( Z\right) =\varphi _{b}\left( \left\langle \theta
_{b},\phi ^{b}\left( Z\right) \right\rangle \right) ,$ where $\varphi
_{a}\left( \cdot \right) $ and $\varphi _{b}\left( \cdot \right) $ are
given, possibly non-linear, link functions, with possibly distinct, $%
p_{a}\times 1$ and $p_{b}\times 1$ covariate vectors $\phi ^{a}\left(
Z\right) $ and $\phi ^{b}\left( Z\right) $. Our proposed estimators of $a$
and $b$ are of the form $\varphi _{a}\left( \left\langle \widehat{\theta }%
_{a},\phi \right\rangle \right) $ and $\varphi _{b}\left( \left\langle 
\widehat{\theta }_{b},\phi \right\rangle \right) $, where $\phi (z)=\left(
\phi _{1}(z),\dots ,\phi _{p}(z)\right) ^{\top },$ $\left\{ \phi
_{j}\right\} _{1\leq j\leq p}=\left\{ \phi _{j}^{a}\right\} _{1\leq j\leq
p_{a}}\cup \left\{ \phi _{j}^{b}\right\} _{1\leq j\leq p_{b}}$ is the union
of the covariates in both models, with $p$ possibly larger than the sample
size $n$.

The estimators $\widehat{\theta }_{a}$ and $\widehat{\theta }_{b}$ are $l_{1}
$-regularized regression estimators of the form 
\begin{equation}
\widehat{\theta }_{c}\in \arg \min_{\theta \in \mathbb{R}^{p}}\mathbb{P}_{n}%
\left[ Q_{c}\left( \theta ,\phi ,w\right) \right] +\lambda \Vert \theta
\Vert _{1}  \label{eq:loss}
\end{equation}%
for $c\in \left\{ a,b\right\} $, with objective function 
\begin{equation}
Q_{c}\left( \theta ,\phi ,w\right) \equiv S_{ab}w\left( Z\right) \psi
_{c}\left( \langle \theta ,\phi \left( Z\right) \rangle \right) +\langle
\theta ,m_{\overline{c}}(O,w\cdot \phi )\rangle   \label{eq:lossDef}
\end{equation}%
where $\overline{c}=\left\{ a,b\right\} \backslash \left\{ c\right\} ,\psi
_{c}\left( u\right) $ is an antiderivative of $\varphi _{c}\left( u\right) ,$
i.e. $\frac{d}{du}\psi _{c}\left( u\right) =\varphi _{c}\left( u\right) ,$
and $w\left( Z\right) $ is a scalar valued, possibly data dependent, weight
function satisfying that $S_{ab}w\left( Z\right) \psi _{c}\left( \langle
\theta ,\phi \left( Z\right) \rangle \right) $ is a convex function of $%
\theta $. The specific choice of weight $w$ depends whether none, one or two
of the given links $\varphi _{a}$ and $\varphi _{b}$ are non-linear
functions. When $\varphi _{a}$ and $\varphi _{b}$ are the identity links, we
use $w\left( Z\right) =1.$ When both $\varphi _{a}$ and $\varphi _{b}$ are
non-linear functions the weights are functions of $Z$ that depend on
preliminary estimators of $\theta _{a}$ and $\theta _{b},$ which, in turn,
also solve minimization problems like $\left( \ref{eq:loss}\right) $ with
weights $w\left( Z\right) =1$. When, say $\varphi _{a}$ is the identity and $%
\varphi _{b}$ is a non-linear link, then the estimator of $\theta _{b}$
solves $\left( \ref{eq:loss}\right) $ with weight $w\left( Z\right) =1$ and $%
\theta _{a}$ solves $\left( \ref{eq:loss}\right) $ with weight $w\left(
Z\right) $ that depends on the estimator of $\theta _{b}$.

The objective function has the key property that $dQ_{c}\left( \theta ,\phi
,w\right) /d\theta $ is an unbiased estimating function under the models
entertained by the investigator for $c\in \left\{ a,b\right\} ,$ since 
\begin{equation*}
\frac{d}{d\theta }Q_{c}\left( \theta ,\phi ,w\right) =S_{ab}w\left( Z\right)
\phi \left( Z\right) \varphi _{c}\left( \langle \theta ,\phi \left( Z\right)
\rangle \right) +m_{\overline{c}}(O,w.\phi )
\end{equation*}%
which, by Proposition \ref{prop:DR functional}, has mean 0 when $c\left(
Z\right) =a\left( Z\right) =$ $\varphi _{a}\left( \langle \theta ,\phi
\left( Z\right) \rangle \right) $ and when $c\left( Z\right) =b\left(
Z\right) =\varphi _{b}\left( \langle \theta ,\phi \left( Z\right) \rangle
\right) $.

Moreover, the objective function is convex if, as we have assumed and as it
is satisfied in all our examples, $S_{ab}w\left( Z\right) \psi _{c}\left(
\langle \theta ,\phi \left( Z\right) \rangle \right) $ is a convex function
of $\theta $. Thus, under this condition our estimators $\widehat{\theta }%
_{c}$ can be computed efficiently.

As an illustration, consider the missing data example \ref{ex:MAR}. In that example,
$a\left( Z\right) \equiv E_{\eta }\left( DY|Z\right) /E_{\eta }\left(
D|Z\right),$ $b\left( Z\right) =1/E_{\eta }\left( D|Z\right) $, $S_{ab}=-D$, $%
m_{a}\left( O,h\right) \equiv h$ and $m_{b}\left( O,h\right) \equiv DYh$.
Thus, when $\varphi _{a}\left( u\right) =u\,\ $and $\varphi _{b}\left(
u\right) =\exp \left( -u\right) +1$, i.e. when one assumes working linear and
logistic regression models for $E_{\eta }\left( Y|D=1,Z\right) $ and $%
P_{\eta }\left( D=1|Z\right) $ respectively, the loss function $Q_{a}$ is
\begin{eqnarray*}
Q_{a}\left( \theta ,\phi ,w\right)  &=&Dw\left( Z\right) \left\{ Y\langle
\theta ,\phi \left( Z\right) \rangle -\langle \theta ,\phi \left( Z\right)
\rangle ^{2}/2\right\}  \\
&=&-2^{-1}Dw\left( Z\right) \left\{ Y-\langle \theta ,\phi \left( Z\right)
\rangle \right\} ^{2}+2^{-1}Dw\left( Z\right) Y^{2}
\end{eqnarray*}%
The weight $w\left( Z\right) $ that our estimators of $a$ will use is equal
to $w\left( Z\right) =-\exp \left( -\langle \widetilde{\theta }_{b},\phi
\left( Z\right) \rangle \right) $ for a preliminary estimator $\widetilde{\theta }_{b}$ of $\theta
_{b}.$ Thus, $Q_{a}$ yields the same $\ell_1$
regularized estimator of $\theta _{a}$ as the one that uses the weighted
least squares loss $\mathbb{P}_{n}\left[ D\exp \left( -\langle \widetilde{%
\theta }_{b},\phi \left( Z\right) \rangle \right) \left\{ Y-\langle \theta
,\phi \left( Z\right) \rangle \right\} ^{2}\right] .$ On the other hand,    
$
Q_{b}\left( \theta ,\phi ,w\right) =w\left( Z\right) \left[ \left\{ -D\left[
\langle \theta ,\phi \left( Z\right) \rangle -\exp \left( -\langle \theta
,\phi \left( Z\right) \rangle \right) \right] \right\} +\langle \theta ,\phi
\left( Z\right) \rangle \right].
$
The weight $w\left( Z\right) $ that our estimators of $b$ will use is equal
to $w\left( Z\right) =1.$ Notice that for such weight, $\frac{d}{d\theta }%
Q_{b}\left( \theta ,\phi ,w=1\right) =\left\{ 1-{D}/{\text{expit}\left(
\langle \theta ,\phi \left( Z\right) \rangle \right) }\right\} \phi \left(
Z\right) .$ Thus, our estimator $\widehat{\theta }_{b}$ of $\theta _{b}$
approximately balances the covariates $\phi \left( Z\right) $ among the
units with observed data, i.e. it satisfies the "balancing" property   
\begin{equation*}
\mathbb{P}_{n}\left[ \phi \left( Z\right) \right] \approx \mathbb{P}_{n}%
\left[ \frac{D}{\text{expit}\left( \langle \widehat{\theta }_{b},\phi \left(
Z\right) \rangle \right) }\phi \left( Z\right) \right].
\end{equation*}%
Estimators satisfying this property have been used in \cite{zubizarreta, athey, changlobal, tan17, imaicov, hirshberg}, among others.

We describe the three procedures next, i.e. the estimation algorithm when
both, one or none of the links are non-linear. In what follows, 
\begin{eqnarray*}
\Upsilon \left( a,b\right) &\equiv &\left( \chi +\chi ^{1}\right) _{\left(
a,b\right) } \\
&=&S_{ab}a\left( Z\right) b\left( Z\right) +m_{a}\left( O,a\right)
+m_{b}\left( O,b\right) +S_{0},
\end{eqnarray*}%
$\mathcal{D}_{n}=\left\{ O_{1},...,O_{n}\right\} $ denotes the entire
sample, and for a subsample $\mathcal{D}_{m}$ of $\mathcal{D}_{n},$ $\mathbb{%
P}_{m}$ and $\mathbb{P}_{\overline{m}}$ denote the empirical sample
operators with respect to $\mathcal{D}_{m}$ and $\mathcal{D}_{\overline{m}%
}\equiv \mathcal{D}_{n}\backslash \mathcal{D}_{m}.$ Finally, for $k\in
\left\{ 1,2,3\right\} ,$ we let $j_{1}\left( k\right) $ and $j_{2}\left(
k\right) $ denote the two distinct, ordered, elements of the set $\left\{
1,2,3\right\} $ that are not equal to $k,$ i.e. $j_{1}\left( 2\right)
=1,j_{2}\left( 2\right) =3.$

\subsection{Procedure when both given links are linear \label{Algo both ALS}}

The estimator, denoted throughout $\widehat{\chi }_{lin},$ when both given
links are linear is the cross-fitted one-step estimator with, as indicated
earlier, $\widehat{\theta }_{a}$ and $\widehat{\theta }_{b}$, solving the
minimization problem \eqref{eq:loss} using weights $w\left( Z\right) =1.$
This estimator coincides with the estimator proposed by \cite{NeweyCherno19}
in the special case in which $\chi \left( \eta \right) $ falls in their
class and $a\left( Z\right) $ is estimated by $\ell_1$ regularization.
Specifically, for parameters $\chi \left( \eta \right) $ in the class
studied in \cite{NeweyCherno19}, $a\left( Z\right) $ is a conditional mean
of some outcome given $Z.$ These authors considered generic machine learning
estimators of $a\left( Z\right) $ and estimators of $b\left( Z\right) $ like
ours. In contrast, we use $\ell_{1}$ regularization for estimation of both $%
a $ and $b$. It can be easily checked that when $a\left( Z\right) $ is a
conditional mean, the loss function $\left( \ref{eq:lossDef}\right) $ is
exactly the weighted least squares loss. So, for such $a$ our estimator is
precisely the Lasso estimator \citep{Lasso96}.

We now give the algorithm for constructing $\widehat{\chi }_{lin}$. For
simplicity assume $n$ is even.

Randomly split the data into two disjoint, equally sized samples, $\mathcal{D%
}_{n1}$ and $\mathcal{D}_{n2}$.

\begin{enumerate}
\item For $k=1,2,$ $c\in \left\{ a,b\right\} $ and $w_{0}\left( Z\right)
=1,\,\ $compute 
\begin{equation*}
\widehat{\theta }_{c,\left( \overline{k}\right) }\in \arg \min_{\theta \in 
\mathbb{R}^{p}}\mathbb{P}_{\overline{nk}}\left[ Q_{c}\left( \theta ,\phi
,w_{0}\right) \right] +\lambda \Vert \theta \Vert _{1}
\end{equation*}%
where $\lambda $ is a given tuning parameter and define $\widehat{c}_{\left( 
\overline{k}\right) }(z)\equiv \langle \widehat{\theta }_{c,\left( \overline{%
k}\right) },\phi (z)\rangle $

\item Estimate $\chi \left( \eta \right) $ with $\widehat{\chi }_{lin}\equiv 
\frac{1}{2}\sum_{k=1}^{2}\widehat{\chi }_{lin}^{\left( k\right) }$ where $%
\widehat{\chi }_{lin}^{\left( k\right) }\equiv \mathbb{P}_{nk}\{\Upsilon (%
\widehat{a}_{\left( \overline{k}\right) },\widehat{b}_{\left( \overline{k}%
\right) })\}$ and estimate the asymptotic variance of $\widehat{\chi }_{lin}$
with $\widehat{V}_{lin}\equiv \frac{1}{2}\sum_{k=1}^{2}\mathbb{P}%
_{nk}\left\{ \left( \Upsilon (\widehat{a}_{\left( \overline{k}\right) },%
\widehat{b}_{\left( \overline{k}\right) })-\widehat{\chi }_{lin}^{\left(
k\right) }\right) ^{2}\right\} $
\end{enumerate}

\subsection{Procedure when both given links are non-linear \label{Algo both
GALS}}

The formal algorithm for constructing the estimator, throughout denoted as $%
\widehat{\chi }_{nonlin},$ when both given links are non-linear is given
below. We first provide a verbal description to facilitate its reading. For
simplicity assume $n$ is divisible by 3.

\begin{itemize}
\item \medskip Randomly split the sample into three subsamples, designate
one as the main estimation sample and the other two as the nuisance
estimation samples.\medskip\ Assume working models $a\left( Z\right)
=\varphi _{a}\left( \left\langle \theta _{a},\phi \left( Z\right)
\right\rangle \right) $ and $b\left( Z\right) =\varphi _{b}\left(
\left\langle \theta _{b},\phi \left( Z\right) \right\rangle \right) .$

\begin{itemize}
\item[(i)] In the first nuisance estimation sample, compute $\widehat{\theta 
}_{a}^{0}$ and $\widehat{\theta }_{b}^{0}$ solving the minimization problem $%
\left( \ref{eq:loss}\right) $ (for $c=$ $a$ and $b),$ using weight $w\left(
Z\right) =1.$

\item[(ii)] In the second nuisance estimation sample, compute now second
stage estimators $\widehat{\theta }_{a}$ and $\widehat{\theta }_{b}$ again
solving $\left( \ref{eq:loss}\right) $ (for $c=$ $a$ and $b)$ but this time
using weights $\widehat{w}_{a}\left( Z\right) =\varphi _{b}^{\prime }\left(
\left\langle \widehat{\theta }_{b}^{0},\phi \left( Z\right) \right\rangle
\right) $ and $\widehat{w}_{b}\left( Z\right) =\varphi _{a}^{\prime }\left(
\left\langle \widehat{\theta }_{a}^{0},\phi \left( Z\right) \right\rangle
\right) $ $\medskip $\ where $\varphi _{c}^{\prime }\left( u\right) =d$ $%
\varphi _{c}\left( u\right) /du,$ $c=a$ or $b.$

\item[(iii)] Repeat steps (i) and (ii), interchanging the roles of the first
and second nuisance samples, to obtain new second stage estimators of $%
\theta _{a}$ and $\theta _{b}.$ \medskip

\item[(iv)] Re-estimate $\theta _{a}$ with the average of the two previously
obtained second stage estimators of $\theta _{a}.$ Use this final estimator
of $\theta _{a}$ to compute $\widehat{a}\left( \cdot \right) .$ Carry out
the analogous steps to compute $\widehat{b}\left( \cdot \right) $.

\item[(v)] Compute the one step estimator in the main estimation sample
using the $\widehat{a}\left( \cdot \right) $ and $\widehat{b}\left( \cdot
\right) $ computed in step (iv).
\end{itemize}

\item Repeat steps (i)-(v) twice, interchanging each time one of the
nuisance samples with the main estimation sample.

\item The final estimator $\widehat{\chi }_{nonlin}$ of $\chi \left(
\eta\right) $ is the average of the three one step estimators.
\end{itemize}

We now give the algorithm for constructing $\widehat{\chi }_{nonlin}$.

Randomly split the data into three disjoint, equally sized samples, $%
\mathcal{D}_{n1}$, $\mathcal{D}_{n2}$ and $\mathcal{D}_{n3}$. Assume $%
\varphi _{a}\left( u\right) $ and $\varphi _{b}\left( u\right) $ are
non-linear functions.

\begin{enumerate}
\item For $k\in \left\{ 1,2,3\right\} $ and $c\in \left\{ a,b\right\} $ let $%
w_{0}\left( Z\right) =1$ and compute 
\begin{equation*}
\left. \widehat{\theta }_{c,\left( k\right) }^{0}\in \arg \min_{\theta \in 
\mathbb{R}^{p}}\mathbb{P}_{nk}\left[ Q_{c}\left( \theta ,\phi ,w_{0}\right) %
\right] +\lambda \Vert \theta \Vert _{1}.\right.
\end{equation*}%
where $\lambda $ is a given tuning parameter. Let $\widehat{c}_{\left(
k\right) }(z)\equiv \varphi _{c}\left( \langle \widehat{\theta }_{c,\left(
k\right) },\phi (z)\rangle \right) .$

\item For $c\in \left\{ a,b\right\} ,$ $l\in \left\{ 1,2\right\} $ and $k\in
\left\{ 1,2,3\right\} $ let $\widehat{w}_{\overline{c},\left( j_{l}\left(
k\right) \right) }\left( Z\right) =\varphi _{\overline{c}}^{\prime }(\langle 
\widehat{\theta }_{\overline{c},\left( j_{l}\left( k\right) \right) },\phi
\left( Z\right) \rangle )$ and compute 
\begin{equation*}
\widehat{\theta }_{c,\left( k\right) ,j_{l}\left( k\right) }\in \arg
\min_{\theta \in \mathbb{R}^{p}}\mathbb{P}_{nk}\{Q_{c}\left( \theta ,\phi ,%
\widehat{w}_{\overline{c},\left( j_{l}\left( k\right) \right) }\right)
\}+\lambda \Vert \theta \Vert _{1}.
\end{equation*}%
where $\lambda $ is a given tuning parameter. Let 
\begin{equation*}
\widehat{c}_{\left( \overline{k}\right) }\left( z\right) \equiv \varphi
_{c}\left( \left\langle \frac{\widehat{\theta }_{c,\left( j_{1}(k)\right)
,j_{2}\left( k\right) }+\widehat{\theta }_{c,\left( j_{2}(k)\right)
,j_{1}\left( k\right) }}{2},\phi (z)\right\rangle \right) .
\end{equation*}

\item Estimate $\chi \left( \eta \right) $ with $\widehat{\chi }%
_{nonlin}\equiv \frac{1}{3}\sum_{k=1}^{3}\widehat{\chi }_{nonlin}^{\left(
k\right) }$ where $\widehat{\chi }_{nonlin}^{\left( k\right) }\equiv \mathbb{%
P}_{nk}\{\Upsilon (\widehat{a}_{\left( \overline{k}\right) },\widehat{b}%
_{\left( \overline{k}\right) })\}$ and estimate the asymptotic variance of $%
\widehat{\chi }_{nonlin}$ with $\widehat{V}_{nonlin}\equiv \frac{1}{3}%
\sum_{k=1}^{3}\mathbb{P}_{nk}\left\{ \left( \Upsilon (\widehat{a}_{\left( 
\overline{k}\right) },\widehat{b}_{\left( \overline{k}\right) })-\widehat{%
\chi }_{nonlin}^{\left( k\right) }\right) ^{2}\right\} $
\end{enumerate}

\subsection{Procedure when one given link is linear and the other is
non-linear \label{Algo mixed}}

The formal algorithm for constructing the estimator, throughout denoted as $%
\widehat{\chi }_{mix},$ when, say, $\varphi _{a}\left( u\right) =u$ and,
say, $\varphi _{b}\left( u\right) $ is non-linear, is given below. As
earlier, we first provide a verbal description to facilitate its reading.
The algorithm again is based on a three way sample splitting followed by
cross-fitting strategy. The main difference with the algorithm when both
links are non-linear is that the estimator of $b$ used in the main
estimation sample is computed only once from the combined two nuisance
estimation samples using weight $w\left( Z\right) =1.$ However, we estimate $%
\theta _{a}$ as the average of two estimators of it, each being computed
from each nuisance estimation sample, using estimated weights that rely
additional preliminary estimators of $\theta _{b}$ computed from the
remaining nuisance estimation sample. For simplicity, assume $n$ is
divisible by 3.

\begin{itemize}
\item \medskip Randomly split the sample into three subsamples, designate
one as the main estimation sample and the other two as the nuisance
estimation samples. Assume working models $a\left( Z\right) = \left\langle
\theta _{a},\phi \left( Z\right) \right\rangle $ and $b\left( Z\right)
=\varphi _{b}\left( \left\langle \theta _{b},\phi \left( Z\right)
\right\rangle \right) .$

\begin{itemize}
\item[(i)] Using combined data from the first and second nuisance sample,
compute $\widehat{\theta }_{b}$ and let $\widehat{b}\left( \cdot \right)
\equiv \varphi _{b}\left( \left\langle \widehat{\theta }_{b},\phi \left(
\cdot \right) \right\rangle \right) .$

\item[(ii)] In the first nuisance estimation sample, compute $\widehat{%
\theta }_{b}^{0}$ solving the minimization problem $\left( \ref{eq:loss}%
\right) $ (for $c=b),$ using weight $w\left( Z\right) =1.$

\item[(iii)] In the second nuisance estimation sample, compute $\widehat{%
\theta }_{a}$ solving $\left( \ref{eq:loss}\right) $ (for $c=a),$ using
weight $\widehat{w}_{a}\left( Z\right) =\varphi _{b}^{\prime }\left(
\left\langle \widehat{\theta }_{b}^{0},\phi \left( Z\right) \right\rangle
\right) $ where $\varphi _{b}^{\prime }\left( u\right) =d$ $\varphi
_{b}\left( u\right) /du.$

\item[(iv)] Repeat steps (ii) and (iii), interchanging the roles of the
first and second nuisance samples, to obtain a new estimators of $\theta
_{a}.$

\item[(v)] Re-estimate $\theta _{a}$ with the average of the two previously
obtained estimators of $\theta _{a}.$ Use this final estimator of $\theta
_{a}$ to compute $\widehat{a}\left( \cdot \right) .$

\item[(vi)] Compute the one step estimator in the main estimation sample
using the $\widehat{a}\left( \cdot \right) $ and $\widehat{b}\left( \cdot
\right) $ computed in step (iv).
\end{itemize}

\item Repeat steps (i)-(vi) twice, interchanging each time one of the
nuisance samples with the main estimation sample.

\item The final estimator $\widehat{\chi }_{mix}$ of $\chi \left(
\eta\right) $ is the average of the three one step estimators.
\end{itemize}

We now give the algorithm for constructing $\widehat{\chi }_{mix}$.

Randomly split the data into three disjoint, equally sized samples, $%
\mathcal{D}_{n1}$, $\mathcal{D}_{n2}$ and $\mathcal{D}_{n3}$.

\begin{enumerate}
\item For $k=1,2,3,$ and $w_{0}\left( Z\right) =1$ compute 
\begin{align*}
& \left. \widehat{\theta }_{b,\left( \overline{k}\right) }\in \arg
\min_{\theta \in \mathbb{R}^{p}}\mathbb{P}_{\overline{nk}}\left[ Q_{b}\left(
\theta ,\phi ,w_{0}\right) \right] +\lambda \Vert \theta \Vert _{1},\right.
\\
& \left. \widehat{\theta }_{b,\left( k\right) }\in \arg \min_{\theta \in 
\mathbb{R}^{p}}\mathbb{P}_{nk}\left[ Q_{b}\left( \theta ,\phi ,w_{0}\right) %
\right] +\lambda \Vert \theta \Vert _{1}.\right.
\end{align*}%
where $\lambda $ is a given tuning parameter. Let $\widehat{b}_{\left(
k\right) }(z)\equiv \varphi _{b}\left( \langle \widehat{\theta }_{b,\left(
k\right) },\phi (z)\rangle \right) $ and $\widehat{b}_{\left( \overline{k}%
\right) }(z)\equiv \varphi _{b}\left( \langle \widehat{\theta }_{b,\left( 
\overline{k}\right) },\phi (z)\rangle \right) .$

\item For $l\in \left\{ 1,2\right\} $ and $k\in \left\{ 1,2,3\right\} ,$ let 
$\widehat{w}_{b,\left( j_{l}\left( k\right) \right) }\left( Z\right)
=\varphi _{b}^{\prime }(\langle \widehat{\theta }_{b,\left( j_{l}\left(
k\right) \right) },\phi \left( Z\right) \rangle )$ and compute 
\begin{equation*}
\widehat{\theta }_{a,\left( k\right) ,j_{l}\left( k\right) }\in \arg
\min_{\theta \in \mathbb{R}^{p}}\mathbb{P}_{nk}\{Q_{a}\left( \theta ,\phi ,%
\widehat{w}_{b,\left( j_{l}\left( k\right) \right) }\right) \}+\lambda \Vert
\theta \Vert _{1}.
\end{equation*}%
where $\lambda $ is a given tuning parameter. Let 
\begin{equation*}
\widehat{a}_{\left( \overline{k}\right) }\left( z\right) \equiv \varphi
_{a}\left( \left\langle \frac{\widehat{\theta }_{a,\left( j_{1}(k)\right)
,j_{2}\left( k\right) }+\widehat{\theta }_{a,\left( j_{2}(k)\right)
,j_{1}\left( k\right) }}{2},\phi (z)\right\rangle \right) .
\end{equation*}

\item Estimate $\chi \left( \eta \right) $ with $\widehat{\chi }_{mix}\equiv 
\frac{1}{3}\sum_{k=1}^{3}\widehat{\chi }_{mix}^{\left( k\right) }$ where $%
\widehat{\chi }_{mix}^{\left( k\right) }\equiv \mathbb{P}_{nk}\{\Upsilon (%
\widehat{a}_{\left( \overline{k}\right) },\widehat{b}_{\left( \overline{k}%
\right) })\}$ and estimate the asymptotic variance of $\widehat{\chi }_{mix}$
with $\widehat{V}_{mix}\equiv \frac{1}{3}\sum_{k=1}^{3}\mathbb{P}%
_{nk}\left\{ \left( \Upsilon (\widehat{a}_{\left( \overline{k}\right) },%
\widehat{b}_{\left( \overline{k}\right) })-\widehat{\chi }_{mix}^{\left(
k\right) }\right) ^{2}\right\} $
\end{enumerate}

\section{Models}

\label{sec:models}

In order to give precise results about the asymptotic properties of our
estimators of $\chi \left( \eta \right) $ we will need to define the
following classes of functions.

We define a sequence of classes of functions $\mathcal{G}_{n}\left( \phi
^{n},s\left( n\right) ,j,\varphi ^{n}\right) $, $n\in \mathbb{N}$. Each
class in the sequence is indexed by an $\mathbb{R}^{p}-$valued function $%
\phi \left( Z\right) =\phi ^{n}\left( Z\right) =\left( \phi _{1}^{n}\left(
Z\right) ,...,\phi _{p}^{n}\left( Z\right) \right) ^{T}$ where $p=p\left(
n\right) ,$ by an integer $s=s\left( n\right) ,$ by $j=1$ or 2 and by a,
possibly non-linear, link function $\varphi =\varphi ^{n}$ with range in a,
possibly strict, subset of $\mathbb{R}$. The convergence of our estimators
of $\chi \left( \eta \right) $ will require that 
\begin{equation*}
s\log \left( p\right) /n\rightarrow 0\text{ as }n\rightarrow \infty .
\end{equation*}%
From now on, we will suppress the dependence on $n$ from the notation. In
what follows for $\theta \in \mathbb{R}^{p},$ $\Vert \theta \Vert _{0}\equiv
\sum\limits_{j=1}^{p}I\left\{ \theta _{j}\neq 0\right\} $ and $\Vert \theta
\Vert _{2}\equiv \left( \sum_{j=1}^{p}\theta _{j}^{2}\right) ^{1/2}$ and, $k$
and $K$ are fixed constants, i.e. not dependent on $n$, that may differ in
different contexts.

\begin{definition}[Approximately generalized linear-sparse class (AGLS)]
\label{def:AGLS} The class $\mathcal{G}\left( \phi ,s,j,\varphi \right) $ is
comprised by functions $c\left( Z\right) $ such that there exists $\theta
^{\ast }\in \mathbb{R}^{p}$ and a function $r\left( Z\right) $ satisfying 
\begin{equation}
c\left( Z\right) =\varphi \left( \left\langle \theta ^{\ast },\phi \left(
Z\right) \right\rangle \right) +r\left( Z\right)  \label{eq:AGLSmain}
\end{equation}%
where $\left\Vert \theta ^{\ast }\right\Vert _{0}\leq s$ and $E_{\eta }\left[
r\left( Z\right) ^{2}\right] \leq K\left( s\log \left( p\right) /n\right) $ $%
^{j}$.
\end{definition}

In what follows we give several examples of models that can be viewed as
AGLS class sequences.

\begin{example}
\label{ex:parametric_sparse} The parametric class of functions $c\left(
Z\right) =\varphi \left( \left\langle \theta ^{\ast },\phi \left( Z\right)
\right\rangle \right) $ for $\theta ^{\ast }\in \mathbb{R}^{p}$ satisfying $%
\left\Vert \theta ^{\ast }\right\Vert _{0}\leq s$ and a given $\phi \left(
Z\right) $ of dimension $p$ such that $s\log (p)/n\rightarrow 0$ is
trivially included in $\mathcal{G}\left( \phi ,s,j,\varphi \right) $ for $%
j=1 $ or $2$ since one can take $r=0.$
\end{example}

In what follows $\alpha $ is a given constant greater than $1/2.$

\begin{example}
\label{ex:parametric_weakly_sparse} For $l\in\left\lbrace 1, 2\right\rbrace$
and $M>0$, consider the classes $\mathcal{W}_{n}\left( \phi,t, \alpha , l,
M,\varphi\right) $ of functions $c(Z)=\varphi \left( \left\langle \theta
,\phi \left( Z\right) \right\rangle \right) $ for $\theta \in \mathbb{R}^{p}$
satisfying 
\begin{equation*}
\max_{j\leq p}j^{\alpha }|\theta _{(j)}|\leq t\left( n\right) \quad \text{and%
} \quad \Vert \theta\Vert_{2}\leq M
\end{equation*}%
for a sequence $t\left( n\right) $ satisfying%
\begin{equation*}
t\left( n\right) ^{1/\alpha }\left( \frac{\log (p)}{n}\right) ^{1-1/(l\alpha
)}\underset{n\rightarrow \infty }{\rightarrow }0 \quad \text{and} \quad
t(n)^{1/\alpha} \left(\frac{n}{\log(p)}\right)^{1/(l\alpha)}\leq p .
\end{equation*}%
Note that when ${\log (p)}/{n}$ converges to 0, the last display requires $%
\alpha >1/l$ if $t\left( n\right) $ does not converge to 0. Let $\overline{l}%
=2$ if $l=1$ and $\overline{l}=1$ if $l=2$. In Proposition \ref%
{prop:approx_sparse_ex} in Appendix B we show that under regularity
assumptions on $\phi \left( Z\right) $ and on the link $\varphi ,$ $\mathcal{%
W}_{n}\left( \phi ,t,\alpha ,l, M,\varphi\right) $ is included in the $%
\mathcal{G}_{n}\left( \phi ,s,\overline{l},\varphi \right)$ AGLS class
sequence with 
\begin{equation*}
s= t(n)^{1/\alpha} \left(\frac{n}{\log(p)}\right)^{1/(l\alpha)}.
\end{equation*}%
Furthermore, $s\log \left( p\right) /n\rightarrow 0$. Interestingly, the
class $\mathcal{W}_{n}\left( \phi, t ,\alpha ,l, M,\varphi\right) $ includes
the so called `weakly sparse' parametric models of \cite{NegahbanSS}.
Specifically, for each $q\in \left( 0,2\right) $ consider the class $%
\mathcal{N}_{n}\left( \phi, t,q,M,\varphi\right) $ of functions $%
c(Z)=\varphi \left( \left\langle \theta ,\phi \left( Z\right) \right\rangle
\right) $ where $\theta \in \mathbb{R}^{p}$ satisfies 
\begin{equation*}
\sum_{j=1}^{p}|\theta _{j}|^{q}\leq t\left( n\right) \quad \text{and} \quad
\Vert \theta\Vert_{2}\leq M
\end{equation*}%
for a sequence $t\left( n\right) $ satisfying%
\begin{equation*}
t\left( n\right) \left( \frac{\log (p)}{n}\right) ^{1-q/2}\underset{%
n\rightarrow \infty }{\rightarrow }0 \quad \text{and} \quad t(n) \left(\frac{%
n}{\log(p)}\right)^{q/2}\leq p .
\end{equation*}%
\cite{NegahbanTR} and \cite{NegahbanSS} studied $\ell_1$ regularized
estimation of $\theta $ under class $\mathcal{N}_{n}\left( \phi,
t,q,M,\varphi\right) $ for a regression function $c\left( Z\right) =E\left(
Y|Z\right)$, but these authors required that $q\in \left( 0,1\right)$. In
Proposition \ref{prop:approx_sparse_ex} we show that $\mathcal{N}_{n}\left(
\phi,t ,q,M,\varphi\right) $ is included in $\mathcal{W}_{n}\left( \phi,
t^{1/q},1/q,2,M,\varphi\right) $ \ for any $q\in \left( 0,2\right)$. In
fact, in Theorem \ref{theo:rate_main} in Appendix A we extend the
convergence rate results obtained by \cite{NegahbanTR} for $\ell_1$
regularized estimation of high-dimensional GLMs to the case $q\in(0,2)$.
\end{example}

\begin{example}
\label{ex:nonparam} Given $\{\phi _{j}\}_{j\in \mathbb{N}}$ an orthonormal
basis of $L^{2}([-1,1]^{d})$ with $d$ fixed, for $l=1$ and $2$ consider the
class $\mathcal{S}_{n}\left( \phi ,t,\alpha ,l\right) $ of functions such
that there exists $p$ and a permutation $\left\{ \phi _{\pi \left( j\right)
}\right\} $ of the first $p=p\left( n\right) $ elements of the basis such
that 
\begin{equation*}
c(Z)=\sum\limits_{j=1}^{p}\theta _{j}\phi _{\pi \left( j\right)
}(Z)+\sum\limits_{j=p+1}^{\infty }\theta _{j}\phi _{j}(Z)\quad \text{ }
\end{equation*}%
where $|\theta _{j}|j^{\alpha }\leq t\left( n\right) $ for all $j$ and $t(n)$
is such that%
\begin{equation*}
t\left( n\right) ^{1/\alpha }\left( \frac{\log (p)}{n}\right) ^{1-1/(l\alpha
)}\underset{n\rightarrow \infty }{\rightarrow }0\quad \text{and}\quad
t(n)^{1/\alpha }\left( \frac{n}{\log (p)}\right) ^{1/(l\alpha )}\leq p.
\end{equation*}%
As in the preceding example, when ${log(p)}/{n}$ converges to 0, \ the last
display requires $\alpha >1/l$ if $t\left( n\right) $ does not converge to
0. Let $\overline{l}=2$ if $l=1$ and $\overline{l}=1$ if $l=2$. In
Proposition \ref{prop:approx_sparse_ex} in Appendix B we show that when the
density of $Z$ is bounded away from infinity, $\mathcal{S}_{n}\left( \phi
,t,\alpha ,l\right) $ is included in the $\mathcal{G}_{n}\left( \phi
^{\left( n\right) },s,\overline{l},id\right) $ AGLS class sequence with $%
\phi ^{\left( n\right) }=\left( \phi _{1},...,\phi _{p\left( n\right)
}\right) $ and 
\begin{equation*}
s=t(n)^{1/\alpha }\left( \frac{n}{\log (p)}\right) ^{1/(l\alpha )}.
\end{equation*}%
Furthermore, $s\log \left( p\right) /n\rightarrow 0$. 

 Suppose $c\left( Z\right) $ stands for $%
E\left( Y|Z\right) $ for some outcome $Y$ and that one estimates $c\left(
Z\right) $ by the series estimator $\widehat{c}\left( Z\right) $ obtained by
least squares regression on the first $k\left( n\right) <n$ elements $\phi
_{1}(Z),...,\phi _{k\left( n\right) }(Z)$ of the basis. Interestingly, as
pointed out by \cite{BelloniChernoSparse}, the class $\mathcal{S}_{n}\left(
\phi ,t,\alpha ,l\right) ,l=1$ or 2, is so large that one can find for each $%
n$ a function $c_{n}\left( Z\right) \in \mathcal{S}_{n}\left( \phi ,\alpha
,k\right) $ such that $E\left[ \left\{ c_{n}\left( Z\right) -\widehat{c}%
_{n}\left( Z\right) \right\} ^{2}\right] $ converges to non-zero constant.
Specifically, let $p\left( n\right) =2n$ and consider the function $%
c_{n}\left( Z\right) =\sum_{j=1}^{\infty }\nu _{j}^{n}\phi _{j}\left(
Z\right) $ where $\nu _{1}^{n}=...=\nu _{n}^{n}=0$ and $\nu
_{j}^{n}=1/\left( j-n\right) ^{\alpha }$ for $j>n.$ Clearly $c_{n}\left(
Z\right) $ is in $\mathcal{S}_{n}\left( \phi ,t,\alpha ,l\right) $ but the
least squares estimator $\widehat{c}\left( Z\right) $ on the first $k\left(
n\right) <n$ basis elements converges to the function identically equal to 0.
\end{example}

\begin{example}\label{ex:holder}
Consider the class of
functions of $Z\in \mathbb{R}^{d}$, where $d$ does not change with $n$,
which lie in a Holder smoothness ball $H\left( \beta ,C\right) $ and admit a
periodic extension. A function $f$ lies in the Holder ball $H\left( \beta
,C\right) $, with exponent $\beta >0$ and radius $C>0,$ if and only if $f$ is
bounded in supremum norm by $C$, all partial derivatives of $f$ up to order $%
\lfloor \beta \rfloor $ exist, and all partial derivatives of order $\lfloor
\beta \rfloor $ are Lipschitz with exponent $\beta -\lfloor \beta \rfloor $
and constant $C$. This class is included in 
$\mathcal{G}\left( \phi^{n} ,s,j=1,\varphi= id \right)$, where $\phi^{n}$ has the first $p$
terms in the trigonometric basis  and $s \asymp \left(n/\log(p)\right)^{1/(2(\beta/d)+1)}$.
This follows easily from the fact if $c(z) \in H\left( \beta ,C\right) $ and $c(z)$ admits a periodic extension then there exists $\theta^{\ast} \in \mathbb{R}^{p}$ with $\Vert \theta^{\ast}\Vert_{0}\leq s$ such that
$E^{1/2}\left[ (c(Z) - \langle \theta^{\ast}, \phi(Z)\rangle)^{2} \right]\leq L s^{-\beta/d}$ for a constant $L>0$ (see Section 3.1 in \cite{BelloniChernoSeries} for example).
\end{example}

In Examples \ref{ex:parametric_weakly_sparse} and \ref{ex:nonparam} a larger 
$\alpha$ correspond to a faster decay to 0 of the coefficients $\theta _{j}.$
In Example \ref{ex:nonparam}\thinspace\ this is connected to the smoothness
of the functions in the class. Thus, in the latter example this implies that 
$\mathcal{G}_{n}\left( \phi ,s,1,id \right) $ contains functions that are
less smooth than those in $\mathcal{G}_{n}\left( \phi ,s,2,id \right) $.

\subsection{Heuristic motivation for our proposal}

\label{sec:heuristic}

Before giving the rigorous results on the model and rate DR properties of
our estimators, we will discuss here the heuristics of these results. We
will discuss the properties of $\widehat{\chi}_{nonlin}$, since it
is the most involved.

Recall that functionals in the BIF class are such that 
\begin{equation*}
\left( \chi +\chi ^{1}\right) _{\left( a,b\right) }=S_{ab}ab+m_{a}\left(
o,a\right) +m_{b}\left( o,b\right) +S_{0}
\end{equation*}%
where $h\rightarrow m_{a}\left( o,h\right) $ and $h\rightarrow m_{b}\left(
o,h\right) $ are linear maps. Then the map $t\rightarrow $ $\left( \chi
+\chi ^{1}\right) _{\left( a+t\left( \widetilde{a}-a\right) ,b+t\left( 
\widetilde{b}-b\right) \right) }$ is quadratic for any fixed $a,b,\widetilde{%
a}$ and $\widetilde{b}$. Consider a generic estimator $\widetilde{\chi }%
\equiv \mathbb{P}_{m}\left( \chi +\chi ^{1}\right) _{\left( \widetilde{a},%
\widetilde{b}\right) }$ where $\mathbb{P}_{m}$ is the empirical mean over a
subsample $\mathcal{D}_{m}$ and $\widetilde{a}$ and $\widetilde{b}$ are
estimators computed using data from $\mathcal{D}_{\overline{m}}=$ $\mathcal{D%
}-\mathcal{D}_{m}$. Simple algebra gives for any $a$ and $b^{1}$, 
\begin{equation}
\widetilde{\chi }-\chi \left( \eta \right) =N_{m}+\Gamma _{a,m}+\Gamma
_{b^{1},m}+\Gamma _{ab^{1},m}  \label{expansion}
\end{equation}%
where 
\begin{equation*}
N_{m}\equiv \mathbb{P}_{m}\left[ \left( \chi +\chi ^{1}\right) _{\left(
a,b^{1}\right) }-\chi \left( \eta \right) \right]
\end{equation*}%
\begin{eqnarray*}
\Gamma _{a,m} &\equiv &\mathbb{P}_{m}\left. \frac{d}{dt}\left( \chi +\chi
^{1}\right) _{\left( a+t\left( \widetilde{a}-a\right) ,b^{1}\right)
}\right\vert _{t=0} \\
&=&\mathbb{P}_{m}\left[ S_{ab}\left( \widetilde{a}-a\right) \left( Z\right)
b^{1}\left( Z\right) +m_{a}\left( O,\widetilde{a}-a\right) \right]
\end{eqnarray*}%
\begin{eqnarray*}
\Gamma _{b^{1},m} &\equiv &\mathbb{P}_{m}\left. \frac{d}{dt}\left( \chi
+\chi ^{1}\right) _{\left( a,b^{1}+t\left( \widetilde{b}-b^{1}\right)
\right) }\right\vert _{t=0} \\
&=&\mathbb{P}_{m}\left[ S_{ab}\left( \widetilde{b}-b^{1}\right) \left(
Z\right) a\left( Z\right) +m_{b}\left( O,\widetilde{b}-b^{1}\right) \right]
\end{eqnarray*}%
\begin{eqnarray*}
\Gamma _{ab^{1},m} &\equiv &\mathbb{P}_{m}\left. \frac{d^{2}}{dt^{2}}\left(
\chi +\chi ^{1}\right) _{\left( a+t\left( \widetilde{a}-a\right)
,b^{1}+t\left( \widetilde{b}-b^{1}\right) \right) }\right\vert _{t=0} \\
&=&\mathbb{P}_{m}\left[ S_{ab}\left( \widetilde{a}-a\right) \left( Z\right)
\left( \widetilde{b}-b^{1}\right) \left( Z\right) \right] .
\end{eqnarray*}%
Note that the term $N_{m}$ is an average of mean zero i.i.d.. random
variables and therefore is approximately normal for large $m$, even if $%
b^{1} $ is not equal to $b$ because by equation \eqref{bias1}, 
$E_{\eta }\left[ \left( \chi +\chi ^{1}\right) _{\left( a,b^{1}\right) }%
\right] =\chi \left( \eta \right) $ even if $b^{1}\not=b$.

\subsubsection*{Heuristics for rate double robustness}

Suppose $\widetilde{a}$ and $\widetilde{b}$ converge to the true functions $%
a $ and $b.$ We analyze the behavior of $\widetilde{\chi }-\chi \left( \eta
\right) $ by setting in the preceding expansion $\left( \ref{expansion}%
\right) $, $b^{1}=b$. The terms $\Gamma _{a,m}$ and $\Gamma _{b,m}$ would
ordinarily be analyzed using tools from empirical processes theory under
Donsker type assumptions on the size of the classes where $a$ and $b$ lie.
However, these assumptions would put a bound on the complexity of these
functions which would then defeat the purpose of rate double robustness,
namely trading off the complexity of one function with the simplicity of the
other. We avoid imposing Donsker assumptions by employing sample-splitting,
i.e. we use estimators $\widetilde{a}$ and $\widetilde{b}$ computed from $%
\mathcal{D}_{\overline{m}}.$ Now, because $\widetilde{a}$ and $\widetilde{b}$
depend on data $\mathcal{D}_{\overline{m}}$ which is independent from $%
\mathcal{D}_{m},$ then invoking Proposition \ref{prop:DR functional}, $\sqrt{%
m}\Gamma _{a,m}$ and $\sqrt{m}\Gamma _{b,m}$ have mean zero given $\mathcal{D%
}_{\overline{m}}.$ We then show that$\sqrt{m}\Gamma _{a,m}$ and $\sqrt{m}%
\Gamma _{b,m}$ are $o_{p}\left( 1\right) $ by checking that, under our
regularity conditions, $\sqrt{m}\Gamma _{a,m}$ and $\sqrt{m}\Gamma _{b,m}$
have, conditionally on $\mathcal{D}_{\overline{m}},$ variance that converges
that converges to 0. Finally, a Cauchy-Schwartz type argument gives that $%
\Gamma _{ab,m}$ converges at a rate equal to the product of the rates of
convergence of $\widetilde{a}$ and $\widetilde{b},$ thus yielding the rate
double robustness property.

Notice that this analysis is valid for any estimators $\widetilde{a}$ and $%
\widetilde{b}$ computed from data $\mathcal{D}_{m},$ not just for the $\ell
_{1}$ regularized estimators in our proposed algorithms. Notice also that we
use estimators $\widetilde{a}$ and $\widetilde{b}$ from a two stage weighted 
$\ell _{1}$ regularized estimation procedure with data dependent weights.
Because we further employ sample splitting of $\mathcal{D}_{\overline{m}}$
to estimate the weights from a separate sample from the one where we compute 
$\widetilde{a}$ and $\widetilde{b}$, we can show that the fact that we use
data dependent weights does not affect the rate of convergence of $%
\widetilde{a}$ and $\widetilde{b}$ to $a$ and $b.$ If we had been only
interested in estimators with the rate double robustness property we could
have used weights $w(Z)=1$ and we would have not needed to sample split $%
\mathcal{D}_{\overline{m}}$ to compute the estimators of $a$ and $b.$ The
use of data dependent weights is needed to further ensure the model double
robustness property.

\subsubsection*{Heuristics for model double robustness}

Suppose now that the analyst has correctly guessed an AGLS model for $a$ but
incorrectly guessed an AGLS\ model for $b,$ and for simplicity assume that $%
\phi _{a}=\phi _{b}=\phi .$ Then, one would expect that for any reasonable
estimators $\widetilde{a}$ and $\widetilde{b},$ $\widetilde{a}$ converges to 
$a$ but $\widetilde{b}$ converges to $b^{1}\not=b.$ In the expansion %
\eqref{expansion}, as indicated earlier, the term $N_{m}$ is an average of
mean zero i.i.d.. random variables and therefore is approximately normal for
large $m.$ Also, by the same arguments as for the case of both models
correctly specified, $\sqrt{m}\Gamma _{b^{1},m}$ should be $o_{p}\left(
1\right) .$ However, the term $\sqrt{m}\Gamma _{a,m}$ is problematic because
unlike the term $\sqrt{m}\Gamma _{b^{1},m}$ we cannot invoke Proposition \ref%
{prop:DR functional} to argue that it is an average of terms that have mean
zero given $\mathcal{D}_{\overline{m}}.$ This is because 
\begin{equation}
\Gamma _{a,m}=\mathbb{P}_{m}\left[ S_{ab}\left( \widetilde{a}-a\right)
\left( Z\right) b^{1}\left( Z\right) +m_{a}\left( O,\widetilde{a}-a\right) %
\right]  \label{eq:gammaA}
\end{equation}%
and Proposition \ref{prop:DR functional} tells us that $S_{ab}h\left(
Z\right) b^{1}\left( Z\right) +m_{a}\left( O,h\right) $ has mean zero for
any $h\left( Z\right) $ -and in particular for $h=\widetilde{a}-a$ - only if 
$b^{1}$ is equal to $b,$ which is not the case. Nevertheless, we can
overcome this difficulty if we cleverly choose our estimator of $b$.
Specifically, if $\widetilde{b}\left( Z\right) $ is of the form $\varphi
_{b}\left( \left\langle \widetilde{\theta }_{b},\phi \left( Z\right)
\right\rangle \right) $ where $\widetilde{\theta }_{b}$ solves the
minimization problem $\left( \ref{eq:loss}\right) $ (with $c=b)$ for a given
weight function $w^{1}\left( Z\right) ,$ then under regularity conditions, $%
\widetilde{b}\left( Z\right) $ converges to $b^{1}\left( Z\right) =\varphi
_{b}\left( \left\langle \theta _{b}^{1},\phi \left( Z\right) \right\rangle
\right) $ where 
\begin{equation}
\theta _{b}^{1}\in \arg \min_{\theta \in \mathbb{R}^{p}}E_{\eta }\left[
S_{ab}w^{1}\left( Z\right) \psi _{b}\left( \langle \theta ,\phi \left(
Z\right) \rangle \right) +\langle \theta ,m_{a}(O,w^{1}\cdot \phi )\rangle %
\right]  \label{eq:thetab1}
\end{equation}%
provided $b^{1}\in \mathcal{G}\left( \phi ,s_{b},j=1,\varphi _{b}\right) $
with $s_{b}\log \left( p\right) /m=o\left( 1\right) .$ So, even though $%
S_{ab}h\left( Z\right) b^{1}\left( Z\right) +m_{a}\left( O,h\right) $ does
not have mean zero for all $h\left( Z\right) ,$ it does have mean zero for $%
h\left( Z\right) =w^{1}\left( Z\right) \phi _{j}\left( Z\right) $ for all $%
j=1,\dots ,p.$

On the other hand, if $\widetilde{a}\left( Z\right) $ is of the form $%
\varphi _{a}\left( \left\langle \widetilde{\theta }_{a},\phi \left( Z\right)
\right\rangle \right) $ where $\widetilde{\theta }_{a}$ is the estimator of $%
\theta _{a}$ computed in step (iv) of the verbal description of the
algorithm for non-linear links, then expanding $\widetilde{a}$ around $a$ in 
$\left( \ref{eq:gammaA}\right) ,$ we have that 
\begin{eqnarray*}
\left\vert \Gamma _{a,m}\right\vert &\approx &\left\vert \mathbb{P}_{m}\left[
S_{ab}\varphi _{a}^{\prime }\left( \left\langle \theta _{a}^{\ast },\phi
\left( Z\right) \right\rangle \right) \left\langle \widetilde{\theta }%
_{a}-\theta _{a}^{\ast },\phi \left( Z\right) \right\rangle b^{1}\left(
Z\right) +m_{a}\left( O,\varphi _{a}^{\prime }\left( \left\langle \theta
_{a}^{\ast },\phi \left( Z\right) \right\rangle \right) \left\langle 
\widetilde{\theta }_{a}-\theta _{a}^{\ast },\phi \left( Z\right)
\right\rangle \right) \right] \right\vert \\
&=&\left\vert \left\langle \widetilde{\theta }_{a}-\theta _{a}^{\ast },%
\mathbb{P}_{m}\left[ S_{ab}\varphi _{a}^{\prime }\left( \left\langle \theta
_{a}^{\ast },\phi \left( Z\right) \right\rangle \right) \phi \left( Z\right)
b^{1}\left( Z\right) +m_{a}\left( O,\varphi _{a}^{\prime }\left(
\left\langle \theta _{a}^{\ast },\phi \left( Z\right) \right\rangle \right)
\phi \left( Z\right) \right) \right] \right\rangle \right\vert \\
&\leq &\left\Vert \widetilde{\theta }_{a}-\theta _{a}^{\ast }\right\Vert
_{1}\left\Vert \mathbb{M}_{m}\right\Vert _{\infty }
\end{eqnarray*}%
where%
\begin{equation*}
\mathbb{M}_{m}\equiv \mathbb{P}_{m}\left[ S_{ab}\varphi _{a}^{\prime }\left(
\left\langle \theta _{a}^{\ast },\phi \left( Z\right) \right\rangle \right)
\phi \left( Z\right) b^{1}\left( Z\right) +m_{a}\left( O,\varphi
_{a}^{\prime }\left( \left\langle \theta _{a}^{\ast },\phi \left( Z\right)
\right\rangle \right) \phi \left( Z\right) \right) \right]
\end{equation*}%
with $\theta _{a}^{\ast }$ the parameter associated with $a$ in the AGLS
model.

Suppose in $\left( \ref{eq:thetab1}\right) $ we had used $w^{1}(Z)=\varphi
_{a}^{\prime }\left( \left\langle \theta _{a}^{\ast },\phi \left( Z\right)
\right\rangle \right) .$ Then $\mathbb{M}_{m}$ would be a $p\times 1$ vector
whose entries are each a sample average of mean zero random variables. Thus,
under moment assumptions, Nemirovski's inequality (see Lemma 14.24 from \cite%
{Buhlmann-book}) would yield $\left\Vert \mathbb{M}_{m}\right\Vert _{\infty
}=O_{p}\left( \sqrt{\log \left( p\right) /m}\right) $. On the other hand,
standard arguments for $\ell _{1}$ regularized regression yield $\left\Vert 
\widetilde{\theta }_{a}-\theta _{a}^{\ast }\right\Vert _{1}=O_{p}\left( s_{a}%
\sqrt{{\log \left( p\right) }/{m}}\right) $ where $s_{a}$ is the sparsity $s$
parameter of the AGLS model for $a$. Therefore, we would obtain 
\begin{equation*}
\sqrt{m}\Gamma _{a,m}=\sqrt{m}O_{p}\left( s_{a}\sqrt{\frac{\log \left(
p\right) }{m}}\right) O_{p}\left( \sqrt{\frac{\log \left( p\right) }{m}}%
\right) =O_{p}\left( \frac{s_{a}\log (p)}{\sqrt{m}}\right) .
\end{equation*}%
Consequently, $\sqrt{m}\Gamma _{a,m}$ would be $o_{p}\left( 1\right) $ if $%
s_{a}\log \left( p\right) =o\left( \sqrt{m}\right) ,$ i.e. if the AGLS model
for $a$ is ultra sparse. Note that this analysis only requires ultra
sparsity of the correctly modeled function $a.$ It does not require ultra
sparsity of the probability limit $b^{1}$ of the estimator $\widetilde{b}$
computed assuming an incorrect model.

Unfortunately, we cannot use the ideal weight $w^{1}(Z)=\varphi _{a}^{\prime
}\left( \left\langle \theta _{a}^{\ast },\phi \left( Z\right) \right\rangle
\right) $ because $\theta _{a}^{\ast }$ is unknown. Instead, we estimate
this ideal weight by replacing $\theta _{a}^{\ast }$ with a first stage
estimator of $\theta _{a}^{\ast }$ computed by solving the minimization
problem $\left( \ref{eq:loss}\right) $ (with $c=a),$ using weight $w\left(
Z\right) =1.$ Since the model for $a$ is correct, then any choice of weight,
and in particular $w\left( Z\right) =1,$ yields a consistent estimator of $%
\theta _{a}^{\ast }.$ Thus, this strategy yields a consistent estimator of
the ideal weight. Because we employ sample splitting of $\mathcal{D}_{%
\overline{m}}$ to estimate the ideal weights from a separate sample from the
one where we compute $\widetilde{b}$, we avoid imposing conditions on the
complexity of the ideal weight function. However, because the very
definition of $b^{1}$ depends on the weights being the ideal ones, small
perturbations from the ideal weight in the estimated weights used to compute 
$\widetilde{b}$ affect the rate of convergence of $\widetilde{b}$ to $b^{1}:$
whereas the rate would be $\sqrt{s_{b}\log \left( p\right) /m}$ if the ideal
weights had been used, the rate becomes $\sqrt{\max \left(
s_{a},s_{b}\right) \log \left( p\right) /m}$ when the estimated weights are
used.

By the Cauchy-Schwartz inequality, the term $\Gamma _{ab^{1},m}$ is bounded
by the product of the convergence rates of $\widetilde{a}$ to $a$ and $%
\widetilde{b}$ to $b^{1},$ thus yielding that $\sqrt{m}\left\{ \widetilde{%
\chi }-\chi \left( \eta \right) \right\} $ is asymptotically normal under
regularity conditions, provided%
\begin{equation}
\sqrt{s_{a}\max \left( s_{a},s_{b}\right)} \log \left( p\right)=o\left(
m^{1/2}\right)  \label{eq:rate1}
\end{equation}

The preceding analysis had assumed that the analyst has correctly guessed an
AGLS model for $a$ but incorrectly guessed an AGLS\ model for $b.$ If the
reverse had happened, an identical argument would have given that 
\begin{equation}
\sqrt{s_{b}\max \left( s_{a},s_{b}\right)} \log \left( p\right)=o\left(
m^{1/2}\right)  \label{eq:rate2}
\end{equation}%
as the condition for asymptotic normality of $\sqrt{m}\left\{ \widetilde{%
\chi }-\chi \left( \eta \right) \right\} $ because our algorithm treats $a$
and $b$ symmetrically. This then shows that $\widetilde{\chi }$ has the
model double robust property so long as corresponding rate condition $\left( %
\ref{eq:rate1}\right) $ or $\left( \ref{eq:rate2}\right) $ holds, depending
on which of the two nuisance functions has been correctly modeled.

\section{Asymptotic results}

\label{sec:asym_results_chi} In this section, we prove that, under
regularity assumptions, the estimators defined by Algorithm \ref{Algo both
ALS} and \ref{Algo both GALS} are simultaneously rate and model doubly
robust. Throughout, we assume that the target parameter $\chi(\eta)$ belongs
to the BIF class. We will need the following additional notation.

\subsection*{Notation}

We will say that a random variable $W$ is \emph{sub-Gaussian} if 
\begin{equation}
\Vert W\Vert _{\psi _{2}}=\sup_{k\in \mathbb{N}}\frac{\left( E\left(
|W|^{k}\right) \right) ^{1/k}}{\sqrt{k}}<\infty .  \notag
\end{equation}%
We will say that $W$ is \emph{sub-Exponential} if 
\begin{equation}
\Vert W\Vert _{\psi _{1}}=\sup_{k\in \mathbb{N}}\frac{\left( E\left(
|W|^{k}\right) \right) ^{1/k}}{k}<\infty .  \notag
\end{equation}%
See \cite{compressed-book} for the numerous equivalent ways of defining
sub-Gaussian and sub-Exponential random variables. For $r\geq 1$ and $\theta
\in \mathbb{R}^{p}$ we let $\Vert \theta \Vert _{r}=\left(
\sum_{j=1}^{p}|\theta _{j}|^{r}\right) ^{1/r}$, $\Vert \theta \Vert _{\infty
}=\max\limits_{j\leq p}|\theta _{j}|$ and $\Vert \theta \Vert
_{0}=\sum\limits_{j=1}^{p}I\left\{ \theta _{j}\neq 0\right\} $. Moreover we
let $\vert \theta_{(p)}\vert\leq \vert \theta_{(p-1)}\vert,\dots,\vert
\theta_{(1)}\vert$ be the sorted absolute values of the entries of $\theta$.
Given $T\subset \left\{ 1,\dots ,p\right\} $, we let $\theta _{T}$ be the
vector with coordinate $j$ equal to $\theta _{j}$ when $j\in T$ and zero
otherwise. $\mathbb{R}_{+}$ will stand for the non-negative reals. If $M$ is
a square matrix, $\lambda _{min}(M)$ and $\lambda _{max}(M)$ stand for the
smallest and largest (in absolute value) eigenvalues of $M$. Moreover, we
let $\Vert v\Vert _{M}=(v^{\prime }Mv)^{1/2}$ for any conformable vector $v$%
. $\mathbb{P}_{n}$ will stand for the sample average operator and $\mathbb{G}%
_{n}\left( \cdot \right) \equiv \sqrt{n}\left\{ \mathbb{P}_{n}\left( \cdot
\right) -E_{\eta }\left( \cdot \right) \right\} $. 

Let 
\begin{equation*}
\widehat{\Sigma }_{2}=\mathbb{P}_{n}\left\{ S_{ab}\phi \left( Z\right) \phi
\left( Z\right) ^{\prime }\right\} ,\quad \Sigma _{2}=E\left\{ S_{ab}\phi
\left( Z\right) \phi \left( Z\right) ^{\prime }\right\} ,
\end{equation*}%
\begin{equation*}
\widehat{\Sigma }_{1}=\mathbb{P}_{n}\left\{ \phi \left( Z\right) \phi \left(
Z\right) ^{\prime }\right\} \quad \text{ and }\quad \Sigma _{1}=E\left\{
\phi \left( Z\right) \phi \left( Z\right) ^{\prime }\right\} .
\end{equation*}%
For each $T\subset \left\{ 1,...,p\right\} $ define the minimal and maximal
restricted sparse eigenvalues \citep{Bickel-Lasso} relative to $T$ of a
matrix $M\in \mathbb{R}^{p\times p}$ respectively as 
\begin{equation*}
\kappa _{l}(M,m,T)=\min\limits_{\Vert \Delta _{T^{c}}\Vert _{0}\leq m,\Delta
\neq 0}\frac{\Vert \Delta \Vert _{M}^{2}}{\Vert \Delta \Vert _{2}^{2}},\quad
\kappa _{u}(M,m,T)=\max\limits_{\Vert \Delta _{T^{c}}\Vert _{0}\leq m,\Delta
\neq 0}\frac{\Vert \Delta \Vert _{M}^{2}}{\Vert \Delta \Vert _{2}^{2}}.
\end{equation*}%
We will say that two sequences a two sequences $x_{n}$ and $y_{n}$ are $%
x_{n}\asymp y_{n}$ if $x_{n}=O(y_{n})$ and $y_{n}=O(x_{n})$. We will use $%
\rightsquigarrow $ to denote convergence in law. Finally, recall that 
\begin{equation*}
\Upsilon \left( a,b\right) \equiv S_{ab}a\left( Z\right) b\left( Z\right)
+m_{a}\left( O,a\right) +m_{b}\left( O,b\right) +S_{0}.
\end{equation*}

\subsection{Asymptotic results for the estimator $\widehat{\protect\chi }%
_{lin}$}

\label{sec:asym_lin}

In Sections \ref{sec:dr_lin} and \ref{sec:mod_dr_lin} we state theorems that
establish the rate and model double robustness properties of the estimator $%
\widehat{\chi }_{lin}$. The rate double robustness property is established
under conditions Lin.L, Lin.E and Lin.V below. We separate these conditions
into those that are needed to analyze the convergence of the $\ell_1$%
-regularized estimators of $a$ and $b$ (Condition Lin.L), those that are
additionally needed to show the asymptotic normality of $\widehat{\chi }%
_{lin}$ (Condition Lin.E) and a last additional condition (Condition Lin.V)
which is needed to show the convergence of the variance estimator. The model
double robustness property is established, essentially, under conditions
Lin.L, Lin.E and Lin.V for the nuisance function whose model was correctly
specified and for the probability limit of estimator of the nuisance
function that was incorrectly modelled.

\subsubsection{Rate double robustness for the estimator $\widehat{\protect%
\chi }_{lin}$}

\label{sec:dr_lin}

\begin{condition}[Condition Lin.L]
\label{cond:Lin.L} There exists fixed constants $0<k<K$ such that for $%
\left( c=a,\overline{c}=b\right) $ and for $\left( c=b,\overline{c}=a\right) 
$ the following conditions hold

\begin{itemize}
\item \textbf{(Lin.L.1}) $c\left( Z\right) \in \mathcal{G}\left( \phi
,s_{c},j=1,\varphi =id\right) $ with associated parameter value denoted as $%
\theta _{c}^{\ast }$ whose support is denoted with $S_{c}$. Furthermore, $%
s_{c}\log \left( p\right) /n\rightarrow 0$.

\item \textbf{(Lin.L.2}) 
\begin{equation*}
E_{\eta }\left[ \max_{1\leq j\leq p}\sqrt{\mathbb{P}_{n}\left[ \left\{
S_{ab}c\left( Z\right) \phi _{j}\left( Z\right) +m_{\overline{c}}\left(
O,\phi _{j}\right) \right\} ^{2}\right] }\right] \leq K
\end{equation*}

\item \textbf{(Lin.L.3)} For sufficiently large $n$ 
\begin{equation*}
k\leq \kappa _{l}({\Sigma }_{1},\left\lceil s_{c}\log (n)\right\rceil
,S_{c})\leq \kappa _{u}({\Sigma }_{1},\left\lceil s_{c}\log (n)\right\rceil
,S_{c})\leq K.
\end{equation*}

\item \textbf{(Lin.L.4)} With probability tending to 1, 
\begin{equation*}
k\leq \kappa _{l}(\widehat{\Sigma }_{2},\left\lceil s_{c}\log
(n)/2\right\rceil ,S_{c})\leq \kappa _{u}(\widehat{\Sigma }_{2},\left\lceil
s_{c}\log (n)/2\right\rceil ,S_{c})\leq K.
\end{equation*}

\item \textbf{(Lin.L.5)} $k\leq E_{\eta }\left( S_{ab}|Z\right) \leq K$ with
probability one.
\end{itemize}
\end{condition}

Condition Lin.L.2, holds in particular, when there exists a statistic $S_{%
\overline{c}}$ such that$\ m_{\overline{c}}\left( O,h\right) =S_{\overline{c}%
}h,$ for all $h,$ $\ \max_{1\leq j\leq p}\left\Vert \phi _{j}\right\Vert
_{\infty }\leq K$ and $E_{\eta }\left[ \left\{ \left\vert S_{ab}\right\vert
\left\vert c\left( Z\right) \right\vert +\left\vert S_{\overline{c}%
}\right\vert \right\} ^{2}\right] \leq K$. Condition Lin.L.2 can be replaced
by 
\begin{equation*}
\max_{1\leq j\leq p}\left\Vert S_{ab}c\left( Z\right) \phi _{j}\left(
Z\right) +m_{\overline{c}}\left( O,\phi _{j}\right) \right\Vert _{\psi
_{1}}\leq K,
\end{equation*}
which essentially requests that all random variables $S_{ab}c\left( Z\right)
\phi _{j}\left( Z\right) +m_{\overline{c}}\left( O,\phi _{j}\right)
,j=1,\dots,p$ have tails decaying at least as fast as the tails of an
exponential random variable.

Condition Lin.L.3 holds if all sub-matrices of ${\Sigma }_{1}$ of size $%
\left\lceil s_{c}\log (n)\right\rceil $ are well-conditioned. This is turn
holds if 
\begin{equation*}
k\leq \lambda _{min}(\Sigma _{1})\leq \lambda _{max}(\Sigma _{1})\leq K.
\end{equation*}%
An example of a $\Sigma _{1}$ satisfying the last display is the Toeplitz
matrix defined as $\Sigma _{1,k,l}=\rho ^{|k-l|}$ for some fixed $\rho \in
(0,1)$ not depending on $n$.

Condition Lin.L.4 is similar to Lin.L.3 but for the weighted sample
covariance matrix $\widehat{\Sigma}_{2}$. It can be shown that Condition
Lin.L.4 holds if Conditions Lin.L.3 and Lin.L.5 hold, and either: $%
\max_{1\leq j\leq p}\left\Vert \phi _{j}\right\Vert _{\infty }\leq K$ for
some $K>0$ or $\phi \left( Z\right) $ is multivariate normal. See Lemmas 1
and 2 from \cite{BelloniCherno11}. \ In addition, Lemma 2 in \cite%
{BelloniCherno11} implies that Conditions Lin.L.3 and Lin.L.4 hold if $%
\left\{ \phi _{j}:j\in \mathbb{N}\right\} $ is an orthonormal basis of $%
L^{2}([0,1]^{d})$ satisfying $\max_{1\leq j\leq p}\left\Vert \phi
_{j}\right\Vert _{\infty }\leq K$ for some $K>0$, e.g. the basis is the
trigonometric basis, Condition Lin.L.5 holds and the density of $Z$ is
uniformly bounded away from zero and infinity.

\begin{condition}[Condition Lin.E]
\label{cond:Lin.E} $\ $

\begin{itemize}
\item \textbf{(Lin.E.1)} 
\begin{equation*}
E_{\eta }\left\{ \left[ S_{ab}a\left( b^{\prime }-b\right) +m_{b}\left(
O,b^{\prime }\right) -m_{b}\left( O,b\right) \right] ^{2}\right\}
\rightarrow 0 \quad \text{as} \quad E_{\eta }\left[ \left( b^{\prime
}-b\right) ^{2}\right] \rightarrow 0
\end{equation*}
and 
\begin{equation*}
E_{\eta }\left\{ \left[ S_{ab}b\left( a^{\prime }-a\right) +m_{a}\left(
O,a^{\prime }\right) -m_{a}\left( O,a\right) \right] ^{2}\right\}
\rightarrow 0 \quad \text{as} \quad E_{\eta }\left[ \left( a^{\prime
}-a\right) ^{2}\right] \rightarrow 0.
\end{equation*}

\item \textbf{(Lin.E.2)} There exists fixed constants $0<k<K$ such that

\begin{itemize}
\item[a)] $k\leq {E\left\{ \left( \chi _{\eta }^{1}\right) ^{2}\right\} }$
and $E\left\{ \left\vert \chi _{\eta }^{1}\right\vert ^{3}\right\} \leq K$.

\item[b)] ${E\left\{ \left( \chi _{\eta }^{1}\right) ^{4}\right\} }\leq K$.
\end{itemize}
\end{itemize}
\end{condition}

Condition Lin.E.1 holds in particular if there exists statistics $S_{a}$ and 
$S_{b}$ such that $m_{a}\left( O,h\right) =S_{a}h$ and $m_{b}\left(
O,h\right) =S_{b}h$ and for some $K>0,$ $E_{\eta }\left\{ \left(
S_{ab}a+S_{b}\right) ^{2}|Z\right\} \leq K$ and $E_{\eta }\left\{ \left(
S_{ab}b+S_{a}\right) ^{2}|Z\right\} \leq K$ almost surely.

\begin{condition}[Condition Lin.V]
\label{cond:Lin.V} There exists a fixed constant $K>0$ such that the
following conditions hold

\begin{itemize}
\item \textbf{(Lin.V.1)} $\max_{1\leq j\leq p}\left\Vert \phi
_{j}\right\Vert _{\infty }\leq K.$

\item \textbf{(Lin.V.2)} $E_{\eta}\left\lbrace S_{ab}^{2} \mid
Z\right\rbrace \leq K$ with probability one.

\item \textbf{(Lin.V.3)} Condition Lin.L.1 for $c\in \left\{ a,b\right\} $
holds and moreover at least one of the following holds 
\begin{equation*}
\max_{i\leq n}|a(Z_{i})-\langle \theta _{a}^{\star },\phi (Z)\rangle
|=O_{P}\left( s_{a}\sqrt{\frac{\log (p)}{n}}\right) \quad \text{or}\quad
\max_{i\leq n}|b(Z_{i})-\langle \theta _{b}^{\star },\phi (Z)\rangle
|=O_{P}\left( s_{b}\sqrt{\frac{\log (p)}{n}}\right) .
\end{equation*}
\end{itemize}
\end{condition}

Condition Lin.V.3 holds trivially in Example \ref{ex:parametric_sparse},
since in that example $a(Z_{i}) - \langle \theta^{\star}_{a}, \phi(Z)\rangle$
is exactly equal to zero and similarly for $b$. In Examples \ref%
{ex:parametric_weakly_sparse} and \ref{ex:nonparam} the condition holds when 
$\alpha>1$ for at least one of $a$ or $b$ and Condition Lin.V.1 holds. We
emphasize that, in the context of this examples, Condition Lin.V.3 requires
that only one of the functions $a $ or $b$, regardless of which one, has $%
\alpha>1$.

\begin{theorem}
\label{theo:rate_double_lin}There exists $\lambda \asymp \sqrt{{\log \left(
p\right) }/{n}}$ such that if Algorithm $\ref{Algo both ALS}$ uses such $%
\lambda ,$ then

(1) under Conditions Lin.L and Lin.E.1 
\begin{equation}
\sqrt{n}\left\{ \widehat{\chi }_{lin}-\chi \left( \eta \right) \right\} =%
\mathbb{G}_{n}\left[ \Upsilon \left( a,b\right) \right] +O_{p}\left( \sqrt{%
\frac{s_{a}s_{b}}{n}}\log (p)\right) +o_{p}\left( 1\right) .
\label{eq:lin_rate_DR_expansion}
\end{equation}

(2) If Conditions Lin.L and Lin.E hold and 
\begin{equation}
\sqrt{\frac{s_{a}s_{b}}{n}}\log (p)\rightarrow 0  \label{eq:sasb}
\end{equation}%
then, 
\begin{equation*}
\frac{\sqrt{n}\left\{ \widehat{\chi }_{lin}-\chi \left( \eta \right)
\right\} }{\sqrt{E_{\eta }\left[ \left( \chi _{\eta }^{1}\right) ^{2}\right] 
}}\rightsquigarrow N\left( 0,1\right)
\end{equation*}

(3) If Conditions Lin.L, Lin.E and Lin.V hold and $\left( \ref{eq:sasb}%
\right) $ holds then 
\begin{equation*}
\frac{\sqrt{n}\left\{ \widehat{\chi }_{lin}-\chi \left( \eta \right)
\right\} }{\sqrt{\widehat{V}_{lin}}}\rightsquigarrow N\left( 0,1\right)
\end{equation*}
\end{theorem}

\begin{remark}
\label{remark:alphas} Let $K>0$, $\alpha _{a}>1/2$ and $\alpha _{b}>1/2$ be
fixed and suppose that we assume that $a$ and $b$ belong to the classes 
\begin{equation*}
\mathcal{S}_{n}\left( \phi ,t=K,\alpha _{a},l=2\right) ,\quad \text{ and }%
\quad \mathcal{S}_{n}\left( \phi ,t=K,\alpha _{b},l=2\right)
\end{equation*}%
respectively of Example \ref{ex:nonparam}. It follows from the discussion in
that example that, for $c=a$ and $c=b$, $\mathcal{S}_{n}\left( \phi
,t=K,\alpha _{c},l=2\right) $ is included in the $\mathcal{G}_{n}\left( \phi
^{\left( n\right) },s_{c},j=1,id\right) $ AGLS class sequence with $\phi
^{\left( n\right) }=\left( \phi _{1},...,\phi _{p\left( n\right) }\right) $
and 
\begin{equation*}
s_{c}=K^{1/\alpha _{c}}\left( \frac{n}{\log (p)}\right) ^{1/(2\alpha _{c})}.
\end{equation*}%
Therefore, up to $\log (p)$ factors, 
\begin{equation*}
\frac{s_{a}s_{b}}{n}\log (p)^{2}\asymp n^{\frac{1}{2\alpha _{a}}+\frac{1}{%
2\alpha _{b}}-1}.
\end{equation*}%
It then follows from the preceding theorem, that the estimator $\widehat{%
\chi }_{lin}$ is asymptotically normal if 
\begin{equation*}
\frac{1}{\alpha _{a}}+\frac{1}{\alpha _{b}}<2.
\end{equation*}%
In particular, one of $\alpha _{a}$ or $\alpha _{b}$ can be less than 1 (but
greater than 1/2) so long as the other is sufficiently large so as to
satisfy the inequality constraint. Recall that the coefficients $\alpha _{a}$
or $\alpha _{b}$ are connected with the smoothness of the functions $a$ and $%
b$, the smoothness increasing with the coefficients. A similar comment is in
place for Examples \ref{ex:parametric_weakly_sparse} and \ref{ex:holder}. 
For the theorems establishing the rate double robustness property
of the estimators $\widehat{\chi }_{nonlin}$ and $\widehat{\chi }_{mix}$ in
the next subsections, the same comment regarding Example \ref{ex:parametric_weakly_sparse} applies.
\end{remark}

\subsubsection{Model double robustness for the estimator $\widehat{\protect%
\chi }_{lin}$}

\label{sec:mod_dr_lin}

We shall next show that $\widehat{\chi }_{lin}$ satisfies also the model
robustness property. We will require the following conditions.

\begin{condition}[Condition Lin.L.W]
\label{cond:Lin.L.W} There exists fixed constants $0<k<K$ such that for $%
\left( c=b,\overline{c}=b\right) $ and for $\left( c=b,\overline{c}=a\right) 
$ the following conditions hold

\begin{itemize}
\item \textbf{(Lin.L.W.1}) There exists $\theta _{b}\in \mathbb{R}^{p}$ such
that 
\begin{equation*}
\theta _{b}\in \arg \min_{\theta \in \mathbb{R}^{p}}E_{\eta }\left[
Q_{b}\left( \theta ,\phi ,w=1\right) \right]
\end{equation*}%
and $b^{0}\left( Z\right) \equiv \left\langle \theta _{b},\phi \left(
Z\right) \right\rangle $ belongs to $\mathcal{G}\left( \phi
,s_{b},j=1,\varphi =id\right) $ with associated parameter value $\theta
_{b}^{\ast }$ whose support is denoted with $S_{b}$. Furthermore, $s_{b}\log
\left( p\right) /n\rightarrow 0$. Additionally, $a\left( Z\right) \in 
\mathcal{G}\left( \phi ,s_{a},j=2,\varphi =id\right) $ with associated
parameter value $\theta _{a}^{\ast }$ whose support is denote with $S_{a}$.
Furthermore, $s_{a}\log \left( p\right) /\sqrt{n}\rightarrow 0$.

\item \textbf{(Lin.L.W.2}) There exists fixed constants $0<k<K$ such that
Conditions Lin.L.2- Lin.L.5 hold for $\left( c=a,\overline{c}=b\right) $ and
for $\left( c=b^{0},\overline{c}=a\right)$.
\end{itemize}
\end{condition}

\bigskip

Condition {Lin.L.W.1 }differs from condition {Lin.L.1 }in three important
ways. First, {Lin.L.W.1} requires the ultra-sparsity condition $s_{a}\log
\left( p\right) /\sqrt{n}\rightarrow 0$ for the class where the function $a$
lies, whereas {Lin.L.1 the less stringent sparsity condition }$s_{a}\log
\left( p\right) /n\rightarrow 0.$ Second, like Lin.L.1, it assumes that $a$
lies in a $\mathcal{G}\left( \phi ,s_{a},j,\varphi =id\right) $ class but in 
{Lin.L.1} $j=1$ whereas in Condition Lin.L.W.1 $j=2$. This distinction is
important because $\mathcal{G}_{n}\left( \phi ,s,j=2,\varphi \right) $ is a
more restrictive class than $\mathcal{G}_{n}\left( \phi ,s,j=1,\varphi
\right) $. Third, {Lin.L.W.1 }does not require that $b$ belongs to the class 
$\mathcal{G}\left( \phi ,s_{b},j=1,\varphi =id\right) $ for the covariates $%
\phi $ and link function $\varphi =id$ used in the computation of $\widehat{%
\chi }_{lin},$ thus allowing for the possibility that the data analyst chose
the wrong set of covariates and/or the wrong link function$.$ However,
Condition {Lin.L.W.1} requires that the function $b^{0}\left( Z\right)
\equiv \left\langle \theta _{b},\phi \left( Z\right) \right\rangle $, where $%
\theta _{b}$ is a minimizer of the expectation of the loss function used to
construct our estimators of $b$, is in a class $\mathcal{G}\left( \phi
,s_{b},j=1,\varphi =id\right)$. Note that under regularity conditions, the $%
\ell _{1}$ regularized estimator of $b$ converges to $b^{0}.$ To summarize,
Condition Lin.L.W.1 essentially requires that the analyst guess correctly
the model for $a,$ and not for $b.$ However, it requires that $a$ lies in an
ultra-sparse approximate linear class and that the $\ell _{1}$ regularized
estimator of $b$ converges to a linear function that belongs to a sparse,
but not necessarily ultra sparse, approximate linear class.

Condition Lin.L.W.1, but with the stringent requirement that the
approximation error $r\left( Z\right) $ be equal to 0 has been assumed by 
\cite{Tan} and \cite{Peng} to prove the model double robustness property of
their proposed estimators. An instance in which Condition Lin.L.W.1 holds
but Lin.L.1 does not hold is if $b\left( Z\right) =E\left( Y|Z\right)
=\varphi ^{\dag }\left( \left\langle \theta _{b}^{\dag },\phi \left(
Z\right) \right\rangle \right) $ with $\left\Vert \theta _{b}^{\dag
}\right\Vert _{0}\leq $ $s_{b}$ for some non-linear strictly increasing link 
$\varphi ^{\dag }$. Results in \cite{glm-misp} imply that, under further
regularity assumptions, there exists a minimizer $\theta _{b}$ of the loss
function $E_{\eta }\left[ Q_{b}\left( \theta ,\phi ,w=1\right) \right] $
which incorrectly uses the linear link instead of $\varphi ^{\dag },$
satisfying $C\theta _{b}^{\dag }=\theta _{b}$ for some constant $C,$ which
then implies that $\theta _{b}$ has the same support as $\theta _{b}^{\dag }$
and consequently that $b^{0}\left( Z\right) \equiv \left\langle \theta
_{b},\phi \left( Z\right) \right\rangle $ belongs to $\mathcal{G}\left( \phi
,s_{b},j=1,\varphi =id\right)$. See also \cite{vdgmisp} for further examples.

\bigskip

\begin{condition}[Condition Lin.E.W]
\label{cond:Lin.E.W} $\ $Condition Lin.L.W.1 holds and

\begin{itemize}
\item \textbf{(Lin.E.W.1)} 
\begin{equation*}
E_{\eta}\left[(S_{ab}b^{0})^{2} \right]\leq K,
\end{equation*}
\begin{equation*}
E_{\eta }\left\{ \left[ S_{ab}a\left( b^{\prime }-b^{0}\right) +m_{b}\left(
O,b^{\prime }\right) -m_{b}\left( O,b^{0}\right) \right] ^{2}\right\}
\rightarrow 0\quad \text{as}\quad E_{\eta }\left[ \left( b^{\prime
}-b^{0}\right) ^{2}\right] \rightarrow 0,
\end{equation*}%
and 
\begin{equation*}
E_{\eta }\left\{ \left[ S_{ab}b^{0}\left( a^{\prime }-a\right) +m_{a}\left(
O,a^{\prime }\right) -m_{a}\left( O,a\right) \right] ^{2}\right\}
\rightarrow 0\quad \text{as}\quad E_{\eta }\left[ \left( a^{\prime
}-a\right) ^{2}\right] \rightarrow 0.
\end{equation*}

\item \textbf{(Lin.E.W.2)} There exists fixed constants $0<k<K$ such that

\begin{itemize}
\item[a)] $k\leq {E\left[ \left\{ \Upsilon \left( a,b^{0}\right) -\chi
\left( \eta \right) \right\} ^{2}\right] }$ and $E\left\{ \left\vert
\Upsilon \left( a,b^{0}\right) -\chi \left( \eta \right) \right\vert
^{3}\right\} \leq K$.

\item[b)] ${E\left[ \left\{ \Upsilon \left( a,b^{0}\right) -\chi \left( \eta
\right) \right\} ^{4}\right] }\leq K$.
\end{itemize}
\end{itemize}
\end{condition}

\begin{condition}[Condition Lin.V.W]
\label{cond:Lin.V.W} There exists a fixed constant $K>0$ such that
Conditions Lin.V.1, Lin.V.2 and the following conditions hold:

\begin{itemize}
\item \textbf{(Lin.V.W.3)} Condition Lin.L.W.1 for $c\in \left\{
a,b^{0}\right\} $ holds and moreover at least one of the following holds 
\begin{equation*}
\max_{i\leq n}|a(Z_{i})-\langle \theta _{a}^{\ast },\phi (Z)\rangle
|=O_{P}\left( s_{a}\sqrt{\frac{\log (p)}{n}}\right) \quad \text{or}\quad
\max_{i\leq n}|b^{0}(Z_{i})-\langle \theta _{b}^{\ast },\phi (Z)\rangle
|=O_{P}\left( s_{b}\sqrt{\frac{\log (p)}{n}}\right) .
\end{equation*}
\end{itemize}
\end{condition}

\begin{condition}[Condition M.W]
\label{cond:M.W} There exists a linear mapping $h\in L_{2}(P_{Z,\eta
})\rightarrow m_{a}^{\ddagger }(o,h)$ such that $h\in L_{2}(P_{Z,\eta
})\rightarrow E_{\eta }\left[ m_{a}^{\ddagger }(O,h)\right] $ is continuous
with Riesz representer $\mathcal{R}_{a}^{\ddagger }$ that satisfies 
\begin{equation*}
|m_{a}(o,h)|\leq m_{a}^{\ddagger }(o,|h|)\text{ for all }o\text{ and all }%
h\in L_{2}(P_{Z,\eta })
\end{equation*}%
and 
\begin{equation*}
E_{\eta }\left[ (\mathcal{R}_{a}^{\ddagger })^{2}\right] \leq K.
\end{equation*}
\end{condition}

Conditions {Lin.E.W }and {Lin.V.W }are essentially the same as {Lin.E }and{\
Lin.V }but with the non-trivial subtlety that the condition must hold with $%
b^{0}$ instead of $b$ and the additional requirement that $E_{\eta}\left[%
(S_{ab}b^{0})^{2} \right]\leq K$. We have already described a realistic
example in which $b^{0} $ satisfied Condition {Lin.L.W.1}. i.e. one in which
the investigator used the wrong link function and the true $b$ followed an
exactly sparse generalized non-linear model. However, for this example, we
have been able to find only a somewhat artificial setting in which the
additional conditions {Lin.E.W} and {Lin.V.W }also hold. See Proposition \ref%
{prop:ex_misp} in Appendix B. Finally, it is easy to show that Condition M.W
holds in all the examples discussed in Section \ref{sec:setup}.

\begin{theorem}
\label{theo:model_double_lin} There exists $\lambda \asymp \sqrt{{\log
\left( p\right) }/{n}}$ such that if Algorithm $\ref{Algo both ALS}$ uses
such $\lambda $, then

(1) under Conditions Lin.L.W, Lin.E.W.1 and M.W 
\begin{equation}
\sqrt{n}\left\{ \widehat{\chi }_{lin}-\chi \left( \eta \right) \right\} =%
\mathbb{G}_{n}\left[ \Upsilon \left( a,b^{0}\right) \right] +O_{p}\left( 
\sqrt{\frac{s_{a}s_{b}}{n}}\log (p)\right) +o_{p}\left( 1\right) .
\label{eq:lin_model_DR_expansion}
\end{equation}

(2) If Conditions Lin.L.W, Lin.E.W and M.W hold and 
\begin{equation}
\sqrt{\frac{s_{a}s_{b}}{n}}\log (p)\rightarrow 0  \label{eq:sasb_model}
\end{equation}%
then, 
\begin{equation*}
\frac{\sqrt{n}\left\{ \widehat{\chi }_{lin}-\chi \left( \eta \right)
\right\} }{\sqrt{E_{\eta }\left[ \left\{ \Upsilon \left( a,b^{0}\right)
-\chi \left( \eta \right) \right\} ^{2}\right] }}\rightsquigarrow N\left(
0,1\right) .
\end{equation*}

(3) If Conditions Lin.L.W, Lin.E.W, M.W and Lin.V.W hold and $\left( \ref%
{eq:sasb_model}\right) $ holds then 
\begin{equation}
\frac{\sqrt{n}\left\{ \widehat{\chi }_{lin}-\chi \left( \eta \right)
\right\} }{\sqrt{\widehat{V}_{lin}}}\rightsquigarrow N\left( 0,1\right) .
\label{eq:normalLimit_lin}
\end{equation}
\end{theorem}

Because the structure of the influence function is symmetric relative to $a$
and $b$, if in Conditions Lin.L.W, Lin.E.W, M.W and Lin.V.W we change the
roles of $a$ and $b,$ then Theorem \ref{theo:model_double_lin} remains valid
but with $\Upsilon \left( a^{0},b\right) $ instead of $\Upsilon \left(
a,b^{0}\right)$. We thus arrive at the following result that encapsulates
the model double robust property of $\widehat{\chi }_{lin}.$

\begin{corollary}
\label{coro:lin_sym} If in Algorithm $\ref{Algo both ALS}$ $\lambda \asymp 
\sqrt{{\log \left( p\right) }/{n}},$ then if conditions (1)-(3) of Theorem %
\ref{theo:model_double_lin} hold, or if the same conditions (1)-(3) hold but
with the roles of $a$ and $b$ reversed, then \eqref{eq:normalLimit_lin}
holds.
\end{corollary}

\subsection{Asymptotic results for the estimator $\widehat{\protect\chi }%
_{nonlin}$}

\label{sec:asym_nonlin}

In Sections \ref{sec:dr_nonlin} and \ref{sec:mod_dr_Nonlin} we state
theorems that establish the rate and model double robustness properties of
the estimator $\widehat{\chi }_{nonlin}$\thinspace\ that uses both links $%
\varphi _{a}\left( u\right) $ and $\varphi _{b}\left( u\right) $ possibly
non-linear. For this case we require more stringent assumptions. In the $%
\ell _{1}$ regularized estimation literature, one of two alternative
assumptions is typically made in order to obtain fast rates of convergence,
i.e. rates of order $\sqrt{s\log \left( p\right) /n}$ in $\ell _{2}-$ norm %
\citep{vdgNotes}. One such assumption is the often referred to as the ultra
sparsity condition that $s=o\left( \sqrt{n}\right) $ (\cite{bach}, \cite%
{vdgquasi}). We cannot impose this condition because it would defeat the
purpose of rate double robustness. Specifically, if $s_{a}$ and $s_{b}$
were, up to logarithmic terms, each of order $o\left( \sqrt{n}\right) $ then 
$s_{a}s_{b}$ would be $o\left( n\right) $ and no trade off of model
complexity could be achieved. The second such assumption is based on higher
order isotropy conditions (see \cite{vdgNotes}) and is satisfied in
particular by covariates $\left\{ \phi _{j}\left( Z\right) \right\} _{1\leq
j\leq p}$ that are jointly sub-gaussian (\cite{NegahbanSS}, \cite{Loh-smooth}%
) with $E_{\eta }\left[ \phi \left( Z\right) \phi \left( Z\right) ^{\prime }%
\right] $ having smallest eigenvalue bounded away from 0. We follow this
approach (see Conditions NLin.L.3, and NLin.L.6).

To show our results we will continue to assume condition Lin.E of Section %
\ref{sec:dr_lin}. To prove the model double robustness property, we will
continue to assume Condition M.W of Section \ref{sec:dr_lin}.

In what follows we assume that the link functions $\varphi _{a}\left(
u\right) $ and $\varphi _{b}\left( u\right) $ are continuously
differentiable in $\mathbb{R}$.

\subsubsection{Rate double robustness for the estimator $\widehat{\protect%
\chi }_{nonlin}$}

\label{sec:dr_nonlin}

\begin{condition}[Condition NLin.L]
\label{cond:NonLin.L} There exists fixed constants $0<k<K$ such that for $%
\left( c=a,\overline{c}=b\right) $ and for $\left( c=b,\overline{c}=a\right) 
$ the following conditions hold

\begin{itemize}
\item \textbf{(NLin.L.1}) $c\left( Z\right) \in \mathcal{G}\left( \phi
,s_{c},j=1,\varphi _{c}\right) $ with associated parameter value denoted as $%
\theta _{c}^{\ast }$. Furthermore, $s_{c}\log \left( p\right) /n\rightarrow
0,$ $\left\Vert \theta _{c}^{\ast }\right\Vert _{2}\leq K$ and 
\begin{equation*}
E^{1/8}_{\eta }\left[ \left\{ c\left( Z\right) -\varphi _{c}\left(
\left\langle \theta _{c}^{\ast },\phi \left( Z\right) \right\rangle \right)
\right\} ^{8}\right] \leq \sqrt{\frac{Ks_{c}\log \left( p\right)}{n}}.
\end{equation*}

\item \textbf{(NLin.L.2}) 
\begin{equation*}
E_{\eta }\left[ \max\limits_{1\leq j\leq p}\left( \mathbb{P}%
_{n}(S_{ab,i}c\left( Z\right) \phi _{j}\left( Z\right) +\mathcal{R}_{%
\overline{c}}\left( Z\right) \phi _{j}\left( Z\right) )^{4}\right) ^{1/2}%
\right] \leq K.
\end{equation*}

\item \textbf{(NLin.L.3) }$k\leq \lambda _{min}\left( \Sigma _{1}\right)$.

\item \textbf{(NLin.L.4)} $k\leq \lambda _{min}\left( \Sigma _{2}\right)$.

\item \textbf{(NLin.L.5) }$E_{\eta }\left( S_{ab}^{4}\right) $ $\leq K$ and $%
E_{\eta }\left( \left[ S_{ab}c\left( Z\right) +\mathcal{R}_{\overline{c}%
}\left( Z\right) \right] ^{4}\right) \leq K$.

\item \textbf{(NLin.L.6) }$\sup_{\Vert \Delta \Vert _{2}=1}\Vert \langle
\Delta ,\phi \left( Z\right) \rangle \Vert _{\psi _{2}}\leq K$.

\item \textbf{(NLin.L.7) }For $\overset{\cdot }{\varphi }_{\overline{c}%
,\theta }\left( Z\right) \equiv \overset{\cdot }{\varphi }_{\overline{c}%
}(\langle \theta ,\phi \left( Z\right) \rangle )$ it holds that 
\begin{equation*}
\sup_{\Vert \theta -\theta _{c}^{\ast }\Vert _{2}\leq 1}E\left\{ \max_{1\leq
j\leq p}\left( \mathbb{P}_{n}(\mathcal{R}_{\overline{c}}\left( Z\right) \phi
_{j}\left( Z\right) \overset{\cdot }{\varphi }_{\overline{c},\theta }\left(
Z\right) -m_{\overline{c}}(O,\phi _{j}\overset{\cdot }{\varphi }_{\overline{c%
},\theta }))^{2}\right) ^{1/2}\right\} \leq K.
\end{equation*}
\end{itemize}
\end{condition}

\begin{remark}
If there exists a statistic $S_{\overline{c}}$ such that$\ m_{\overline{c}%
}\left( O,h\right) =S_{\overline{c}}h,$ for all $h,\,\ $then $\mathcal{R}_{%
\overline{c}}\left( Z\right) $ can be replaced by $S_{\overline{c}},$ in
Conditions NLin.L.2 and NLin.L.5 and Condition NLin.L.7 is not needed.
\end{remark}

\begin{condition}[Condition NLin.Link]
\label{cond:links} There exists fixed constants $0<k<K$ such that

\begin{itemize}
\item \textbf{(NLin.Link.1) }${\varphi }_{a}^{\prime}(u)>0$ and ${\varphi }%
_{b}^{\prime}(u)>0$ for all $u\in \mathbb{R}$ if $P_{\eta }\left( S_{ab}\geq
0\right) =1$ and ${\varphi }_{a}^{\prime}(u)<0$ and ${\varphi }%
^{\prime}_{b}(u)<0$ for all $u\in \mathbb{R}$ if $P_{\eta }\left( S_{ab}\leq
0\right) =1$.

\item \textbf{(NLin.Link.2)} For all $u,v$ 
\begin{equation*}
|\varphi _{a}(u)-\varphi _{a}(v)|\leq K\exp \left( K(|u|+|v|)\right) |u-v|
\end{equation*}
and 
\begin{equation*}
|\varphi _{b}(u)-\varphi _{b}(v)|\leq K\exp \left( K(|u|+|v|)\right) |u-v|.
\end{equation*}

\item \textbf{(NLin.Link.3)} For all $u,v$ 
\begin{equation*}
|{\varphi }^{\prime}_{a}(u)-{\varphi }^{\prime}_{a}(v)|\leq K\exp \left(
K(|u|+|v|)\right) |u-v|
\end{equation*}
and 
\begin{equation*}
|{\varphi }^{\prime}_{b}(u)-{\varphi }^{\prime}_{b}(v)|\leq K\exp \left(
K(|u|+|v|)\right) |u-v|.
\end{equation*}

\item \textbf{(NLin.Link.4)} $\varphi_{a}$ and $\varphi_{b}$ are twice
continuously differentiable and for all $u,v$ 
\begin{equation*}
|{\varphi }_{a}^{\prime\prime}(u)-{\varphi }_{a}^{\prime\prime}(v)|\leq
K\exp \left( K(|u|+|v|)\right) |u-v|
\end{equation*}
and 
\begin{equation*}
|{\varphi }_{b}^{\prime\prime}(u)-{\varphi }_{b}^{\prime\prime}(v)|\leq
K\exp \left( K(|u|+|v|)\right) |u-v|.
\end{equation*}
\end{itemize}
\end{condition}

Condition {NLin.L.1} is like {Lin.L.1 }except that for non-linear links $%
\varphi _{a}$ and $\varphi _{b}$ we additionally require that the $\ell_2$
norm of the coefficients of the sparse linear approximations to the nuisance
functions be bounded. In Example \ref{ex:parametric_sparse} of Section \ref%
{sec:models}, the model is exactly sparse. In this case a necessary
condition for the $\ell_2$ norm to be bounded is that the coefficient $s_{c}$
does not grow with $n$. In Example \ref{ex:parametric_weakly_sparse}, the
model is parametric and we know the rate of decay of the coefficients. In
this case, the $\ell_2$ norm of the coefficient vector is bounded for any $%
\alpha >1/2,$ so long as $t\left( n\right)$ is bounded. Note also that
whereas in {Lin.L.1} we required that the norm $\left\Vert \cdot \right\Vert
_{L_{2}\left( P_{\eta ,Z}\right) }$ of the approximation error to converge
to zero at rate $\sqrt{s_{c}\log \left( p\right)/n}$, in {NLin.L.1} we
require the more stringent condition that the norm $\left\Vert \cdot
\right\Vert _{L_{8}\left( P_{\eta ,Z}\right) }$ of the approximation error
converges at this rate. We note that this condition is trivially satisfied
in Example \ref{ex:parametric_sparse} because the approximation error is
zero. In Example \ref{ex:parametric_weakly_sparse}, it can be shown that the
condition is satisfied under the sub-gaussianity condition {NLin.L.6 }(to be
discussed shortly), Condition {NLin.Link.2 }on the modulus of continuity of
the link function and Condition NLin.L.3 as long as 
\begin{equation*}
E^{1/2}_{\eta }\left[ \left\{ c\left( Z\right) -\varphi _{c}\left(
\left\langle \theta _{c}^{\ast },\phi \left( Z\right) \right\rangle \right)
\right\}^{2}\right] \leq \sqrt{\frac{Ks_{c}\log \left( p\right)}{n}}.
\end{equation*}%
See Lemma \ref{lemma:lipschitz} in Appendix C.

Condition {NLin.L.2} is similar in spirit to Condition {Lin.L.2}. It holds
in particular, when 
\begin{equation*}
\max_{1\leq j\leq p}\left\Vert \phi _{j}\right\Vert _{\infty } \leq K
\end{equation*}
and 
\begin{equation*}
E_{\eta }\left[ \left\{ \left\vert S_{ab}\right\vert \left\vert c\left(
Z\right) \right\vert +\left\vert \mathcal{R}_{\overline{c}}\left( Z\right)
\right\vert \right\} ^{4}\right] \leq K.
\end{equation*}
Conditions NLin.L.3 and {NLin.L.4 }are more demanding than conditions {%
Lin.L.3 }and{\ Lin.L.4. }An instance in which{\ }conditions{\ NLin.L.3 }and{%
\ NLin.L.4 }was discussed in Section \ref{sec:dr_lin}.

Condition {NLin.L.5 }imposes mild moment assumptions. Condition {NLin.L.6 }%
requires that all linear combinations of the components of $\phi \left(
Z\right) $ have tails that decay at least as fast as the tail of a normal
random variable. This holds, in particular, if there exists a random vector $%
R\in \mathbb{R}^{t}$ and a matrix $A\in \mathbb{R}^{p\times t} $ such that $%
\phi \left( Z\right) =AR$, the coordinates of $R$ are independent with
bounded sub-Gaussian norm and the singular values of $A$ are bounded away
from zero and infinity. For instance, the condition is satisfied if $\phi
\left( Z\right) $ is multivariate normal, with a covariance matrix that has
eigenvalues bounded away from zero and infinity.

Conditions {NLin.Link.1-NLin.Link.4} are satisfied for instance, when $%
P_{\eta }\left( S_{ab}\geq 0\right) =1$ if $\varphi _{a}(u)$ and $\varphi
_{b}(u)$ are each either $id\left( u\right) ,\exp (u)$ or $\exp (u)/(1+\exp
(u)),$ and when $P_{\eta }\left( S_{ab}\leq 0\right) =1$ if $\varphi _{a}(u)$
and $\varphi _{b}(u)$ are each either $-id\left( u\right) ,\exp (-u)$ or $%
1+\exp (-u)$. Condition NLin.Link.4 will only be needed to prove the model
double robustness property of $\widehat{\chi }_{nonlin}$. In fact, we make
assumption {NLin.Link.1 for simplicity, since it can be replaced without
affecting the conclusion of the Theorems, if we replace it by the weaker
condition }$S_{ab}w\left( Z\right) \psi _{c}\left( \left\langle \theta ,\phi
\left( Z\right) \right\rangle \right) $ is convex in $\theta $ for $c=a$ and 
$c=b.$

\begin{theorem}
\label{theo:rate_double_nonlin} There exists $\lambda \asymp \sqrt{{\log
\left( p\right) }/{n}}$ such that if Algorithm $\ref{Algo both GALS}$ uses
such $\lambda $ in both steps, then

(1) under Conditions NLin.L, NLin.Link.1-NLin.Link.3 and Lin.E.1 
\begin{equation}
\sqrt{n}\left\{ \widehat{\chi }_{nonlin}-\chi \left( \eta \right) \right\} =%
\mathbb{G}_{n}\left[ \Upsilon \left( a,b\right) \right] +O_{p}\left( \sqrt{%
\frac{s_{a}s_{b}}{n}}\log (p)\right) +o_{p}\left( 1\right) .
\label{eq:nonlin_rate_DR_expansion}
\end{equation}

(2) If Conditions NLin.L, NLin.Link.1-NLin.Link.3 and Lin.E hold and 
\begin{equation}
\sqrt{\frac{s_{a}s_{b}}{n}}\log (p)\rightarrow 0  \notag
\end{equation}%
then, 
\begin{equation*}
\frac{\sqrt{n}\left\{ \widehat{\chi }_{nonlin}-\chi \left( \eta \right)
\right\} }{\sqrt{E_{\eta }\left[ \left( \chi _{\eta }^{1}\right) ^{2}\right] 
}}\rightsquigarrow N\left( 0,1\right)
\end{equation*}%
and 
\begin{equation*}
\frac{\sqrt{n}\left\{ \widehat{\chi }_{nonlin}-\chi \left( \eta \right)
\right\} }{\sqrt{\widehat{V}_{nonlin}}}\rightsquigarrow N\left( 0,1\right)
\end{equation*}
\end{theorem}

\subsubsection{Model double robustness for the estimator $\widehat{\protect%
\chi }_{nonlin}$}

\label{sec:mod_dr_Nonlin}

We shall next show that $\widehat{\chi }_{nonlin}$ satisfies also the model
robustness property. In what follows for $c=a$ or $c=b,$ we define ${\varphi 
}_{c,\theta }^{\prime}\left( Z\right) \equiv {\varphi }_{c}^{\prime}(\langle
\theta ,\phi \left( Z\right) \rangle )$. We will require the following
conditions.

\begin{condition}[Condition NLin.L.W]
\label{cond:NonLin.L.W} Conditions NLin.L.3, NLin.L.4 and NLin.L.6 hold.
Moreover, there exists a fixed constant $K>0$ such that the following
conditions hold

\begin{itemize}
\item \textbf{(NLin.L.W.1}) There exists $\theta _{b}^{0}\in \mathbb{R}^{p} $
such that 
\begin{equation*}
\theta _{b}^{0}\in \arg \min_{\theta \in \mathbb{R}^{p}}E_{\eta }\left[
Q_{b}\left( \theta ,\phi ,w=1\right) \right]
\end{equation*}%
and $b^{0}\left( Z\right) \equiv \varphi _{b}\left( \left\langle \theta
_{b}^{0},\phi \left( Z\right) \right\rangle \right) $ belongs to $\mathcal{G}%
\left( \phi ,s_{b},j=1,\varphi _{b}\right) $ with associated parameter value
denoted as $\theta _{b}^{0\ast }$. Furthermore, $s_{b}\log \left( p\right)
/n\rightarrow 0,$ $\left\Vert \theta _{b}^{0\ast }\right\Vert _{2}\leq K$
and 
\begin{equation*}
E_{\eta }^{1/8}\left[ \left\{ b^{0}\left( Z\right) -\varphi _{b}\left(
\left\langle \theta _{b}^{0\ast },\phi \left( Z\right) \right\rangle \right)
\right\} ^{8}\right] \leq \sqrt{\frac{Ks_{b}\log \left( p\right)}{n}}.
\end{equation*}%
Additionally, $a\left( Z\right) \in \mathcal{G}\left( \phi
,s_{a},j=2,\varphi _{a}\right) $ with associated parameter value denoted as $%
\theta _{a}^{\ast }$. Furthermore, $s_{a}\log \left( p\right) /\sqrt{n}%
\rightarrow 0,$ $\left\Vert \theta _{a}^{\ast }\right\Vert _{2}\leq K$ and 
\begin{equation*}
E_{\eta }^{1/8}\left[ \left\{ a\left( Z\right) -\varphi _{a}\left(
\left\langle \theta _{a}^{\ast },\phi \left( Z\right) \right\rangle \right)
\right\} ^{8}\right] \leq \sqrt{\frac{Ks_{a}\log \left( p\right)}{n}}.
\end{equation*}%
Also, there exists $\theta _{b}^{1}\in \mathbb{R}^{p}$ such that 
\begin{equation*}
\theta _{b}^{1}\in \arg \min_{\theta \in \mathbb{R}^{p}}E_{\eta }\left[
Q_{b}\left( \theta ,\phi ,w={\varphi }^{\prime}_{a,\theta _{a}^{\ast
}}\right) \right]
\end{equation*}%
where $b^{1}\left( Z\right) \equiv \varphi _{b}\left( \left\langle \theta
_{b}^{1},\phi \left( Z\right) \right\rangle \right) $ belongs to $\mathcal{G}%
\left( \phi ,s_{b},j=1,\varphi _{b}\right) $ with associated parameter value
denoted as $\theta _{b}^{1\ast }$. Furthermore, $s_{b}\log \left( p\right)
/n\rightarrow 0,\left\Vert \theta _{b}^{1\ast }\right\Vert _{2}\leq K$ and 
\begin{equation*}
E_{\eta }^{1/8}\left[ \left\{ b^{1}\left( Z\right) -\varphi _{b}\left(
\left\langle \theta _{b}^{1\ast },\phi \left( Z\right) \right\rangle \right)
\right\} ^{8}\right] \leq \sqrt{\frac{Ks_{b}\log \left( p\right)}{n}}.
\end{equation*}

\item \textbf{(NLin.L.W.2}) 
\begin{equation*}
E_{\eta }\left[ \max\limits_{1\leq j\leq p}\left( \mathbb{P}%
_{n}(S_{ab,i}a\left( Z\right) \phi _{j}\left( Z\right) +\mathcal{R}%
_{b}\left( Z\right) \phi _{j}\left( Z\right) )^{4}\right) ^{1/2}\right] \leq
K.
\end{equation*}%
\begin{equation*}
E_{\eta }\left[ \max\limits_{1\leq j\leq p}\left( \mathbb{P}%
_{n}(S_{ab,i}b^{0}\left( Z\right) \phi _{j}\left( Z\right) +\mathcal{R}%
_{a}\left( Z\right) \phi _{j}\left( Z\right) )^{4}\right) ^{1/2}\right] \leq
K.
\end{equation*}%
\begin{equation*}
E_{\eta }\left[ \max\limits_{1\leq j\leq p}\left( \mathbb{P}%
_{n}(S_{ab,i}b^{1}\phi _{j}\left( Z\right) +\mathcal{R}_{a}\left( Z\right)
\phi _{j}\left( Z\right) )^{4}\right) ^{1/2}\right] \leq K.
\end{equation*}

\item \textbf{(NLin.L.W.3) }$E_{\eta }\left( S_{ab}^{4}\right)\leq K$, $%
E_{\eta }\left( \left[ S_{ab}a\left( Z\right) +\mathcal{R}_{b}\left(
Z\right) \right] ^{4}\right) \leq K$, $E_{\eta }\left( \left[
S_{ab}b^{0}\left( Z\right) +\mathcal{R}_{a}\left( Z\right) \right]
^{4}\right) \leq K$ and $E_{\eta }\left( \left[ S_{ab}b^{1}\left( Z\right) +%
\mathcal{R}_{a}\left( Z\right) \right] ^{4}\right) \leq K$.

\item \textbf{(NLin.L.W.4) } 
\begin{equation*}
\sup_{\Vert \theta -\theta _{b}^{0\ast }\Vert _{2}\leq 1}E\left\{
\max_{1\leq j\leq p}\left( \mathbb{P}_{n} \left\lbrace (\mathcal{R}%
_{b}\left( Z\right) \phi _{j}\left( Z\right) {\varphi }^{\prime}_{b,\theta
}\left( Z\right) -m_{b}(O,\phi _{j}{\varphi }^{\prime}_{b,\theta }))^{2}
\right\rbrace\right)^{1/2}\right\} \leq K.
\end{equation*}%
\begin{equation*}
\sup_{\Vert \theta -\theta _{b}^{1\ast }\Vert _{2}\leq 1}E\left\{
\max_{1\leq j\leq p}\left( \mathbb{P}_{n} \left\lbrace (\mathcal{R}%
_{b}\left( Z\right) \phi _{j}\left( Z\right) {\varphi }^{\prime}_{b,\theta
}\left( Z\right) -m_{b}(O,\phi _{j}{\varphi }^{\prime}_{b,\theta }))^{2}
\right\rbrace\right)^{1/2}\right\} \leq K.
\end{equation*}%
\begin{equation*}
\sup_{\Vert \theta -\theta_{a}^{\ast }\Vert _{2}\leq 1}E\left\{ \max_{1\leq
j\leq p}\left( \mathbb{P}_{n}\left\lbrace(\mathcal{R}_{a}\left( Z\right)
\phi _{j}\left( Z\right) {\varphi }^{\prime}_{a,\theta }\left( Z\right)
-m_{a}(O,\phi _{j}{\varphi }^{\prime}_{a,\theta }))^{2} \right\rbrace
\right)^{1/2} \right\} \leq K.
\end{equation*}
\end{itemize}
\end{condition}

\begin{remark}
If for $c=a$ or for $c=b$ there exists a statistic $S_{c}$ such that$\
m_{c}\left( O,h\right) =S_{c}h,$ for all $h,\,\ $then $\mathcal{R}_{c}\left(
Z\right) $ in conditions NLin.L.W.2 and NLin.L.W.3 can be replaced by $%
S_{c}, $ and the inequality where $\mathcal{R}_{c}\left( Z\right) $ appears
in condition NLin.L.W.4 can be removed.
\end{remark}

Condition N{Lin.L.W.1 }differs from Condition N{Lin.L.1 }in same three
important ways as discussed for the distinctions between {Lin.L.W.1} and {%
Lin.L.1, except that now we make assumptions on two limit functions }$b^{0}$
and $b^{1}$ corresponding to estimators of $b$ computed at the first and
second stage of the algorithm.

\begin{condition}[Condition NLin.E.W]
\label{cond:NonLin.E.W} $\ $Condition NLin.L.W.1 holds and

\begin{itemize}
\item \textbf{(NLin.E.W.1)} 
\begin{equation*}
E_{\eta}\left[ (S_{ab}b^{1})^{2}\right]\leq K,
\end{equation*}
\begin{equation*}
E_{\eta }\left\{ \left[ S_{ab}a\left( b^{\prime }-b^{1}\right) +m_{b}\left(
O,b^{\prime }\right) -m_{b}\left( O,b^{1}\right) \right] ^{2}\right\}
\rightarrow 0\quad \text{as}\quad E_{\eta }\left[ \left( b^{\prime
}-b^{1}\right) ^{2}\right] \rightarrow 0
\end{equation*}%
and 
\begin{equation*}
E_{\eta }\left\{ \left[ S_{ab}b^{1}\left( a^{\prime }-a\right) +m_{a}\left(
O,a^{\prime }\right) -m_{a}\left( O,a\right) \right] ^{2}\right\}
\rightarrow 0\quad \text{as}\quad E_{\eta }\left[ \left( a^{\prime
}-a\right) ^{2}\right] \rightarrow 0.
\end{equation*}

\item \textbf{(NLin.E.W.2)} There exists fixed constants $0<k<K$ such that

\begin{itemize}
\item[a)] $k\leq {E\left[ \left\{ \Upsilon \left( a,b^{1}\right) -\chi
\left( \eta \right) \right\} ^{2}\right] }$ and $E\left\{ \left\vert
\Upsilon \left( a,b^{1}\right) -\chi \left( \eta \right) \right\vert
^{3}\right\} \leq K$.

\item[b)] ${E\left[ \left\{ \Upsilon \left( a,b^{1}\right) -\chi \left( \eta
\right) \right\} ^{4}\right] }\leq K$.
\end{itemize}
\end{itemize}
\end{condition}

\begin{theorem}
\label{theo:model_DR_Nonlin} There exists $\lambda \asymp \sqrt{{\log \left(
p\right) }/{n}}$ such that if Algorithm $\ref{Algo both GALS}$ $\ $uses such 
$\lambda $ in both steps, then

(1) under Conditions NLin.L.W, NLin.Link, NLin.E.W.1 and M.W 
\begin{equation}
\sqrt{n}\left\{ \widehat{\chi }_{nonlin}-\chi \left( \eta \right) \right\} =%
\mathbb{G}_{n}\left[ \Upsilon \left( a,b^{1}\right) \right] +O_{p}\left( 
\sqrt{\frac{s_{a}\max(s_{b},s_{a})}{n}}\log (p)\right) +o_{p}\left( 1\right)
.  \label{eq:Nonlin_model_DR_expansion}
\end{equation}

(2) If Conditions NLin.L.W, NLin.Link, Lin.E.W and M.W hold and 
\begin{equation}
\sqrt{\frac{s_{a}\max (s_{b},s_{a})}{n}}\log (p)\rightarrow 0
\label{eq:sasb_nonlin_model}
\end{equation}%
then, 
\begin{equation*}
\frac{\sqrt{n}\left\{ \widehat{\chi }_{nonlin}-\chi \left( \eta \right)
\right\} }{\sqrt{E_{\eta }\left[ \left\{ \Upsilon \left( a,b^{1}\right)
-\chi \left( \eta \right) \right\} ^{2}\right] }}\rightsquigarrow N\left(
0,1\right) .
\end{equation*}

(3) If Conditions NLin.L.W, NLin.Link, NLin.E.W and M.W hold and $\left( \ref%
{eq:sasb_nonlin_model}\right) $ holds then 
\begin{equation}
\frac{\sqrt{n}\left\{ \widehat{\chi }_{nonlin}-\chi \left( \eta \right)
\right\} }{\sqrt{\widehat{V}_{nonlin}}}\rightsquigarrow N\left( 0,1\right) .
\label{eq:normalLimit_nonlin}
\end{equation}
\end{theorem}

Note that in expansion \eqref{eq:Nonlin_model_DR_expansion} the term $%
O_{p}\left( \sqrt{{s_{a}\max(s_{b},s_{a})}/{n}}\log (p)\right) $ appears,
instead of the term $O_{p}\left( \sqrt{{s_{a}s_{b}}/{n}}\log (p)\right) $
that appears in Theorems \ref{theo:rate_double_lin}, \ref%
{theo:model_double_lin} and \ref{theo:rate_double_nonlin}. This is due to
the following. The second stage $\ell_{1}$-regularized estimators for $b$ in
Algorithm \ref{Algo both GALS} use weights defined using the estimators of $%
a $ obtained from the first stage. When the model for $b$ is misspecified,
the rate of convergence of the first stage estimators of $a$ affects the
rate of convergence of the second stage estimators of $b$, since the very
definition of the probability limit of the second stage estimators of $b$,
namely $b^{1} $, depends directly on the probability limit of the estimators
of $a$. In fact, Theorem \ref{theo:rate_main} implies that the rate of
convergence of the second stage estimators of $b$ to $b^{1}$ is $\sqrt{%
\max(s_a,s_b)\log(p)/n}$ instead of the usual $\sqrt{s_{b}\log(p)/n}$. This
is where the $\max(s_a,s_b)$ term comes from.

Because the structure of the influence function is symmetric relative to $a$
and $b$, if in Conditions NLin.L.W, NLin.E.W and M.W we change the roles of $%
a$ and $b$ then Theorem \ref{theo:model_DR_Nonlin} remains valid but with $%
\Upsilon \left( a^{1},b\right) $ instead of $\Upsilon \left( a,b^{1}\right) $%
. We thus arrive at the following result that encapsulates the model double
robust property of $\widehat{\chi }_{nonlin}.$

\begin{corollary}
\label{coro:Nonlin_sym} If in Algorithm $\ref{Algo both GALS}$ $\lambda
\asymp \sqrt{{\log \left( p\right) }/{n}},$ then if conditions (1)-(3) of
Theorem \ref{theo:model_DR_Nonlin} hold, or if the same conditions (1)-(3)
hold but with the roles of $a$ and $b$ reversed, then %
\eqref{eq:normalLimit_nonlin} holds.
\end{corollary}

\subsection{Asymptotic results for the estimator $\widehat{\protect\chi }%
_{mix}$}

\label{sec:asym_mix}

We will state without proof the results for the asymptotic behavior of $%
\widehat{\chi }_{mix}.$ The proofs involve a combination of the strategies
used in the proofs for the properties of $\widehat{\chi }_{lin}$ and $%
\widehat{\chi }_{nonlin}$. The assumptions will be stated assuming that $%
\varphi _{a}\left( u\right) =u$ and $\varphi _{b}$ is non-linear.

\subsubsection{Rate double robustness for the estimator $\widehat{\protect%
\chi }_{mix}$}

\begin{theorem}
\label{theo:rate_double_mix} There exists $\lambda \asymp \sqrt{{\log \left(
p\right) }/{n}}$ such that if Algorithm \ref{Algo mixed} uses such $\lambda $
in both steps, then

(1) under Conditions NLin.L, NLin.Link.1-NLin.Link.3 and Lin.E.1 
\begin{equation}
\sqrt{n}\left\{ \widehat{\chi }_{mix}-\chi \left( \eta \right) \right\} =%
\mathbb{G}_{n}\left[ \Upsilon \left( a,b\right) \right] +O_{p}\left( \sqrt{%
\frac{s_{a}s_{b}}{n}}\log (p)\right) +o_{p}\left( 1\right) .
\label{eq:mix_rate_DR_expansion}
\end{equation}

(2) If Conditions NLin.L, NLin.Link.1-NLin.Link.3 and Lin.E hold and 
\begin{equation}
\sqrt{\frac{s_{a}s_{b}}{n}}\log (p)\rightarrow 0  \notag
\end{equation}%
then, 
\begin{equation*}
\frac{\sqrt{n}\left\{ \widehat{\chi }_{mix}-\chi \left( \eta \right)
\right\} }{\sqrt{E_{\eta }\left[ \left( \chi _{\eta }^{1}\right) ^{2}\right] 
}}\rightsquigarrow N\left( 0,1\right)
\end{equation*}%
and 
\begin{equation*}
\frac{\sqrt{n}\left\{ \widehat{\chi }_{mix}-\chi \left( \eta \right)
\right\} }{\sqrt{\widehat{V}_{mix}}}\rightsquigarrow N\left( 0,1\right)
\end{equation*}
\end{theorem}

We note that conditions NLin.Link.1-NLin.Link.3 are trivially true for $%
\varphi _{a}\left( u\right) =u.$ We state them in Theorem \ref%
{theo:rate_double_mix} to avoid redundantly writing separately a new set of
several assumptions for $\varphi _{a}$ and another for $\varphi _{b}.$

\subsubsection{Model double robustness for the estimator $\widehat{\protect%
\chi }_{mix}$}

\label{sec:mod_dr_mix}

Unlike the theorem for rate double robustness, the assumptions for the
theorems stating the model double robustness property of $\widehat{\chi }%
_{mix}$ need to be modified to reflect the fact that $b$ (i.e. the nuisance
for which the working model uses a non-linear link) is estimated only once:
whereas in the case of two non-linear links we had to make assumptions for
two probability limits, namely the limits of the first and second stage
estimators of $b,$ in the case of $\widehat{\chi }_{mix}$ we only need to
make assumptions about the first and -unique- stage estimator of $b.$
Furthermore, the regularity assumptions needed for the convergences of $\ 
\widehat{\chi }_{mix}$ to a normal distribution are different when the
correctly specified model is the one for $a$ (i.e. for the nuisance function
modeled with a linear link) than when $b$ is correctly modeled. For clarity
we state two different theorems, each assuming one of the two nuisance
functions is correctly specified.

\begin{condition}[Condition Mix.L.W]
\label{cond:Mix.L.W} Conditions NLin.L.3, NLin.L.4 and NLin.L.6 hold.
Moreover, there exists a fixed constant $K>0$ such that the following
conditions hold

\begin{itemize}
\item \textbf{(Mix.L.W.1}) There exists $\theta _{b}^{0}\in \mathbb{R}^{p}$
such that 
\begin{equation*}
\theta _{b}^{0}\in \arg \min_{\theta \in \mathbb{R}^{p}}E_{\eta }\left[
Q_{b}\left( \theta ,\phi ,w=1\right) \right]
\end{equation*}%
and $b^{0}\left( Z\right) \equiv \varphi _{b}\left( \left\langle \theta
_{b}^{0},\phi \left( Z\right) \right\rangle \right) $ belongs to $\mathcal{G}%
\left( \phi ,s_{b},j=1,\varphi _{b}\right) $ with associated parameter value
denoted as $\theta _{b}^{0\ast }$. Furthermore, $s_{b}\log \left( p\right)
/n\rightarrow 0,$ $\left\Vert \theta _{b}^{0\ast }\right\Vert _{2}\leq K$
and 
\begin{equation*}
E_{\eta }^{1/8}\left[ \left\{ b^{0}\left( Z\right) -\varphi _{b}\left(
\left\langle \theta _{b}^{0\ast },\phi \left( Z\right) \right\rangle \right)
\right\} ^{8}\right] \leq \sqrt{\frac{Ks_{b}\log \left( p\right) }{n}}.
\end{equation*}%
Additionally, $a\left( Z\right) \in \mathcal{G}\left( \phi
,s_{a},j=2,\varphi _{a}=id\right) $ with associated parameter value denoted
as $\theta _{a}^{\ast }$. Furthermore, $s_{a}\log \left( p\right) /\sqrt{n}%
\rightarrow 0,$ $\left\Vert \theta _{a}^{\ast }\right\Vert _{2}\leq K$ and 
\begin{equation*}
E_{\eta }^{1/8}\left[ \left\{ a\left( Z\right) -\left\langle \theta
_{a}^{\ast },\phi \left( Z\right) \right\rangle \right\} ^{8}\right] \leq 
\sqrt{\frac{Ks_{a}\log \left( p\right) }{n}}.
\end{equation*}

\item \textbf{(Mix.L.W.2}) 
\begin{equation*}
E_{\eta }\left[ \max\limits_{1\leq j\leq p}\left( \mathbb{P}%
_{n}(S_{ab,i}a\left( Z\right) \phi _{j}\left( Z\right) +\mathcal{R}%
_{b}\left( Z\right) \phi _{j}\left( Z\right) )^{4}\right) ^{1/2}\right] \leq
K.
\end{equation*}%
\begin{equation*}
E_{\eta }\left[ \max\limits_{1\leq j\leq p}\left( \mathbb{P}%
_{n}(S_{ab,i}b^{0}\left( Z\right) \phi _{j}\left( Z\right) +\mathcal{R}%
_{a}\left( Z\right) \phi _{j}\left( Z\right) )^{4}\right) ^{1/2}\right] \leq
K.
\end{equation*}

\item \textbf{(Mix.L.W.3) }$E_{\eta }\left( S_{ab}^{4}\right) \leq K$, $%
E_{\eta }\left( \left[ S_{ab}a\left( Z\right) +\mathcal{R}_{b}\left(
Z\right) \right] ^{4}\right) \leq K$and $E_{\eta }\left( \left[
S_{ab}b^{0}\left( Z\right) +\mathcal{R}_{a}\left( Z\right) \right]
^{4}\right) \leq K$.

\item \textbf{(Mix.L.W.4) } 
\begin{equation*}
\sup_{\Vert \theta -\theta _{b}^{0\ast }\Vert _{2}\leq 1}E\left\{
\max_{1\leq j\leq p}\left( \mathbb{P}_{n}\left\{ (\mathcal{R}_{b}\left(
Z\right) \phi _{j}\left( Z\right) {\varphi }_{b,\theta }^{\prime }\left(
Z\right) -m_{b}(O,\phi _{j}{\varphi }_{b,\theta }^{\prime }))^{2}\right\}
\right) ^{1/2}\right\} \leq K.
\end{equation*}%
\begin{equation*}
E\left\{ \max_{1\leq j\leq p}\left( \mathbb{P}_{n}\left\{ (\mathcal{R}%
_{a}\left( Z\right) \phi _{j}\left( Z\right) -m_{a}(O,\phi
_{j}))^{2}\right\} \right) ^{1/2}\right\} \leq K.
\end{equation*}
\end{itemize}
\end{condition}

\begin{remark}
If for $c=a$ or for $c=b$ there exists a statistic $S_{c}$ such that$\
m_{c}\left( O,h\right) =S_{c}h,$ for all $h$ then $\mathcal{R}_{c}\left(
Z\right) $ in conditions Mix.L.W.2 and Mix.L.W.3 can be replaced by $S_{c},$
and the inequality where $\mathcal{R}_{c}\left( Z\right) $ appears in
condition Mix.L.W.4 can be removed.
\end{remark}

\begin{condition}[Condition Mix.E.W]
\label{cond:Mix.E.W} $\ $Condition Mix.L.W.1 holds and Condition NLin.E.W
holds with $b^{0}$ instead of $b^{1}$ everywhere.
\end{condition}

We are now ready to state the theorem that establishes the asymptotic
normality of $\widehat{\chi }_{mix}$ when $b$ is the incorrectly modelled
nuisance.

\begin{theorem}
\label{theo:model_DR_mix} There exists $\lambda \asymp \sqrt{{\log \left(
p\right) }/{n}}$ such that if Algorithm \ref{Algo mixed} uses such $\lambda $
in both steps, then

(1) under Conditions Mix.L.W, NLin.Link, Mix.E.W and M.W 
\begin{equation}
\sqrt{n}\left\{ \widehat{\chi }_{mix}-\chi \left( \eta \right) \right\} =%
\mathbb{G}_{n}\left[ \Upsilon \left( a,b^{0}\right) \right] +O_{p}\left( 
\sqrt{\frac{s_{a}s_{b}}{n}}\log (p)\right) +o_{p}\left( 1\right) .  \notag
\end{equation}

(2) If Conditions Mix.L.W, NLin.Link, Mix.E.W and M.W hold and 
\begin{equation}
\sqrt{\frac{s_{a}s_{b}}{n}}\log (p)\rightarrow 0  \label{eq:sasb_mix_model}
\end{equation}%
then, 
\begin{equation*}
\frac{\sqrt{n}\left\{ \widehat{\chi }_{mix}-\chi \left( \eta \right)
\right\} }{\sqrt{E_{\eta }\left[ \left\{ \Upsilon \left( a,b^{0}\right)
-\chi \left( \eta \right) \right\} ^{2}\right] }}\rightsquigarrow N\left(
0,1\right) .
\end{equation*}

(3) If Conditions Mix.L.W, NLin.Link, Mix.E.W and M.W hold and $\left( \ref%
{eq:sasb_mix_model}\right) $ holds then 
\begin{equation}
\frac{\sqrt{n}\left\{ \widehat{\chi }_{mix}-\chi \left( \eta \right)
\right\} }{\sqrt{\widehat{V}_{mix}}}\rightsquigarrow N\left( 0,1\right) . 
\notag
\end{equation}
\end{theorem}

Next, we give the regularity conditions for the case when $a$ is the
incorrectly modelled nuisance.

\begin{condition}[Condition Mix.L.W.a]
\label{cond:Mix.L.W.a} Conditions NLin.L.3, NLin.L.4 and NLin.L.6 hold.
Moreover, there exists a fixed constant $K>0$ such that the following
conditions hold

\begin{itemize}
\item \textbf{(Mix.L.W.a.1}) $b\left( Z\right) \in \mathcal{G}\left( \phi
,s_{b},j=2,\varphi _{b}\right) $ with associated parameter value denoted as $%
\theta _{b}^{\ast }$. Furthermore, $s_{b}\log \left( p\right) /\sqrt{n}%
\rightarrow 0,$ $\left\Vert \theta _{b}^{\ast }\right\Vert _{2}\leq K$ and 
\begin{equation*}
E_{\eta }^{1/8}\left[ \left\{ b\left( Z\right) -\varphi _{b}\left(
\left\langle \theta _{b}^{\ast },\phi \left( Z\right) \right\rangle \right)
\right\} ^{8}\right] \leq \sqrt{\frac{Ks_{b}\log \left( p\right) }{n}}.
\end{equation*}

There exists $\theta _{a}^{0}\in \mathbb{R}^{p}$ such that 
\begin{equation*}
\theta _{a}^{0}\in \arg \min_{\theta \in \mathbb{R}^{p}}E_{\eta }\left[
Q_{a}\left( \theta ,\phi ,w={\varphi }_{b,\theta _{b}^{\ast }}^{\prime
}\right) \right]
\end{equation*}%
and $a^{0}\left( Z\right) \equiv \left\langle \theta _{a}^{0},\phi \left(
Z\right) \right\rangle $ belongs to $\mathcal{G}\left( \phi
,s_{a},j=1,\varphi _{a}=id\right) $ with associated parameter value denoted
as $\theta _{a}^{0\ast }$. Furthermore, $s_{a}\log \left( p\right)
/n\rightarrow 0,$ $\left\Vert \theta _{a}^{0\ast }\right\Vert _{2}\leq K$
and 
\begin{equation*}
E_{\eta }^{1/8}\left[ \left\{ a^{0}\left( Z\right) -\left\langle \theta
_{a}^{0\ast },\phi \left( Z\right) \right\rangle \right\} ^{8}\right] \leq 
\sqrt{\frac{Ks_{a}\log \left( p\right) }{n}}.
\end{equation*}

\item \textbf{(Mix.L.W.a.2}) 
\begin{equation*}
E_{\eta }\left[ \max\limits_{1\leq j\leq p}\left( \mathbb{P}%
_{n}(S_{ab,i}a^{0}\left( Z\right) \phi _{j}\left( Z\right) +\mathcal{R}%
_{b}\left( Z\right) \phi _{j}\left( Z\right) )^{4}\right) ^{1/2}\right] \leq
K.
\end{equation*}%
\begin{equation*}
E_{\eta }\left[ \max\limits_{1\leq j\leq p}\left( \mathbb{P}%
_{n}(S_{ab,i}b\left( Z\right) \phi _{j}\left( Z\right) +\mathcal{R}%
_{a}\left( Z\right) \phi _{j}\left( Z\right) )^{4}\right) ^{1/2}\right] \leq
K.
\end{equation*}

\item \textbf{(Mix.L.W.a.3) }$E_{\eta }\left( S_{ab}^{4}\right) \leq K$, $%
E_{\eta }\left( \left[ S_{ab}a^{0}\left( Z\right) +\mathcal{R}_{b}\left(
Z\right) \right] ^{4}\right) \leq K$ and $E_{\eta }\left( \left[
S_{ab}b\left( Z\right) +\mathcal{R}_{a}\left( Z\right) \right] ^{4}\right)
\leq K$.

\item \textbf{(Mix.L.W.a.4) } 
\begin{equation*}
\sup_{\Vert \theta -\theta _{b}^{\ast }\Vert _{2}\leq 1}E\left\{ \max_{1\leq
j\leq p}\left( \mathbb{P}_{n}\left\{ (\mathcal{R}_{b}\left( Z\right) \phi
_{j}\left( Z\right) {\varphi }_{b,\theta }^{\prime }\left( Z\right)
-m_{b}(O,\phi _{j}{\varphi }_{b,\theta }^{\prime }))^{2}\right\} \right)
^{1/2}\right\} \leq K.
\end{equation*}%
\begin{equation*}
E\left\{ \max_{1\leq j\leq p}\left( \mathbb{P}_{n}\left\{ (\mathcal{R}%
_{a}\left( Z\right) \phi _{j}\left( Z\right) -m_{a}(O,\phi
_{j}))^{2}\right\} \right) ^{1/2}\right\} \leq K.
\end{equation*}
\end{itemize}
\end{condition}

\begin{condition}[Condition Mix.E.W.a]
\label{cond:Mix.E.W.a} $\ $Condition Mix.L.W.a.1 holds and

\begin{itemize}
\item \textbf{(Mix.E.W.a.1)} 
\begin{equation*}
E_{\eta }\left( (S_{ab}a^{0}(Z))^{2}\right) \leq K,
\end{equation*}
\begin{equation*}
E_{\eta }\left\{ \left[ S_{ab}a^{0}\left( b^{\prime }-b\right) +m_{b}\left(
O,b^{\prime }\right) -m_{b}\left( O,b\right) \right] ^{2}\right\}
\rightarrow 0\quad \text{as}\quad E_{\eta }\left[ \left( b^{\prime
}-b\right) ^{2}\right] \rightarrow 0
\end{equation*}%
and 
\begin{equation*}
E_{\eta }\left\{ \left[ S_{ab}b\left( a^{\prime }-a^{0}\right) +m_{a}\left(
O,a^{\prime }\right) -m_{a}\left( O,a^{0}\right) \right] ^{2}\right\}
\rightarrow 0\quad \text{as}\quad E_{\eta }\left[ \left( a^{\prime
}-a^{0}\right) ^{2}\right] \rightarrow 0.
\end{equation*}

\item \textbf{(Mix.E.W.a.2)} There exists fixed constants $0<k<K$ such that

\begin{itemize}
\item[a)] $k\leq {E\left[ \left\{ \Upsilon \left( a^{0},b\right) -\chi
\left( \eta \right) \right\} ^{2}\right] }$ and $E\left\{ \left\vert
\Upsilon \left( a^{0},b\right) -\chi \left( \eta \right) \right\vert
^{3}\right\} \leq K$.

\item[b)] ${E\left[ \left\{ \Upsilon \left( a^{0},b\right) -\chi \left( \eta
\right) \right\} ^{4}\right] }\leq K$.
\end{itemize}
\end{itemize}
\end{condition}

\begin{condition}[Condition M.W.a]
\label{cond:M.W.a} There exists a linear mapping $h\in L_{2}(P_{Z,\eta
})\rightarrow m_{b}^{\ddagger }(o,h)$ such that $h\in L_{2}(P_{Z,\eta
})\rightarrow E_{\eta }\left[ m_{b}^{\ddagger }(O,h)\right] $ is continuous
with Riesz representer $\mathcal{R}_{b}^{\ddagger }$ that satisfies 
\begin{equation*}
|m_{b}(o,h)|\leq m_{b}^{\ddagger }(o,|h|)\text{ for all }o\text{ and all }%
h\in L_{2}(P_{Z,\eta })
\end{equation*}%
and 
\begin{equation*}
E_{\eta }\left[ (\mathcal{R}_{b}^{\ddagger })^{2}\right] \leq K.
\end{equation*}
\end{condition}

We are now ready to state the Theorem that establishes the asymptotic
normality of $\widehat{\chi }_{mix}$ when $a$ is the incorrectly modelled
nuisance.

\begin{theorem}
\label{theo:model_DR_mix_a}There exists $\lambda \asymp \sqrt{{\log \left(
p\right) }/{n}}$ such that if Algorithm \ref{Algo mixed} uses such $\lambda $
in both steps, then

(1) under Conditions Mix.L.W.a, NLin.Link, Mix.E.W.a.1 and M.W.a 
\begin{equation*}
\sqrt{n}\left\{ \widehat{\chi }_{mix}-\chi \left( \eta \right) \right\} =%
\mathbb{G}_{n}\left[ \Upsilon \left( a^{0},b\right) \right] +O_{p}\left( 
\sqrt{\frac{s_{b}\max(s_{a},s_{b})}{n}}\log (p)\right) +o_{p}\left( 1\right)
.
\end{equation*}

(2) If Conditions Mix.L.W.a, NLin.Link, Mix.E.W.a and M.W.a hold and 
\begin{equation}
\sqrt{\frac{s_{b}\max (s_{a},s_{b})}{n}}\log (p)\rightarrow 0
\label{eq:sasb_mix_model_a}
\end{equation}%
then, 
\begin{equation*}
\frac{\sqrt{n}\left\{ \widehat{\chi }_{mix}-\chi \left( \eta \right)
\right\} }{\sqrt{E_{\eta }\left[ \left\{ \Upsilon \left( a^{0},b\right)
-\chi \left( \eta \right) \right\} ^{2}\right] }}\rightsquigarrow N\left(
0,1\right) .
\end{equation*}

(3) If Conditions Mix.L.W.a, NLin.Link, Mix.E.W.a and M.W.a hold and $\left( %
\ref{eq:sasb_mix_model_a}\right) $ holds then 
\begin{equation*}
\frac{\sqrt{n}\left\{ \widehat{\chi }_{mix}-\chi \left( \eta \right)
\right\} }{\sqrt{\widehat{V}_{mix}}}\rightsquigarrow N\left( 0,1\right) .
\end{equation*}
\end{theorem}

\section{Literature review}

\label{sec:litrev}

There is a vast literature on non parametric estimation of ATE\ starting
from the earlier work on series estimation of the propensity and outcome
regression models to the modern approaches using machine learning techniques
for estimating these functions. Here we will restrict attention to the
proposals connected with our proposal, i.e. those in which the nuisance
functions are estimated using $\ell _{1}$ regularized methods and the
estimators have some kind of double robustness property.

As far as we know, with the exception of \cite{NeweyCherno19}, all existing
doubly robust estimation proposals based on $\ell _{1}$ regularized
regression are confined to the estimation of ATE and/or ATT. A first
distinction then is that our proposal is general in that it accommodates all
estimands in the BIF class, in particular, all the examples in Section \ref%
{sec:setup}. In addition, none of the existing proposals based on $\ell _{1}$
regularized methods have the model and rate double robustness properties
simultaneously. We will therefore review first the articles that propose
estimators with the model double robustness property and subsequently review
those that have the rate double robustness property.

Three articles, namely \cite{stijn}, \cite{Tan} and \cite{Peng}, proposed
estimators of ATE\ with the model DR property, but not the rate DR property.
In these articles the nuisance parameters $a\left( L\right) =E\left( \left.
Y\right\vert D=1,L\right) $ and $b\left( L\right) =1/P\left( \left.
D=1\right\vert L\right) $ are estimated using $\ell _{1}$ regularization. As
in our proposal, the loss functions\ for $b$ are essentially based on the
seminal result of \cite{vermeulen} that the derivative of $\chi \left( \eta
\right) +\chi _{\eta }^{1}$ with respect to $a$ is an unbiased estimating
function for $b.$ To the best of our knowledge \cite{stijn} are the first to
have noticed that this estimating function could be used in the sparse
regression setting to obtain estimators of $\chi \left( \eta \right) $ with
the model double robustness property.

The three articles only consider models that are exactly sparse for $a\left(
L\right) =E\left( \left. Y\right\vert D=1,L\right) $ and $b\left( L\right)
=1/P\left( \left. D=1\right\vert L\right) ,$ as in our Example \ref%
{ex:parametric_sparse}. Moreover, the assumptions in the theorems in these
three articles that state the asymptotic normality of the estimators of $%
\chi \left( \eta \right) $ explicitly require that the probability limits of
the estimators of $a$ and $b$ (regardless of whether or not these
probability limits are the true functions$)$ be ultra sparse functions in
the sense that, up to logarithmic terms, $s_{a}=o\left(\sqrt{n}\right) $ and 
$s_{b}=o\left( \sqrt{n}\right) $. In contrast, we prove the asymptotic
normality of our estimators requiring only that the correctly specified
nuisance function be ultra sparse, without imposing ultra sparsity of the
probability limit of the estimator of the incorrectly modelled nuisance
function.

The proposal of \cite{stijn} is based on iterating $\ell_1$ regularized
estimation of $a\left( L\right) =E\left( \left. Y\right\vert D=1,L\right) $
and $b\left( L\right) =1/P\left( \left. D=1\right\vert L\right) ,$ each time
using a special loss for one of the nuisance functions which, for non-linear
links, depends on the estimator of the other nuisance parameter at the
previous iteration. \cite{stijn} do not discuss the convergence properties
of their iterative algorithm, and at present it remains an open question
whether or not convergence is guaranteed. In contrast, our algorithm is a
two or three step (depending on whether or not the link functions are non
linear) non-iterative algorithm. At each step we solve a convex $\ell_1$
regularized optimization problem, for which there are available several
computation algorithms with known convergence guarantees (see for example 
\cite{Tseng}).

The estimators of the nuisance parameters proposed by \cite{Tan}, like ours,
are non-iterative and are based on solving convex optimization problems.
However, unlike our proposal, Tan's proposal only yields the model DR
property if the model for $a\left( L\right) =E\left( \left. Y\right\vert
D=1,L\right) $ is linear. Tan provides rigorous proofs of the asymptotic
behavior of his estimator of ATE and mentions that his proposal can be
easily extended to compute model DR estimators of ATT.

\cite{Peng} propose estimators of the nuisance parameters that, like Tan's
and ours, are non-iterative and based on solving convex optimization
problems. Unlike Tan but like our proposal, the models for $a\left( L\right)
=E\left( \left. Y\right\vert D=1,L\right) $ and $b\left( L\right) =1/P\left(
\left. D=1\right\vert L\right) $ can be non-linear. At the moment the
publicly available article does not offer proofs of the claims in it.

The \cite{Farrell}, \cite{NeweyCherno19} and \cite{cherno2} discuss estimators
with the rate DR property with nuisance functions estimated by $\ell_1$
regularized methods.

Farrell discusses estimation of population average counterfactual means
under treatments that can take more than two levels. The extension from two
to several levels of treatment is inconsequential for the points that we
want to discuss here so we will assume that treatment has two levels. The
author considers models for the propensity score and the outcome regression
that are much like our general classes of functions in Definition \ref%
{def:AGLS}. However, in his discussion of the example similar to our Example %
\ref{ex:nonparam}, he indicates that $\alpha$ must be greater than 1. In
contrast, we allow in more generality $\alpha >1/2$. We note also that
Farrell restricts the link for the outcome regression to be the identity and
the link for the propensity score to be multinomial logistic.

In Farrell the propensity score is estimated via $\ell _{1}$ regularization,
but unlike us, using the loss from the logistic regression likelihood.
Because the estimators are not computed using the special loss as in \cite%
{stijn}, \cite{Tan} and \cite{Peng}, the resulting estimators of the
counterfactual means do not have the model DR property. Furthermore, because
the procedure does not use sample splitting to estimate the nuisance
parameters, the resulting estimator of counterfactual means do not have the
rate doubly robust \ property. Nevertheless, Farrell's estimators have a one
sided rate double robustness property in the sense that they are
asymptotically normal when $s_{a}s_{b}=o\left( n\right) $ so long as $%
s_{a}=o\left( \sqrt{n}\right) $, up to logarithmic terms. That is, the
estimators of the counterfactual means are asymptotically normal even when
the propensity score is severely non-sparse provided that the outcome
regression is sufficiently ultra sparse. This one sided rate double
robustness property is a result of the fact that $b\left( L\right)
=1/P\left( D=1\mid L\right) $ is a function only of the distribution of
treatment and covariates. We note that in the statement of the properties of
the $\ell _{1}$ regularized estimators of the propensity score Farrell
assumes that $s_{b}=o\left( \sqrt{n}\right) $ up to logarithmic terms. Such
requirement defeats the very purpose of rate double robustness. As we
indicated in Section \ref{sec:asym_nonlin}, one need not impose ultra
sparsity.

\cite{NeweyCherno19} (see also \cite{cherno2}) consider estimation of parameters in a subclass of the
BIF class, specifically, those parameters for which $a$ is a conditional
expectation. These authors consider models for $b$ that are much like our
general classes of functions defined in Definition \ref{def:AGLS} but
specifically with the identity link. However, like Farrell, in their
discussion of an example similar to our Example \ref{ex:nonparam}, they
indicate that $\alpha $ must be greater than 1. Because they restrict
attention to $a$ being a conditional expectation, they consider the
possibility that $a$ be estimated via an arbitrary machine learning
algorithm. On the other hand, they estimate $b$ using $\ell _{1}$
regularization and using the same the loss function as in \cite{stijn}, \cite%
{Tan} and \cite{Peng} and the present paper. A key feature of their
procedure is that they use sample splitting. They thus obtain estimators
with the rate double robustness property. However, because they do not
necessarily assume that $a$ is estimated via $\ell _{1}$ regularization,
their estimators do not have the model DR property. They provide rigorous asymptotic
theory which follows from results in \cite{local-dr}.

\section{Discussion}

\label{sec:conclusion}

We have proposed a unified procedure for computing estimators of parameters
in the BIF class with the model and rate double robustness properties.

Interestingly, starting from a parameter in the BIF class one can construct
infinitely many parameters not in the BIF class that nevertheless admit rate
and model doubly robust estimators. Specifically, if $\chi \left( \eta
\right) $ is in the BIF class, then $\psi \left( \eta \right) =g\left( \chi
\left( \eta \right) \right) $ is not in the BIF class for any continuously
differentiable function $g\left( \cdot \right) $, because the influence
function of $\psi \left( \eta \right) $ is $\psi _{\eta }^{1}=g^{\prime
}\left( \chi \left( \eta \right) \right) \chi _{\eta }^{1}$ which does not
meet the requirements for influence functions of parameters in the BIF
class. Nevertheless, by the delta method, $\widehat{\psi }=g\left( \widehat{%
\chi }\right) $ where $\widehat{\chi }$ is one of the estimators $\widehat{%
\chi }_{lin},$ $\widehat{\chi }_{nonlin}$ or $\widehat{\chi }_{mix}$
depending on the form of the link functions, is a simultaneously rate and
model double robust estimator of $\psi \left( \eta \right) .$

Parameters in the BIF class are additively separable in the sense that their
influence function is equal to a function of the data that depends on $%
P_{\eta }$ only through two nuisance functions $a$ and $b$ minus the
parameter itself. However, there exist parameters which are defined
implicitly as solutions of population moment equations $E_{\eta }\left[
u\left( O,a,b,\chi \right) \right] =0$ involving two unknown nuisances $%
a\left( Z\right) \equiv a\left( Z;\eta \right) $ and $b\left( Z\right)
\equiv b\left( Z;\eta \right) ,$ and such that $E_{\eta }\left[ u\left(
O,a^{\prime },b,\chi \right) \right] =E_{\eta }\left[ u\left( O,a,b^{\prime
},\chi \right) \right] =0$ for any $a^{\prime }\left( Z\right) \equiv
a\left( Z;\eta ^{\prime }\right) $ and $b^{\prime }\left( Z\right) \equiv
b\left( Z;\eta ^{\prime }\right) $ with $\eta \not=\eta ^{\prime }$ (see,
for instance, \cite{oddsratio}). These identities suggest that estimators
that simultaneously have the rate and model double robustness property may
exist, it remains an open question if this is true and in such case how to
construct them. In principle, applying ideas similar to the one in this
article one could construct tests of the pointwise hypothesis $H_{0}:\chi
\left( \eta \right) =\chi ^{\ast }$ for an arbitrary $\chi ^{\ast }$ that
are model and rate double robust and then invert the tests to obtain a
confidence interval for $\chi \left( \eta \right) .$ However, this approach
has two difficulties. First, for $\ell _{1}$ regularized estimation of the
nuisance functions, this proposal will be impractical as it will involve
carrying out many tests over a grid of possible values of $\chi \left( \eta
\right) .$ Second, working models for $a$ and $b$ will depend on the
specific value $\chi ^{\ast }$ of $\chi \left( \eta \right) $ of the null
hypothesis being tested. It is unclear what model double robustness will
mean in such case.

Finally, in \cite{globalclass} it is shown that for parameters in the BIF
class there exist two loss functions, one whose expectation is minimized at $%
a$ and another whose expectation is minimized at $b.$ This opens the
possibility of constructing machine-learning loss-based estimators of $a$
and $b,$ such as support vector machine estimators and raises the question
of whether estimators of $\chi \left( \eta \right) $ that are simultaneously
rate and model doubly robust can be constructed based on these machine
learning estimators. 

\section*{Acknowledgments}

Ezequiel Smucler was partially supported by Grant 20020170100330BA from
Universidad de Buenos Aires and PICT-201-0377 from ANPYCT, Argentina. He
wants to thank Aristas S.R.L. for providing the computational resources for
running the simulations.

\newpage

\section{Appendix A: Asymptotic results for the estimators of the nuisance
functions}

In this section we study the asymptotic properties of a class of $\ell_1$%
-regularized regression estimators (see \eqref{eq:def_estim}), that includes
the estimators of the nuisance parameters used in the algorithms described
in Section \ref{sec:doubly_robust_algo}. The rate of convergence derived
from applying Theorem \ref{theo:rate_main} to $\ell_1$-regularized maximum
likelihood estimators for GLMs with a canonical link is, to the best of our
knowledge, novel. Thus, Theorem \ref{theo:rate_main} may be of independent
interest.

In order to obtain a single result that accommodates the behavior of the
estimators appearing in all steps of the algorithms in Section \ref%
{sec:doubly_robust_algo}, we will consider the more general setting in which
the weights defined using $\varphi _{a}^{\prime }$ and $\varphi _{b}^{\prime
}$ are replaced by a general weight function. Moreover, we will work with a
general i.i.d. sample $O_{i}=(Z_{i},Y_{i})$, $i=1,\dots ,n$, which need not
be the full sample used in Section \ref{sec:doubly_robust_algo}.

We will consider triangular array asymptotics, and hence unless explicitly
stated, all quantities appearing in what follows may depend on $n$. However,
to keep the notation simple, this dependence will not be made explicit in
general. All expectations are taken with respect to the law of $O$. We will
work in the following setting.

\begin{condition}
\label{cond:setting_nuis}

\begin{itemize}
\item[a)] $S_{ab}$ is a known statistic such that $P\left(S_{ab}\geq 0
\right)=1$.

\item[b)] $h\in L_{2}(P_{Z,\eta})\to m_{a}(O,h)$ and $h\in
L_{2}(P_{Z,\eta})\to m_{b}(O,h)$ are linear maps with probability one.

\item[c)] $h\in L_{2}(P_{Z,\eta}) \to E\left[m_{a}(O,h)\right]$ and $h\in
L_{2}(P_{Z,\eta})\to E\left[m_{b}(O,h)\right]$ are continuous with Riesz
representers $\mathcal{R}_{a}$ and $\mathcal{R}_{b}$ respectively.

\item[d)] There exist $a\in L_{2}(P_{Z,\eta})$ and $b\in L_{2}(P_{Z,\eta})$
such that $E(S_{ab}\mid Z)a(Z)\in L_{2}(P_{Z,\eta})$, $E(S_{ab}\mid
Z)b(Z)\in L_{2}(P_{Z,\eta})$ and 
\begin{equation}
E\left[S_{ab}a h + m_{b}(O,h) \right]=0 \quad \text{and} \quad E\left[%
S_{ab}b h + m_{a}(O,h) \right]=0 \quad \text{for all } h \in
L_{2}(P_{Z,\eta}).  \notag
\end{equation}
\end{itemize}
\end{condition}

Condition \ref{cond:setting_nuis} holds in particular when $\chi(\eta)$ is a
functional in the BIF class with an influence function of the form %
\eqref{eq:mainIF} and $P\left(S_{ab}\geq 0 \right)=1$. The case in which $%
P\left(S_{ab} \leq 0 \right)=1$ can be handled similarly.

Our goal is to study estimates of $a$ and $b$ built using $\ell_1$
regularization, under possible model misspecification. Since the problem is
completely symmetric in $a$ and $b$, we will study estimators for a $c$
equal to $a$ or to $b$. If $c=a$ we let $\overline{c}=b$, if $c=b$ we let $%
\overline{c}=a$.

We let $X_{i,j}=\phi_{j}(Z_{i})$, $j=1,\dots,p$ and $X_{i}=\left(X_{i,1},%
\dots,X_{i,p}\right)^{\top}$, $i=1,\dots, n$. Given $z\in\mathbb{R}^{d}$,
let 
\begin{equation*}
\phi(z)=\left(\phi_{1}(z),\dots,\phi_{p}(z) \right)^{\top}.
\end{equation*}
Finally, we let $m_{a,i}(h)=m_{a}(O_{i},h)$ and $m_{b,i}(h)=m_{b}(O_{i},h)$, 
$i=1,\dots, n$.

Let $\varphi_{c}$ be a link function and let $\psi_{c}=\int \varphi_{c}$.
Let ${w}(u)$ be a fixed non-negative weight function, let $\widehat{\beta}\in%
\mathbb{R}^{p}$ be random vector and for $\beta\in\mathbb{R}^{p}$ let $%
w_{\beta}=w(\langle \beta, X\rangle)$. Let 
\begin{equation}
Q_{c}(\theta, \phi, w_{\widehat{\beta}}) = S_{ab} w_{\widehat{\beta}}
\psi_{c}(\langle \theta, X \rangle)+ \langle \theta, m_{\overline{c}}(w_{%
\widehat{\beta}} \phi) \rangle  \label{eq:def_Q}
\end{equation}
and 
\begin{equation*}
L_{c}(\theta, \phi, w_{\widehat{\beta}})= \mathbb{P}_{n}\left\lbrace
Q_{c}(\theta, \phi, w_{\widehat{\beta}}) \right\rbrace.
\end{equation*}
We consider estimates defined by 
\begin{equation}
\widehat{\theta}_c \in \arg\min_{\theta\in\mathbb{R}^{p}} L_{c}(\theta,
\phi, w_{\widehat{\beta}}) + \lambda_{c} \Vert \theta \Vert_{1}.
\label{eq:def_estim}
\end{equation}
We emphasize that there may be more than one solution in general to %
\eqref{eq:def_estim} and that our results hold for any one of them.

We use $k, K$ to denote fixed (\emph{i.e. not changing with $n$}) positive
constants that appear in lower and upper bounds respectively in the
assumptions we need to obtain rates of convergence for $\widehat{\theta}_{c}$%
. The assumptions we make are grouped into those regarding the link
functions in Condition \ref{cond:glm}, those regarding the weight function
in Condition \ref{cond:weight}, those regarding the nuisance functions in
Condition \ref{cond:approx_sparse} for the linear case and Condition \ref%
{cond:approx_sparse_S} for the (possibly) non-linear case, and finally those
regarding the data generating process in Condition \ref{cond:subgauss_lin}
for the linear case and Condition \ref{cond:subgauss} for the (possibly)
non-linear case. All these assumptions appeared, organized differently, in
Section \ref{sec:asym_results_chi}, since they are needed to prove our
results regarding the asymptotic properties of the estimators of $\chi(\eta)$
we propose. In particular the assumptions were already discussed in Section %
\ref{sec:asym_results_chi}.

We will need the following assumptions on $\varphi_{c}$

\begin{condition}
\label{cond:glm}

\begin{itemize}
\item[a)] $\varphi_{c}$ is continuously differentiable and $\dot{\varphi}%
_c(u) > 0$ for all $u\in\mathbb{R}$.

\item[b)] $\vert {\varphi}_c(u)- {\varphi}_c(v)\vert \leq K\exp\left( K(
\vert u\vert + \vert v \vert) \right) \vert u - v \vert$ for all $u,v \in 
\mathbb{R}$.
\end{itemize}
\end{condition}

The assumptions on the link functions are satisfied by $\exp(u), -\exp(-u),
G(u)$ and $-1/G(u)$, where $G(u)=\exp(u)/(1+\exp(u))$ is the expit function.
The assumptions are also clearly satisfied by the identity function.

We will need the following assumption on the weights to ensure that the
optimization problem \eqref{eq:def_estim} is well behaved. In particular,
under this assumption and Condition \ref{cond:glm} a), the optimization
problem in \eqref{eq:def_estim} is convex and can be solved efficiently, for
example, using coordinate descent optimization \citep{glmnet}.

\begin{condition}
\label{cond:weight}

\begin{enumerate}
\item[a)] $w(u)> 0$ for all $u \in \mathbb{R}$.

\item[b)] $\vert w(u)- w(v)\vert \leq K\exp\left( K (\vert u\vert + \vert v
\vert) \right) \vert u - v \vert$.

\item[c)] $\widehat{\beta}$ is independent of the data and there exists $%
\beta^{\ast}$ such that 
\begin{equation*}
\Vert\widehat{\beta} - \beta^{\ast}\Vert_{2}=O_{P}\left(\sqrt{\frac{%
s_{\beta}\log(p)}{n}}\right),\quad \frac{s_{\beta}\log(p)}{n}\to 0 \quad 
\text{ and } \quad \Vert \beta^{\ast}\Vert_{2}\leq K.
\end{equation*}
\end{enumerate}
\end{condition}

Conditions \ref{cond:glm} a) and b) are satisfied by $\exp(u), \exp(-u),
G(u) $ and $1/G(u)$, where $G(u)=\exp(u)/(1+\exp(u))$ is the expit function.
Part c) requires that $\widehat{\beta}$ converge to a limit $\beta^{\ast}$.

At this point we split the analysis of the estimates defined by %
\eqref{eq:def_estim} according to whether $\varphi_{c}$ is the identity
function or an arbitrary function satisfying Condition \ref{cond:glm}. We do
this because the analysis of the latter case is more involved and requires
stronger assumptions.

\subsubsection{The linear case}

\label{sec:linear_case} Throughout this subsection we assume that $%
\varphi_{c}(u)=u$ and $w(u)\equiv 1$. The general case is dealt with in
subsection \ref{sec:nonlinear_case}.

When $c\left( Z\right) \in \mathcal{G}\left( \phi ,s_{c},j=1,\varphi
=id\right) $ with associated parameter value denoted as $\theta _{c}^{\ast }$%
, under technical assumptions $\widehat{\theta }_{c}$ will converge to $%
\theta _{c}^{\ast }$. However, when the model for $c$ is misspecified, again
under technical assumptions, $\widehat{\theta }_{c}$ will converge to the
minimizer of $E_{\eta }\left[ Q_{c}\left( \theta ,\phi ,w=1\right) \right] $%
. In order to accommodate both asymptotic behaviors in a single theoretical
result we will use the same notation for possibly different quantities that
appear in the asymptotic analysis of $\widehat{\theta }_{c}$, according to
whether the model is correctly specified or not. Hence, for example, in
Condition \ref{cond:approx_sparse} below the meaning of $\theta _{c}^{\ast }$
depends on whether the model for $c$ is correctly specified or not.

\begin{condition}
\label{cond:approx_sparse} Condition \ref{cond:setting_nuis} holds and
either a) or b) holds:

\begin{enumerate}
\item[a)] $c\left( Z\right) \in \mathcal{G}\left( \phi ,s_{c},j=1,\varphi
=id\right) $ with associated parameter value denoted as $\theta _{c}^{\ast}$%
. Furthermore, 
\begin{equation*}
s_{c}\log \left( p\right) /n\rightarrow 0.
\end{equation*}

\item[b)] There exists $\theta_{c}\in \mathbb{R}^{p}$ such that 
\begin{equation*}
\theta _{c}\in \arg \min_{\theta \in \mathbb{R}^{p}}E_{\eta }\left[
Q_{c}\left( \theta ,\phi ,w=1\right) \right]
\end{equation*}%
and the function $\left\langle \theta _{c},\phi \left( Z\right)
\right\rangle $ belongs to $\mathcal{G}\left( \phi ,s_{c},j=1,\varphi
=id\right) $ with associated parameter value denoted as $\theta_{c}^{\ast }$%
. Furthermore, $s_{c}\log \left( p\right) /n\rightarrow 0$.
\end{enumerate}
\end{condition}

Recall that 
\begin{align*}
\widehat{\Sigma}_{2}=\mathbb{P}_{n}\left\lbrace
S_{ab}XX^{\prime}\right\rbrace,\quad \Sigma_{2} = E\left\lbrace
S_{ab}XX^{\prime}\right\rbrace,
\end{align*}
\begin{align*}
&\widehat{\Sigma}_{1}=\mathbb{P}_{n}\left\lbrace XX^{\prime}\right\rbrace
\quad \text{ and } \quad \Sigma_{1} = E\left\lbrace XX^{\prime}\right\rbrace.
\end{align*}
Let $S_{c}$ be the support of $\theta_{c}^{\ast}$ in Condition \ref%
{cond:approx_sparse}.

We will need the following assumptions.

\begin{condition}
\label{cond:subgauss_lin} Condition \ref{cond:setting_nuis} holds and

\begin{enumerate}
\item[a)] If Condition \ref{cond:approx_sparse} a) holds 
\begin{equation*}
E\left[ \max_{1\leq j\leq p}\sqrt{\mathbb{P}_{n}\left[ \left\{ S_{ab}c\left(
Z\right) X_{i,j} +m_{\overline{c}}\left( O,\phi _{j}\right) \right\} ^{2}%
\right] }\right] \leq K.
\end{equation*}

If Condition \ref{cond:approx_sparse} b) holds 
\begin{equation*}
E\left[ \max_{1\leq j\leq p}\sqrt{\mathbb{P}_{n}\left[ \left\{ S_{ab}\langle
\theta_{c},X_{i} \rangle X_{i,j} +m_{\overline{c}}\left( O,\phi _{j}\right)
\right\} ^{2}\right] }\right] \leq K.
\end{equation*}

\item[b)] For sufficiently large $n$ 
\begin{equation*}
k\leq \kappa_{l}({\Sigma}_{1}, \left\lceil s_{c}\log(n) \right\rceil, S_{c})
\leq \kappa_{u}({\Sigma}_{1}, \left\lceil s_{c}\log(n)\right\rceil, S_{c})
\leq K.
\end{equation*}

\item[c)] 
\begin{equation*}
k\leq \kappa_{l}(\widehat{\Sigma}_{2}, \left\lceil s_{c}\log(n)/2
\right\rceil,S_{c}) \leq \kappa_{u}(\widehat{\Sigma}_{2}, \left\lceil
s_{c}\log(n)/2\right\rceil, S_{c}) \leq K,
\end{equation*}
with probability tending to one.

\item[d)] $k\leq E\left( S_{ab}\vert Z\right) \leq K$ almost surely.
\end{enumerate}
\end{condition}

Theorem \ref{theo:rate_main_lin} can be proven by a straightforward
adaptation of the techniques used by \cite{Bickel-Lasso} and \cite%
{BelloniCherno11}. For this reason, we omit the proof.

\begin{theorem}
\label{theo:rate_main_lin} Assume $\varphi _{c}(u)=\varphi _{\overline{c}%
}(u)=u$ and $w(u)=1$. Assume Conditions \ref{cond:setting_nuis}, \ref%
{cond:approx_sparse} and \ref{cond:subgauss_lin} hold. Then there exists $%
\lambda _{c}\asymp \sqrt{{\log \left( p\right) }/{n}}$ such that

\begin{itemize}
\item 
\begin{equation*}
\Vert \widehat{\theta}_{c} - \theta^{\ast}_{c}\Vert_{{\Sigma}_{1}}
=O_{P}\left( \sqrt{\frac{s_{c}\log(p)}{n}}\right) ,
\end{equation*}

\item 
\begin{equation*}
\Vert \widehat{\theta}_{c} - \theta^{\ast}_{c}\Vert_{2}=O_{P}\left( \sqrt{%
\frac{s_{c}\log(p)}{n}}\right) ,
\end{equation*}

\item 
\begin{equation*}
\Vert \widehat{\theta}_{c} - \theta^{\ast}_{c}\Vert_{1} =O_{P}\left( s_{c} 
\sqrt{\frac{\log(p)}{n}}\right).
\end{equation*}
\end{itemize}
\end{theorem}

\subsubsection{The general case}

\label{sec:nonlinear_case}

To handle the case in which $\varphi_{c}$ is non-linear, we will need the
following modification of Condition \ref{cond:approx_sparse}.

\begin{condition}
\label{cond:approx_sparse_S} Condition \ref{cond:setting_nuis} holds and
either a) or b) holds:

\begin{enumerate}
\item[a)] $c\left( Z\right) \in \mathcal{G}\left( \phi ,s_{c},j=1,\varphi
_{c}\right) $ with associated parameter value denoted as $\theta _{c}^{\ast
} $. Furthermore, $s_{c}\log \left( p\right) /n\rightarrow 0$, $\left\Vert
\theta _{c}^{\ast }\right\Vert _{2}\leq K$ and 
\begin{equation*}
E^{1/8}\left[ \left\{ c\left( Z\right) -\varphi _{c}\left( \left\langle
\theta _{c}^{\ast },\phi \left( Z\right) \right\rangle \right) \right\} ^{8}%
\right] \leq \sqrt{\frac{Ks_{c}\log \left( p\right)}{n}}.
\end{equation*}

\item[b)] 

There exists $\theta_{c}\in \mathbb{R}^{p} $ such that 
\begin{equation*}
\theta_{c}\in \arg \min_{\theta \in \mathbb{R}^{p}}E_{\eta }\left[
Q_{c}\left( \theta ,\phi ,w_{\beta^{\ast}}\right) \right]
\end{equation*}%
and $\varphi_{c}\left( \left\langle \theta _{c},\phi \left( Z\right)
\right\rangle \right) $ belongs to $\mathcal{G}\left( \phi
,s_{c},j=1,\varphi _{c}\right) $ with associated parameter value denoted as $%
\theta _{c}^{\ast}$. Furthermore, $s_{c}\log \left( p\right) /n\rightarrow 0$%
, $\left\Vert \theta _{c}\right\Vert _{2}\leq K$, $\left\Vert \theta
_{c}^{\ast }\right\Vert _{2}\leq K$ and 
\begin{equation*}
E^{1/8}\left[ \left\{ \varphi_{c}\left( \left\langle \theta _{c},\phi \left(
Z\right) \right\rangle \right) -\varphi _{c}\left( \left\langle \theta
_{c}^{\ast},\phi \left( Z\right) \right\rangle \right) \right\} ^{8}\right]
\leq \sqrt{\frac{Ks_{c}\log \left( p\right)}{n}}.
\end{equation*}
\end{enumerate}
\end{condition}

If Condition \ref{cond:approx_sparse_S} a) holds we let $\varphi^{%
\ast}_{i}=c(Z_i) \quad i=1,\dots, n $ and if Condition \ref%
{cond:approx_sparse_S} b) holds we let $\varphi^{\ast}_{i}=\varphi_{c}\left(%
\langle \theta_{c}, \phi(Z_{i})\rangle\right), i=1,\dots, n $ We emphasize
that the meanings of $\theta^{\ast}_{c}, \theta_{c}$ and $\varphi^{\ast}$
depend on whether Condition \ref{cond:approx_sparse_S} a) or b) holds. We
will also need the following assumptions.

\begin{condition}
\label{cond:subgauss} Condition \ref{cond:setting_nuis} holds and

\begin{enumerate}
\item[a)] 
\begin{equation*}
E\left[\max\limits_{1\leq j \leq p} \left( \frac{1}{n}\sum\limits_{i=1}^{n}
(S_{ab,i} \varphi^{\ast} X_{i,j} + \mathcal{R}_{\overline{c}} X_{i,j}
)^{4}\right)^{1/2} \right] \leq K.
\end{equation*}

\item[b)] $\sup_{\Vert \Delta\Vert_{2}=1}\Vert \langle \Delta, X
\rangle\Vert_{\psi_2}\leq K$.

\item[c)] 
\begin{equation*}
k\leq \lambda_{min}\left(\Sigma_{1} \right) \quad \text{and} \quad k\leq
\lambda_{min}\left(\Sigma_{2} \right)
\end{equation*}

\item[d)] $E^{1/4}\left(S_{ab}^{4} \right)\leq K$.

\item[e)] $E^{1/2}\left( \left[ S_{ab}\varphi^{\ast}+\mathcal{R}_{\overline{c%
}} \right]^{4}\right) \leq K$.

\item[f)] 
\begin{equation*}
\sup_{\Vert \beta - \beta^{\ast}\Vert_{2}\leq 1}E \left\lbrace \max_{1\leq j
\leq p} \left( \frac{1}{n} \sum\limits_{i=1}^{n} (\mathcal{R}_{\overline{c}%
,i} X_{i,j} w(\langle \beta, X_{i}\rangle) - m_{\overline{c},i}(\phi_{j}
w_{\beta}))^{2} \right)^{1/2} \right\rbrace \leq K.
\end{equation*}
\end{enumerate}
\end{condition}

\begin{remark}
If there exists a statistic $S_{\overline{c}}$ such that $m_{\overline{c}%
}(O,h)=S_{\overline{c}}h$ then $\mathcal{R}_{\overline{c}}$ can be replaced
by $S_{\overline{c}}$ in Conditions \ref{cond:subgauss} a) and e). Moreover,
in this case Condition \ref{cond:subgauss} f) can be removed.
\end{remark}

\begin{theorem}
\label{theo:rate_main} Assume Conditions \ref{cond:setting_nuis}, \ref%
{cond:glm}, \ref{cond:weight}, \ref{cond:approx_sparse_S} and \ref%
{cond:subgauss}. Then, there exists $\lambda _{c}\asymp \sqrt{{\log \left(
p\right) }/{n}}$ such that:

(i) If Condition \ref{cond:approx_sparse_S} a) holds or if Condition \ref%
{cond:approx_sparse_S} b) holds and $w(u)=1$ then 
\begin{equation*}
\Vert \widehat{\theta }-\theta _{c}^{\ast }\Vert _{2}=O_{P}\left( \sqrt{%
\frac{s_{c}\log (p)}{n}}\right) ,\quad \Vert \widehat{\theta }-\theta
_{c}^{\ast }\Vert _{1}=O_{P}\left( s_{c}\sqrt{\frac{\log (p)}{n}}\right) .
\end{equation*}%
(ii) For a general weight function $w(u)$, if Condition \ref%
{cond:approx_sparse_S} b) holds then 
\begin{equation*}
\Vert \widehat{\theta }-\theta _{c}^{\ast }\Vert _{2}=O_{P}\left( \sqrt{%
\frac{\max (s_{c},s_{\beta })\log (p)}{n}}\right) ,\quad \Vert \widehat{%
\theta }-\theta _{c}^{\ast }\Vert _{1}=O_{P}\left( \max (s_{c},s_{\beta })%
\sqrt{\frac{\log (p)}{n}}\right) .
\end{equation*}
\end{theorem}

\begin{remark}
When Condition \ref{cond:approx_sparse_S} b) holds and the weight function
is not constantly equal to one, the rate of convergence of $\widehat{\beta}$
may affect the rate of convergence of $\widehat{\theta}_{c}$. This is due to
the fact that in this case (and in contrast to what happens when Condition %
\ref{cond:approx_sparse_S} a) holds) the very definitions of $\theta_{c}$
and $\theta^{\ast}_{c}$ actually depend on the `limit' of $\widehat{\beta}$.
\end{remark}

\begin{remark}
\label{remark:comparison_negahban} Suppose $O_{i}=(Z_{i},Y_{i})$ follows a
generalized linear model with outcome $Y_{i}$ and covariate $Z_{i}$, with
canonical link function $\varphi _{a}^{-1}$ and regression parameter $\theta
_{a}$ that satisfies $\Vert \theta _{a}\Vert _{2}\leq K$. Then if $S_{ab}=1$%
, $m_{b}(O,h)=-Y_{i}h$, $\phi (Z)=Z$, $w(u)\equiv 1$, $\widehat{\beta }%
\equiv \beta ^{\ast }\equiv 0$ and $a(Z)=E\left( Y|Z\right) $ we have that $%
\widehat{\theta }_{a}$ is the $\ell _{1}-$penalized maximum likelihood
estimator of $\theta _{a}$. Hence Theorem \ref{theo:rate_main} also provides
rates of convergence for $\ell _{1}-$penalized maximum likelihood estimates.
As we mentioned earlier (see Example \ref{ex:parametric_weakly_sparse}), our
results cover the case in which $R_{q}\left( {\log (p)}/{n}\right)
^{1-q/2}\rightarrow 0,$ where $R_{q}=\sum_{j=1}^{p}|\theta _{a,j}|^{q}$ and $%
q\in (0,2)$, by taking $s_{a}=\left( n/\log (p)\right) ^{q/2}R_{q}$. In this
case, Theorem \ref{theo:rate_main} yields a rate of convergence of 
\begin{equation*}
R_{q}^{1/2}\left( \frac{\log (p)}{n}\right) ^{1/2-q/4}
\end{equation*}%
for $\Vert \widehat{\theta }_{a}-\theta _{a}\Vert _{2}$. Hence, Theorem \ref%
{theo:rate_main} extends the results of \cite{NegahbanSS} (see their
Corollary 3 and Section 4,4) which require $q<1$. The key point that allows
us to obtain this rate for $q\in (1,2)$ is that we essentially require that $%
\theta _{a}$ be well approximated in $\ell _{2}$ norm by a sparse vector,
rather than in $\ell _{1}$ norm. \cite{Raskutti-minimax} prove that these
rates are minimax for linear regression and $q<1$. \cite{donoho1994minimax}
prove a related result (but with a sharp control of the constants involved)
covering the case $q\in (0,2)$ for the simpler gaussian sequence model.
\end{remark}

\subsection{Proofs}

Here we provide the proofs of all our results regarding the estimates of the
nuisance parameters. In Section \ref{sec:det} we prove deterministic
results, providing the claimed rate of convergence of $\widehat{\theta}_{c}$
under high-level assumptions on the size of the penalty parameter, on the
approximation error, and the rate of convergence of $\widehat{\beta}$. In
Section \ref{sec:cond_high_prob} we show that these high-level assumptions
hold with high probability under Conditions \ref{cond:setting_nuis}, \ref%
{cond:glm}, \ref{cond:weight}, \ref{cond:approx_sparse_S} and \ref%
{cond:subgauss}.

\subsection*{Notation}

We introduce further notation. Let 
\begin{align*}
&\varrho_{i}=\varphi_{c}(\langle \theta^{\ast}_{c}, X_{i} \rangle) -
\varphi^{\ast}_{i} \quad i=1,\dots, n, \\
& \widetilde{\varrho}_{i}= S_{ab,i}w(\langle \widehat{\beta}, X_i\rangle)
\varrho_{i} \quad i=1,\dots, n, \\
& \Vert\widetilde{\varrho}\Vert_{2,n}= \left(\mathbb{P}_{n}\lbrace 
\widetilde{\varrho}^{2} \rbrace\right)^{1/2}, \\
& R(\beta)= \mathbb{P}_{n}\left( X w_{\beta} S_{ab} \varrho\right), \\
& \widehat{\Delta}=\widehat{\theta}_{c}-\theta^{\ast}_{c}, \\
& S_{c}=\left\lbrace j\in\left\lbrace 1,\dots,p\right\rbrace :
\theta^{\ast}_{cj}\neq 0\right\rbrace, \\
& \overline{S}_{c}=\left\lbrace j\in\left\lbrace 1,\dots,p\right\rbrace :
\theta^{\ast}_{cj} =0\right\rbrace.
\end{align*}
Note that $s_c = \# S_{c}$. If Condition \ref{cond:approx_sparse_S} a) holds
we let 
\begin{equation*}
J= \Vert \mathbb{P}_{n}\left\lbrace X w(\langle \widehat{\beta}, X\rangle) %
\left[S_{ab}\varphi^{\ast} +\mathcal{R}_{\overline{c}} \right] \right\rbrace
\Vert_{\infty} + \Vert \mathbb{P}_{n}\left\lbrace m_{\overline{c}}(\phi w_{%
\widehat{\beta}}) -\mathcal{R}_{\overline{c}} X w(\langle \widehat{\beta},
X\rangle) \right\rbrace \Vert_{\infty} \quad \text{and} \quad H=0
\end{equation*}
and if Condition \ref{cond:approx_sparse_S} b) holds we let 
\begin{align*}
&J= \Vert \mathbb{P}_{n}\left\lbrace X w(\langle \beta^{\ast}, X\rangle) %
\left[S_{ab}\varphi^{\ast} +\mathcal{R}_{\overline{c}} \right] \right\rbrace
\Vert_{\infty} + \Vert \mathbb{P}_{n}\left\lbrace m_{\overline{c}}(\phi
w_{\beta^{\ast}}) -\mathcal{R}_{\overline{c}} X w(\langle \beta^{\ast},
X\rangle) \right\rbrace \Vert_{\infty}, \\
&H=\mathbb{P}^{1/2}_{n}\left\lbrace \left[(w_{\widehat{\beta}%
}-w_{\beta^{\ast}})\left(S_{ab}\varphi^{\ast} + \mathcal{R}_{\overline{c}%
}\right)\right]^{2} \right\rbrace.
\end{align*}

\subsection{Deterministic results}

\label{sec:det}

In this section we prove deterministic results for $\widehat{\theta}_{c}$.
Assuming essentially that: $\lambda_{c}$ overrules the noise level as
measured by $J$, $\widehat{\beta}$ is close to $\beta^{\ast}$, the
approximation error as measured by $\Vert \widetilde{\varrho} \Vert_{2}$ is
small and ${\widehat{\Sigma}_{1}}$ satisfies a restricted upper eigenvalue
type condition, then $\widehat{\theta}_{c}$ lies in a special cone and
satisfies the rates of convergence in Theorem \ref{theo:rate_main}. Later on
in Section \ref{sec:cond_high_prob} we will show that these conditions hold
with high probability under the assumptions made in the paper.

We begin with a known bound (see e.g., Lemma 3 from \cite{NegahbanSS}) on
the difference of the $\ell_{1}$ norms of $\theta^{\ast}_{c}+\Delta$ and $%
\theta^{\ast}_{c}$.

\begin{lemma}
\label{lemma:bound_l1_diff} For any $\Delta \in \mathbb{R}^{p}$ 
\begin{equation*}
\Vert \theta^{\ast}_{c}+\Delta \Vert_{1}-\Vert \theta^{\ast}_{c} \Vert_{1}
\geq \Vert {\Delta}_{\overline{S}_{c}} \Vert_{1}-\Vert{\Delta}%
_{S_{c}}\Vert_{1}
\end{equation*}
\end{lemma}

\begin{proof}
\lbrack Proof of Lemma \ref{lemma:bound_l1_diff}] 
\begin{align*}
\Vert \theta^{\ast}_{c} + {\Delta}\Vert_{1}&=\Vert(\theta^{%
\ast}_{c})_{S_{c}} + (\theta^{\ast}_{c})_{\overline{S}_{c}}+ {\Delta}%
_{S_{c}}+ {\Delta}_{\overline{S}_{c}} \Vert_{1} \\
& \geq \Vert (\theta^{\ast}_{c})_{S_{c}} + {\Delta}_{\overline{S}_{c}}
\Vert_{1}- \Vert (\theta^{\ast}_{c})_{\overline{S}_{c}} +{\Delta}_{S_{c}}
\Vert_{1} \\
& = \Vert (\theta^{\ast}_{c})_{S_{c}} +{\Delta}_{\overline{S}_{c}}
\Vert_{1}- \Vert {\Delta}_{S_{c}} \Vert_{1} \\
& = \Vert (\theta^{\ast}_{c})_{S_{c}}\Vert_{1} + \Vert{\Delta}_{\overline{S}%
_{c}} \Vert_{1}- \Vert {\Delta}_{S_{c}} \Vert_{1}
\end{align*}
Hence 
\begin{equation*}
\Vert \theta^{\ast}_{c}+\Delta \Vert_{1}-\Vert \theta^{\ast}_{c} \Vert_{1}
\geq \Vert {\Delta}_{S^{c}} \Vert_{1}-\Vert{\Delta}_{S_{c}}\Vert_{1}.
\end{equation*}
\end{proof}

Given constants $C,s>0$ define the set 
\begin{equation*}
\mathbb{D}(C,s)=\left\lbrace \Delta \in \mathbb{R}^{p} : \Vert
\Delta\Vert_{1}\leq C \sqrt{s} \Vert \Delta\Vert_{2} \right\rbrace.
\end{equation*}
It is easy to see that $\mathbb{D}(C,s)$ is a cone: if $\Delta \in \mathbb{D}%
(C,s) $ then $t\Delta \in \mathbb{D}(C,s) $ for all positive $t$. We aim at
showing that $\widehat{\Delta} \in \mathbb{D}(C,s)$ for some $C,s>0$. We
will need the following preliminary result.

\begin{lemma}
\label{lemma:pre_cone} Assume Conditions \ref{cond:setting_nuis}, \ref%
{cond:glm} a), \ref{cond:weight} a) and \ref{cond:approx_sparse_S} hold.
Then 
\begin{equation}
\lambda_{c} \Vert \widehat{\Delta}_{\overline{S}_{c}} \Vert_{1} \leq
\lambda_{c} \Vert \widehat{\Delta}_{S_{c}}\Vert_{1} + J \Vert \widehat{\Delta%
} \Vert_{1} + H\Vert \widehat{\Delta}\Vert_{\widehat{\Sigma}_{1}} + \Vert 
\widehat{\Delta}\Vert_{\widehat{\Sigma}_{1}} \Vert \widetilde{\varrho}%
\Vert_{2,n}  \label{eq:pre_cone}
\end{equation}
\end{lemma}

\begin{proof}
The definition of $\widehat{\Delta}$ and Lemma \ref{lemma:bound_l1_diff}
imply 
\begin{equation}
L_{c}(\theta^{\ast}_{c}+\widehat{\Delta}, \phi, w_{\widehat{\beta}%
})-L_{c}(\theta^{\ast}_{c}, \phi, w_{\widehat{\beta}}) \leq
\lambda_{c}\left( \Vert \theta^{\ast}_{c} \Vert_{1}- \Vert \widehat{\theta}%
_{c} \Vert_{1}\right)\leq \lambda_{c}\left( \Vert \widehat{\Delta}%
_{S_{c}}\Vert_{1} - \Vert \widehat{\Delta}_{\overline{S}_{c}}
\Vert_{1}\right).  \label{eq:bound_L_def}
\end{equation}
By Conditions \ref{cond:glm} a) and \ref{cond:weight} a) ($\psi_{c}$ is
convex and $w$ is positive), the convexity of $L_{c}(\cdot, \phi, w_{%
\widehat{\beta}})$ implies that 
\begin{equation}
L_{c}(\theta^{\ast}_{c}+\widehat{\Delta}, \phi, w_{\widehat{\beta}%
})-L_{c}(\theta^{\ast}_{c}, \phi, w_{\widehat{\beta}})\geq \langle \nabla
L_{c}(\theta^{\ast}_{c}, \phi, w_{\widehat{\beta}}), \widehat{\Delta}
\rangle.  \label{eq:lowbound_convex}
\end{equation}
If Condition \ref{cond:approx_sparse_S} a) holds, we decompose $\nabla
L_{c}(\theta^{\ast}_{c}, \phi, w_{\widehat{\beta}})$ as 
\begin{align*}
\nabla L_{c}(\theta^{\ast}_{c}, \phi, w_{\widehat{\beta}})&= \mathbb{P}%
_{n}\left\lbrace S_{ab}X w_{\widehat{\beta}} {\varphi_{c}}(\langle
\theta^{\ast}_{c}, X\rangle) + m_{\overline{c}}(\phi w_{\widehat{\beta}})
\right\rbrace \\
&= \mathbb{P}_{n}\left\lbrace S_{ab}X w_{\widehat{\beta}} \varphi^{\ast} +
m_{\overline{c}}(\phi w_{\widehat{\beta}}) \right\rbrace + \mathbb{P}%
_{n}\left\lbrace X w_{\widehat{\beta}} S_{ab} ({\varphi_{c}}(\langle
\theta^{\ast}_{c}, X\rangle) -\varphi^{\ast} )\right\rbrace \\
&= \mathbb{P}_{n}\left\lbrace S_{ab}X w_{\widehat{\beta}} \varphi^{\ast} +
m_{\overline{c}}(\phi w_{\widehat{\beta}}) \right\rbrace + \mathbb{P}%
_{n}\left\lbrace X w_{\widehat{\beta}} S_{ab} \varrho\right\rbrace \\
&= \mathbb{P}_{n}\left\lbrace S_{ab}X w_{\widehat{\beta}} \varphi^{\ast} + 
\mathcal{R}_{\overline{c}}X w_{\widehat{\beta}} \right\rbrace + \mathbb{P}%
_{n}\left\lbrace m_{\overline{c}}(\phi w_{\widehat{\beta}}) -\mathcal{R}_{%
\overline{c}}X w_{\widehat{\beta}} \right\rbrace +\mathbb{P}_{n}\left\lbrace
X w_{\widehat{\beta}} S_{ab} \varrho\right\rbrace. \\
&= \mathbb{P}_{n}\left\lbrace S_{ab}X w_{\widehat{\beta}} \varphi^{\ast} + 
\mathcal{R}_{\overline{c}}X w_{\widehat{\beta}} \right\rbrace + \mathbb{P}%
_{n}\left\lbrace m_{\overline{c}}(\phi w_{\widehat{\beta}}) -\mathcal{R}_{%
\overline{c}}X w_{\widehat{\beta}} \right\rbrace + R(\widehat{\beta}).
\end{align*}
Holder's inequality implies 
\begin{equation}
\vert \langle \mathbb{P}_{n}\left\lbrace S_{ab}X w_{\widehat{\beta}}
\varphi^{\ast} + \mathcal{R}_{\overline{c}}X w_{\widehat{\beta}}
\right\rbrace + \mathbb{P}_{n}\left\lbrace m_{\overline{c}}(\phi w_{\widehat{%
\beta}})-\mathcal{R}_{\overline{c}}X w_{\widehat{\beta}} \right\rbrace , 
\widehat{\Delta}\rangle \vert \leq J \Vert \widehat{\Delta} \Vert_{1}.
\label{eq:bound_J_Delta}
\end{equation}
Using Cauchy-Schwartz 
\begin{align}
\vert \langle R(\widehat{\beta}), \widehat{\Delta}\rangle \vert= \vert 
\mathbb{P}_{n}\left\lbrace \langle \widehat{\Delta},X\rangle w_{\widehat{%
\beta}} S_{ab} \varrho\right\rbrace\vert & \leq \left( \mathbb{P}%
_{n}\left\lbrace \langle \widehat{\Delta},X\rangle^{2}
\right\rbrace\right)^{1/2} \left( \mathbb{P}_{n}\left\lbrace (S_{ab} w_{%
\widehat{\beta}} \varrho)^{2}\right\rbrace\right)^{1/2}  \notag \\
& = \Vert \widehat{\Delta}\Vert_{\widehat{\Sigma}_{1}} \Vert \widetilde{%
\varrho}\Vert_{2,n}.  \label{eq:bound_R_Delta}
\end{align}
On the other hand if Condition \ref{cond:approx_sparse_S} b) holds we
decompose $\nabla L_{c}(\theta^{\ast}_{c}, \phi, w_{\widehat{\beta}})$ as 
\begin{align*}
\nabla L_{c}(\theta^{\ast}_{c}, \phi, w_{\widehat{\beta}})&= \mathbb{P}%
_{n}\left\lbrace S_{ab}X w_{\widehat{\beta}} \varphi^{\ast} + \mathcal{R}_{%
\overline{c}}X w_{\widehat{\beta}} \right\rbrace + \mathbb{P}%
_{n}\left\lbrace m_{\overline{c}}(\phi w_{\widehat{\beta}}) -\mathcal{R}_{%
\overline{c}}X w_{\widehat{\beta}} \right\rbrace + R(\widehat{\beta}) \\
& =\mathbb{P}_{n}\left\lbrace S_{ab}X w_{\beta^{\ast}} \varphi^{\ast} + 
\mathcal{R}_{\overline{c}}X w_{\beta^{\ast}} \right\rbrace + \mathbb{P}%
_{n}\left\lbrace m_{\overline{c}}(\phi w_{\widehat{\beta}}) -\mathcal{R}_{%
\overline{c}}X w_{\widehat{\beta}} \right\rbrace \\
&+ \mathbb{P}_{n}\left\lbrace (w_{\widehat{\beta}}-w_{\beta^{\ast}})%
\left(S_{ab} \varphi^{\ast} + \mathcal{R}_{\overline{c}}\right) X
\right\rbrace + R(\widehat{\beta}).
\end{align*}
The Cauchy-Schwartz inequality yields 
\begin{equation}
\left\vert \langle \mathbb{P}_{n}\left\lbrace (w_{\widehat{\beta}%
}-w_{\beta^{\ast}})\left(S_{ab} \varphi^{\ast} + \mathcal{R}_{\overline{c}%
}\right)X \right\rbrace, \Delta\rangle \right\vert \leq H \Vert \widehat{%
\Delta}\Vert_{\widehat{\Sigma}_{1}}.  \label{eq:bound_H}
\end{equation}
Holder's inequality implies 
\begin{equation}
\vert \langle \mathbb{P}_{n}\left\lbrace S_{ab}X w_{\beta^{\ast}}
\varphi^{\ast} + \mathcal{R}_{\overline{c}}X w_{\beta^{\ast}} \right\rbrace
+ \mathbb{P}_{n}\left\lbrace m_{\overline{c}}(\phi w_{\beta^{\ast}})-%
\mathcal{R}_{\overline{c}}X w_{\beta^{\ast}} \right\rbrace , \widehat{\Delta}%
\rangle \vert \leq J \Vert \widehat{\Delta} \Vert_{1}.
\label{eq:bound_J_Delta_betastar}
\end{equation}
Putting together equations \eqref{eq:bound_J_Delta}, \eqref{eq:bound_R_Delta}%
, \eqref{eq:bound_H} and \eqref{eq:bound_J_Delta_betastar} we get that if
either Condition \ref{cond:approx_sparse_S} a) or b) holds 
\begin{equation}
\langle \nabla L_{c}(\theta^{\ast}_{c}, \phi, w_{\widehat{\beta}}), \widehat{%
\Delta} \rangle \geq - J\Vert \widehat{\Delta} \Vert_{1} - H \Vert \widehat{%
\Delta}\Vert_{\widehat{\Sigma}_{1}} - \Vert \widehat{\Delta}\Vert_{\widehat{%
\Sigma}_{1}} \Vert \widetilde{\varrho}\Vert_{2,n}.  \label{eq:lowbound_grad}
\end{equation}
This together with equations \eqref{eq:bound_L_def}, %
\eqref{eq:lowbound_convex} implies 
\begin{align*}
-J \Vert \widehat{\Delta} \Vert_{1} - H \Vert \widehat{\Delta}\Vert_{%
\widehat{\Sigma}_{1}}- \Vert \widehat{\Delta}\Vert_{\widehat{\Sigma}_{1}}
\Vert \widetilde{\varrho}\Vert_{2,n} \leq \lambda_{c}\left( \Vert \widehat{%
\Delta}_{S_{c}}\Vert_{1} - \Vert \widehat{\Delta}_{\overline{S}_{c}}
\Vert_{1}\right).
\end{align*}
Rearranging this last equation leads to 
\begin{equation*}
\lambda_{c} \Vert \widehat{\Delta}_{\overline{S}_{c}} \Vert_{1} \leq
\lambda_{c} \Vert \widehat{\Delta}_{S_{c}}\Vert_{1} + J \Vert \widehat{\Delta%
} \Vert_{1} + H \Vert \widehat{\Delta}\Vert_{\widehat{\Sigma}_{1}}+\Vert 
\widehat{\Delta}\Vert_{\widehat{\Sigma}_{1}} \Vert \widetilde{\varrho}%
\Vert_{2,n}.
\end{equation*}
\end{proof}

The following proposition is the key result of this section. It shows that
if: $\lambda_{c}$ overrules the noise level as measured by $J$, the
approximation error as measured by $\Vert \widetilde{\varrho} \Vert_{2}$ is
small, $\widehat{\beta}$ is close to $\beta^{\ast}$ and ${\widehat{\Sigma}%
_{1}}$ satisfies a restricted upper eigenvalue type condition, then there
exists $C>0$ such that either 
\begin{equation*}
\widehat{\Delta}\in\mathbb{D}(C, s_{c}) \quad \text{ or } \quad \widehat{%
\Delta}\in\mathbb{D}(C, \max(s_{c}, s_{\beta})).
\end{equation*}
It will be the case that $\widehat{\Delta}\in\mathbb{D}(C, s_{c})$ when
Condition \ref{cond:approx_sparse_S} a) holds or when a constant weight
function is used.

\begin{proposition}
\label{prop:cone} Assume Conditions \ref{cond:setting_nuis}, \ref{cond:glm}
a), \ref{cond:weight} and \ref{cond:approx_sparse_S} hold. Assume there
exist fixed non-negative constants $c_{\lambda}, c_{H},
c_{\rho},c_{\Sigma_1} $ and $n_0\in\mathbb{N}$ such that $c_{\lambda}>0$ and
for all $n\geq n_0$ 
\begin{align*}
\lambda_{c}=c_{\lambda}\sqrt{\frac{\log(p)}{n}},\quad J \leq \frac{%
\lambda_{c}}{2}, \quad H\leq c_{H} \sqrt{\frac{s_{\beta}\log(p)}{n}}, \quad
\Vert \widetilde{\varrho}\Vert_{2,n}\leq c_{\rho}\sqrt{\frac{s_{c}\log(p)}{n}%
}
\end{align*}
and 
\begin{align*}
\Vert {\Delta}\Vert_{\widehat{\Sigma}_{1}}^{2}\leq c_{\Sigma_1} \left( \Vert 
{\Delta}\Vert_{2}^{2} + \frac{\log(p)}{n}\Vert {\Delta}\Vert_{1}^{2}\right)
\quad \text{ for all } \Delta.
\end{align*}
Then there exist $n_1\geq n_0$ and $c_{\mathcal{C}}>0$ depending only on $%
c_{\lambda}, c_{H}, c_{\rho},c_{\Sigma_1}$ such that if $n\geq n_1$ then

\begin{itemize}
\item If $c_{H}=0$ 
\begin{equation*}
\widehat{\Delta} \in \mathbb{D}(c_{\mathcal{C}}, s_{c}).
\end{equation*}

\item If $c_{H}>0$ 
\begin{equation*}
\widehat{\Delta} \in \mathbb{D}(c_{\mathcal{C}}, \max(s_{c}, s_{\beta})).
\end{equation*}
\end{itemize}
\end{proposition}

\begin{proof}
\lbrack Proof of Proposition \ref{prop:cone}] By Lemma \ref{lemma:pre_cone} 
\begin{equation*}
\lambda_{c} \Vert \widehat{\Delta}_{\overline{S}_{c}} \Vert_{1} \leq
\lambda_{c} \Vert \widehat{\Delta}_{S_{c}}\Vert_{1} + J \Vert \widehat{\Delta%
} \Vert_{1} + H \Vert \widehat{\Delta}\Vert_{\widehat{\Sigma}_{1}} + \Vert 
\widehat{\Delta}\Vert_{\widehat{\Sigma}_{1}} \Vert \widetilde{\varrho}%
\Vert_{2,n}.
\end{equation*}
Take $n$ larger than $n_0$. Using the assumptions we get 
\begin{align*}
\lambda_{c} \Vert \widehat{\Delta}_{\overline{S}_{c}} \Vert_{1} \leq
\lambda_{c} \Vert \widehat{\Delta}_{S_{c}}\Vert_{1} + \frac{\lambda_{c}}{2}
\Vert \widehat{\Delta} \Vert_{1} + c_{H} \sqrt{\frac{s_{\beta}\log(p)}{n}}
\Vert \widehat{\Delta}\Vert_{\widehat{\Sigma}_{1}} + c_{\rho}\sqrt{\frac{%
s_{c}\log(p)}{n}} \Vert \widehat{\Delta}\Vert_{\widehat{\Sigma}_{1}}.
\end{align*}
Since $\Vert \widehat{\Delta} \Vert_{1} = \Vert \widehat{\Delta}_{S_{c}}
\Vert_{1}+ \Vert \widehat{\Delta} _{\overline{S}_{c}}\Vert_{1}$ it follows
that 
\begin{align*}
\frac{\lambda_{c}}{2} \Vert \widehat{\Delta}_{\overline{S}_{c}} \Vert_{1}
&\leq \frac{3\lambda_{c}}{2} \Vert \widehat{\Delta}_{S_{c}}\Vert_{1} + c_{H} 
\sqrt{\frac{s_{\beta}\log(p)}{n}} \Vert \widehat{\Delta}\Vert_{\widehat{%
\Sigma}_{1}} + c_{\rho}\sqrt{\frac{s_{c}\log(p)}{n}} \Vert \widehat{\Delta}%
\Vert_{\widehat{\Sigma}_{1}}.
\end{align*}
Since $\lambda_{c}=c_{\lambda}\sqrt{{\log(p)}/{n}}$ the equation above
implies 
\begin{align*}
\frac{1}{2} \Vert \widehat{\Delta}_{\overline{S}_{c}} \Vert_{1} &\leq \frac{3%
}{2} \Vert \widehat{\Delta}_{S_{c}}\Vert_{1} + \frac{\sqrt{s_{\beta}}}{%
c_{\lambda}} c_{H} \Vert \widehat{\Delta}\Vert_{\widehat{\Sigma}_{1}} + 
\frac{\sqrt{s_{c}}}{c_{\lambda}} c_{\rho} \Vert \widehat{\Delta}\Vert_{%
\widehat{\Sigma}_{1}}.
\end{align*}
Then at least one of the following has to hold: 
\begin{align}
\frac{1}{6}\Vert \widehat{\Delta}_{\overline{S}_{c}} \Vert_{1} &\leq \frac{3%
}{2} \Vert \widehat{\Delta}_{S_{c}} \Vert_{1},  \label{eq:cone_alt1} \\
\frac{1}{6}\Vert \widehat{\Delta}_{\overline{S}_{c}} \Vert_{1} &\leq \frac{%
\sqrt{s_{\beta}}}{c_{\lambda}} c_{H} \Vert \widehat{\Delta}\Vert_{\widehat{%
\Sigma}_{1}},  \label{eq:cone_alt2} \\
\frac{1}{6}\Vert \widehat{\Delta}_{\overline{S}_{c}} \Vert_{1} &\leq \frac{%
\sqrt{s_{c}}}{c_{\lambda}} c_{\rho} \Vert \widehat{\Delta}\Vert_{\widehat{%
\Sigma}_{1}}.  \label{eq:cone_alt3}
\end{align}
If \eqref{eq:cone_alt1} holds then 
\begin{equation}
\Vert \widehat{\Delta}_{\overline{S}_{c}} \Vert_{1} \leq 9 \sqrt{s_{c}}
\Vert \widehat{\Delta}\Vert_{2}  \notag
\end{equation}
and hence 
\begin{equation}
\Vert \widehat{\Delta} \Vert_{1}= \Vert \widehat{\Delta}_{S_{c}}\Vert_{1} +
\Vert \widehat{\Delta}_{\overline{S}_{c}}\Vert_{1}\leq 10 \sqrt{s_{c}}\Vert 
\widehat{\Delta}\Vert_{2}.  \label{eq:cone_alt1_end}
\end{equation}
If \eqref{eq:cone_alt2} holds then 
\begin{equation*}
\Vert \widehat{\Delta}\Vert_{1}=\Vert \widehat{\Delta}_{S_{c}}\Vert_{1} +
\Vert \widehat{\Delta}_{\overline{S}_{c}}\Vert_{1}\leq \sqrt{s_{c}}\Vert 
\widehat{\Delta}\Vert_{2} + \frac{6c_{H}\sqrt{s_{\beta}}}{c_{\lambda}} \Vert 
\widehat{\Delta}\Vert_{\widehat{\Sigma}_{1}}.
\end{equation*}
Hence 
\begin{equation}
\Vert \widehat{\Delta}\Vert_{1}^{2}\leq 2 s_{c}\Vert \widehat{\Delta}%
\Vert_{2}^{2} + 2 \left(\frac{6c_{H}}{c_{\lambda}} \right)^{2}s_{\beta}
\Vert \widehat{\Delta}\Vert_{\widehat{\Sigma}_{1}}^{2}.
\label{eq:cone_ineq_l1}
\end{equation}
Now by assumption 
\begin{equation*}
\Vert \widehat{\Delta}\Vert_{\widehat{\Sigma}_{1}}^{2}\leq c_{\Sigma_1}
\left( \Vert \widehat{\Delta}\Vert_{2}^{2} + \frac{\log(p)}{n}\Vert \widehat{%
\Delta}\Vert_{1}^{2}\right)
\end{equation*}
and then using \eqref{eq:cone_ineq_l1} 
\begin{equation*}
\Vert \widehat{\Delta}\Vert_{\widehat{\Sigma}_{1}}^{2}\leq c_{\Sigma_1}
\Vert \widehat{\Delta}\Vert_{2}^{2} + 2c_{\Sigma_1} \frac{s_{c}\log(p)}{n}
\Vert \widehat{\Delta}\Vert_{2}^{2} + 2c_{\Sigma_1} \left(\frac{6c_{H}}{%
c_{\lambda}} \right)^{2} \frac{s_{\beta}\log(p)}{n}\Vert \widehat{\Delta}%
\Vert_{\widehat{\Sigma}_{1}}^{2}.
\end{equation*}
It follows that 
\begin{equation*}
\Vert \widehat{\Delta}\Vert_{\widehat{\Sigma}_{1}}^{2} \left(1 -
2c_{\Sigma_1} \left(\frac{6c_{H}}{c_{\lambda}} \right)^{2} \frac{%
s_{\beta}\log(p)}{n}\right)\leq \Vert \widehat{\Delta}\Vert_{2}^{2}
\left(c_{\Sigma_1} + 2c_{\Sigma_1} \frac{s_{c}\log(p)}{n} \right)
\end{equation*}
Since by Condition \ref{cond:weight} c) and Condition \ref%
{cond:approx_sparse_S} we have ${s_{\beta}\log(p)}/{n}\to 0$ and ${%
s_{c}\log(p)}/{n}\to 0$, for sufficiently large $n$ 
\begin{equation*}
\Vert \widehat{\Delta}\Vert_{\widehat{\Sigma}_{1}}^{2}\leq 9 c_{\Sigma_1}
\Vert \widehat{\Delta}\Vert_{2}^{2}
\end{equation*}
and hence 
\begin{equation*}
\Vert \widehat{\Delta}\Vert_{\widehat{\Sigma}_{1}} \leq
3c_{\Sigma_1}^{1/2}\Vert \widehat{\Delta}\Vert_{2}.
\end{equation*}
Going back to \eqref{eq:cone_alt2}, we get 
\begin{equation*}
\Vert \widehat{\Delta}_{\overline{S}_{c}} \Vert_{1} \leq \frac{6c_{H} }{%
c_{\lambda}}\sqrt{s_{\beta}} \Vert \widehat{\Delta}\Vert_{\widehat{\Sigma}%
_{1}} \leq \frac{18c_{H} c_{\Sigma_1}^{1/2}}{c_{\lambda}}\sqrt{s_{\beta}}
\Vert \widehat{\Delta}\Vert_{2}
\end{equation*}
and 
\begin{equation}
\Vert \widehat{\Delta}\Vert_{1}=\Vert \widehat{\Delta}_{S_{c}}\Vert_{1} +
\Vert \widehat{\Delta}_{\overline{S}_{c}}\Vert_{1}\leq \sqrt{s_{c}}\Vert 
\widehat{\Delta}\Vert_{2} + \frac{18c_{H} c_{\Sigma_1}^{1/2}}{c_{\lambda}}%
\sqrt{s_{\beta}} \Vert \widehat{\Delta}\Vert_{2}.  \label{eq:cone_alt2_end}
\end{equation}
Similarly, if \eqref{eq:cone_alt3} holds, for sufficiently large $n$ 
\begin{equation}
\Vert \widehat{\Delta}\Vert_{1} \leq \left(1 + \frac{18c_{\rho}
c_{\Sigma_1}^{1/2}}{c_{\lambda}}\right)\sqrt{s_{c}} \Vert \widehat{\Delta}%
\Vert_{2}.  \label{eq:cone_alt3_end}
\end{equation}
Now, if $c_{H}=0$, \eqref{eq:cone_alt1_end}, \eqref{eq:cone_alt2_end} and %
\eqref{eq:cone_alt3_end} imply that there exist $n_1\in\mathbb{N}$ such that
if $n\geq n_1$ then 
\begin{equation*}
\Vert \widehat{\Delta}\Vert_{1}\leq c_{\mathcal{C}} \sqrt{s_{c}}\Vert 
\widehat{\Delta}\Vert_{2},
\end{equation*}
where $c_{\mathcal{C}}$ depends only on $c_{\lambda}, c_{\rho}, c_{\Sigma_1}$%
. If $c_{H}>0$ we get 
\begin{equation*}
\Vert \widehat{\Delta}\Vert_{1}\leq c_{\mathcal{C}} \max(\sqrt{s_{c}},\sqrt{%
s_{\beta}})\Vert \widehat{\Delta}\Vert_{2},
\end{equation*}
where $c_{\mathcal{C}}$ depends only on $c_{\lambda}, c_{H},c_{\rho},
c_{\Sigma_1}$. The proposition is proven.
\end{proof}

Let 
\begin{equation*}
F(\Delta,\widehat{\beta})=L_{c}(\theta^{\ast}_{c}+\Delta, \phi, w_{\widehat{%
\beta}})-L_{c}(\theta^{\ast}_{c}, \phi, w_{\widehat{\beta}})+
\lambda_{c}\left( \Vert\theta^{\ast}_{c}+\Delta
\Vert_{1}-\Vert\theta^{\ast}_{c} \Vert_{1} \right).
\end{equation*}

The following lemma is similar to Lemma 4 from \cite{NegahbanSS}.

\begin{lemma}
\label{lemma:lowbound_loc} Assume Conditions \ref{cond:setting_nuis}, \ref%
{cond:glm} a) and \ref{cond:weight} a) hold. Let $C,s>0$ and assume $%
\widehat{\Delta}\in \mathbb{D}(C,s)$ . Let $\delta >0$. If $F(\Delta, 
\widehat{\beta}) >0 $ for all $\Delta\in\mathbb{D}(C,s)$ with $\Vert \Delta
\Vert_{2} = \delta$ then $\Vert \widehat{\Delta}\Vert_{2} \leq \delta$.
\end{lemma}

\begin{proof}
\lbrack Proof of Lemma \ref{lemma:lowbound_loc}] Suppose that $\Vert 
\widehat{\Delta}\Vert_{2} > \delta$. We will show that there exists $\Delta
\in \mathbb{D}(C,s)$ with $\Vert \Delta\Vert_{2}=\delta$ such that $%
F(\Delta, \widehat{\beta})\leq 0$. Let $\Delta = \delta \widehat{\Delta} /
\Vert \widehat{\Delta} \Vert_{2}$. Clearly $\Vert \Delta \Vert_{2} = \delta$%
. Moreover, $\Delta \in \mathbb{D}(C,s)$. Using that by Conditions \ref%
{cond:glm} a) and \ref{cond:weight} a), $F(\cdot,\widehat{\beta})$ is
convex, that $F(\widehat{\Delta},\widehat{\beta})\leq 0$ and $F(0,\widehat{%
\beta})=0$ we see that 
\begin{align*}
F\left( \Delta ,\widehat{\beta}\right)&= F\left( \frac{\delta}{\Vert 
\widehat{\Delta} \Vert_{2}} \widehat{\Delta} + \left(1-\frac{\delta}{\Vert 
\widehat{\Delta} \Vert_{2})}\right)0,\widehat{\beta}\right) \\
&\leq \frac{\delta}{\Vert \widehat{\Delta}\Vert_{2}} F(\widehat{\Delta},%
\widehat{\beta}) + \left(1-\frac{\delta}{\Vert \widehat{\Delta} \Vert_{2})}%
\right)F(0,\widehat{\beta})=\frac{\delta}{\Vert \widehat{\Delta}\Vert_{2}} F(%
\widehat{\Delta},\widehat{\beta}) \leq 0.
\end{align*}
Hence $F\left( \Delta ,\widehat{\beta}\right)\leq 0$, $\Vert \Delta
\Vert_{2} = \delta$ and $\Delta \in \mathbb{D}(C,s)$, what we wanted to show.
\end{proof}

For $\Delta\in\mathbb{R}^{p}$ let 
\begin{equation*}
\delta L_{c}(\Delta, \theta^{\ast}_{c},\widehat{\beta}) =
L_{c}(\theta^{\ast}_{c}+\Delta, \phi, w_{\widehat{\beta}})-L_{c}(\theta^{%
\ast}_{c}, \phi, w_{\widehat{\beta}}) - \langle \nabla
L_{c}(\theta^{\ast}_{c}, \phi, w_{\widehat{\beta}}), \Delta \rangle.
\end{equation*}
Following \cite{NegahbanSS}, we will say that $L_{c}(\cdot, \phi, w_{%
\widehat{\beta}})$ satisfies Restricted Strong Convexity (RSC) with
curvature $\kappa_{RSC}$ over a set $\mathcal{S}$ if 
\begin{equation*}
\delta L_{c}(\Delta, \theta^{\ast}_{c},\widehat{\beta}) \geq
\kappa_{RSC}\Vert \Delta\Vert_{2}^{2} \quad \text{ for all } \Delta\in 
\mathcal{S}.
\end{equation*}

\begin{theorem}
\label{theo:rate_gen} Assume the setting of Proposition \ref{prop:cone}. Let 
$s$ stand for $s_{c}$ if $c_{H}=0$ and for $\max(s_{c},s_{\beta})$ if $%
c_{H}>0$. Let $n_1$ and $c_{\mathcal{C}}$ be as in Proposition \ref%
{prop:cone} and assume there exists $n_2\geq n_1$ such that for all $n\geq
n_2$, $L_{c}(\cdot, \phi, w_{\widehat{\beta}})$ satisfies RSC over $\mathbb{D%
}(c_{\mathcal{C}},s) \cap \lbrace \Delta : \Vert\Delta\Vert_{2}\leq 1
\rbrace $ with curvature $\kappa_{RSC}$. Let 
\begin{equation*}
\upsilon_{n}= \left( \frac{3c_{\lambda}}{2 }+c_{\rho}\left(c_{\Sigma_1}
\left( 1+ c_{\mathcal{C}}^{2} \right)\right)^{1/2} \right) \sqrt{\frac{%
s_{c}\log(p)}{n}}+ c_{H}\left(c_{\Sigma_1} \left( 1+ c_{\mathcal{C}}^{2}
\right)\right)^{1/2} \sqrt{\frac{s_{\beta}\log(p)}{n}}.
\end{equation*}
Then there exists $n_3\in\mathbb{N}$ such that if $n\geq n_3$ 
\begin{equation*}
\widehat{\Delta} \in \mathbb{D}(c_{\mathcal{C}},s), \quad \Vert \widehat{%
\Delta}\Vert_{2}\leq \frac{2\upsilon_{n}}{\kappa_{RSC}}\quad \text{ and }
\quad \Vert \widehat{\Delta}\Vert_{1}\leq \frac{2\upsilon_{n} c_{\mathcal{C}%
} \sqrt{s}}{\kappa_{RSC}} .
\end{equation*}
\end{theorem}

\begin{proof}
\lbrack Proof of Theorem \ref{theo:rate_gen}] By Conditions \ref{cond:weight}
c) and \ref{cond:approx_sparse_S} we can take $n\geq n_2$ large enough such
that 
\begin{equation*}
\frac{2 \upsilon_{n}}{\kappa_{RSC}} < 1 \quad \text{ and } \quad \sqrt{\frac{%
s\log(p)}{n}} < 1 .
\end{equation*}
By Proposition \ref{prop:cone}, $\widehat{\Delta}\in\mathbb{D}(c_{\mathcal{C}%
},s)$. Take $\Delta\in\mathbb{D}(c_{\mathcal{C}},s)$ with 
\begin{equation*}
\Vert \Delta \Vert_{2} = \frac{2 \upsilon_{n}}{\kappa_{RSC}}.
\end{equation*}
We will show that $F(\Delta, \widehat{\beta})>0$. By the assumption on RSC 
\begin{equation*}
L_{c}(\theta^{\ast}_{c}+\Delta, \phi, w_{\widehat{\beta}})-L_{c}(\theta^{%
\ast}_{c}, \phi, w_{\widehat{\beta}}) - \langle \nabla
L_{c}(\theta^{\ast}_{c}, \phi, w_{\widehat{\beta}}), \Delta \rangle\geq
\kappa_{RSC}\Vert \Delta\Vert_{2}^{2}.
\end{equation*}
Arguing as in the proof of Lemma \ref{lemma:pre_cone} (see %
\eqref{eq:lowbound_grad}) 
\begin{align*}
\langle \nabla L_{c}(\theta^{\ast}_{c}, \phi, w_{\widehat{\beta}}), \Delta
\rangle &\geq - J \Vert {\Delta} \Vert_{1} - H \Vert {\Delta}\Vert_{\widehat{%
\Sigma}_{1}}- \Vert {\Delta}\Vert_{\widehat{\Sigma}_{2}} \Vert \widetilde{%
\varrho}\Vert_{2,n} \\
&\geq -\frac{\lambda_{c}}{2} \Vert {\Delta} \Vert_{1} - c_{H} \sqrt{\frac{%
s_{\beta}\log(p)}{n}} \Vert {\Delta}\Vert_{\widehat{\Sigma}_{1}} - c_{\rho}%
\sqrt{\frac{s_{c}\log(p)}{n}} \Vert {\Delta}\Vert_{\widehat{\Sigma}_{1}}.
\end{align*}
By Lemma \ref{lemma:bound_l1_diff} 
\begin{equation*}
\Vert\theta^{\ast}_{c}+{\Delta} \Vert_{1} - \Vert\theta^{\ast}_{c} \Vert_{1}
\geq \Vert {\Delta}_{\overline{S}_{c}} \Vert_{1} - \Vert {\Delta}%
_{S_{c}}\Vert_{1}.
\end{equation*}
Hence 
\begin{align}
F(\Delta, \widehat{\beta})&= L_{c}(\theta^{\ast}_{c}+\Delta, \phi, w_{%
\widehat{\beta}})-L_{c}(\theta^{\ast}_{c}, \phi, w_{\widehat{\beta}})+
\lambda_{c}\left( \Vert\theta^{\ast}_{c}+\Delta
\Vert_{1}-\Vert\theta^{\ast}_{c} \Vert_{1} \right)  \notag \\
& = L_{c}(\theta^{\ast}_{c}+\Delta, \phi, w_{\widehat{\beta}%
})-L_{c}(\theta^{\ast}_{c}, \phi, w_{\widehat{\beta}}) - \langle \nabla
L_{c}(\theta^{\ast}_{c}, \phi, w_{\widehat{\beta}}), \Delta \rangle +
\langle \nabla L_{c}(\theta^{\ast}_{c}, \phi, w_{\widehat{\beta}}), \Delta
\rangle  \notag \\
&+ \lambda_{c}\left( \Vert\theta^{\ast}_{c}+\Delta \Vert_{1}-
\Vert\theta^{\ast}_{c} \Vert_{1} \right)  \notag \\
&\geq \kappa_{RSC}\Vert \Delta\Vert_{2}^{2}-\frac{\lambda_{c}}{2} \Vert {%
\Delta} \Vert_{1} - c_{H} \sqrt{\frac{s_{\beta}\log(p)}{n}} \Vert {\Delta}%
\Vert_{\widehat{\Sigma}_{1}} - c_{\rho}\sqrt{\frac{s_{c}\log(p)}{n}} \Vert {%
\Delta}\Vert_{\widehat{\Sigma}_{1}}+ \lambda_{c} \Vert {\Delta}_{\overline{S}%
_{c}} \Vert_{1} - \lambda_{c}\Vert {\Delta}_{S_{c}}\Vert_{1}  \notag \\
&\geq \kappa_{RSC}\Vert \Delta\Vert_{2}^{2}-\frac{3\lambda_{c}\sqrt{s_{c}}}{2%
} \Vert {\Delta}\Vert_{2} - c_{H} \sqrt{\frac{s_{\beta}\log(p)}{n}} \Vert {%
\Delta}\Vert_{\widehat{\Sigma}_{1}} - c_{\rho}\sqrt{\frac{s_{c}\log(p)}{n}}
\Vert {\Delta}\Vert_{\widehat{\Sigma}_{1}},  \label{eq:firstbound_F}
\end{align}
where in the last inequality we used 
\begin{align*}
-\frac{\lambda_{c}}{2} \Vert {\Delta} \Vert_{1}+ \lambda_{c} \Vert {\Delta}_{%
\overline{S}_{c}} \Vert_{1} - \lambda_{c}\Vert {\Delta}_{S_{c}}\Vert_{1}&= -%
\frac{\lambda_{c}}{2} \Vert {\Delta}_{S_{c}} \Vert_{1}-\frac{\lambda_{c}}{2}
\Vert {\Delta}_{\overline{S}_{c}} \Vert_{1}+ \lambda_{c} \Vert {\Delta}_{%
\overline{S}_{c}} \Vert_{1} - \lambda_{c}\Vert {\Delta}_{S_{c}}\Vert_{1} \\
&= -\frac{3\lambda_{c}}{2} \Vert {\Delta}_{S_{c}} \Vert_{1} +\frac{%
\lambda_{c}}{2} \Vert {\Delta}_{\overline{S}_{c}} \Vert_{1} \\
& \geq -\frac{3\lambda_{c}}{2} \Vert {\Delta}_{S_{c}} \Vert_{1} \\
& \geq -\frac{3\lambda_{c}\sqrt{s_{c}}}{2} \Vert {\Delta}\Vert_{2} .
\end{align*}
By assumption 
\begin{equation*}
\Vert {\Delta}\Vert_{\widehat{\Sigma}_{1}}^{2} \leq c_{\Sigma_1} \left(
\Vert {\Delta}\Vert_{2}^{2} + \frac{\log(p)}{n}\Vert {\Delta}%
\Vert_{1}^{2}\right)
\end{equation*}
and since $\Delta \in \mathbb{D}(c_{\mathcal{C}},s)$ and $s\log(p)/n < 1$ 
\begin{align}
\Vert {\Delta}\Vert_{\widehat{\Sigma}_{1}}^{2} \leq c_{\Sigma_1} \left(
\Vert {\Delta}\Vert_{2}^{2} + \frac{\log(p)}{n}\Vert {\Delta}%
\Vert_{1}^{2}\right) \leq c_{\Sigma_1} \left( \Vert {\Delta}\Vert_{2}^{2} +
c_{\mathcal{C}}^{2}\frac{s\log(p)}{n}\Vert {\Delta}\Vert_{2}^{2}\right)&=
c_{\Sigma_1} \left( 1+ c_{\mathcal{C}}^{2}\frac{s\log(p)}{n} \right)\Vert {%
\Delta}\Vert_{2}^{2}  \notag \\
& \leq c_{\Sigma_1} \left( 1+ c_{\mathcal{C}}^{2} \right)\Vert {\Delta}%
\Vert_{2}^{2}.  \label{eq:bound_Sigma1norm}
\end{align}
Plugging \eqref{eq:bound_Sigma1norm} back in \eqref{eq:firstbound_F} and
using that $\lambda_{c}=c_{\lambda}\sqrt{\log(p)/n}$ we get 
\begin{align*}
F(\Delta, \widehat{\beta}) \geq \kappa_{RSC}\Vert \Delta\Vert_{2}^{2}-\frac{%
3\lambda_{c}\sqrt{s_{c}}}{2} \Vert {\Delta}\Vert_{2} &- c_{H} \sqrt{\frac{%
s_{\beta}\log(p)}{n}} \left(c_{\Sigma_1} \left( 1+ c_{\mathcal{C}}^{2}
\right)\right)^{1/2}\Vert {\Delta}\Vert_{2} \\
&- c_{\rho}\sqrt{\frac{s_{c}\log(p)}{n}} \left(c_{\Sigma_1} \left( 1+ c_{%
\mathcal{C}}^{2} \right)\right)^{1/2}\Vert {\Delta}\Vert_{2} \\
&= \kappa_{RSC}\Vert \Delta\Vert_{2}^{2} - \upsilon_{n}\Vert {\Delta}%
\Vert_{2}.
\end{align*}
Since 
\begin{equation*}
\Vert \Delta\Vert_{2}=\frac{2 \upsilon_{n}}{\kappa_{RSC}}> \frac{\upsilon_{n}%
}{\kappa_{RSC}}
\end{equation*}
we conclude that $F(\Delta, \widehat{\beta})>0$. Hence by Lemma \ref%
{lemma:lowbound_loc} 
\begin{equation*}
\Vert \widehat{\Delta} \Vert_{2} \leq 2\frac{\upsilon_{n}}{\kappa_{RSC}}.
\end{equation*}
Moreover, since $\widehat{\Delta} \in \mathbb{D}(c_{\mathcal{C}},s)$ 
\begin{equation*}
\Vert \widehat{\Delta} \Vert_{1} \leq c_{\mathcal{C}}\sqrt{s} \Vert \widehat{%
\Delta}\Vert_{2}\leq 2\sqrt{s}\frac{c_{\mathcal{C}}\upsilon_{n}}{\kappa_{RSC}%
}.
\end{equation*}
\end{proof}

\subsection{The assumptions in Theorem \protect\ref{theo:rate_gen}}

\label{sec:cond_high_prob}

In this section we show that under Conditions \ref{cond:setting_nuis}, \ref%
{cond:glm}, \ref{cond:weight}, \ref{cond:approx_sparse_S} and \ref%
{cond:subgauss} the assumptions in Theorem \ref{theo:rate_gen} hold with
high probability. The following lemma implies that $\lambda_{c}$ can be
chosen to be of order $\sqrt{\log(p)/n}$.

\begin{lemma}
\label{lemma:lambda_choice} Assume Conditions \ref{cond:setting_nuis}, \ref%
{cond:glm}, \ref{cond:weight}, \ref{cond:approx_sparse_S}, \ref%
{cond:subgauss} a)-c) and f) hold. Then 
\begin{equation*}
J = O_{P}\left(\sqrt{ \frac{\log(p)}{n}}\right).
\end{equation*}
\end{lemma}

\begin{proof}
\lbrack Proof of Lemma \ref{lemma:lambda_choice}] Assume first that
Condition \ref{cond:approx_sparse_S} a) holds. Then 
\begin{equation*}
J=\Vert \mathbb{P}_{n}\left\lbrace X w(\langle \widehat{\beta}, X\rangle) %
\left[S_{ab}\varphi^{\ast} +\mathcal{R}_{\overline{c}} \right] \right\rbrace
\Vert_{\infty} + \Vert \mathbb{P}_{n}\left\lbrace \mathcal{R}_{\overline{c}}
X w(\langle \widehat{\beta}, X\rangle) - m_{\overline{c}}(\phi w_{\widehat{%
\beta}}) \right\rbrace \Vert_{\infty}.
\end{equation*}
Now fix $\varepsilon>0$ and take $L_0>0$ to be chosen later. Write 
\begin{align*}
&P\left( J \geq L_0 \sqrt{\frac{\log(p)}{n}}\right) = E\left\lbrace P\left(
J \geq L_0 \sqrt{\frac{\log(p)}{n}} \mid \widehat{\beta} \right)
\right\rbrace = \\
& E\left\lbrace P\left( J \geq L_0 \sqrt{\frac{\log(p)}{n}} \mid \widehat{%
\beta} \right) \mid \Vert \widehat{\beta} - \beta^{\ast}\Vert_{2}\leq 1
\right\rbrace P\left( \Vert \widehat{\beta} - \beta^{\ast}\Vert_{2}\leq 1
\right)+ \\
& E\left\lbrace P\left( J \geq L_0 \sqrt{\frac{\log(p)}{n}} \mid \widehat{%
\beta} \right) \mid \Vert \widehat{\beta} - \beta^{\ast}\Vert_{2}> 1
\right\rbrace P\left( \Vert \widehat{\beta} - \beta^{\ast}\Vert_{2}> 1
\right) \leq \\
& E\left\lbrace P\left( J \geq L_0 \sqrt{\frac{\log(p)}{n}} \mid \widehat{%
\beta} \right) \mid \Vert \widehat{\beta} - \beta^{\ast}\Vert_{2}\leq 1
\right\rbrace +P\left( \Vert \widehat{\beta} - \beta^{\ast}\Vert_{2} > 1
\right).
\end{align*}
By Condition \ref{cond:weight} c), the second term in the last display
converges to zero and hence can be made smaller than $\varepsilon/2$ for
sufficiently large $n$. We will show that we can choose $L_0$ such that the
first term is smaller than $\varepsilon/2$ too. Since $\widehat{\beta}$ is
independent of the data, by Markov's inequality, it suffices to bound 
\begin{equation*}
E\left( J \right)=E\left(\Vert \mathbb{P}_{n}\left\lbrace X w(\langle \beta,
X\rangle) \left[ S_{ab}\varphi^{\ast}+\mathcal{R}_{\overline{c}} \right]%
\right\rbrace\Vert_{\infty} \right) + E\left(\Vert \mathbb{P}%
_{n}\left\lbrace \mathcal{R}_{\overline{c}} X w(\langle \beta, X\rangle) -
m_{\overline{c}}(\phi w_{\beta}) \right\rbrace \Vert_{\infty}\right).
\end{equation*}
for all $\beta$ such that $\Vert \beta - \beta^{\ast}\Vert_{2}\leq 1$.
Hence, let $\beta$ be such that $\Vert \beta - \beta^{\ast}\Vert_{2}\leq 1$.
We first bound 
\begin{equation*}
E\left(\Vert \mathbb{P}_{n}\left\lbrace X w(\langle \beta, X\rangle) \left[
S_{ab}\varphi^{\ast}+\mathcal{R}_{\overline{c}} \right]\right\rbrace\Vert_{%
\infty} \right).
\end{equation*}
Since in this case $\varphi^{\ast}(Z)=c(Z)$, by Condition \ref%
{cond:setting_nuis} d) we have that 
\begin{equation*}
E\left(X w(\langle \beta, X\rangle) \left[ S_{ab}\varphi^{\ast}+\mathcal{R}_{%
\overline{c}} \right] \right)=0.
\end{equation*}
Nemirovski's inequality (Lemma 14.24 from \cite{Buhlmann-book}) yields 
\begin{align}
&E \left(\Vert X w(\langle \beta, X\rangle) \left[ S_{ab}\varphi^{\ast}+%
\mathcal{R}_{\overline{c}} \right]\Vert_{\infty} \right) \leq  \notag \\
& \left( \frac{8\log(2p)}{n} \right)^{1/2} E \left\lbrace \left( \max_{1\leq
j \leq p} \frac{1}{n}\sum\limits_{i=1}^{n} \left( X_{i,j} w(\langle \beta,
X_{i}\rangle) \left[ S_{ab,i}\varphi_{i}^{\ast}+\mathcal{R}_{\overline{c},i} %
\right] \right)^{2}\right)^{1/2} \right\rbrace.  \label{eq:nemirovski}
\end{align}
Using Cauchy-Schwartz, for each $j$ 
\begin{equation*}
\frac{1}{n}\sum\limits_{i=1}^{n} \left(X_{i,j} w(\langle \beta,
X_{i}\rangle) \left[ S_{ab,i}\varphi_{i}^{\ast}+\mathcal{R}_{\overline{c},i} %
\right] \right)^{2}\leq \left(\frac{1}{n}\sum\limits_{i=1}^{n} w^{4}(\langle
\beta, X_{i}\rangle)\right)^{1/2} \left(\frac{1}{n}\sum\limits_{i=1}^{n}
\left(X_{i,j} \left[ S_{ab,i}\varphi_{i}^{\ast}+\mathcal{R}_{\overline{c},i} %
\right] \right)^{4} \right)^{1/2}.
\end{equation*}
Hence using Cauchy-Schwartz once more 
\begin{align*}
&E \left\lbrace \left( \max_{1\leq j \leq p} \frac{1}{n}\sum%
\limits_{i=1}^{n} \left( X_{i,j} w(\langle \beta, X_{i}\rangle) \left[
S_{ab,i}\varphi_{i}^{\ast}+\mathcal{R}_{\overline{c},i} \right]
\right)^{2}\right)^{1/2} \right\rbrace \leq \\
& E\left\lbrace \left(\frac{1}{n}\sum\limits_{i=1}^{n} w^{4}(\langle \beta,
X_{i}\rangle)\right)^{1/4} \max\limits_{1\leq j\leq p}\left(\frac{1}{n}%
\sum\limits_{i=1}^{n} \left(X_{i,j} \left[ S_{ab,i}\varphi_{i}^{\ast}+%
\mathcal{R}_{\overline{c},i} \right] \right)^{4} \right)^{1/4}\right\rbrace
\leq \\
& E^{1/2}\left\lbrace \left(\frac{1}{n}\sum\limits_{i=1}^{n} w^{4}(\langle
\beta, X_{i}\rangle)\right)^{1/2}\right\rbrace E^{1/2}\left\lbrace
\max\limits_{1\leq j\leq p}\left(\frac{1}{n}\sum\limits_{i=1}^{n}
\left(X_{i,j} \left[ S_{ab,i}\varphi_{i}^{\ast}+\mathcal{R}_{\overline{c},i} %
\right] \right)^{4} \right)^{1/2}\right\rbrace.
\end{align*}
By Condition \ref{cond:subgauss} a) 
\begin{equation*}
E^{1/2}\left\lbrace \max\limits_{1\leq j\leq p}\left(\frac{1}{n}%
\sum\limits_{i=1}^{n} \left(X_{i,j} \left[ S_{ab,i}\varphi_{i}^{\ast}+%
\mathcal{R}_{\overline{c},i} \right] \right)^{4}
\right)^{1/2}\right\rbrace\leq K^{1/2}.
\end{equation*}
On the other hand by Jensen's inequality 
\begin{equation*}
E^{1/2}\left\lbrace \left(\frac{1}{n}\sum\limits_{i=1}^{n} w^{4}(\langle
\beta, X_{i}\rangle)\right)^{1/2} \right\rbrace \leq E^{1/4}\left\lbrace 
\frac{1}{n}\sum\limits_{i=1}^{n} w^{4}(\langle \beta, X_{i}\rangle)
\right\rbrace = E^{1/4}\left\lbrace w^{4}(\langle \beta, X_{i}\rangle)
\right\rbrace.
\end{equation*}
Then Lemma \ref{lemma:bound_moment_lipschitz} implies 
\begin{equation*}
E^{1/4}\left\lbrace w^{4}(\langle \beta, X_{i}\rangle) \right\rbrace \leq
B_{2}(\vert f(0)\vert, \Vert \beta \Vert_{2},4, k,K),
\end{equation*}
where $B_2$ is a function that is increasing in $\Vert \beta\Vert_{2}$. By
Condition \ref{cond:weight} c) we have 
\begin{equation*}
\Vert \beta \Vert_{2}\leq \Vert \beta - \beta^{\ast}\Vert_{2} + \Vert
\beta^{\ast}\Vert_{2}\leq 1 + K.
\end{equation*}
Hence 
\begin{equation*}
E^{1/4}\left\lbrace w^{4}(\langle \beta, X_{i}\rangle) \right\rbrace \leq
B_{2}(\vert f(0)\vert, 1 + K,4, k,K)
\end{equation*}
Going back to \eqref{eq:nemirovski}, we have shown that 
\begin{equation}
E \left(\Vert\mathbb{P}_{n}\left\lbrace X w(\langle \beta, X\rangle) \left[
S_{ab}\varphi^{\ast}+\mathcal{R}_{\overline{c}} \right] \right\rbrace\Vert_{%
\infty} \right) \leq \left( \frac{8\log(2p)}{n} \right)^{1/2} B_{2}(\vert
f(0)\vert, 1 + K,4, k,K) K^{1/2}.  \label{eq:bound_J_first}
\end{equation}
For the second term, note that by the definition of $\mathcal{R}_{\overline{c%
}}$ 
\begin{equation*}
E\left( \mathcal{R}_{\overline{c}} X w(\langle \beta,
X\rangle)\right)=E\left( m_{\overline{c}}(\phi w_{\beta}) \right)
\end{equation*}
and hence 
\begin{equation*}
E\left( \mathcal{R}_{\overline{c}} X w(\langle \beta, X\rangle) - m_{%
\overline{c}}(\phi w_{\beta}) \right)=0.
\end{equation*}
Nemirovski's inequality now yields 
\begin{align}
&E\left(\Vert \mathbb{P}_{n}\left\lbrace \mathcal{R}_{\overline{c}} X
w(\langle \beta, X\rangle) - m_{\overline{c}}(\phi w_{\beta}) \right\rbrace
\Vert_{\infty}\right) \leq  \notag \\
&\left( \frac{8\log(2p)}{n} \right)^{1/2} E \left\lbrace \max_{1\leq j \leq
p} \left( \frac{1}{n} \sum\limits_{i=1}^{n} (\mathcal{R}_{\overline{c},i}
X_{i,j} w(\langle \beta, X_{i}\rangle) - m_{\overline{c},i}(\phi_{j}
w_{\beta}))^{2} \right)^{1/2} \right\rbrace  \notag
\end{align}
Hence by Condition \ref{cond:subgauss} f) we have 
\begin{equation}
E\left(\Vert \mathbb{P}_{n}\left\lbrace \mathcal{R}_{\overline{c}} X
w(\langle \beta, X\rangle) - m_{\overline{c}}(\phi w_{\beta}) \right\rbrace
\Vert_{\infty}\right) \leq \left( \frac{8\log(2p)}{n} \right)^{1/2} K.
\label{eq:bound_J_second}
\end{equation}
Hence putting together \eqref{eq:bound_J_first} and \eqref{eq:bound_J_second}
we get that for all sufficiently large $n$ 
\begin{equation*}
E\left( J \right) \leq \sqrt{\frac{\log(p)}{n}} L_{1},
\end{equation*}
where $L_1$ depends only on $k$ and $K$. Choosing a sufficiently large $L_0$
finishes the proof.

On the other hand, if Condition \ref{cond:approx_sparse_S} b) holds, recall
that 
\begin{equation*}
J=\Vert \mathbb{P}_{n}\left\lbrace X w(\langle \beta^{\ast}, X\rangle) \left[%
S_{ab}\varphi^{\ast} +\mathcal{R}_{\overline{c}} \right] \right\rbrace
\Vert_{\infty} + \Vert \mathbb{P}_{n}\left\lbrace \mathcal{R}_{\overline{c}}
X w(\langle \beta^{\ast}, X\rangle) - m_{\overline{c}}(\phi
w_{\beta^{\ast}}) \right\rbrace \Vert_{\infty}.
\end{equation*}
Now by the definition of $\mathcal{R}_{\overline{c}}$ 
\begin{equation*}
E\left( X w(\langle \beta^{\ast}, X\rangle) \left[S_{ab}\varphi^{\ast} +%
\mathcal{R}_{\overline{c}} \right]\right)=E\left( X w(\langle \beta^{\ast},
X\rangle) S_{ab}\varphi^{\ast} + m_{\overline{c}}(w_{\beta^{\ast}}\phi)%
\right).
\end{equation*}
Moreover, since in this case $\varphi^{\ast}=\varphi_{c}(\langle \theta_{c},
X\rangle)$ and because of the way $\theta_{c}$ was defined, we have 
\begin{equation*}
E\left( X w(\langle \beta^{\ast}, X\rangle) S_{ab}\varphi^{\ast} + m_{%
\overline{c}}(w_{\beta^{\ast}}\phi)\right)=0.
\end{equation*}
The rest of the proof follows as before.
\end{proof}

In the following lemma we show that the approximation error as measured by $%
\Vert \widetilde{\varrho}\Vert_{2,n} $ is of order at most $\sqrt{%
s_{c}\log(p)/n}$.

\begin{lemma}
\label{lemma:bound_res} Assume Conditions \ref{cond:setting_nuis}, \ref%
{cond:glm}, \ref{cond:weight}, \ref{cond:approx_sparse_S} and \ref%
{cond:subgauss} b)-d) hold. Then 
\begin{equation*}
\Vert \widetilde{\varrho}\Vert_{2,n}= O_{P}\left( \sqrt{\frac{s_{c}\log(p)}{n%
}} \right).
\end{equation*}
\end{lemma}

\begin{proof}
\lbrack Proof of Lemma \ref{lemma:bound_res}] Recall that 
\begin{equation*}
\varrho_{i}=\varphi_{c}(\langle \theta^{\ast}_{c}, X_{i} \rangle) -
\varphi^{\ast}_{i}
\end{equation*}
and 
\begin{equation*}
\Vert \widetilde{\varrho}\Vert_{2,n} = \left( \mathbb{P}_{n}\left\lbrace
(S_{ab} w(\langle \widehat{\beta}, X\rangle) {\varrho})^{2}
\right\rbrace\right)^{1/2}.
\end{equation*}
Fix $\varepsilon>0$. Take $L_0>0$ to be chosen later and write 
\begin{align*}
&P\left( \Vert \widetilde{\varrho}\Vert_{2,n}^{2} \geq L_0\left( \frac{%
s_{c}\log(p)}{n} \right)\right)=E\left\lbrace P\left( \Vert \widetilde{%
\varrho}\Vert_{2,n}^{2} \geq L_0\left( \frac{s_{c}\log(p)}{n} \right) \mid 
\widehat{\beta} \right) \right\rbrace= \\
& E\left\lbrace P\left( \Vert \widetilde{\varrho}\Vert_{2,n}^{2} \geq L_0
\left( \frac{s_{c}\log(p)}{n} \right) \mid \widehat{\beta} \right) \mid
\Vert \widehat{\beta}-\beta^{\ast}\Vert_{2} \leq K \right\rbrace P\left(
\Vert \widehat{\beta}-\beta^{\ast}\Vert_{2} \leq K\right)+ \\
& E\left\lbrace P\left( \Vert \widetilde{\varrho}\Vert_{2,n}^{2} \geq L_0
\left( \frac{s_{c}\log(p)}{n} \right) \mid \widehat{\beta} \right) \mid
\Vert \widehat{\beta}-\beta^{\ast}\Vert_{2} > K \right\rbrace P\left( \Vert 
\widehat{\beta}-\beta^{\ast}\Vert_{2} > K\right)\leq \\
& E\left\lbrace P\left( \Vert \widetilde{\varrho}\Vert_{2,n}^{2} \geq
L_0\left( \frac{s_{c}\log(p)}{n} \right) \mid \widehat{\beta} \right) \mid
\Vert \widehat{\beta}-\beta^{\ast}\Vert_{2} \leq K \right\rbrace +P\left(
\Vert \widehat{\beta}-\beta^{\ast}\Vert_{2} > K\right).
\end{align*}
The second term in the last display converges to zero by Condition \ref%
{cond:weight} c) and hence can be made smaller than $\varepsilon/2$ for
sufficiently large $n$. We will prove that by choosing a sufficiently large $%
L_0$ the first term can be made smaller than $\varepsilon/2$ too. Take any
fixed $\beta\in\mathbb{R}^{p}$ such that $\Vert \beta - \beta^{\ast}
\Vert_{2}\leq K$. Condition \ref{cond:weight} c) implies that $\Vert \beta
\Vert_{2} \leq 2K$. Using Cauchy-Schwartz 
\begin{align*}
E\left\lbrace S_{ab}^{2} w^{2}(\langle \beta, X\rangle) {\varrho}^{2}
\right\rbrace &\leq E^{1/2}\left\lbrace S_{ab}^{2} w(\langle \beta,
X\rangle)^{4} \right\rbrace E^{1/2}\left\lbrace S_{ab}^{2} {\varrho}^{4}
\right\rbrace \\
&\leq E^{1/4}\left\lbrace S_{ab}^{4} \right\rbrace E^{1/4}\left\lbrace
w(\langle \beta, X\rangle)^{8} \right\rbrace E^{1/4}\left\lbrace S_{ab}^{4}
\right\rbrace E^{1/4}\left\lbrace{\varrho}^{8} \right\rbrace \\
&= E^{1/2}\left\lbrace S_{ab}^{4} \right\rbrace E^{1/4}\left\lbrace
w(\langle \beta, X\rangle)^{8} \right\rbrace E^{1/4}\left\lbrace{\varrho}%
^{8} \right\rbrace \\
&= I \times II \times III .
\end{align*}
By Condition \ref{cond:subgauss} d) 
\begin{equation}
I \leq K^{2}.  \label{eq:bound_approx_I}
\end{equation}
By Lemma \ref{lemma:bound_moment_lipschitz} 
\begin{equation*}
II \leq B_{2}^{2}(\vert w(0)\vert, \Vert \beta \Vert_{2},8, k, K),
\end{equation*}
where $B_{2}$ is a function that is increasing in $\Vert \beta \Vert_{2}$.
Since $\Vert \beta\Vert_{2}\leq 2K$ 
\begin{equation}
II \leq B_{2}^{2}(\vert w(0)\vert, 2K,4, k, K).  \label{eq:bound_approx_II}
\end{equation}
By Condition \ref{cond:approx_sparse_S} 
\begin{equation}
III=E^{1/4}\left\lbrace \varrho^{8} \right\rbrace \leq \frac{Ks_{c}\log(p)}{n%
}.  \label{eq:bound_approx_III}
\end{equation}

Putting together \eqref{eq:bound_approx_I}, \eqref{eq:bound_approx_II} and %
\eqref{eq:bound_approx_III} we get 
\begin{equation}
E\left\lbrace (S_{ab} w(\langle \beta, X\rangle) {\varrho})^{2}
\right\rbrace \leq L_1 \frac{s_{c}\log(p)}{n},  \label{eq:bound_approx_cond}
\end{equation}
where $L_1$ depends only on $k$ and $K$. Since $\widehat{\beta}$ is
independent of the data, Markov's inequality and \eqref{eq:bound_approx_cond}
imply that whenever $\Vert \widehat{\beta} - \beta^{\ast}\Vert_{2}\leq K$ 
\begin{align*}
P\left( \Vert \widetilde{\varrho}\Vert_{2,n}^{2} \geq L_0\left( \frac{%
s_{c}\log(p)}{n} \right) \mid \widehat{\beta} \right) \leq L_1 \frac{%
s_{c}\log(p)}{n}\frac{n}{L_0 s_{c}\log(p)}=\frac{L_1}{L_0}.
\end{align*}
Hence if $L_0=L_1 2/\varepsilon$ 
\begin{equation*}
E\left\lbrace P\left( \Vert \widetilde{\varrho}\Vert_{2,n}^{2} \geq
L_0\left( \frac{s_{c}\log(p)}{n} \right) \mid \widehat{\beta} \right) \mid
\Vert \widehat{\beta}-\beta^{\ast}\Vert_{2} \leq K \right\rbrace \leq \frac{%
L_1}{L_0} = \varepsilon/2
\end{equation*}
and the result is proven.
\end{proof}

Proposition \ref{prop:upper_RSE} will be used to show that the restricted
upper eigenvalue conditions imposed on $\widehat{\Sigma}_{1}$ by Theorem \ref%
{theo:rate_gen} hold with high probability. Its proof makes use of two key
lemmas from \cite{Loh-smooth}. Let 
\begin{equation*}
\mathbb{K}(l)=\left\lbrace \Delta\in\mathbb{R}^{p}: \Vert
\Delta\Vert_{2}\leq 1, \Vert \Delta \Vert_{0}\leq l \right\rbrace.
\end{equation*}

\begin{proposition}
\label{prop:upper_RSE} Let $W_{1},\dots,W_{n}$ be i.i.d. random vectors in $%
\mathbb{R}^{p}$ such that $\Vert \langle W_{1}, \Delta \rangle
\Vert_{\psi_2}\leq C_{0}$ for all $\Delta$ with $\Vert \Delta \Vert_{2}=1$.
Moreover assume that $0<C_{1}\leq \lambda_{min}(E(WW^{\prime}))\leq
\lambda_{max}(E(WW^{\prime})) \leq C_{0}$ and $\log(p)/n \to 0$. Then there
exists fixed constants $c_{w}>0$ and $N\in\mathbb{N}$ depending only on $%
C_{1}$ and $C_{0}$, and a universal constant $C>0$ such that 
\begin{equation*}
\mathbb{P}_{n}(\langle W,\Delta \rangle^{2}) \leq c_{w} \left( \Vert
\Delta\Vert_{2}^{2} + \frac{\log(p)}{n}\Vert \Delta\Vert_{1}^{2}\right)
\quad \text{ for all } \Delta\in\mathbb{R}^{p}
\end{equation*}
with probability at least 
\begin{equation*}
1-2\exp\left\lbrace -\frac{nC}{2}\min\left(\frac{C_{1}^{2}}{54^2 16 C_{0}^{4}%
}, \frac{C_{1}}{216 C_{0}^{2}}\right)\right\rbrace
\end{equation*}
for all $n\geq N$.
\end{proposition}

\begin{proof}
\lbrack Proof of Proposition \ref{prop:upper_RSE}] Take 
\begin{equation*}
l=c \frac{n}{\log(p)},
\end{equation*}
where $c>0$, depending only on $C_{1}$ and $C_{0}$, will be determined
shortly. By Lemma \ref{lemma:concent_cov_sparse} in Appendix C, if $l\geq 1$
then 
\begin{equation*}
P\left( \sup\limits_{\Delta \in \mathbb{K}(2l)}\vert \mathbb{P}_{n}\left(
\langle W, \Delta \rangle^{2}\right) - E\left( \langle W, \Delta
\rangle^{2}\right) \vert\geq \frac{C_{1}}{54} \right) \leq 2\exp\left\lbrace
-nC\min\left(\frac{C_{1}^{2}}{54^2 16 C_{0}^{4}}, \frac{C_{1}}{216 C_{0}^{2}}%
\right) + 2c n\right\rbrace,
\end{equation*}
where $C>0$ is a fixed universal constant. If we take 
\begin{equation*}
c = \frac{C}{4} \min\left(\frac{C_{1}^{2}}{54^2 16C_{0}^{4}}, \frac{C_{1}}{%
216 C_{0}^{2}}\right)
\end{equation*}
we have that for all $n$ such that 
\begin{equation*}
l=c \frac{n}{\log(p)} \geq 1,
\end{equation*}
\begin{equation*}
P\left( \sup\limits_{\Delta \in \mathbb{K}(2l)}\vert \mathbb{P}_{n}\left(
\langle W, \Delta \rangle^{2}\right) - E\left( \langle W, \Delta
\rangle^{2}\right) \vert\geq \frac{C_{1}}{54} \right) \leq 2\exp\left\lbrace
-\frac{nC}{2}\min\left(\frac{C_{1}^{2}}{54^2 16C_{0}^{4}}, \frac{C_{1}}{216
C_{0}^{2}}\right)\right\rbrace.
\end{equation*}
By Lemma 13 in \cite{Loh-smooth}, whenever 
\begin{equation*}
\sup\limits_{\Delta \in \mathbb{K}(2l)}\vert \mathbb{P}_{n}\left( \langle W,
\Delta \rangle^{2}\right) - E\left( \langle W, \Delta \rangle^{2}\right)
\vert\leq \frac{C_{1}}{54},
\end{equation*}
we have that for all $\Delta \in \mathbb{R}^{p}$ 
\begin{align*}
\mathbb{P}_{n}(\langle W,\Delta \rangle^{2}) &\leq \frac{3}{2}%
\lambda_{max}(E(WW^{\prime})) \Vert \Delta \Vert_{2}^{2} + \frac{%
\lambda_{min}(E(WW^{\prime}))}{2l}\Vert\Delta\Vert_{1}^{2} \\
&\leq \frac{3}{2}C_{0}\Vert \Delta \Vert_{2}^{2} + \frac{C_0}{2l}%
\Vert\Delta\Vert_{1}^{2} \\
& \leq \max\left(\frac{3C_0}{2}, \frac{C_0}{2c} \right)\left( \Vert
\Delta\Vert_{2}^{2} + \frac{\log(p)}{n}\Vert \Delta\Vert_{1}^{2}\right).
\end{align*}
The result now follows taking 
\begin{equation*}
c_{w}=\max\left(\frac{3C_0}{2}, \frac{C_0}{2c} \right).
\end{equation*}
\end{proof}

\begin{lemma}
\label{lemma:Hbeta} Assume Conditions \ref{cond:setting_nuis}, \ref{cond:glm}%
, \ref{cond:weight}, \ref{cond:approx_sparse_S}, \ref{cond:subgauss} b)-e)
hold. Then 
\begin{equation*}
\mathbb{P}_{n}^{1/2}\left\lbrace \left[ S_{ab}\varphi^{\ast}+\mathcal{R}_{%
\overline{c}} \right]^{2} \left( w(\langle \widehat{\beta},
X\rangle)-w(\langle \beta^{\ast}, X\rangle) \right)^{2} \right\rbrace
=O_{P}\left(\sqrt{\frac{s_{\beta}\log(p)}{n}}\right).
\end{equation*}
\end{lemma}

\begin{proof}
\lbrack Proof of Lemma \ref{lemma:Hbeta}] Let 
\begin{equation*}
H=\mathbb{P}_{n}^{1/2}\left\lbrace \left[ S_{ab}\varphi^{\ast}+\mathcal{R}_{%
\overline{c}} \right]^{2} \left( w(\langle \widehat{\beta},
X\rangle)-w(\langle \beta^{\ast}, X\rangle) \right)^{2} \right\rbrace.
\end{equation*}
Take $\varepsilon>0$. Let $L_2,L_0>0$, to be chosen later. Then 
\begin{align*}
&P\left( H > L_2\sqrt{\frac{s_{\beta}\log(p)}{n}} \right) = E\left\lbrace
P\left( H > L_2\sqrt{\frac{s_{\beta}\log(p)}{n}} \mid \widehat{\beta}\right)
\right\rbrace = \\
& E\left\lbrace P\left( H > L_2\sqrt{\frac{s_{\beta}\log(p)}{n}} \mid 
\widehat{\beta}\right) \mid\Vert \widehat{\beta} - \beta^{\ast}\Vert_{2}
\leq L_0 \sqrt{\frac{s_{\beta}\log(p)}{n}} \right\rbrace P\left( \Vert 
\widehat{\beta} - \beta^{\ast}\Vert_{2} \leq L_0 \sqrt{\frac{s_{\beta}\log(p)%
}{n}} \right) + \\
& E\left\lbrace P\left( H > L_2\sqrt{\frac{s_{\beta}\log(p)}{n}} \mid 
\widehat{\beta}\right) \mid \Vert \widehat{\beta} - \beta^{\ast}\Vert_{2} >
L_0 \sqrt{\frac{s_{\beta}\log(p)}{n}} \right\rbrace P\left( \Vert \widehat{%
\beta} - \beta^{\ast}\Vert_{2} > L_0 \sqrt{\frac{s_{\beta}\log(p)}{n}}
\right)\leq \\
& E\left\lbrace P\left( H >L_2\sqrt{\frac{s_{\beta}\log(p)}{n}} \mid 
\widehat{\beta}\right) \mid \Vert \widehat{\beta} - \beta^{\ast}\Vert_{2}
\leq L_0 \sqrt{\frac{s_{\beta}\log(p)}{n}} \right\rbrace + P\left( \Vert 
\widehat{\beta} - \beta^{\ast}\Vert_{2} > L_0 \sqrt{\frac{s_{\beta}\log(p)}{n%
}} \right).
\end{align*}
By Condition \ref{cond:weight} c) we can choose $L_0$ large enough such that
the second term in the last display is smaller than $\varepsilon/2$ for all
sufficiently large $n$. We will show that we can choose $L_2$ large enough
such that the first term in the last display is smaller than $\varepsilon/2$
for all sufficiently large $n$ too. This will prove the lemma. Choose $n$
large enough such that 
\begin{equation*}
\frac{s_{\beta}\log(p)}{n} \leq 1
\end{equation*}
and take a fixed $\beta$ such that 
\begin{equation*}
\Vert \beta - \beta^{\ast}\Vert_{2} \leq L_0 \sqrt{\frac{s_{\beta}\log(p)}{n}%
}.
\end{equation*}
Using Cauchy-Schwartz write 
\begin{equation*}
E\left\lbrace \left[S_{ab}\varphi^{\ast}+\mathcal{R}_{\overline{c}}\right]%
^{2} \left( w(\langle \beta, X\rangle)-w(\langle \beta^{\ast}, X\rangle)
\right)^{2}\right\rbrace \leq E^{1/2}\left\lbrace \left[S_{ab}\varphi^{\ast}+%
\mathcal{R}_{\overline{c}}\right]^{4} \right\rbrace E^{1/2}\left\lbrace
\left( w(\langle \beta, X\rangle)-w(\langle \beta^{\ast}, X\rangle)
\right)^{4} \right\rbrace.
\end{equation*}
By Condition \ref{cond:subgauss} e) 
\begin{equation*}
E^{1/2}\left\lbrace \left[S_{ab}\varphi^{\ast}+\mathcal{R}_{\overline{c}} %
\right]^{4} \right\rbrace \leq K.
\end{equation*}
By Lemma \ref{lemma:lipschitz} and Condition \ref{cond:subgauss} b) 
\begin{equation*}
E^{1/2}\left\lbrace \left( w(\langle \beta, X\rangle)-w(\langle
\beta^{\ast}, X\rangle) \right)^{4} \right\rbrace \leq B_{1}^{2}(\Vert
\beta^{\ast}\Vert_{2}, \Vert \beta - \beta^{\ast}\Vert_{2}, 4, k, K) 2 K^{2}
\Vert \beta - \beta^{\ast}\Vert_{2}^{2},
\end{equation*}
where $B_1$ is a function that is increasing in $\Vert \beta^{\ast}\Vert_{2}$
and $\Vert \beta - \beta^{\ast}\Vert_{2}$. By Condition \ref{cond:weight}
c), $\Vert \beta^{\ast}\Vert_{2}\leq K$ and by assumption 
\begin{equation*}
\Vert \beta - \beta^{\ast}\Vert_{2}\leq L_0\sqrt{\frac{s_{\beta}\log(p)}{n}}
\leq L_0.
\end{equation*}
Hence 
\begin{equation*}
E^{1/2}\left\lbrace \left( w(\langle \beta, X\rangle)-w(\langle
\beta^{\ast}, X\rangle) \right)^{4} \right\rbrace \leq B_{1}^{2}(K, K + L_0,
4, k, K) 2K^{2} L_{0}^{2} \frac{s_{\beta}\log(p)}{n}.
\end{equation*}
It follows that 
\begin{equation*}
E\left\lbrace \left[S_{ab}\varphi^{\ast}+\mathcal{R}_{\overline{c}}\right]%
^{2} \left( w(\langle \beta, X\rangle)-w(\langle \beta^{\ast}, X\rangle)
\right)^{2}\right\rbrace \leq L_{1} L_{0}^{2} \frac{s_{\beta}\log(p)}{n},
\end{equation*}
where $L_1$ depends only on $k,K$. Since $\widehat{\beta}$ is independent of
the data, an application of Markov's inequality finishes the proof.
\end{proof}

The following lemma is needed to show that $L_{c}(\cdot,\phi, w_{\widehat{%
\beta}})$ satisfies RSC over $\mathbb{D}(c_{\mathcal{C}},s) \cap \lbrace
\Delta : \Vert\Delta\Vert_{2}\leq 1\rbrace$ with high probability for
suitable choices $s$.

\begin{lemma}
\label{lemma:eigen_w} Assume Conditions \ref{cond:weight} and \ref%
{cond:subgauss} b)-d) hold. Then there exists a constant $c(k,K)$ depending
only on $k,K$ such that 
\begin{equation*}
c(k,K) \leq \inf_{\Vert\Delta\Vert_{2}=1}E\left\lbrace w(\langle
\beta^{\ast}, X \rangle) S_{ab}\langle \Delta, X \rangle^{2} \right\rbrace
\end{equation*}
\end{lemma}

\begin{proof}
\lbrack Proof of Lemma \ref{lemma:eigen_w}] Take $\Delta$ with $\Vert
\Delta\Vert_{2}=1$. Take $T>0$, a truncation parameter to be chosen later.
Since $w$ is positive 
\begin{equation*}
E\left\lbrace w(\langle \beta^{\ast}, X \rangle) S_{ab}\langle \Delta, X
\rangle^{2} \right\rbrace \geq E\left\lbrace w(\langle \beta^{\ast}, X
\rangle) S_{ab}\langle \Delta, X \rangle^{2} I\left\lbrace \vert \langle
\beta^{\ast}, X \rangle \vert \leq T \right\rbrace\right\rbrace.
\end{equation*}
By Condition \ref{cond:weight} b), $w$ is continuous. Then 
\begin{equation*}
\eta=\inf\limits_{\vert u \vert \leq T} w(u) >0.
\end{equation*}
Hence 
\begin{equation*}
E\left\lbrace w(\langle \beta^{\ast}, X \rangle) S_{ab}\langle \Delta, X
\rangle^{2} I\left\lbrace \vert \langle \beta^{\ast}, X \rangle \vert \leq T
\right\rbrace\right\rbrace \geq \eta E\left\lbrace S_{ab}\langle \Delta, X
\rangle^{2} I\left\lbrace \vert \langle \beta^{\ast}, X \rangle \vert \leq T
\right\rbrace\right\rbrace.
\end{equation*}
Now, by Condition \ref{cond:subgauss} c) 
\begin{align*}
E\left\lbrace S_{ab}\langle \Delta, X \rangle^{2} I\left\lbrace \vert
\langle \beta^{\ast}, X \rangle \vert \leq T \right\rbrace\right\rbrace &=
E\left\lbrace S_{ab}\langle \Delta, X \rangle^{2} \right\rbrace -
E\left\lbrace S_{ab}\langle \Delta, X \rangle^{2} I\left\lbrace \vert
\langle \beta^{\ast}, X \rangle \vert > T \right\rbrace\right\rbrace \\
&\geq k - E\left\lbrace S_{ab}\langle \Delta, X \rangle^{2} I\left\lbrace
\vert \langle \beta^{\ast}, X \rangle \vert > T \right\rbrace\right\rbrace.
\end{align*}
Using Cauchy-Schwartz 
\begin{align*}
E\left\lbrace S_{ab}\langle \Delta, X \rangle^{2} I\left\lbrace \vert
\langle \beta^{\ast}, X \rangle \vert > T \right\rbrace\right\rbrace &\leq
E^{1/2}\left\lbrace S_{ab}^{2}\langle \Delta, X \rangle^{4} \right\rbrace
P^{1/2}\left \vert \langle \beta^{\ast}, X \rangle \vert > T\right) \\
& \leq E^{1/4}\left\lbrace S_{ab}^{4}\right\rbrace E^{1/4}\left\lbrace
\langle \Delta, X \rangle^{8}\right\rbrace P^{1/2}\left \vert \langle
\beta^{\ast}, X \rangle \vert > T\right).
\end{align*}
By Condition \ref{cond:subgauss} b) 
\begin{equation*}
E^{1/4}\left\lbrace \langle \Delta, X \rangle^{8} \right\rbrace \leq 8 K^{2}
\end{equation*}
and by Condition \ref{cond:subgauss} d) 
\begin{equation*}
E^{1/4}\left\lbrace S_{ab}^{4}\right\rbrace \leq K.
\end{equation*}
By Markov's inequality, Condition \ref{cond:weight} c) and Condition \ref%
{cond:subgauss} b) 
\begin{equation*}
P^{1/2}\left( \vert \langle \beta^{\ast}, X \rangle \vert > T\right) \leq 
\frac{E^{1/2}\left\lbrace \langle \beta^{\ast}, X \rangle^{2}\right\rbrace }{%
T}\leq \frac{\sqrt{2}K^{2}}{T}.
\end{equation*}
Thus 
\begin{equation*}
E\left\lbrace w(\langle \beta^{\ast}, X \rangle) S_{ab}\langle \Delta, X
\rangle^{2} \right\rbrace \geq \eta \left( k - \frac{\sqrt{2}\: 8K^{5}}{T}%
\right).
\end{equation*}
Choosing $T= \sqrt{2} \:16K^{5}/ k$ yields the desired result.
\end{proof}

The following proposition is a key step in showing that $L_{c}(\cdot,\phi,
w_{\widehat{\beta}})$ satisfies RSC over $\mathbb{D}(c_{\mathcal{C}},s) \cap
\lbrace \Delta : \Vert\Delta\Vert_{2}\leq 1\rbrace$ with high probability
for all sufficiently large $n$ when $s=s_{c}$ or $s=\max(s_c,s_{\beta})$.

\begin{proposition}
\label{prop:preRSC} Assume Conditions \ref{cond:setting_nuis}, \ref{cond:glm}%
, \ref{cond:weight}, \ref{cond:approx_sparse_S} and \ref{cond:subgauss}
b)-d) hold. Then there exist positive constants $k^{RSC}_1$, $k^{RSC}_2$ and 
$n_3\in\mathbb{N}$ depending only on $k$ and $K$ such that if $n \geq n_3$ 
\begin{equation*}
\delta\mathcal{L}(\Delta, \theta_{c}^{\ast},\widehat{\beta})\geq k^{RSC}_1
\Vert \Delta\Vert_{2}^{2} -k^{RSC}_{2} \left( \frac{\log(p)}{n}%
\right)^{1/2}\Vert \Delta\Vert_{1}\Vert \Delta\Vert_{2} \text{ for all }
\Delta \text{ with } \Vert \Delta \Vert_{2}\leq 1
\end{equation*}
with probability tending to one.
\end{proposition}

\begin{proof}
\lbrack Proof of Proposition \ref{prop:preRSC}]

The outline of the proof is similar to that of Proposition 2 from \cite%
{NegahbanTR}. However, several technical complications have to be dealt with
due to the use of the random weights $w(\langle \widehat{\beta}, X_{i}
\rangle)$ in the loss function. Since by Condition \ref{cond:weight} c) $%
\widehat{\beta}$ is independent of the data and $\Vert \widehat{\beta} -
\beta^{\ast} \Vert_{2}=o_{P}(1)$, it suffices to show that the result holds
for $L_{c}(\cdot,\phi, w_{\beta})$ for any fixed $\beta\in\mathbb{R}^{p}$
such that $\Vert \beta - \beta^{\ast}\Vert_{2} \leq L_0, $ where $L_0$ will
depend only on $k$ and $K$ and will be chosen later. Let $\beta\in\mathbb{R}%
^{p}$ satisfy $\Vert \beta - \beta^{\ast}\Vert_{2} \leq L_0$ and to lighten
the notation let $w_{i}=w(\langle \beta, X_{i} \rangle)$, $%
w_{i}^{\ast}=w(\langle \beta^{\ast}, X_{i} \rangle)$. By the standard
formula for the remainder in a Taylor expansion 
\begin{equation*}
\delta\mathcal{L}(\Delta, \theta_{c}^{\ast},\widehat{\beta})= \frac{1}{2n}%
\sum\limits_{i=1}^{n} w_{i}S_{ab}\varphi_{c}^{\prime}\left(\langle
\theta_{c}^{\ast}, X_{i}\rangle + t_{i} \langle \Delta, X_{i} \rangle
\right) \langle \Delta, X_{i} \rangle^{2},
\end{equation*}
where $t_{i}\in[0,1]$. Take $\Delta \in \mathbb{R}^{p}$ with $\Vert
\Delta\Vert_{2}=\delta \in (0,1]$. Consider four truncation parameters, $%
\tau=\tau(\delta)=h \delta$, $\zeta, t$ and $T$, where $h, \zeta, t$ and $T$
will be chosen shortly and will depend only on $k$ and $K$. Let $\mathcal{T}%
_{\tau}(u)= u^{2} I\lbrace\vert u\vert \leq 2 \tau \rbrace$, $\mathcal{T}%
_{\zeta}(u)=\min(u, \zeta)$ and $\mathcal{T}_{t}(u)=u I\left\lbrace \vert u
\vert \leq t \right\rbrace$. Since by Condition \ref{cond:glm} a) $%
\varphi_{c}^{\prime}>0$ and $\mathcal{T}_{\tau}(u) \leq u^{2}$ for all $u$
we have 
\begin{equation*}
\delta\mathcal{L}(\Delta, \theta_{c}^{\ast},\widehat{\beta}) \geq \frac{1}{2n%
}\sum\limits_{i=1}^{n} w_{i}S_{ab,i}\varphi_{c}^{\prime}\left(\langle
\theta_{c}^{\ast}, X_{i}\rangle + t_{i} \langle \Delta, X_{i} \rangle
\right) \mathcal{T}_{\tau}\left(\langle \Delta, X_{i} \rangle \right)
I\lbrace \vert \langle \theta_{c}^{\ast}, X_i \rangle \vert \leq T \rbrace.
\end{equation*}
If $\mathcal{T}_{\tau}\left(\langle \Delta, X_{i} \rangle \right)\neq 0$ and 
$\vert \langle \theta_{c}^{\ast}, X_i \rangle \vert \leq T$ then, since $%
\tau \leq h$, $\vert \langle \theta_{c}^{\ast}, X_{i}\rangle + t_{i} \langle
\Delta, X_{i} \rangle \vert \leq T+2h$. Let $\epsilon=
(1/2)\min\limits_{\vert u \vert \leq T + 2h}\varphi_{c}^{\prime}(u)$. Then 
\begin{equation*}
\delta\mathcal{L}(\Delta, \theta_{c}^{\ast}) \geq \epsilon \frac{1}{n}%
\sum\limits_{i=1}^{n} w_{i}S_{ab,i}\mathcal{T}_{\tau}\left(\langle \Delta,
X_{i} \rangle \right) I\lbrace \vert \langle \theta_{c}^{\ast}, X_i \rangle
\vert \leq T \rbrace .
\end{equation*}
Moreover, since $\mathcal{T}_{\zeta}(u) \leq u$ and $\mathcal{T}_{t}(u) \leq
u$ for non-negative $u$, 
\begin{equation*}
\delta\mathcal{L}(\Delta, \theta_{c}^{\ast},\widehat{\beta}) \geq \epsilon 
\frac{1}{n}\sum\limits_{i=1}^{n} \mathcal{T}_{t}(w_{i})\mathcal{T}%
_{\zeta}(S_{ab,i})\mathcal{T}_{\tau}\left(\langle \Delta, X_{i} \rangle
\right) I\lbrace \vert \langle \theta_{c}^{\ast}, X_i \rangle \vert \leq T
\rbrace .
\end{equation*}
Let $\mu_{n}=\sqrt{\log(p)/n}$. We will show now that for some constants $%
k_1,k_2$ that depend only on $k$ and $K$, for all sufficiently large $n$ 
\begin{equation}
\frac{1}{n}\sum\limits_{i=1}^{n} \mathcal{T}_{t}(w_{i})\mathcal{T}%
_{\zeta}(S_{ab,i}) \mathcal{T}_{\tau(\delta)}\left(\langle \Delta, X_{i}
\rangle \right) I\lbrace \vert \langle \theta_{c}^{\ast}, X_i \rangle \vert
\leq T \rbrace\geq k_1 \delta^{2} - k_{2} \mu_{n} \Vert \Delta \Vert_{1}
\delta \quad \text{ for all } \Delta \text{ with } \Vert
\Delta\Vert_{2}=\delta  \label{eq:bound_trun_delta}
\end{equation}
with high probability. It suffices to show that \eqref{eq:bound_trun_delta}
holds for $\delta=1$. In fact, if $\Vert \Delta \Vert_{2} =\delta$ the bound
in \eqref{eq:bound_trun_delta} applied to $\Delta/\delta$ gives 
\begin{equation}
\frac{1}{n}\sum\limits_{i=1}^{n}\mathcal{T}_{t}(w_{i})\mathcal{T}%
_{\zeta}(S_{ab,i})\mathcal{T}_{\tau(1)}\left(\langle \Delta/\delta, X_{i}
\rangle \right) I\lbrace \vert \langle \theta_{c}^{\ast}, X_i \rangle \vert
\leq T \rbrace \geq k_1 - k_{2} \mu_{n} \frac{\Vert \Delta \Vert_{1}}{\delta}%
.  \notag
\end{equation}
It is easy to verify that $\mathcal{T}_{\tau(1)}(u/\delta)=(1/\delta)^{2}%
\mathcal{T}_{\tau(\delta)}(u)$. Hence the claim follows from multiplying the
equation above by $\delta^{2}$ on both sides. Thus, we will prove that %
\eqref{eq:bound_trun_delta} holds for $\delta=1$.

Define a new truncation function 
\begin{equation*}
\widehat{\mathcal{T}}_{\tau}(u)=u^{2} I\lbrace \vert u \vert \leq \tau
\rbrace + (2\tau - u)^{2} I\lbrace \tau < u \leq 2 \tau \rbrace +
(u+2\tau)^{2} I\lbrace -2\tau \leq u < -\tau \rbrace.
\end{equation*}
It is easy to verify that $\widehat{\mathcal{T}}_{\tau}$ is $2\tau$%
-Lipschitz. Since $\widehat{\mathcal{T}}_{\tau}(u) \leq \mathcal{T}%
_{\tau}(u) $ for all $u$, it suffices to show that 
\begin{equation}
\frac{1}{n}\sum\limits_{i=1}^{n} \mathcal{T}_{t}(w_{i})\mathcal{T}%
_{\zeta}(S_{ab,i})\widehat{\mathcal{T}}_{\tau(1)}\left(\langle \Delta, X_{i}
\rangle \right) I\lbrace \vert \langle \theta_{c}^{\ast}, X_i \rangle \vert
\leq T \rbrace \geq k_1 - k_{2} \mu_{n} \Vert \Delta \Vert_{1} .  \notag
\end{equation}

Let 
\begin{align*}
& \mathcal{M}(\Delta ,\beta ;O_{i})=\mathcal{T}_{t}(w_{i})\mathcal{T}_{\zeta
}(S_{ab,i})\widehat{\mathcal{T}}_{\tau (1)}\left( \langle \Delta
,X_{i}\rangle \right) I\{|\langle \theta _{c}^{\ast },X_{i}\rangle |\leq T\},
\\
& M(\Delta ,\beta )=E\left( \mathcal{M}(\Delta ,\beta ;O_{i})\right) , \\
& f_{n}(\Delta ,\beta )=\left\vert \mathbb{P}_{n}\left\{ \mathcal{M}(\Delta
,\beta ;O_{i})\right\} -M(\Delta ,\beta )\right\vert .
\end{align*}%
For $r\geq 1$ let 
\begin{equation*}
Z_{n}(r;\beta )=\sup\limits_{\Vert \Delta \Vert _{2}=1,\Vert \Delta \Vert
_{1}\leq r}f_{n}(\Delta ,\beta ).
\end{equation*}%
We will show that there exist constants $k_{2}$, $c(k,K)$, depending only on 
$k,K$, such that for all sufficiently large $n$, for all $\Delta $ with $%
\Vert \Delta \Vert _{2}=1$ and $r\geq 1$ 
\begin{align}
& M(\Delta ,\beta )\geq \frac{c(k,K)}{2}  \label{eq:RSC_lower_bound_exp_fin}
\\
& P\left( Z_{n}(r;\beta )>c(k,K)/8+(k_{2}/2)r\mu _{n}\right) \leq \exp
\left( \frac{-n}{64(\zeta th^{2})^{2}}\left( c(k,K)/8+(k_{2}/2)r\mu
_{n}\right) ^{2}\right) .  \label{eq:RSC_concentration}
\end{align}%
Let 
\begin{equation*}
g_{n}(r)=c(k,K)/8+(k_{2}/2)r\mu _{n},
\end{equation*}%
and 
\begin{equation*}
E_{n}=\left\{ \exists \Delta ,\Vert \Delta \Vert _{2}=1:f_{n}(\Delta ,\beta
)\geq 2g_{n}(\Vert \Delta \Vert _{1})\right\}
\end{equation*}%
Assume for the moment that \eqref{eq:RSC_concentration} holds. Using the
peeling argument in Lemma 9 from \cite{Raskutti-minimax-TR} then yields 
\begin{equation}
P\left( E_{n}\right) \leq 2\frac{\exp ((-ng_{n}(1)^{2})/(64(\zeta
th^{2})^{2}))}{1-\exp ((-ng_{n}(1)^{2})/(64(\zeta th^{2})^{2}))}\leq 2\frac{%
\exp ((-nc(k,h)^{2})/(4096(\zeta th^{2})^{2}))}{1-\exp
((-nc(k,h)^{2})/(4096(\zeta th^{2})^{2}))}.  \label{eq:bound_peeling}
\end{equation}%
This together with \eqref{eq:RSC_lower_bound_exp_fin} in turn implies 
\begin{align*}
\frac{1}{n}\sum\limits_{i=1}^{n}\mathcal{T}_{t}(w_{i})\mathcal{T}_{\zeta
}(S_{ab,i})\widehat{\mathcal{T}}_{\tau (1)}\left( \langle \Delta
,X_{i}\rangle \right) I\{|\langle \theta _{c}^{\ast },X_{i}\rangle |\leq
T\}& =M(\Delta ,\beta )+\mathbb{P}_{n}\left\{ \mathcal{M}(\Delta ,\beta
;O_{i})\right\} -M(\Delta ,\beta ) \\
& \geq M(\Delta ,\beta )-f_{n}(\Delta ,\beta ) \\
& \geq \frac{c(k,K)}{2}-\frac{c(k,K)}{4}-k_{2}\mu _{n}\Vert \Delta \Vert _{1}
\\
& =\frac{c(k,K)}{4}-k_{2}\mu _{n}\Vert \Delta \Vert _{1}\text{ for all }%
\Delta ,\Vert \Delta \Vert _{2}=1,
\end{align*}%
with probability at least 
\begin{equation*}
1-2\frac{\exp ((-nc(k,K)^{2})/(4096(\zeta th^{2})^{2}))}{1-\exp
((-nc(k,K)^{2})/(4096(\zeta th^{2})^{2}))}.
\end{equation*}%
To summarize, if we prove \eqref{eq:RSC_lower_bound_exp_fin} and %
\eqref{eq:RSC_concentration}, the Proposition will be proven with $%
k_{1}^{RSC}=\epsilon c(k,K)/4$, $k_{2}^{RSC}=\epsilon k_{2}$ and $k_{2}$, $%
h,t$ and $\zeta $ to be defined below (depending only on $k,K$).

\bigskip

\textbf{Proof of \eqref{eq:RSC_lower_bound_exp_fin}}

\bigskip

Write 
\begin{align*}
M(\Delta, \beta)&=E\left\lbrace \mathcal{T}_{t}(w_{i})\mathcal{T}%
_{\zeta}(S_{ab,i})\widehat{\mathcal{T}}_{\tau(1)}\left(\langle \Delta, X_{i}
\rangle \right) I\lbrace \vert \langle \theta_{c}^{\ast}, X_i \rangle \vert
\leq T \rbrace \right\rbrace \\
&= E\left\lbrace \mathcal{T}_{t}(w_{i})\mathcal{T}_{\zeta}(S_{ab,i})\widehat{%
\mathcal{T}}_{\tau(1)}\left(\langle \Delta, X_{i} \rangle \right)
\right\rbrace - E\left\lbrace \mathcal{T}_{t}(w_{i})\mathcal{T}%
_{\zeta}(S_{ab,i})\widehat{\mathcal{T}}_{\tau(1)}\left(\langle \Delta, X_{i}
\rangle \right) I\lbrace \vert \langle \theta_{c}^{\ast}, X_i \rangle \vert
> T \rbrace \right\rbrace.
\end{align*}
To bound the first term write 
\begin{align*}
&E\left\lbrace \mathcal{T}_{t}(w_{i})\mathcal{T}_{\zeta}(S_{ab,i})\widehat{%
\mathcal{T}}_{\tau(1)}\left(\langle \Delta, X_{i} \rangle \right)
\right\rbrace \geq E\left\lbrace \mathcal{T}_{t}(w_{i})\mathcal{T}%
_{\zeta}(S_{ab,i})\langle \Delta, X_{i} \rangle^{2} I\left\lbrace
\vert\langle \Delta, X_{i} \rangle\vert \leq \tau \right\rbrace
\right\rbrace = \\
& E\left\lbrace \mathcal{T}_{t}(w_{i})\mathcal{T}_{\zeta}(S_{ab,i})\langle
\Delta, X_{i} \rangle^{2} \right\rbrace - E\left\lbrace \mathcal{T}%
_{t}(w_{i})\mathcal{T}_{\zeta}(S_{ab,i})\langle \Delta, X_{i} \rangle^{2}
I\left\lbrace \vert\langle \Delta, X_{i} \rangle\vert > \tau \right\rbrace
\right\rbrace = \\
& E\left\lbrace w^{\ast}_{i}\mathcal{T}_{\zeta}(S_{ab,i})\langle \Delta,
X_{i} \rangle^{2} \right\rbrace + E\left\lbrace (w_{i}-w^{\ast}_{i})\mathcal{%
T}_{\zeta}(S_{ab,i})\langle \Delta, X_{i} \rangle^{2} \right\rbrace -
E\left\lbrace w_{i}\mathcal{T}_{\zeta}(S_{ab,i})\langle \Delta, X_{i}
\rangle^{2} I\left\lbrace \vert w_{i}\vert > t \right\rbrace \right\rbrace \\
&- E\left\lbrace \mathcal{T}_{t}(w_{i})\mathcal{T}_{\zeta}(S_{ab,i})\langle
\Delta, X_{i} \rangle^{2} I\left\lbrace \vert\langle \Delta, X_{i}
\rangle\vert > \tau \right\rbrace \right\rbrace \\
& = I + II -III - IV
\end{align*}
where in the first inequality we used that $\widehat{\mathcal{T}}%
_{\tau(1)}\left(u \right)\geq u^{2} I\left\lbrace u\leq \tau \right\rbrace$
for non-negative $u$. \bigskip

\textbf{Bounding I}

\bigskip Using that $\mathcal{T}_{\zeta}(u)\geq u I\left\lbrace \vert u
\vert \leq \zeta \right\rbrace$ for non-negative $u$ 
\begin{align*}
E\left\lbrace w^{\ast}_{i}\mathcal{T}_{\zeta}(S_{ab,i})\langle \Delta, X_{i}
\rangle^{2} \right\rbrace &\geq E\left\lbrace w^{\ast}_{i}
S_{ab,i}I\left\lbrace S_{ab,i} \leq \zeta \right\rbrace \langle \Delta,
X_{i} \rangle^{2} \right\rbrace \\
&= E\left\lbrace w^{\ast}_{i} S_{ab,i}\langle \Delta, X_{i} \rangle^{2}
\right\rbrace - E\left\lbrace w^{\ast}_{i} S_{ab,i}I\left\lbrace S_{ab,i}
>\zeta \right\rbrace \langle \Delta, X_{i} \rangle^{2} \right\rbrace .
\end{align*}
By Lemma \ref{lemma:eigen_w} 
\begin{equation*}
E\left\lbrace w^{\ast}_{i} S_{ab,i}\langle \Delta, X_{i} \rangle^{2}
\right\rbrace \geq c(k, K).
\end{equation*}
Using Cauchy-Schwartz 
\begin{align*}
&E\left\lbrace w^{\ast}_{i} S_{ab,i}I\left\lbrace S_{ab,i} >\zeta
\right\rbrace \langle \Delta, X_{i} \rangle^{2} \right\rbrace \leq
E^{1/2}\left\lbrace (w^{\ast}_{i})^{2} S_{ab,i} \langle \Delta, X_{i}
\rangle^{4} \right\rbrace E^{1/2}\left\lbrace S_{ab,i}I\left\lbrace S_{ab,i}
>\zeta \right\rbrace \right\rbrace \\
& \leq E^{1/4}\left\lbrace S_{ab,i}(w^{\ast}_{i})^{4}\right\rbrace
E^{1/4}\left\lbrace S_{ab,i} \langle \Delta, X_{i} \rangle^{8} \right\rbrace
E^{1/4}\left\lbrace S_{ab,i}^{2} \right\rbrace P^{1/4}\left( S_{ab,i} >\zeta
\right) \\
& \leq E^{1/8}\left\lbrace S_{ab,i}^{2}\right\rbrace E^{1/8}\left\lbrace
(w^{\ast}_{i})^{8}\right\rbrace E^{1/8}\left\lbrace S_{ab,i}^{2}
\right\rbrace E^{1/8}\left\lbrace \langle \Delta, X_{i} \rangle^{16}
\right\rbrace E^{1/4}\left\lbrace S_{ab,i}^{2} \right\rbrace P^{1/4}\left(
S_{ab,i} >\zeta \right) \\
&= E^{1/2}\left\lbrace S_{ab,i}^{2}\right\rbrace E^{1/8}\left\lbrace
(w^{\ast}_{i})^{8}\right\rbrace E^{1/8}\left\lbrace \langle \Delta, X_{i}
\rangle^{16} \right\rbrace P^{1/4}\left( S_{ab,i} >\zeta \right).
\end{align*}
By Condition \ref{cond:subgauss} d) $E^{1/2}\left\lbrace
S_{ab,i}^{2}\right\rbrace\leq K$. By Lemma \ref{lemma:bound_moment_lipschitz}
\begin{equation*}
E^{1/8}\left\lbrace (w^{\ast}_{i})^{8}\right\rbrace \leq B_{2}(\vert w(0)
\vert, \Vert \beta^{\ast}\Vert_{2}, 8, k, K),
\end{equation*}
where $B_2$ is a function that is increasing in $\Vert \beta^{\ast}\Vert_{2}$%
. Since by Condition \ref{cond:weight} c) $\Vert \beta^{\ast}\Vert_{2}\leq,
K $ 
\begin{equation*}
E^{1/8}\left\lbrace (w^{\ast}_{i})^{8}\right\rbrace \leq B_{2}(\vert w(0)
\vert, K, 8, k, K),
\end{equation*}
By Condition \ref{cond:subgauss} b) 
\begin{equation*}
E^{1/8}\left\lbrace \langle \Delta, X_{i} \rangle^{16} \right\rbrace \leq
16K^{2}.
\end{equation*}
Using Markov's inequality 
\begin{equation*}
P\left( S_{ab,i} >\zeta \right) \leq \frac{E(S_{ab}^{2})}{\zeta^{2}}\leq 
\frac{K^{2}}{\zeta^{2}}.
\end{equation*}
Hence 
\begin{equation*}
E\left\lbrace w^{\ast}_{i} S_{ab,i}I\left\lbrace S_{ab,i} >\zeta
\right\rbrace \langle \Delta, X_{i} \rangle^{2} \right\rbrace \leq 16
K^{7/2} B_{2}(\vert w(0) \vert, K, 8, k, K) \frac{1}{\zeta^{1/2}}.
\end{equation*}
Choosing $\zeta$ such that 
\begin{equation*}
16 K^{7/2} B_{2}(\vert w(0) \vert, K, 8, k, K) \frac{1}{\zeta^{1/2}} \leq 
\frac{c(k, K)}{8}
\end{equation*}
yields 
\begin{equation}
I=E\left\lbrace w^{\ast}_{i}\mathcal{T}_{\zeta}(S_{ab,i})\langle \Delta,
X_{i} \rangle^{2} \right\rbrace \geq \frac{7}{8}c(k, K).
\label{bound:preRSC_I}
\end{equation}

\bigskip

\textbf{Bounding II}

\bigskip Using Cauchy-Schwartz and that $\mathcal{T}_{\zeta}(S_{ab,i}) \leq
\zeta$ 
\begin{align*}
\vert II \vert = \vert E\left\lbrace (w_{i}-w^{\ast}_{i})\mathcal{T}%
_{\zeta}(S_{ab,i})\langle \Delta, X_{i} \rangle^{2} \right\rbrace \vert \leq
\zeta E^{1/2} \left\lbrace (w_{i}-w^{\ast}_{i})^{2}\right\rbrace E^{1/2}
\left\lbrace \langle \Delta, X_{i} \rangle^{4}\right\rbrace.
\end{align*}
By Lemma \ref{lemma:lipschitz} 
\begin{equation*}
E^{1/2} \left\lbrace (w_{i}-w^{\ast}_{i})^{2}\right\rbrace \leq B_{1}(\Vert
\beta^{\ast}\Vert_{2}, \Vert \beta - \beta^{\ast}\Vert_{2}, 2, k, K) E^{1/2}
\left\lbrace \langle \beta - \beta^{\ast}, X_{i} \rangle^{2}\right\rbrace,
\end{equation*}
where $B_1$ is a function that is increasing in $\Vert \beta^{\ast}\Vert_{2}$
and in $\Vert \beta - \beta^{\ast}\Vert_{2}$. By Condition \ref{cond:weight}
c), $\Vert \beta^{\ast}\Vert_{2}\leq K$ and by assumption $\Vert \beta -
\beta^{\ast}\Vert_{2}\leq L_0$. Hence using Condition \ref{cond:subgauss} b) 
\begin{equation*}
E^{1/2} \left\lbrace (w_{i}-w^{\ast}_{i})^{2}\right\rbrace \leq B_{1}(K,
L_{0}, 2, k, K) \sqrt{2} K L_0.
\end{equation*}
On the other hand by Condition \ref{cond:subgauss} b) 
\begin{equation*}
E^{1/2} \left\lbrace \langle \Delta, X_{i} \rangle^{4}\right\rbrace \leq
4K^{2} .
\end{equation*}
We have shown that 
\begin{align*}
\vert II \vert \leq 4\sqrt{2} \zeta B_{1}(K, L_{0}, 2, k, K) K^{3} L_0.
\end{align*}
We now fix $L_0$ small enough such that 
\begin{equation*}
4 \sqrt{2} \zeta B_{1}(K, L_{0}, 2, k, K) K^{3} L_0 \leq \frac{c(k, K)}{32},
\end{equation*}
yielding 
\begin{equation}
\vert II\vert\leq \frac{c(k, K)}{32}.  \label{bound:preRSC_II}
\end{equation}

\bigskip

\textbf{Bounding III}

\bigskip

Using Cauchy-Schwartz and the definition of $\mathcal{T}_{\zeta}$ 
\begin{align*}
\vert III \vert &= \vert E\left\lbrace w_{i}\mathcal{T}_{\zeta}(S_{ab,i})%
\langle \Delta, X_{i} \rangle^{2} I\left\lbrace \vert w_{i} \vert> t
\right\rbrace \right\rbrace \vert \\
&\leq \zeta \vert E\left\lbrace w_{i}\langle \Delta, X_{i} \rangle^{2}
I\left\lbrace \vert w_{i} \vert > t \right\rbrace \right\rbrace \vert \\
&\leq \zeta E^{1/2}\left\lbrace w_{i}^{2} I\left\lbrace \vert w_{i} \vert> t
\right\rbrace \right\rbrace E^{1/2}\left\lbrace \langle \Delta, X_{i}
\rangle^{4} I\left\lbrace \vert w_{i} \vert> t \right\rbrace \right\rbrace \\
&\leq \zeta E^{1/4}\left\lbrace w_{i}^{4} \right\rbrace P^{1/4}\left\lbrace
\vert w_{i} \vert > t \right\rbrace E^{1/4}\left\lbrace \langle \Delta,
X_{i} \rangle^{8}\right\rbrace P^{1/4}\left\lbrace \vert w_{i} \vert> t
\right\rbrace \\
& = \zeta E^{1/4}\left\lbrace w_{i}^{4} \right\rbrace P^{1/2}\left\lbrace
\vert w_{i} \vert> t \right\rbrace E^{1/4}\left\lbrace \langle \Delta, X_{i}
\rangle^{8}\right\rbrace .
\end{align*}
By Lemma \ref{lemma:bound_moment_lipschitz} 
\begin{equation*}
E^{1/4}\left\lbrace w_{i}^{4} \right\rbrace \leq B_{2}(\vert w(0)\vert,
\Vert \beta\Vert_{2}, 4, k, K),
\end{equation*}
where $B_2$ is increasing in $\Vert \beta\Vert_{2}$. Using Condition \ref%
{cond:weight} c) 
\begin{equation*}
\Vert \beta\Vert_{2} \leq \Vert \beta^{\ast }\Vert_{2} + \Vert
\beta-\beta^{\ast}\Vert_{2} \leq K + L_0.
\end{equation*}
Hence 
\begin{equation*}
E^{1/4}\left\lbrace w_{i}^{4} \right\rbrace \leq B_{2}(\vert w(0)\vert, K +
L_0, 4, k, K).
\end{equation*}
Using Condition \ref{cond:subgauss} b) 
\begin{equation*}
E^{1/4}\left\lbrace \langle \Delta, X_{i} \rangle^{8}\right\rbrace \leq
8K^{2} .
\end{equation*}
By Markov's inequality 
\begin{equation*}
P^{1/2}\left\lbrace \vert w_{i} \vert > t \right\rbrace=P^{1/2}\left\lbrace
w_{i}^{4} > t^{4} \right\rbrace \leq \frac{E^{1/2}\left( w_{i}^{4} \right) }{%
t^2} \leq \frac{ B_{2}^{2}(\vert w(0)\vert, K + L_0, 4, k, K)}{t^{2}}.
\end{equation*}
Thus 
\begin{equation*}
\vert III \vert \leq \zeta B_{2}^{3}(\vert w(0)\vert, K + L_0, 4, k, K) 
\frac{1}{t^{2}}.
\end{equation*}
We now choose $t$ large enough such that 
\begin{equation}
\vert III \vert \leq \frac{c(k, K)}{32}.  \label{bound:preRSC_III}
\end{equation}

\bigskip

\textbf{Bounding IV}

\bigskip

Using Cauchy-Schwartz and the definitions of $\mathcal{T}_{t}$ and $\mathcal{%
T}_{\zeta}$ 
\begin{align*}
\vert IV\vert = \vert E\left\lbrace \mathcal{T}_{t}(w_{i})\mathcal{T}%
_{\zeta}(S_{ab,i})\langle \Delta, X_{i} \rangle^{2} I\left\lbrace
\vert\langle \Delta, X_{i} \rangle \vert> \tau \right\rbrace \right\rbrace
\vert &\leq t \zeta E\left\lbrace \langle \Delta, X_{i} \rangle^{2}
I\left\lbrace\vert \langle \Delta, X_{i} \rangle \vert > \tau
\right\rbrace\right\rbrace \\
&\leq t \zeta E^{1/2} \left\lbrace \langle \Delta, X_{i} \rangle^{4}
\right\rbrace P^{1/2} \left( \vert\langle \Delta, X_{i} \rangle \vert > \tau
\right).
\end{align*}
By Condition \ref{cond:subgauss} b) 
\begin{equation*}
E^{1/2} \left\lbrace \langle \Delta, X_{i} \rangle^{4} \right\rbrace \leq
4K^{2}.
\end{equation*}
By Markov's inequality and Condition \ref{cond:subgauss} b) 
\begin{equation*}
P^{1/2} \left(\vert \langle \Delta, X_{i} \rangle \vert> \tau
\right)=P^{1/2} \left( \langle \Delta, X_{i} \rangle^{2} > \tau^{2} \right)
\leq \frac{E^{1/2} \left( \langle \Delta, X_{i} \rangle^{2} \right)}{\tau}
\leq \frac{\sqrt{2}K}{\tau}.
\end{equation*}
Then 
\begin{equation*}
\vert IV\vert \leq 4 \: \sqrt{2} t \zeta \frac{K^{3}}{\tau}.
\end{equation*}
Choosing $h$ such that $\tau=\tau(1)=h$ satisfies 
\begin{equation*}
4 \: \sqrt{2} t \zeta \frac{K^{3}}{\tau} \leq \frac{c(k, K)}{32}
\end{equation*}
yields 
\begin{equation}
\vert IV\vert \leq \frac{c(k, K)}{32}.  \label{bound:preRSC_IV}
\end{equation}

Recall that 
\begin{align*}
M(\Delta, \beta)&=E\left\lbrace \mathcal{T}_{t}(w_{i})\mathcal{T}%
_{\zeta}(S_{ab,i})\widehat{\mathcal{T}}_{\tau(1)}\left(\langle \Delta, X_{i}
\rangle \right) I\lbrace \vert \langle \theta_{c}^{\ast}, X_i \rangle \vert
\leq T \rbrace \right\rbrace \\
&= E\left\lbrace \mathcal{T}_{t}(w_{i})\mathcal{T}_{\zeta}(S_{ab,i})\widehat{%
\mathcal{T}}_{\tau(1)}\left(\langle \Delta, X_{i} \rangle \right)
\right\rbrace - E\left\lbrace \mathcal{T}_{t}(w_{i})\mathcal{T}%
_{\zeta}(S_{ab,i})\widehat{\mathcal{T}}_{\tau(1)}\left(\langle \Delta, X_{i}
\rangle \right) I\lbrace \vert \langle \theta_{c}^{\ast}, X_i \rangle \vert
> T \rbrace \right\rbrace
\end{align*}
To bound the first term, putting together \eqref{bound:preRSC_I}, %
\eqref{bound:preRSC_II}, \eqref{bound:preRSC_III}, \eqref{bound:preRSC_IV}
we get 
\begin{equation*}
E\left\lbrace \mathcal{T}_{t}(w_{i})\mathcal{T}_{\zeta}(S_{ab,i})\widehat{%
\mathcal{T}}_{\tau(1)}\left(\langle \Delta, X_{i} \rangle \right)
\right\rbrace \geq \frac{7c(k, K)}{8} - 3\frac{c(k, K)}{32}\geq \frac{3c(k,
K)}{4}.
\end{equation*}
To bound the second term, by the definitions of the truncation functions 
\begin{equation*}
E\left\lbrace \mathcal{T}_{t}(w_{i})\mathcal{T}_{\zeta}(S_{ab,i})\widehat{%
\mathcal{T}}_{\tau(1)}\left(\langle \Delta, X_{i} \rangle \right) I\lbrace
\vert \langle \theta_{c}^{\ast}, X_i \rangle \vert > T \rbrace \right\rbrace
\leq t \zeta h^{2} P\left\lbrace \vert \langle \theta_{c}^{\ast}, X_i
\rangle \vert > T \right\rbrace.
\end{equation*}
Using Markov's inequality, Conditions \ref{cond:approx_sparse_S} and \ref%
{cond:subgauss} b) 
\begin{equation*}
P\left\lbrace \vert \langle \theta_{c}^{\ast}, X_i \rangle \vert > T
\right\rbrace \leq \frac{2K^{4}}{T^{2}}.
\end{equation*}
Then 
\begin{equation*}
E\left\lbrace \mathcal{T}_{t}(w_{i})\mathcal{T}_{\zeta}(S_{ab,i})\widehat{%
\mathcal{T}}_{\tau(1)}\left(\langle \Delta, X_{i} \rangle \right) I\lbrace
\vert \langle \theta_{c}^{\ast}, X_i \rangle \vert > T \rbrace \right\rbrace
\leq t \zeta h^{2} \frac{2K^{4}}{T^{2}}.
\end{equation*}
Choosing $T$ large enough such that 
\begin{equation*}
2t \zeta h^{2} \frac{K^{4}}{T^{2}} \leq \frac{c(k, K)}{4}
\end{equation*}
we conclude that 
\begin{equation*}
M(\Delta, \beta) \geq \frac{3c(k, K)}{4}- \frac{c(k, K)}{4}=\frac{c(k, K)}{2}%
,
\end{equation*}
which proves \eqref{eq:RSC_lower_bound_exp_fin}. Note that $\zeta, t, T $
and $h$ depend only on $k$ and $K$. \bigskip

\textbf{Proof of \eqref{eq:RSC_concentration}}

\bigskip

The empirical process defining $Z_{n}(r;\beta)$ is formed by a class of
functions bounded by $t \zeta h^{2}$. Then \eqref{eq:RSC_concentration}
follows easily using Massart's inequality in Theorem 14.2 from \cite%
{Buhlmann-book} and standard symmetrization and contraction inequalities for
empirical processes. We omit the details.
\end{proof}

\begin{corollary}
\label{coro:RSC_fin} Assume the setting of Proposition \ref{prop:preRSC}.
Let $s$ stand for either $s_{c}$ or $\max(s_{c},s_{\beta})$. Then there
exists $n_2$ depending only on $k$ and $K$ such that if $n\geq n_2$, $%
L_{c}(\cdot, \phi,\widehat{\beta})$ satisfies RSC over $\mathbb{D}(c_%
\mathcal{C},s) \cap \left\lbrace \Delta : \Vert \Delta \Vert_{2} \leq
1\right\rbrace$ with curvature $k^{RSC}_1/2$ with probability tending to one.
\end{corollary}

\begin{proof}
\lbrack Proof of Corollary \ref{coro:RSC_fin}] By Proposition \ref%
{prop:preRSC}, for $n\geq n_{3}$ 
\begin{equation}
\delta\mathcal{L}(\Delta, \theta_{c}^{\ast}, \widehat{\beta})\geq k^{RSC}_1
\Vert \Delta\Vert_{2}^{2} -k^{RSC}_{2} \left( \frac{\log(p)}{n}%
\right)^{1/2}\Vert \Delta\Vert_{1}\Vert \Delta\Vert_{2} \text{ for all }
\Delta \text{ with } \Vert \Delta \Vert_{2}\leq 1  \label{eq:tesis_preRSC}
\end{equation}
with probability tending to one. Take $\Delta\in \mathbb{D}(c_{\mathcal{C}%
},s)$ with $\Vert \Delta \Vert_{2}\leq 1$. Then \eqref{eq:tesis_preRSC}
implies 
\begin{align*}
\delta\mathcal{L}(\Delta, \theta_{c}^{\ast},\widehat{\beta})&\geq k^{RSC}_1
\Vert \Delta\Vert_{2}^{2} -k^{RSC}_{2} \left( \frac{\log(p)}{n}%
\right)^{1/2}\Vert \Delta\Vert_{1}\Vert \Delta\Vert_{2} \\
&\geq \left( k^{RSC}_1-k^{RSC}_{2} c_{\mathcal{C}}\left( \frac{s \log(p)}{n}%
\right)^{1/2}\right) \Vert \Delta\Vert_{2}^{2}.
\end{align*}
By Condition \ref{cond:weight} c) and Condition \ref{cond:approx_sparse_S}
we can take $n$ large enough such that 
\begin{equation*}
k^{RSC}_{2} c_{\mathcal{C}}\left( \frac{s \log(p)}{n}\right)^{1/2} \leq 
\frac{k_{1}^{RSC}}{2}.
\end{equation*}
The result is proven.
\end{proof}

We are now ready to prove Theorem \ref{theo:rate_main}.

\begin{proof}
\lbrack Proof of Theorem \ref{theo:rate_main}] Fix $\varepsilon>0$. We will
verify that the assumptions in Theorem \ref{theo:rate_gen} hold with
probability at least $1-\varepsilon$ for large enough $n$. If Condition \ref%
{cond:approx_sparse_S} a) holds, let $s=s_{c}$, if Condition \ref%
{cond:approx_sparse_S} b) holds let $s=\max(s_{c},s_{\beta})$. By Lemma \ref%
{lemma:lambda_choice} we can choose $c_{\lambda}$ and such that if $%
\lambda_{a}=c_{\lambda}\sqrt{\log(p)/n}$ then for all sufficiently large $n$ 
\begin{equation*}
P\left( J \geq \frac{\lambda_{a}}{2} \right) \leq \frac{\varepsilon}{5}.
\end{equation*}
By Lemma \ref{lemma:bound_res}, we can choose $c_{\rho}$ such that for
sufficiently large $n$ 
\begin{equation*}
P\left( \Vert \widetilde{\varrho}\Vert_{2,n} \geq c_{\rho}\sqrt{\frac{%
s_{c}\log(p)}{n}} \right) \leq \frac{\varepsilon}{5}.
\end{equation*}
If Condition \ref{cond:approx_sparse_S} a) holds then $H=0$ by definition
and we can take $c_{H}=0$. If Condition \ref{cond:approx_sparse_S} b) holds
and $w\equiv 1$ then $H=0$ and again we can take $c_{H}=0$. In general, if
Condition \ref{cond:approx_sparse_S} b) holds, by Lemma \ref{lemma:Hbeta} we
can choose $c_{H}>0$ such that for all sufficiently large $n$ 
\begin{equation*}
P\left( H \geq c_{H}\sqrt{\frac{s_{\beta}\log(p)}{n}} \right) \leq \frac{%
\varepsilon}{5}.
\end{equation*}
By Conditions \ref{cond:subgauss} b) and c) the random vector $X$ satisfies
the assumptions in Proposition \ref{prop:upper_RSE}. Hence for large enough $%
n$, with probability at least $1-\varepsilon/5$ 
\begin{align*}
\Vert {\Delta}\Vert_{\widehat{\Sigma}_{1}}^{2}\leq c_{\Sigma_1} \left( \Vert 
{\Delta}\Vert_{2}^{2} + \frac{\log(p)}{n}\Vert {\Delta}\Vert_{1}^{2}\right)
\quad \text{for all } \Delta,
\end{align*}
where $c_{\Sigma_1}$ depends only on $k$ and $K$. Then by Proposition \ref%
{prop:cone} there exists $c_{\mathcal{C}}$ depending only on $c_{\lambda},
c_{\rho},c_{H}, c_{\Sigma_1}$ such that, for all sufficiently large $n$, $%
\widehat{\Delta}\in \mathbb{D}(c_{\mathcal{C}},s)$ with probability at least 
$1-4/5\varepsilon$. Finally, by Corollary \ref{coro:RSC_fin}, there exists $%
\kappa_{RSC}$ depending only on $k$ and $K$ such that, for all large enough $%
n$, $L_{c}(\cdot, \phi, \widehat{\beta})$ satisfies RSC over $\mathbb{D}(c_%
\mathcal{C},s) \cap \left\lbrace \Delta : \Vert \Delta \Vert_{2} \leq
1\right\rbrace$ with curvature $\kappa_{RSC}$ with probability at least $%
1-\varepsilon/5$.

Hence, if Condition \ref{cond:approx_sparse_S} a) holds or if Condition \ref%
{cond:approx_sparse_S} b) holds and $w\equiv 1$, taking $c_{H}=0$ and $%
s=s_{c}$, from Theorem \ref{theo:rate_gen} we get that with probability at
least $1-\varepsilon$, for large enough $n$, 
\begin{equation*}
\Vert \widehat{\theta}_{c} - \theta^{\ast}_{c}\Vert_{2}\leq \frac{2}{%
\kappa_{RSC}}\left( \frac{3c_{\lambda}}{2 }+c_{\rho}\left(c_{\Sigma_1}
\left( 1+ c_{\mathcal{C}}^{2} \right)\right)^{1/2} \right) \sqrt{\frac{%
s_{c}\log(p)}{n}}
\end{equation*}
and 
\begin{equation*}
\Vert \widehat{\theta}_{c} - \theta^{\ast}_{c}\Vert_{1}\leq \frac{2c_{%
\mathcal{C}}}{\kappa_{RSC}}\left( \frac{3c_{\lambda}}{2 }+c_{\rho}\left(c_{%
\Sigma_1} \left( 1+ c_{\mathcal{C}}^{2} \right)\right)^{1/2} \right) s_{c} 
\sqrt{\frac{\log(p)}{n}}.
\end{equation*}
On the other hand, if Condition \ref{cond:approx_sparse_S} b) holds and $w$
is a general weight function, taking $s=\max(s_{c},s_{\beta})$ from Theorem %
\ref{theo:rate_gen} we get that with probability at least $1-\varepsilon$,
for large enough $n$, 
\begin{equation*}
\Vert \widehat{\theta}_{c} - \theta^{\ast}_{c}\Vert_{2}\leq \frac{2}{%
\kappa_{RSC}}\left( \frac{3c_{\lambda}}{2 }+c_{\rho}\left(c_{\Sigma_1}
\left( 1+ c_{\mathcal{C}}^{2} \right)\right)^{1/2} +c_{H}\left(c_{\Sigma_1}
\left( 1+ c_{\mathcal{C}}^{2} \right)\right)^{1/2} \right) \sqrt{\frac{%
\max(s_{c},s_{\beta})\log(p)}{n}}
\end{equation*}
and 
\begin{equation*}
\Vert \widehat{\theta}_{c} - \theta^{\ast}_{c}\Vert_{1}\leq \frac{2c_{%
\mathcal{C}}}{\kappa_{RSC}}\left( \frac{3c_{\lambda}}{2 }+c_{\rho}\left(c_{%
\Sigma_1} \left( 1+ c_{\mathcal{C}}^{2} \right)\right)^{1/2}
+c_{H}\left(c_{\Sigma_1} \left( 1+ c_{\mathcal{C}}^{2} \right)\right)^{1/2}
\right) \max(s_{c},s_{\beta})\sqrt{\frac{\log(p)}{n}}.
\end{equation*}
The Theorem is proven.
\end{proof}

\newpage

\section{Appendix B: Asymptotic results for $\widehat{\protect\chi}_{lin}$
and $\widehat{\protect\chi}_{nonlin}$}

This appendix contains the proofs of the theorems regarding the asymptotic
behavior of the estimators $\widehat{\chi }_{lin}$ and $\widehat{\chi }%
_{nonlin}$. It also includes the proofs of several results announced in
Section \ref{sec:models} regarding the AGLS class. Moreover, here we provide
an example of a setting in which Conditions Lin.L.W.1, Lin.E.W and Lin.V.W
for the probability limit of the estimator of the incorrectly specified
nuisance function are satisfied.

\subsection{Approximate sparsity examples}

We will need the following Lemma to show that the `weakly sparse' parametric
models used in \cite{NegahbanSS} are included in the class of functions
defined in Example \ref{ex:parametric_weakly_sparse}.

\begin{lemma}
\label{lemma:Rq} Let $q>0$, $\theta \in \mathbb{R}^{p}$ and suppose $%
\sum_{j=1}^{p} \vert \theta_{j}\vert^{q}\leq R_{q}$. Then 
\begin{equation*}
\max_{j\leq p} j \vert \theta_{(j)}\vert^{q} \leq R_q
\end{equation*}
\end{lemma}

\begin{proof}
\lbrack Proof of Lemma \ref{lemma:Rq}] 
\begin{align*}
\sum\limits_{j=1}^{p} \vert \theta_{j}\vert^{q}=\sum\limits_{j=1}^{p} \vert
\theta_{(j)}\vert^{q}\leq R_{q}.
\end{align*}
Take $i\in\lbrace 1,\dots, p\rbrace$ such that 
\begin{equation*}
\max_{j\leq p} j \vert \theta_{(j)}\vert^{q} = i \vert \theta_{(i)}\vert^{q}.
\end{equation*}
Then 
\begin{equation*}
i \vert \theta_{(i)}\vert^{q} \leq \sum\limits_{j=1}^{i} \vert
\theta_{(j)}\vert^{q} \leq \sum\limits_{j=1}^{p} \vert \theta_{(j)}\vert^{q}
=R_{q}.
\end{equation*}
\end{proof}

The following proposition collects the results announced in the discussion
of Examples \ref{ex:parametric_weakly_sparse} and \ref{ex:nonparam}.

\begin{proposition}
\label{prop:approx_sparse_ex} $\ $

\begin{enumerate}
\item Assume that there exist fixed positive constants $k,K$ such that $%
\sup_{\Vert \Delta\Vert_{2}=1} \Vert \langle \Delta, \phi(Z) \rangle
\Vert_{\psi_{2}}\leq K$, $k\leq
\lambda_{min}\left(\Sigma_{1}\right)\leq\lambda_{max}\left(\Sigma_{1}\right)%
\leq K$ and that $\varphi$ satisfies 
\begin{equation*}
\vert \varphi(u) -\varphi(v)\vert \leq K\exp\left\lbrace K \left( \vert
u\vert+\vert v \vert \right) \right\rbrace \vert u -v\vert \quad \text{for
all } u,v \in\mathbb{R}.
\end{equation*}
Take $l \in \left\lbrace 1,2 \right\rbrace$, $\alpha>1/l$ and $M>0$. Let $%
\overline{l}=2$ if $l=1$ and $\overline{l}=1$ if $l=2$. Consider the class
of functions $\mathcal{W}_{n}(\phi,t,\alpha,l,M,\varphi)$ defined in Example %
\ref{ex:parametric_weakly_sparse}. Let 
\begin{equation*}
s=t(n)^{1/\alpha} \left(\frac{n}{\log(p)}\right)^{1/(l\alpha)}.
\end{equation*}%
Then $\mathcal{W}_{n}(\phi,t,\alpha,l,M,\varphi)$ is contained in the AGLS
class $\mathcal{G}_{n}(\phi, s, \overline{l}, \varphi)$.

Moreover, for each $q\in(0,2)$, $\mathcal{N}_{n}(\phi, t,q,M,\varphi)$ is
contained in $\mathcal{W}_{n}(\phi,t^{1/q}, 1/q,2, \varphi)$.

\item Assume that there exists a fixed positive constants $K$ such that the
density of $Z$ satisfies $f_{Z}(z)\leq K$ for all $z$, and that $%
\left\lbrace \phi_{j} \right\rbrace$ is an orthonormal basis of $%
L_{2}[0,1]^{d}$, for a fixed $d$. Take $l \in \left\lbrace 1,2 \right\rbrace$%
, $\alpha>1/l$. Let $\overline{l}=2$ if $l=1$ and $\overline{l}=1$ if $l=2$.
Consider the class of functions $\mathcal{S}_{n}(\phi,t,\alpha,l)$ defined
in Example \ref{ex:nonparam}. Let 
\begin{equation*}
s=t(n)^{1/\alpha} \left(\frac{n}{\log(p)}\right)^{1/(l\alpha)}.
\end{equation*}%
Then $\mathcal{S}_{n}(\phi,t,\alpha,l)$ is contained in the AGLS class $%
\mathcal{G}_{n}(\phi^{n}, s, \overline{l}, id)$ with $\phi^{n}=\left(%
\phi_{1},\dots,\phi_{p(n)}\right)$.
\end{enumerate}
\end{proposition}

\begin{proof}
\lbrack Proof of Proposition \ref{prop:approx_sparse_ex}] We prove part one
first. Let $c(z)=\varphi\left(\langle\theta, \phi(z) \rangle \right)\in 
\mathcal{W}_{n}(\phi,t,\alpha,l,M,\varphi)$. Then 
\begin{equation*}
\max_{j\leq p}j^{\alpha }|\theta _{(j)}|\leq t\left( n\right) \quad \text{and%
} \quad \Vert \theta\Vert_{2}\leq M,
\end{equation*}%
where the sequence $t\left( n\right) $ satisfies 
\begin{equation*}
t\left( n\right) ^{1/\alpha }\left( \frac{\log (p)}{n}\right) ^{1-1/(l\alpha
)}\underset{n\rightarrow \infty }{\rightarrow }0 \quad \text{and} \quad
t(n)^{1/\alpha} \left(\frac{n}{\log(p)}\right)^{1/(l\alpha)}\leq p .
\end{equation*}%
Let 
\begin{equation*}
s=t(n)^{1/\alpha} \left(\frac{n}{\log(p)}\right)^{1/(l\alpha)}.
\end{equation*}
Then ${s\log(p)}/{n}\to 0, $ which implies $s(\log(p)/{n})^{2/l} \to 0$. In
particular, $s(\log(p)/{n})^{2/l}$ is bounded.

Let $i_{1},\dots,i_{s}$ be the indices corresponding to the $s$ largest
values (in absolute value) of $\theta$. Take $\theta^{\ast}$ to be $p$%
-dimensional, and $\theta^{\ast}_{i_j}$ to be equal to $\theta_{i_j}$ for
the values of $j=1,\dots, s$ and zero elsewhere. Let $S$ be the support of $%
\theta^{\ast}$. Then 
\begin{align*}
E \left[\left( \langle \theta, \phi(Z)\rangle - \langle \theta^{\ast},
\phi(Z)\rangle \right)^{2}\right]= E \left[ \left( \sum\limits_{j\in 
\overline{S}, j < p+1} \theta_{j} \phi_{j}(Z) \right)^{2} \right] &\leq
K^{2} \sum\limits_{j\in \overline{S}, j < p+1} (\theta_{j})^{2} \\
&\leq K^{2} \sum\limits_{j= s+1}^{p} \frac{t(n)^{2}}{j^{2\alpha}} \\
&\leq C_0 K^{2} t(n)^{2}s^{1-2\alpha}=C_0 K^{2} s \left(\frac{\log(p)}{n}%
\right)^{2/l}
\end{align*}
where $C_0$ depends only on $\alpha$ and in the first inequality we used
that $\lambda_{max}(\Sigma_{1})\leq K$. Since $\lambda_{min}(\Sigma_{1})\geq
k$, 
\begin{equation*}
\Vert \theta-\theta^{\ast} \Vert_{2} \leq \sqrt{\frac{C_{0} K^{2}}{k^{2}}
s\left(\frac{\log(p)}{n}\right)^{2/l} }\leq C_{1}
\end{equation*}
for some positive constant $C_{1}$. On the other hand, by Lemma \ref%
{lemma:lipschitz}, 
\begin{equation*}
E^{1/2} \left[\left( c(z)- \varphi\left(\langle \theta^{\ast},
\phi(Z)\rangle\right) \right)^{2}\right]\leq B_{1}(\Vert \theta\Vert_{2},
\Vert \theta-\theta^{\ast} \Vert_{2},2, k, K) E^{1/2}\left(\langle
\theta-\theta^{\ast}, \phi(Z) \rangle^{2} \right),
\end{equation*}
where $B_{1}$ is a function that is increasing in $\Vert \theta\Vert_{2}$
and in $\Vert \theta-\theta^{\ast} \Vert_{2}$. By assumption $\Vert
\theta\Vert_{2}\leq M. $ Thus 
\begin{equation*}
E^{1/2} \left[\left( c(z)- \varphi\left(\langle \theta^{\ast},
\phi(Z)\rangle\right) \right)^{2}\right]\leq B_{1}(M, C_{1},2, k, K) \sqrt{%
C_0 K^{2} s \left(\frac{\log(p)}{n}\right)^{2/l}}.
\end{equation*}
If $l=2$, this implies that $c(z)\in\mathcal{G}_{n}(\phi, s, 1, \varphi)$.
If $l=1$ 
\begin{equation*}
s \left(\frac{\log(p)}{n}\right)^{2/l}=s \left(\frac{\log(p)}{n}\right)^{2}
\leq C_{2} \left(\frac{s\log(p)}{n} \right)^{2},
\end{equation*}
for some fixed constant $C_2$. Thus, if $l=1$, $c(z)\in\mathcal{G}_{n}(\phi,
s, 2, \varphi)$.

Let $q\in(0,2)$. If $c(z)\in\mathcal{N}_{n}(\phi, t, q, M, \varphi)$, then $%
c(z)=\varphi\left(\langle\theta, \phi(z) \rangle \right)$, where 
\begin{equation*}
\sum\limits_{j=1}^{p}\vert \theta_{j}\vert^{q}\leq t(n) \quad \text{and}
\quad \Vert \theta\Vert_{2}\leq M,
\end{equation*}%
and the sequence $t\left( n\right) $ satisfies 
\begin{equation*}
t\left( n\right)\left( \frac{\log (p)}{n}\right) ^{1-q/2}\underset{%
n\rightarrow \infty }{\rightarrow }0 \quad \text{and} \quad t(n) \left(\frac{%
n}{\log(p)}\right)^{q/2}\leq p .
\end{equation*}%
By Lemma \ref{lemma:Rq}, 
\begin{equation*}
\max_{j\leq p}j^{1/q}|\theta _{(j)}|\leq t\left( n\right)^{1/q}.
\end{equation*}
Thus, $c(z)\in\mathcal{W}(\phi, t^{1/q}, 1/q, 2, M, \varphi)$.

Now for part two, let $c(z)\in \mathcal{S}_{n}(\phi,t,\alpha,l)$. Then there
exists $p$ and a permutation $\left\{ \phi _{\pi \left( j\right) }\right\} $
of the first $p=p\left( n\right) $ elements of the basis such that 
\begin{equation*}
c(Z)=\sum\limits_{j=1}^{p}\theta _{j}\phi _{\pi \left( j\right)
}(Z)+\sum\limits_{j=p+1}^{\infty }\theta _{j}\phi _{j}(Z)\quad \text{ }
\end{equation*}%
where $|\theta _{j}|j^{\alpha }\leq t\left( n\right) $ for all $j$ and $t(n)$
is such that%
\begin{equation*}
t\left( n\right) ^{1/\alpha }\left( \frac{\log (p)}{n}\right) ^{1-1/(l\alpha
)}\underset{n\rightarrow \infty }{\rightarrow }0 \quad \text{and} \quad
t(n)^{1/\alpha} \left(\frac{n}{\log(p)}\right)^{1/(l\alpha)}\leq p .
\end{equation*}%
Equivalently 
\begin{equation*}
c(Z)=\sum\limits_{j=1}^{\infty}\beta_{j}\phi _{j}(Z),
\end{equation*}%
where 
\begin{equation*}
j^{\alpha} \vert \beta_{(j)}\vert \leq t(n) \quad \text{for} \quad j\leq p
\quad \text{and} \quad j^{\alpha} \vert \beta_{j}\vert \leq t(n) \quad \text{%
for} \quad j>p
\end{equation*}
and $\vert \beta_{(p)} \vert \leq \dots \vert \beta_{(1)} \vert$. Let 
\begin{equation*}
s=t(n)^{1/\alpha} \left(\frac{n}{\log(p)}\right)^{1/(l\alpha)}.
\end{equation*}
Then $s(\log(p)/{n})^{2/l} \to 0$. Let $i_{1},\dots,i_{s}$ be the indices
corresponding to the $s$ largest values of $\vert \beta_{j}\vert,
j=1,\dots,p $. Take $\theta^{\ast}$ to be $p$-dimensional, and $%
\theta^{\ast}_{i_j}$ to be equal to $\beta_{i_j}$ for the values of $%
j=1,\dots, s$ and zero elsewhere. Let $S$ be the support of $\theta^{\ast}$.
Then 
\begin{align*}
E \left[\left( c(Z) - \langle \theta^{\ast}, \phi(Z)\rangle \right)^{2}%
\right]= E \left[ \left( \sum\limits_{j\in \mathbb{N} \backslash S}
\beta_{j} \phi_{j}(Z) \right)^{2} \right] &\leq K \int\limits_{[0,1]^{d}} %
\left[ \left( \sum\limits_{j\in \mathbb{N} \backslash S} \beta_{j}
\phi_{j}(z) \right)^{2} \right] \\
&= K \sum\limits_{j\in \mathbb{N} \backslash S} (\beta_{j})^{2} \\
&= K \sum\limits_{j\in \overline{S}, j < p+1} \beta_{j}^{2} + K
\sum\limits_{j= p+1}^{\infty} \beta_{j}^{2} \\
&\leq K \sum\limits_{j= s+1}^{p} \frac{t(n)^{2}}{j^{2\alpha}} + K
\sum\limits_{j= p+1}^{\infty} \frac{t(n)^{2}}{j^{2\alpha}} \\
&\leq C_{0} K t(n)^{2}s^{1-2\alpha} = C_0 Ks \left(\frac{\log(p)}{n}%
\right)^{2/l},
\end{align*}
where $C_{0}$ depends only on $\alpha$ and in the first inequality we used
that the density of $Z$ is uniformly bounded by $K$. If $l=2$, this implies
that $c(z)\in\mathcal{G}_{n}(\phi^{n}, s, 1, id)$. If $l=1$ 
\begin{equation*}
s \left(\frac{\log(p)}{n}\right)^{2/l}=s \left(\frac{\log(p)}{n}\right)^{2}
\leq C_{1} \left(\frac{s\log(p)}{n} \right)^{2},
\end{equation*}
for some fixed constant $C_1$. Thus, if $l=1$, $c(z)\in\mathcal{G}%
_{n}(\phi^{n}, s, 2, id)$. This finishes the proof of the proposition.
\end{proof}

\subsection{An example in which the probability limit of the estimator of
the incorrectly modelled nuisance function satisfies the regularity
conditions}

Suppose a researcher wants to estimate the expected conditional covariance
functional of Example \ref{ex:CC}. The following proposition shows an
admittedly rather artificial setting in which if the researcher mistakenly
used a linear link to model $E\left( D | Z \right)$, then the $\ell_{1}$
regularized linear regression estimator used in Algorithm \ref{Algo both ALS}
will converge to the zero vector, which is obviously sparse. In this case,
Conditions Lin.L.W, Lin.E.W and Lin.V.W only impose assumptions on the true
nuisance function.

\begin{proposition}
\label{prop:ex_misp} Consider the expected conditional covariance functional
of Example \ref{ex:CC}. Assume that $Z$ is a zero mean multivariate normal
random vector with a identity covariance matrix. Let $\theta_{a}\in\mathbb{R}%
^{d}$ be such that $\Vert \theta_{a}\Vert_{2}\leq K$. Let $\varphi:\mathbb{R}%
^{2d}\to \mathbb{R}$ be a function such that $\varphi(\theta_{a}, z)$ is an
even function of each coordinate of $z$. Assume that $E\left( D | Z
\right)=\varphi(\theta_{a}, Z)$. Then the minimizer over $\theta$ of 
\begin{equation*}
E\left( \frac{\langle \theta, Z \rangle^{2}}{2}- D \langle \theta, Z \rangle
\right)
\end{equation*}
is the zero vector.
\end{proposition}

\begin{proof}
\lbrack Proof of Proposition \ref{prop:ex_misp} The minimizer of 
\begin{equation*}
E\left( \frac{\langle \theta, Z \rangle^{2}}{2}- D \langle \theta, Z \rangle
\right)
\end{equation*}
is given by $E\left(DZ \right). $ Now 
\begin{align*}
E\left(DZ \right)=E\left( E(D| Z) Z \right)=E\left( \varphi(\theta_{a}, Z) Z
\right).
\end{align*}
Since $Z$ is a vector of independent standard normal random variables and $%
\varphi(\theta_{a}, z)z_{j}$ is an odd function of $j$ for all $%
j\in\left\lbrace 1,\dots,d\right\rbrace$ we have that $E\left(DZ \right)=0$.
\end{proof}

\subsection{Asymptotic results for the estimator $\widehat{\protect\chi }%
_{lin}$}

\label{sec:app_asym_lin}

\subsubsection{Rate double robustness for the estimator $\widehat{\protect%
\chi }_{lin}$}

\label{sec:app_dr_lin}

\begin{proof}
\lbrack Proof of Theorem \ref{theo:rate_double_lin}] Let $n_{k}$ be the size
of sample $\mathcal{D}_{nk}, k=1,2$. To simplify the proof, we assume that $%
n $ is even, so that $n_k=n/2, k=1,2$. Now 
\begin{equation*}
\sqrt{n}\left\{ \widehat{\chi }_{lin}-\chi \left( \eta \right) \right\} =%
\sqrt{n}\left\{ \frac{1}{2}\sum_{k=1}^{2}\left[ \widehat{\chi }%
_{lin}^{\left( k\right) }-\chi \left( \eta \right) \right] \right\}
\end{equation*}%
so, since $n_k=n/2$, to show \eqref{eq:lin_rate_DR_expansion} it suffices to
show that for $k=1,2$ 
\begin{equation*}
\sqrt{n_k}\left[ \widehat{\chi }_{lin}^{\left( k\right) }-\chi \left( \eta
\right) \right] =\mathbb{G}_{nk}\left[ \Upsilon \left( a,b\right) \right]
+O_{p}\left( \sqrt{\frac{s_{a}s_{b}}{n_k}}\log (p)\right) +o_{p}\left(
1\right).
\end{equation*}%
To simplify the notation we let $m=\overline{m}=n/2$, $\mathcal{D}_{m}$
denote $\mathcal{D}_{n1}$ and $\mathcal{D}_{m}^{c}$ denote $\mathcal{D}%
_{n}\setminus \mathcal{D}_{n1}$, and $\widetilde{a}\left( Z\right) \equiv
\left\langle \widetilde{\theta }_{a},\phi \left( Z\right) \right\rangle $
and $\widetilde{b}\left( Z\right) \equiv \left\langle \widetilde{\theta }%
_{b},\phi \left( Z\right) \right\rangle $ denote the estimators $\widehat{a}%
_{\left( \overline{1}\right) }\left( Z\right) $ and $\widehat{b}_{\left( 
\overline{1}\right) }\left( Z\right) $ computed using data $\mathcal{D}%
_{m}^{c}.$ Now,%
\begin{equation*}
\widehat{\chi }^{\left( 1\right) }_{lin}-\chi \left( \eta \right)
=N_{m}+\Gamma _{a,m}+\Gamma _{b,m}+\Gamma _{ab,m}
\end{equation*}
where 
\begin{equation*}
N_{m}\equiv \mathbb{P}_{m}\left[ \Upsilon \left( a,b\right) -\chi \left(
\eta \right) \right]
\end{equation*}%
\begin{equation*}
\Gamma _{a,m}\equiv \mathbb{P}_{m}\left[ S_{ab}\left( \widetilde{a}-a\right)
\left( Z\right) b\left( Z\right) +m_{a}\left( O,\widetilde{a}\right)
-m_{a}\left( O,a\right) \right]
\end{equation*}%
\begin{equation*}
\Gamma _{b,m}\equiv \mathbb{P}_{m}\left[ S_{ab}\left( \widetilde{b}-b\right)
\left( Z\right) a\left( Z\right) +m_{b}\left( O,\widetilde{b}\right)
-m_{b}\left( O,b\right) \right]
\end{equation*}%
\begin{equation*}
\Gamma _{ab,m}\equiv \mathbb{P}_{m}\left[ S_{ab}\left( \widetilde{a}%
-a\right) \left( Z\right) \left( \widetilde{b}-b\right) \left( Z\right) %
\right].
\end{equation*}
Invoking Theorem \ref{theo:rate_main_lin}, Condition Lin.L implies that 
\begin{equation*}
\left\Vert \widetilde{\theta }_{a}-\theta _{a}^{\ast }\right\Vert _{\Sigma
_{1}}=O_{p}\left( \sqrt{\frac{s_{a}\log \left( p\right) }{m}}\right) \text{
and }\left\Vert \widetilde{\theta }_{b}-\theta _{b}^{\ast }\right\Vert
_{\Sigma _{1}}=O_{p}\left( \sqrt{\frac{s_{b}\log \left( p\right) }{m}}%
\right).
\end{equation*}
In addition, 
\begin{eqnarray*}
\left\Vert \widetilde{a}-a\right\Vert _{L_{2}\left( P_{\eta }\right) } &\leq
&\left\Vert \widetilde{a}-\left\langle \theta _{a}^{\ast },\phi
\right\rangle \right\Vert _{L_{2}\left( P_{\eta }\right) }+\left\Vert
a-\left\langle \theta _{a}^{\ast },\phi \right\rangle \right\Vert
_{L_{2}\left( P_{\eta }\right) } \\
&=&\left\Vert \widetilde{\theta }_{a}-\theta _{a}^{\ast }\right\Vert
_{\Sigma _{1}}+\left\Vert a-\left\langle \theta _{a}^{\ast },\phi
\right\rangle \right\Vert _{L_{2}\left( P_{\eta }\right) } \\
&=&O_{p}\left( \sqrt{\frac{s_{a}\log \left( p\right) }{m}}\right)
\end{eqnarray*}%
where the last equality follows because, by assumption $a\in \mathcal{G}%
\left( \phi ,s_{a},j=1,\varphi =id\right) $ and the approximation error $%
\left\Vert a-\left\langle \theta _{a}^{\ast },\phi \right\rangle \right\Vert
_{L_{2}\left( P_{\eta }\right) }$ is $O\left( \sqrt{s_{a}\log \left(
p\right) /{m}}\right) .$ Likewise, 
\begin{equation*}
\left\Vert \widetilde{b}-b\right\Vert _{L_{2}\left( P_{\eta }\right)
}=O_{p}\left( \sqrt{\frac{s_{b}\log \left( p\right) }{m}}\right).
\end{equation*}

Next, we show that $\sqrt{m}\Gamma _{a,m}=o_{p}\left( 1\right)$. By the
Dominated Convergence Theorem it suffices to show that for any $\varepsilon
>0,$ 
\begin{equation*}
P_{\eta }\left[ \sqrt{m}\Gamma _{a,m}>\varepsilon |\mathcal{D}_{m}^{c}\right]
=o_{p}\left( 1\right).
\end{equation*}%
In turn, to prove the latter, by Markov's inequality it suffices to show
that $E_{\eta }\left[ \sqrt{m}\Gamma _{a,m}|\mathcal{D}_{m}^{c}\right] =0$
and $Var_{\eta }\left[ \sqrt{m}\Gamma _{a,m}|\mathcal{D}_{m}^{c}\right]
=o_{p}\left( 1\right) .$ Now, 
\begin{eqnarray*}
E_{\eta }\left[ \sqrt{m}\Gamma _{a,m}|\mathcal{D}_{m}^{c}\right] &=&E_{\eta }%
\left[ S_{ab}\left( \widetilde{a}-a\right) \left( Z\right) b\left( Z\right)
+m_{a}\left( O,\widetilde{a}\right) -m_{a}\left( O,a\right) |\mathcal{D}%
_{m}^{c}\right] \\
&=&0
\end{eqnarray*}%
by Proposition \ref{prop:DR functional}. On the other hand 
\begin{eqnarray*}
Var_{\eta }\left[ \sqrt{m}\Gamma _{a,m}|\mathcal{D}_{m}^{c}\right]
&=&E_{\eta }\left[ \left. \left\{ S_{ab}\left( \widetilde{a}-a\right) \left(
Z\right) b\left( Z\right) +m_{a}\left( O,\widetilde{a}\right) -m_{a}\left(
O,a\right) \right\} ^{2}\right\vert \mathcal{D}_{m}^{c}\right] \\
&=&o_{p}\left( 1\right)
\end{eqnarray*}%
by Condition Lin.E.1 and the fact that $\left\Vert \widetilde{a}%
-a\right\Vert _{L_{2}\left( P_{\eta }\right) }=O_{p}\left( \sqrt{s_{a}\log
\left( p\right)/{m} }\right) =o_{p}\left( 1\right) $ because by assumption ${%
s_{a}\log \left( p\right) }/{m}\rightarrow 0$ as $m\rightarrow \infty $. The
same line of argument proves that $\sqrt{m}\Gamma _{b,m}=o_{p}\left(
1\right) $.

Next, by Cauchy-Schwartz inequality, 
\begin{eqnarray*}
\sqrt{m}\Gamma _{ab,m} &\equiv &\sqrt{m}\mathbb{P}_{m}\left[ S_{ab}\left( 
\widetilde{a}-a\right) \left( Z\right) \left( \widetilde{b}-b\right) \left(
Z\right) \right] \\
&\leq &\sqrt{m}\sqrt{\mathbb{P}_{m}\left[ S_{ab}\left( \widetilde{a}%
-a\right) ^{2}\left( Z\right) \right] }\sqrt{\mathbb{P}_{m}\left[
S_{ab}\left( \widetilde{b}-b\right) ^{2}\left( Z\right) \right] }.
\end{eqnarray*}
We will show that 
\begin{equation}
\sqrt{\mathbb{P}_{m}\left[ S_{ab}\left( \widetilde{a}-a\right) ^{2}\left(
Z\right) \right] } = O_{P}\left(\sqrt{\frac{s_{a}\log(p)}{m}} \right)
\label{eq:DR_CS_a}
\end{equation}
and 
\begin{equation}
\sqrt{\mathbb{P}_{m}\left[ S_{ab}\left( \widetilde{b}-b\right) ^{2}\left(
Z\right) \right] } = O_{P}\left(\sqrt{\frac{s_{b}\log(p)}{m}} \right) .
\label{eq:DR_CS_b}
\end{equation}
We begin with \eqref{eq:DR_CS_a}. Fix $\varepsilon>0$. Take $L_0, L_1>0$.
Let 
\begin{equation*}
A=\left\lbrace \Vert \widetilde{\theta}_{a} -
\theta^{\ast}_{a}\Vert_{\Sigma_{1}}^{2} \leq L_1 \frac{s_{a}\log(p)}{m}%
\right\rbrace.
\end{equation*}
Then 
\begin{align*}
&P_{\eta}\left(\mathbb{P}_{m}\left\lbrace S_{ab}(\widetilde{a}%
-a)^{2}\right\rbrace \geq L_0 \frac{s_{a}\log(p)}{m} \right)=
E_{\eta}\left\lbrace P\left(\mathbb{P}_{m}\left\lbrace S_{ab}(\widetilde{a}%
-a)^{2}\right\rbrace \geq L_0 \frac{s_{a}\log(p)}{m}\mid \widetilde{\theta}%
_{a}\right)\right\rbrace = \\
& E_{\eta}\left\lbrace P\left(\mathbb{P}_{m}\left\lbrace S_{ab}(\widetilde{a}%
-a)^{2}\right\rbrace \geq L_0 \frac{s_{a}\log(p)}{m} \mid \widetilde{\theta}%
_{a}\right) \mid A \right\rbrace P_{\eta}(A) + \\
& E_{\eta}\left\lbrace P\left(\mathbb{P}_{m}\left\lbrace S_{ab}(\widetilde{a}%
-a)^{2}\right\rbrace \geq L_0 \frac{s_{a}\log(p)}{m} \mid \widetilde{\theta}%
_{a}\right) \mid A^{c} \right\rbrace P_{\eta}(A^{c}) \leq \\
& E_{\eta}\left\lbrace P\left(\mathbb{P}_{m}\left\lbrace S_{ab}(\widetilde{a}%
-a)^{2} \right\rbrace\geq L_0 \frac{s_{a}\log(p)}{m} \mid \widetilde{\theta}%
_{a}\right) \mid A \right\rbrace + P_{\eta}(A^{c}).
\end{align*}
By Theorem \ref{theo:rate_main_lin} we may choose $L_1$ such that for all
sufficiently large $n$, $P_{\eta}(A^{c})<\varepsilon/2$. We will show that
we can choose $L_0$ such that the first term in the right hand side of the
last display is smaller than $\varepsilon/2$. This will show that %
\eqref{eq:DR_CS_a} holds. Since $\widetilde{\theta}_{a}$ is independent of
the data in $\mathcal{D}_{m}$, by Markov's inequality it suffices to show
that there exists $L_2>0$ depending only on $K$ and $L_1$ such that for all $%
\theta$ that satisfy 
\begin{equation*}
\Vert\theta - \theta^{\ast}_{a}\Vert_{\Sigma_{1}}^{2} \leq L_1 \frac{%
s_{a}\log(p)}{m},
\end{equation*}
it holds that 
\begin{equation}
E_{\eta}\left\lbrace S_{ab}\left(\langle \theta,X\rangle- a\right)^{2}
\right\rbrace \leq L_{2} \frac{s_{a}\log(p)}{m}  \label{eq:bound_Saab}
\end{equation}
Hence, take $\theta$ such that 
\begin{equation*}
\Vert \theta - \theta^{\ast}_{a}\Vert_{\Sigma_{1}}^{2} \leq L_1 \frac{%
s_{a}\log(p)}{m}.
\end{equation*}
By Condition Lin.L.5 
\begin{equation*}
E_{\eta}\left\lbrace S_{ab}\left(\langle \theta,\phi(Z)\rangle- a\right)^{2}
\right\rbrace \leq K E\left\lbrace \left(\langle \theta,\phi(Z)\rangle-
a\right)^{2} \right\rbrace .
\end{equation*}
Moreover, 
\begin{equation*}
E_{\eta}^{1/2}\left\lbrace \left(\langle \theta,\phi(Z)\rangle- a\right)^{2}
\right\rbrace \leq \Vert \theta - \theta^{\ast}_{a}\Vert_{\Sigma_1} +
E_{\eta}^{1/2}\left\lbrace a - \langle \theta^{\ast}_{a},\phi(Z)
\rangle\right\rbrace \leq \sqrt{\frac{L_{1} s_{a}\log(p)}{m}} + \sqrt{K 
\frac{s_{a}\log(p)}{m}}.
\end{equation*}
Hence \eqref{eq:bound_Saab} holds, which implies \eqref{eq:DR_CS_a}. %
\eqref{eq:DR_CS_b} is shown similarly.

Now \eqref{eq:DR_CS_a} together with \eqref{eq:DR_CS_b} implies 
\begin{equation*}
\sqrt{m}\Gamma _{ab,m}=O_{P}\left( \sqrt{\frac{s_{a}s_{b}}{m}}\log
(p)\right) .
\end{equation*}%
Thus, we have shown that 
\begin{equation*}
\sqrt{m}\left( \widehat{\chi }^{\left( 1\right) }_{lin}-\chi \left( \eta
\right) \right) =\mathbb{G}_{m}\left[ \Upsilon \left( a,b\right) \right]
+O_{P}\left( \sqrt{\frac{s_{a}s_{b}}{m}}\log (p)\right) +o_{P}(1).
\end{equation*}%
This proves the first part of the Theorem. To prove the second part, we
verify the assumptions of Lyapunov's Central Limit Theorem. Let 
\begin{equation*}
T_{i,m}=\frac{S_{ab,i}ab+m_{a,i}(a)+m_{b,i}(b)+S_{0,i}-E_{\eta }\left\{
S_{ab,i}ab+m_{a,i}(a)+m_{b,i}(b)+S_{0,i}\right\} }{\sqrt{m}}
\end{equation*}%
where $m_{a,i}(a)\equiv m_{a}(O_{i},a)$, $m_{b,i}(b)\equiv m_{b}(O_{i},b)$
and 
\begin{equation*}
s_{m}^{2}=\sum\limits_{i=1}^{m}E_{\eta }\left\{ T_{i,m}^{2}\right\} =E_{\eta
}\left\{ \left( \chi _{\eta }^{1}\right) ^{2}\right\} .
\end{equation*}%
Then by Condition Lin.E.2 a) 
\begin{equation*}
\frac{1}{s_{m}^{3}}\sum\limits_{i=1}^{m}E_{\eta }\left\{ T_{i,m}^{3}\right\}
=\frac{1}{m^{1/2}}\frac{E_{\eta }\left\{ \left\vert \chi _{\eta
}^{1}\right\vert ^{3}\right\} }{E_{\eta }^{3/2}\left\{ \left( \chi _{\eta
}^{1}\right) ^{2}\right\} }\leq \frac{1}{m^{1/2}}\frac{K}{k^{3/2}}%
\rightarrow 0.
\end{equation*}%
Lyapunov's Central Limit Theorem together with the assumption that 
\begin{equation*}
\sqrt{\frac{s_{a}s_{b}}{m}}\log (p)\rightarrow 0
\end{equation*}%
implies 
\begin{equation*}
\frac{\sqrt{m}\left( \widehat{\chi }^{\left( 1\right) }_{lin}-\chi \left(
\eta \right) \right) }{E_{\eta }^{1/2}\left\{ \left( \chi _{\eta
}^{1}\right) ^{2}\right\} }=\frac{1}{s_{m}}\sum%
\limits_{i=1}^{m}T_{i,m}+o_{P}(1)\overset{d}{\rightarrow }N(0,1).
\end{equation*}%
This proves the second part.

To prove the last part, by Slutzky's Lemma, it suffices to show that 
\begin{equation*}
\frac{\mathbb{P}_{m}^{1/2}\left\{ \left( \Upsilon (\widetilde{a},\widetilde{b%
})-\widehat{\chi }_{lin}^{(1)}\right) ^{2}\right\} }{E_{\eta }^{1/2}\left\{
(\chi _{\eta }^{1})^{2}\right\} }\overset{P}{\rightarrow }1.
\end{equation*}%
We will show first that 
\begin{equation}
\left\vert \mathbb{P}_{m}^{1/2}\left\{ (\chi _{\eta }^{1})^{2}\right\}
-E^{1/2}\left\{ (\chi _{\eta }^{1})^{2}\right\} \right\vert \overset{P}{%
\rightarrow }0.  \label{eq:estim_var_sample}
\end{equation}%
Clearly 
\begin{equation*}
E_{\eta}\left\{ \mathbb{P}_{m}\left\{ (\chi _{\eta }^{1})^{2}\right\}
\right\} =E_{\eta }\left\{ (\chi _{\eta }^{1})^{2}\right\} .
\end{equation*}%
Moreover by Condition Lin.E.2 b) 
\begin{equation*}
Var_{\eta }\left\{ \mathbb{P}_{m}\left\{ (\chi _{\eta }^{1})^{2}\right\}
\right\} =\frac{1}{m}Var_{\eta }\left\{ (\chi _{\eta }^{1})^{2}\right\} \leq 
\frac{K}{m}\rightarrow 0.
\end{equation*}%
Hence 
\begin{equation*}
\left\vert \mathbb{P}_{m}\left\{ (\chi _{\eta }^{1})^{2}\right\} -E_{\eta
}\left\{ (\chi _{\eta }^{1})^{2}\right\} \right\vert \overset{P}{\rightarrow 
}0.
\end{equation*}%
Since by Condition Lin.E.2 a) 
\begin{equation*}
k^{1/2}\leq E^{1/2}\left\{ (\chi _{\eta }^{1})^{2}\right\} ,
\end{equation*}%
by the Mean Value Theorem \eqref{eq:estim_var_sample} holds. Next, we will
show that 
\begin{equation}
\left\vert \mathbb{P}_{n3}^{1/2}\left\{ (\chi _{\eta }^{1})^{2}\right\} -%
\mathbb{P}_{n3}^{1/2}\left\{ \left( \Upsilon (\widetilde{a},\widetilde{b})-%
\widehat{\chi }^{(1)}_{lin}\right) ^{2}\right\} \right\vert \overset{P}{%
\rightarrow }0.  \label{eq:estim_var_plugin}
\end{equation}%
By the triangle inequality 
\begin{align*}
\left\vert \mathbb{P}_{m}^{1/2}\left\{ (\chi _{\eta }^{1})^{2}\right\} -%
\mathbb{P}_{m}^{1/2}\left\{ \left( \Upsilon (\widetilde{a},\widetilde{b})-%
\widehat{\chi }^{(1)}_{lin}\right) ^{2}\right\} \right\vert & \leq \mathbb{P}%
_{m}^{1/2}\left\{ \left( \chi _{\eta }^{1}-\Upsilon (\widetilde{a},%
\widetilde{b})+\widehat{\chi }_{lin}^{(1)}\right) ^{2}\right\} \\
& =\mathbb{P}_{m}^{1/2}\left\{ \left( \Upsilon (a,b)-\Upsilon (\widetilde{a},%
\widetilde{b})+\widehat{\chi }_{lin}^{(1)}-\chi (\eta )\right) ^{2}\right\}
\\
& \leq \mathbb{P}_{m}^{1/2}\left\{ \left( \Upsilon (a,b)-\Upsilon (%
\widetilde{a},\widetilde{b})\right) ^{2}\right\} +\mathbb{P}%
_{m}^{1/2}\left\{ \left( \widehat{\chi }_{lin}^{(1)}-\chi (\eta )\right)
^{2}\right\} \\
& =\mathbb{P}_{m}^{1/2}\left\{ \left( \Upsilon (a,b)-\Upsilon (\widetilde{a},%
\widetilde{b})\right) ^{2}\right\} +|\widehat{\chi }_{lin}^{(1)}-\chi (\eta
)|.
\end{align*}%
By the second part of this Theorem, the last term in the right hand side of
the last display converges to zero in probability. We will show that the
first term converges to zero too. For the first term, note that $\Upsilon
\left( \widetilde{a},\widetilde{b}\right) -\Upsilon \left( a,b\right) $ is
equal to 
\begin{equation*}
S_{ab}a\left( \widetilde{b}-b\right) +m_{b}(O,\widetilde{b}%
)-m_{b}(O,b)+S_{ab}b\left( \widetilde{a}-a\right) +m_{a}(O,\widetilde{a}%
)-m_{a}(O,a)+S_{ab}\left( \widetilde{b}-b\right) \left( \widetilde{a}%
-a\right) .
\end{equation*}%
Hence, by the triangle inequality, that first term is bounded by 
\begin{align*}
& \mathbb{P}_{m}^{1/2}\left\{ \left( S_{ab}a\left( \widetilde{b}-b\right)
+m_{b}(O,\widetilde{b})-m_{b}(O,b)\right) ^{2}\right\} +\mathbb{P}%
_{m}^{1/2}\left\{ \left( S_{ab}b\left( \widetilde{a}-a\right) +m_{a}(O,%
\widetilde{a})-m_{a}(O,a)\right) ^{2}\right\} + \\
& \mathbb{P}_{m}^{1/2}\left\{ S_{ab}^{2}\left( \widetilde{b}-b\right)
^{2}\left( \widetilde{a}-a\right) ^{2}\right\} .
\end{align*}%
The first two terms in the last display can be shown to be $o_{P}(1)$ using
Condition Lin.E.1 and arguments similar to the ones used earlier in this
proof. By symmetry, we can assume that Condition Lin.V.3 holds for $b$. To
bound the last term, note that by Condition Lin.V.1 
\begin{align*}
\max\limits_{i\leq n}\left\vert \widetilde{b}(Z_{i})-b(Z_{i})\right\vert &
\leq \max\limits_{i\leq n}\left\vert \langle \widetilde{\theta }_{b}-\theta
_{b}^{\ast },\phi (Z)\rangle \right\vert +\max\limits_{i\leq n}\left\vert
b(Z_{i})-\langle \theta _{b}^{\ast },\phi (Z)\rangle \right\vert \\
& \leq K\Vert \widetilde{\theta }_{b}-\theta _{b}^{\ast }\Vert
_{1}+\max\limits_{i\leq n}|R_{b}(Z_{i})|
\end{align*}%
where $R_{b}(Z_{i})=b(Z_{i})-\langle \theta _{b}^{\ast },\phi (Z)\rangle$.
Hence 
\begin{equation*}
\mathbb{P}_{m}\left\{ S_{ab}^{2}\left( \widetilde{b}-b\right) ^{2}\left( 
\widetilde{a}-a\right) ^{2}\right\} \leq 2K^{2}\Vert \widetilde{\theta }%
_{b}-\theta _{b}^{\ast }\Vert _{1}^{2}\mathbb{P}_{m}\left\{ S_{ab}^{2}\left( 
\widetilde{a}-a\right) ^{2}\right\} +2\max\limits_{i\leq n}|R_{b}(Z_{i})|^{2}%
\mathbb{P}_{m}\left\{ S_{ab}^{2}\left( \widetilde{a}-a\right) ^{2}\right\}
\end{equation*}%
By Theorem \ref{theo:rate_main_lin} 
\begin{equation*}
\Vert \widetilde{\theta }_{b}-\theta _{b}^{\ast }\Vert _{1}^{2}=O_{P}\left( 
\frac{s_{b}^{2}\log (p)}{m}\right) .
\end{equation*}%
By Condition Lin.V.3 
\begin{equation*}
\max\limits_{i\leq n}|R_{b}(Z_{i})|^{2}=O_{P}\left( \frac{s_{b}^{2}\log (p)}{%
m}\right) .
\end{equation*}%
Arguing as before and using Condition Lin.V.2, it is easy to show that 
\begin{equation*}
\mathbb{P}_{m}\left\{ S_{ab}^{2}\left( \widetilde{a}-a\right) ^{2}\right\}
=O_{P}\left( {\frac{s_{a}\log (p)}{m}}\right) .
\end{equation*}%
Thus 
\begin{equation*}
\mathbb{P}_{m}\left\{ S_{ab}^{2}\left( \widetilde{b}-b\right) ^{2}\left( 
\widetilde{a}-a\right) ^{2}\right\} \leq O_{P}\left( \frac{%
s_{b}^{2}s_{a}\log (p)^{2}}{m^{2}}\right) =o_{P}(1)
\end{equation*}%
since by assumption 
\begin{equation*}
\frac{s_{a}s_{b} \log(p)}{m}=o(1)
\end{equation*}%
and 
\begin{equation*}
\frac{s_{b}\log (p)}{m}=o(1).
\end{equation*}
This finishes the proof of the third part of the Theorem.
\end{proof}

\subsubsection{Model double robustness for the estimator $\widehat{\protect%
\chi }_{lin}$}

We will need the following lemma.

\begin{lemma}
\label{lemma:bound_m_term} Assume Condition M.W holds and $a\in\mathcal{G}%
_{n}(\phi, s_{a}, j=2, \varphi)$ with associated parameter $%
\theta^{\ast}_{a} $ such that 
\begin{equation*}
\frac{s_{a}\log(p)}{\sqrt{n}}\to 0.
\end{equation*}
Let 
\begin{equation*}
r(Z)=a(Z)-\varphi\left(\langle \theta^{\ast}_{a}, \phi(Z) \rangle\right)
\end{equation*}
Let $h$ be such that 
\begin{equation*}
E_{\eta}((S_{ab}h)^{2})\leq K.
\end{equation*}
Then 
\begin{equation*}
\sqrt{n}E_{\eta}\left[\left\vert S_{ab}h r + m_{a}(0,r) \right\vert\right]
\to 0.
\end{equation*}
\end{lemma}

\begin{proof}
\lbrack Proof of Lemma \ref{lemma:bound_m_term}] By Condition M.W 
\begin{equation*}
E_{\eta}\left[\left\vert S_{ab}h r(Z) + m_{a}(0,r) \right\vert\right] \leq
E_{\eta}\left[\left\vert S_{ab}h r(Z) \right\vert \right] +E_{\eta}\left[
m^{\ddagger}_{a}(0,\left\vert r \right\vert) \right]= E_{\eta}\left[%
\left\vert S_{ab}h r(Z) \right\vert \right] +E_{\eta}\left[ \mathcal{R}%
_{a}^{\ddagger} \left\vert r(Z) \right\vert \right].
\end{equation*}
Now, using Cauchy-Schwartz 
\begin{equation*}
E_{\eta}\left[\left\vert S_{ab}h r(Z) + m_{a}(0,r) \right\vert\right] \leq
E_{\eta}^{1/2}\left[ (S_{ab}h)^{2} \right] E_{\eta}^{1/2}\left[ r(Z)^{2}%
\right]+E_{\eta}^{1/2}\left[ \left( \mathcal{R}_{a}^{\ddagger}\right)^{2} %
\right] E_{\eta}^{1/2}\left[ r(Z)^{2}\right].
\end{equation*}
Since $a\in\mathcal{G}_{n}(\phi, s_{a}, j=2, \varphi)$ 
\begin{equation*}
E_{\eta}^{1/2}\left[ r(Z)^{2}\right] \leq K\frac{s_{a}\log(p)}{n}.
\end{equation*}
Then 
\begin{equation*}
\sqrt{n}E_{\eta}\left[\left\vert S_{ab}h r(Z) + m_{a}(0,r) \right\vert\right]
\leq 2 \sqrt{K}K\frac{s_{a}\log(p)}{\sqrt{n}}\to 0.
\end{equation*}
\end{proof}

\label{sec:proof_dr_model_lin}

\begin{proof}
\lbrack Proof of Theorem \ref{theo:model_double_lin}] To simplify the proof,
we assume that $n$ is even, so that $n_{k}=n/2,k=1,2$. As in the proof of
rate double robustness it suffices to show that for $k=1,2$ 
\begin{equation*}
\sqrt{n_{k}}\left[ \widehat{\chi }_{lin}^{\left( k\right) }-\chi \left( \eta
\right) \right] =\mathbb{G}_{nk}\left[ \Upsilon \left( a,b^{0}\right) \right]
+O_{p}\left( \sqrt{\frac{s_{a}s_{b}}{n_{k}}}\log (p)\right) +o_{p}\left(
1\right).
\end{equation*}%
Here we also use the notation $m=\overline{m}=n/2$, $\mathcal{D}_{m}$
denoting $\mathcal{D}_{n1}$ and $\mathcal{D}_{m}^{c}$ denoting $\mathcal{D}%
_{n}\setminus \mathcal{D}_{n2}$, and $\widetilde{a}\left( Z\right) \equiv
\left\langle \widetilde{\theta }_{a},\phi \left( Z\right) \right\rangle $
and $\widetilde{b}\left( Z\right) \equiv \left\langle \widetilde{\theta }%
_{b},\phi \left( Z\right) \right\rangle $ denoting the estimators $\widehat{a%
}_{\left( \overline{1}\right) }\left( Z\right) $ and $\widehat{b}_{\left( 
\overline{1}\right) }\left( Z\right) $ computed using data $\mathcal{D}%
_{m}^{c}.$ Now,%
\begin{equation*}
\widehat{\chi }^{\left( 1\right)}_{lin}-\chi \left( \eta \right)
=N_{m}^{\ast }+\Gamma _{a,m}^{\ast }+\Gamma _{b,m}^{\ast }+\Gamma
_{ab,m}^{\ast }
\end{equation*}%
where 
\begin{align*}
& N_{m}^{\ast }\equiv \mathbb{P}_{m}\left[ \Upsilon \left( a,b^{0}\right)
-\chi \left( \eta \right) \right] , \\
& \Gamma _{a,m}^{\ast }\equiv \mathbb{P}_{m}\left[ S_{ab}\left( \widetilde{a}%
-a\right) \left( Z\right) b^{0}\left( Z\right) +m_{a}(O,\widetilde{a}%
)-m_{a}(O,a)\right] , \\
& \Gamma _{b,m}^{\ast }\equiv \mathbb{P}_{m}\left[ S_{ab}\left( \widetilde{b}%
-b^{0}\right) \left( Z\right) a\left( Z\right) +m_{b}(O,\widetilde{b}%
)-m_{b}(O,b^{0})\right] , \\
& \Gamma _{ab,m}^{\ast }\equiv \mathbb{P}_{m}\left[ S_{ab}\left( \widetilde{a%
}-a\right) \left( Z\right) \left( \widetilde{b}-b^{0}\right) \left( Z\right) %
\right] .
\end{align*}%
Invoking Theorem \ref{theo:rate_main_lin}, Condition Lin.L.W implies that 
\begin{equation*}
\left\Vert \widetilde{\theta }_{a}-\theta _{a}^{\ast }\right\Vert _{\Sigma
_{1}}=O_{p}\left( \sqrt{\frac{s_{a}\log \left( p\right) }{m}}\right) \text{
and }\left\Vert \widetilde{\theta }_{b}-\theta _{b}^{\ast }\right\Vert
_{\Sigma _{1}}=O_{p}\left( \sqrt{\frac{s_{b}\log \left( p\right) }{m}}\right)
\end{equation*}
Moreover 
\begin{equation*}
\Vert b^{0} - \widetilde{b}\Vert_{L_{2}(P_{Z,\eta})} \leq \Vert
\theta^{\ast}_{b} - \widetilde{\theta}_{b} \Vert_{\Sigma_1} + \Vert
\theta^{\ast}_{b} - \theta_{b} \Vert_{\Sigma_1} \leq O_{p}\left( \sqrt{\frac{%
s_{b}\log \left( p\right) }{m}}\right) + \sqrt{K \frac{s_{b}\log(p)}{m}},
\end{equation*}
since $b^{0}(Z)=\langle \theta_{b},\phi(Z) \rangle\in \mathcal{G}%
_{n}(\phi^{n}, s_{b}, j=1, \varphi=id)$ with associated parameter $%
\theta^{\ast}_{b}$ by Condition Lin.L.W. Then 
\begin{equation*}
\Vert b^{0} - \widetilde{b}\Vert_{L_{2}(P_{Z,\eta})} = O_{P}\left( \sqrt{%
\frac{s_{b}\log \left( p\right) }{m}}\right)=o_{P}(1),
\end{equation*}
since $s_{b}\log(p)/m=o(1)$.

Then, following the arguments in the proof of Theorem \ref%
{theo:rate_double_lin} it can be shown that 
\begin{equation*}
\sqrt{m}\Gamma _{ab,m}^{\ast }=O_{P}\left( \sqrt{\frac{s_{a}s_{b}}{m}}\log
(p)\right)
\end{equation*}%
and 
\begin{equation*}
\sqrt{m}\Gamma _{b,m}^{\ast }=o_{p}\left( 1\right)
\end{equation*}%
We next show that $\sqrt{m}\Gamma _{a,m}^{\ast }=o_{p}\left( 1\right)$.
Recall that $a(Z)=\langle \theta _{a}^{\ast },\phi (Z)\rangle +r(Z)$. Hence 
\begin{align*}
\Gamma _{a,m}^{\ast }& =\mathbb{P}_{m}\left[ S_{ab}\left( \widetilde{a}%
-a\right) \left( Z\right) b^{0}\left( Z\right) +m_{a}(O,\widetilde{a}-a)%
\right] \\
& =\mathbb{P}_{m}\left[ S_{ab}\left( \widetilde{a}\left( Z\right) -\langle
\theta _{a}^{\ast },\phi (Z)\rangle -r(Z)\right) b^{0}\left( Z\right)
+m_{a}(O,\widetilde{a}-\langle \theta _{a}^{\ast },\phi \rangle -r)\right] \\
& =\mathbb{P}_{m}\left[ S_{ab}b^{0}\langle \widetilde{\theta }_{a}-\theta
_{a}^{\ast },\phi (Z)\rangle +m_{a}(O,\langle \widetilde{\theta }_{a}-\theta
_{a}^{\ast },\phi \rangle )\right] -\mathbb{P}_{m}\left[ S_{ab}b^{0
}(Z)r(Z)+m_{a}(O,r)\right] \\
& =\langle \mathbb{P}_{m}\left[ S_{ab}b^{0}\phi (Z)+m_{a}(O,\phi )\right] ,%
\widetilde{\theta }_{a}-\theta _{a}^{\ast }\rangle -\mathbb{P}_{m}\left[
S_{ab}b^{0}(Z)r(Z)+m_{a}(O,r)\right] .
\end{align*}%
By Holder's inequality 
\begin{equation*}
\left\vert \langle \mathbb{P}_{m}\left[ S_{ab}b^{0}\phi (Z)+m_{a}(O,\phi )%
\right] ,\widetilde{\theta }_{a}-\theta _{a}^{\ast }\rangle \right\vert \leq
\Vert \mathbb{P}_{m}\left[ S_{ab}b^{0}\phi (Z)+m_{a}(O,\phi )\right] \Vert
_{\infty }\Vert \widetilde{\theta }_{a}-\theta _{a}^{\ast }\Vert _{1}.
\end{equation*}%
By Theorem \ref{theo:rate_main_lin} 
\begin{equation}
\Vert \widetilde{\theta }_{a}-\theta _{a}^{\ast }\Vert _{1}=O_{P}\left( s_{a}%
\sqrt{\frac{\log (p)}{m}}\right) .  \label{eq:lin_model_rate_l1}
\end{equation}%
Because $b^{0}(Z)=\langle \theta _{b},\phi (Z)\rangle $ and $\theta _{b} $
satisfies 
\begin{equation*}
\theta _{b}\in \arg \min\limits_{\theta \in \mathbb{R}^{p}}E_{\eta }\left[
Q_{b}\left( \theta ,\phi ,w=1\right) \right]
\end{equation*}%
we have that 
\begin{equation*}
E_{\eta }\left[ S_{ab}b^{0}\phi (Z)+m_{a}(O,\phi )\right] =0.
\end{equation*}%
Hence, by Condition Lin.L.W.2, Nemirovski's inequality (see Lemma 14.24 in 
\cite{Buhlmann-book}) implies that 
\begin{equation*}
\Vert \mathbb{P}_{m}\left[ S_{ab}b^{0}\phi (Z)+m_{a}(O,\phi )\right] \Vert
_{\infty }=O_{P}\left( \sqrt{\frac{\log (p)}{m}}\right) .
\end{equation*}%
This together with \eqref{eq:lin_model_rate_l1} 
\begin{equation}
\left\vert \langle \mathbb{P}_{m}\left[ S_{ab}b^{0}\phi (Z)+m_{a}(O,\phi )%
\right] ,\widetilde{\theta }_{a}-\theta _{a}^{\ast }\rangle \right\vert
=O_{P}\left( s_{a}\frac{\log (p)}{m}\right) .
\label{eq:lin_model_first_term}
\end{equation}%
On the other hand, by Conditions M.W, Lin.E.W.1 and Lin.L.W.1, Lemma \ref%
{lemma:bound_m_term} implies 
\begin{equation}
\left\vert \mathbb{P}_{m}\left[ S_{ab}b^{0}(Z)r(Z)+m_{a}(O,r)\right]
\right\vert =o_{P}(m^{-1/2}).  \label{eq:lin_model_second_term}
\end{equation}%
Then \eqref{eq:lin_model_first_term} and \eqref{eq:lin_model_second_term}
imply 
\begin{equation*}
\sqrt{m}\Gamma _{a,m}^{\ast }=O_{P}\left( s_{a}\frac{\log (p)}{\sqrt{m}}%
\right) +o_{P}(1)=o_{P}(1)
\end{equation*}%
since 
\begin{equation*}
O_{P}\left( s_{a}\frac{\log (p)}{\sqrt{m}}\right) =o_{P}(1)
\end{equation*}%
by Condition Lin.L.W.1. This finishes the proof of the first part of the
Theorem.

Note that 
\begin{equation*}
E_{\eta}\left[\Upsilon(a, b^{0}) - \chi(\eta) \right]=0
\end{equation*}
by \eqref{bias1}. Then the second and third parts of
the Theorem are proven following the arguments in the proof of Theorem \ref%
{theo:rate_double_lin}.
\end{proof}

\subsection{Asymptotic results for the estimator $\widehat{\protect\chi }%
_{nonlin}$}

\label{sec:app_asym_nonlin}

\subsubsection{Rate double robustness for the estimator $\widehat{\protect%
\chi }_{nonlin}$}

\label{sec:app_dr_nonlin}

\begin{proof}
\lbrack Proof of Theorem \ref{theo:rate_double_nonlin}] Let $n_{k}$ be the
size of sample $\mathcal{D}_{nk}, k=1,2,3$. To simplify the proof, we assume
that $n_k=n/3, k=1,2,3$. Now 
\begin{equation*}
\sqrt{n}\left\{ \widehat{\chi }_{nonlin}-\chi \left( \eta \right) \right\} =%
\sqrt{n}\left\{ \frac{1}{3}\sum_{k=1}^{3}\left[ \widehat{\chi }%
_{lin}^{\left( k\right) }-\chi \left( \eta \right) \right] \right\}
\end{equation*}%
so, since $n_k=n/3$, to show \eqref{eq:nonlin_rate_DR_expansion} it suffices
to show that for $k=1,2,3$ 
\begin{equation}
\sqrt{n_k}\left[ \widehat{\chi }_{nonlin}^{\left( k\right) }-\chi \left(
\eta \right) \right] =\mathbb{G}_{nk}\left[ \Upsilon \left( a,b\right) %
\right] +O_{p}\left( \sqrt{\frac{s_{a}s_{b}}{n_k}}\log (p)\right)
+o_{p}\left( 1\right).  \label{eq:rate_nonlin_goal}
\end{equation}

Applying Theorem \ref{theo:rate_main} with $w\equiv 1$, $\widehat{\beta }%
\equiv \beta ^{\ast }\equiv 0$, we have that for all $k\in \left\{
1,2,3\right\} $ and $c\in \left\{ a,b\right\} $ 
\begin{equation*}
\Vert \widehat{\theta }_{c,(k)}^{0}-\theta _{c}^{\ast }\Vert
_{2}=O_{P}\left( \sqrt{\frac{s_{c}\log (p)}{n_{k}}}\right) .
\end{equation*}%
Fix any $k\in \left\{ 1,2,3\right\} $, $c\in \left\{ a,b\right\} $ and $l\in
\left\{ 1,2\right\} $. Recall that if $c=a$ then $\overline{c}=b$ and vice
versa. Then applying Theorem \ref{theo:rate_main} with $w=\varphi _{%
\overline{c}}^{\prime }$, $\widehat{\beta }=\widehat{\theta }_{\overline{c}%
,(j_{l}(k))}^{0}$ and $\beta ^{\ast }=\theta _{\overline{c}}^{\ast }$ we get 
\begin{equation*}
\Vert \widehat{\theta }_{c,(k),j_{l}(k)}-\theta _{c}^{\ast }\Vert
_{2}=O_{P}\left( \sqrt{\frac{s_{c}\log (p)}{n_{k}}}\right) .
\end{equation*}

Take $k=1$, we will show that \eqref{eq:rate_nonlin_goal} holds. The proof
for $k=2,3$ is entirely analogous. To simplify the notation let $m=n/3$, let 
$\mathcal{D}_{m}$ denote $\mathcal{D}_{n1}$ and $\mathcal{D}_{m}^{c}$ denote 
$\mathcal{D}_{n}\setminus \mathcal{D}_{n1}$. Let 
\begin{equation*}
\widetilde{\theta}_{a}=\frac{\widetilde{\theta }_{a,(2),3} + \widetilde{%
\theta }_{a,(3),2}}{2} \quad \text{and} \quad \widetilde{\theta}_{b}=\frac{%
\widetilde{\theta }_{b,(2),3} + \widetilde{\theta }_{b,(3),2}}{2}.
\end{equation*}
Let 
\begin{equation*}
\widetilde{a}\left( Z\right) \equiv \varphi_{a}\left(\left\langle \widetilde{%
\theta }_{a},\phi \left( Z\right) \right\rangle\right) \quad \text{and}
\quad \widetilde{b}\left( Z\right) \equiv \varphi_{b}\left(\left\langle 
\widetilde{\theta }_{b},\phi \left( Z\right) \right\rangle\right)
\end{equation*}
denote the estimators $\widehat{a}_{\left( \overline{1}\right) }\left(
Z\right) $ and $\widehat{b}_{\left( \overline{1}\right) }\left( Z\right) $
computed using data $\mathcal{D}_{m}^{c}$.

We first show that $\widetilde{a}$ converges to $a$ and $\widetilde{b}$
converges to $b$. By the triangle inequality $\left\Vert \widetilde{a}%
(Z)-a(Z)\right\Vert _{L_{2}\left( P_{\eta }\right) }$ is bounded by 
\begin{align*}
&\left\Vert \widetilde{a}(Z)-\varphi_{a}\left(\left\langle \theta _{a}^{\ast
},\phi(Z) \right\rangle\right) \right\Vert _{L_{2}\left( P_{\eta }\right)
}+\left\Vert a(Z)-\varphi_{a}\left(\left\langle \theta _{a}^{\ast },\phi(Z)
\right\rangle\right)\right\Vert _{L_{2}\left( P_{\eta }\right) } = \\
& \left\Vert \varphi_{a}\left(\left\langle \widetilde{\theta}_{a},\phi(Z)
\right\rangle\right) -\varphi_{a}\left(\left\langle \theta _{a}^{\ast
},\phi(Z) \right\rangle\right) \right\Vert _{L_{2}\left( P_{\eta }\right)
}+\left\Vert a(Z)-\varphi_{a}\left(\left\langle \theta _{a}^{\ast },\phi(Z)
\right\rangle\right)\right\Vert _{L_{2}\left( P_{\eta }\right) } .
\end{align*}%
By Condition NLin.L.1 and Jensen's inequality we have that 
\begin{equation*}
\left\Vert a(Z)-\varphi_{a}\left(\left\langle \theta _{a}^{\ast },\phi(Z)
\right\rangle\right)\right\Vert _{L_{2}\left( P_{\eta }\right) } \leq
E^{1/8} \left[ \left\{ a\left( Z\right) -\varphi _{a}\left( \left\langle
\theta _{a}^{\ast },\phi \left( Z\right) \right\rangle \right) \right\} ^{8}%
\right] \leq \sqrt{\frac{Ks_{a}\log \left( p\right)}{n}}.
\end{equation*}
Using Conditions NLin.L.3, NLin.L.6 and NLin.Link.2, by Lemma \ref%
{lemma:lipschitz} 
\begin{align*}
&\left\Vert \varphi_{a}\left(\left\langle \widetilde{\theta}_{a},\phi(Z)
\right\rangle\right) -\varphi_{a}\left(\left\langle \theta _{a}^{\ast
},\phi(Z) \right\rangle\right) \right\Vert _{L_{2}\left( P_{\eta }\right) }
\leq B_{1}(\Vert \theta_{a}^{\ast}\Vert_{2}, \Vert \widetilde{\theta}%
_{a}-\theta^{\ast}_{a} \Vert_{2},2, k, K) E^{1/2}\left(\langle \phi(Z), 
\widetilde{\theta}_{a}-\theta^{\ast}_{a}\rangle^{2} \right),
\end{align*}
where $B_1$ is a function that is increasing in $\Vert
\theta_{a}^{\ast}\Vert_{2}$ and in $\Vert \widetilde{\theta}%
_{a}-\theta^{\ast}_{a} \Vert_{2}$. Since 
\begin{equation*}
\Vert \widetilde{\theta}_{a}-\theta^{\ast}_{a} \Vert_{2}=O_{P}\left(\sqrt{%
\frac{s_{a}\log(p)}{m}} \right),
\end{equation*}
and by Condition NLin.L.1 $\Vert \theta^{\ast}_{a}\Vert_{2} \leq K$, we have
that 
\begin{equation*}
B_{1}(\Vert \theta_{a}^{\ast}\Vert_{2}, \Vert \widetilde{\theta}%
_{a}-\theta^{\ast}_{a} \Vert_{2},2, k, K) =O_{P}(1).
\end{equation*}
Moreover, by Condition NLin.L.6 
\begin{equation*}
E^{1/2}\left(\langle \phi(Z), \widetilde{\theta}_{a}-\theta^{\ast}_{a}%
\rangle^{2} \right) = O_{P}\left(\sqrt{\frac{s_{a}\log(p)}{m}} \right).
\end{equation*}
Thus 
\begin{equation}
\left\Vert \widetilde{a}(Z)-a(Z)\right\Vert _{L_{2}\left( P_{\eta }\right)
}=O_{P}\left(\sqrt{\frac{s_{a}\log(p)}{m}} \right)=o_{P}(1),
\label{eq:nonlin_rate_atil}
\end{equation}
since $s_{a}\log(p)/m \to 0$ by assumption. Similarly, 
\begin{equation}
\left\Vert \widetilde{b}(Z)-b(Z)\right\Vert _{L_{2}\left( P_{\eta }\right)
}=O_{P}\left(\sqrt{\frac{s_{b}\log(p)}{m}} \right)=o_{P}(1).
\label{eq:nonlin_rate_btil}
\end{equation}

Now,%
\begin{equation*}
\widehat{\chi }^{\left( 1\right) }_{nonlin}-\chi \left( \eta \right)
=N_{m}+\Gamma _{a,m}+\Gamma _{b,m}+\Gamma _{ab,m}
\end{equation*}
where 
\begin{align*}
&N_{m}\equiv \mathbb{P}_{m}\left[ \Upsilon \left( a,b\right) -\chi \left(
\eta \right) \right] \\
&\Gamma _{a,m}\equiv \mathbb{P}_{m}\left[ S_{ab}\left( \widetilde{a}%
-a\right) \left( Z\right) b\left( Z\right) +m_{a}\left( O,\widetilde{a}%
\right) -m_{a}\left( O,a\right) \right] \\
&\Gamma _{b,m}\equiv \mathbb{P}_{m}\left[ S_{ab}\left( \widetilde{b}%
-b\right) \left( Z\right) a\left( Z\right) +m_{b}\left( O,\widetilde{b}%
\right) -m_{b}\left( O,b\right) \right] \\
& \Gamma _{ab,m}\equiv \mathbb{P}_{m}\left[ S_{ab}\left( \widetilde{a}%
-a\right) \left( Z\right) \left( \widetilde{b}-b\right) \left( Z\right) %
\right].
\end{align*}
We will show that $\sqrt{m} \Gamma_{a,m}=o_{P}(1)$. By the Dominated
Convergence Theorem it suffices to show that for any $\varepsilon >0,$ 
\begin{equation*}
P_{\eta }\left[ \sqrt{m}\Gamma _{a,m}>\varepsilon |\mathcal{D}_{m}^{c}\right]
=o_{p}\left( 1\right).
\end{equation*}%
By Markov's inequality it suffices to show that $E_{\eta }\left[ \sqrt{m}%
\Gamma _{a,m}|\mathcal{D}_{m}^{c}\right] =0$ and $Var_{\eta }\left[ \sqrt{m}%
\Gamma _{a,m}|\mathcal{D}_{m}^{c}\right] =o_{p}\left( 1\right)$. Now, 
\begin{eqnarray*}
E_{\eta }\left[ \sqrt{m}\Gamma _{a,m}|\mathcal{D}_{m}^{c}\right] &=&E_{\eta }%
\left[ S_{ab}\left( \widetilde{a}-a\right) \left( Z\right) b\left( Z\right)
+m_{a}\left( O,\widetilde{a}\right) -m_{a}\left( O,a\right) |\mathcal{D}%
_{m}^{c}\right] \\
&=&0
\end{eqnarray*}%
by Proposition \ref{prop:DR functional}. On the other hand 
\begin{eqnarray*}
Var_{\eta }\left[ \sqrt{m}\Gamma _{a,m}|\mathcal{D}_{m}^{c}\right]
&=&E_{\eta }\left[ \left. \left\{ S_{ab}\left( \widetilde{a}-a\right) \left(
Z\right) b\left( Z\right) +m_{a}\left( O,\widetilde{a}\right) -m_{a}\left(
O,a\right) \right\} ^{2}\right\vert \mathcal{D}_{m}^{c}\right] \\
&=&o_{p}\left( 1\right)
\end{eqnarray*}%
by Condition Lin.E.1 and the fact that by \eqref{eq:nonlin_rate_atil} $%
\left\Vert \widetilde{a}-a\right\Vert _{L_{2}\left( P_{\eta }\right)
}=o_{p}\left( 1\right)$. The same line of argument proves that $\sqrt{m}%
\Gamma _{b,m}=o_{p}\left( 1\right) $.

To bound $\sqrt{m}\Gamma _{ab,m} $, by the Cauchy-Schwartz inequality, 
\begin{eqnarray*}
\sqrt{m}\Gamma _{ab,m} &\equiv &\sqrt{m}\mathbb{P}_{m}\left[ S_{ab}\left( 
\widetilde{a}-a\right) \left( Z\right) \left( \widetilde{b}-b\right) \left(
Z\right) \right] \\
&\leq &\sqrt{m}\sqrt{\mathbb{P}_{m}\left[ S_{ab}\left( \widetilde{a}%
-a\right) ^{2}\left( Z\right) \right] }\sqrt{\mathbb{P}_{m}\left[
S_{ab}\left( \widetilde{b}-b\right) ^{2}\left( Z\right) \right] }.
\end{eqnarray*}
We will show that 
\begin{equation}
\sqrt{\mathbb{P}_{m}\left[ S_{ab}\left( \widetilde{a}-a\right) ^{2}\left(
Z\right) \right] } = O_{P}\left(\sqrt{\frac{s_{a}\log(p)}{m}} \right)
\label{eq:nonlinDR_CS_a}
\end{equation}
and 
\begin{equation}
\sqrt{\mathbb{P}_{m}\left[ S_{ab}\left( \widetilde{b}-b\right) ^{2}\left(
Z\right) \right] } = O_{P}\left(\sqrt{\frac{s_{b}\log(p)}{m}} \right) .
\label{eq:nonlinDR_CS_b}
\end{equation}
We begin with \eqref{eq:nonlinDR_CS_a}. Fix $\varepsilon>0$. Take $L_0,
L_1>0 $. Let 
\begin{equation*}
A=\left\lbrace \Vert \widetilde{\theta}_{a} - \theta^{\ast}_{a}\Vert_{2}^{2}
\leq L_1 \frac{s_{a}\log(p)}{m} \right\rbrace.
\end{equation*}
Then 
\begin{align*}
&P_{\eta}\left(\mathbb{P}_{m}\left\lbrace S_{ab}(\widetilde{a}%
-a)^{2}\right\rbrace \geq L_0 \frac{s_{a}\log(p)}{m} \right)=
E_{\eta}\left\lbrace P\left(\mathbb{P}_{m}\left\lbrace S_{ab}(\widetilde{a}%
-a)^{2}\right\rbrace \geq L_0 \frac{s_{a}\log(p)}{m}\mid \widetilde{\theta}%
_{a} \right)\right\rbrace = \\
& E_{\eta}\left\lbrace P\left(\mathbb{P}_{m}\left\lbrace S_{ab}(\widetilde{a}%
-a)^{2}\right\rbrace \geq L_0 \frac{s_{a}\log(p)}{m} \mid \widetilde{\theta}%
_{a}\right) \mid A \right\rbrace P_{\eta}(A) + \\
& E_{\eta}\left\lbrace P\left(\mathbb{P}_{m}\left\lbrace S_{ab}(\widetilde{a}%
-a)^{2}\right\rbrace \geq L_0 \frac{s_{a}\log(p)}{m} \mid \widetilde{\theta}%
_{a} \right) \mid A^{c} \right\rbrace P_{\eta}(A^{c}) \leq \\
& E_{\eta}\left\lbrace P\left(\mathbb{P}_{m}\left\lbrace S_{ab}(\widetilde{a}%
-a)^{2} \right\rbrace\geq L_0 \frac{s_{a}\log(p)}{m} \mid \widetilde{\theta}%
_{a}\right) \mid A \right\rbrace + P_{\eta}(A^{c}).
\end{align*}
By Theorem \ref{theo:rate_main} we may choose $L_1$ such that for all
sufficiently large $n$, $P_{\eta}(A^{c})<\varepsilon/2$. We will show that
we can choose $L_0$ such that the first term in the right hand side of the
last display is smaller than $\varepsilon/2$. This will show that %
\eqref{eq:nonlinDR_CS_a} holds. Since $\widetilde{\theta}_{a}$ is
independent of the data in $\mathcal{D}_{m}$, by Markov's inequality it
suffices to show that there exists $L_2>0$ depending only on $k, K$ and $L_1$
such that for all $\theta$ that satisfy 
\begin{equation*}
\Vert\theta - \theta^{\ast}_{a}\Vert_{2}^{2} \leq L_1 \frac{s_{a}\log(p)}{m}
\end{equation*}
it holds that 
\begin{equation*}
E_{\eta}\left[ S_{ab} \left(\varphi_{a}(\langle \theta, \phi(Z)\rangle)
-a(Z) \right)^{2} \right] \leq L_{2}\frac{s_{a}\log(p)}{n}.
\end{equation*}

Take then $\theta$ that satisfies 
\begin{equation*}
\Vert\theta - \theta^{\ast}_{a}\Vert_{2}^{2} \leq L_1 \frac{s_{a}\log(p)}{m}.
\end{equation*}
By the Cauchy-Schwartz inequality 
\begin{align*}
E_{\eta}\left[ S_{ab} \left(\varphi_{a}(\langle \theta, \phi(Z)\rangle)
-a(Z) \right)^{2} \right] \leq E_{\eta}^{1/2}\left[S_{ab}^{2} \right]
E_{\eta}^{1/2}\left[ \left(\varphi_{a}(\langle \theta, \phi(Z)\rangle)
-a(Z)\right)^{4}\right].
\end{align*}
By Condition NLin.L.5 and Jensen's inequality 
\begin{equation}
E_{\eta}^{1/2}\left[S_{ab}^{2} \right] \leq K^{1/4}.  \label{eq:nonlin_sab}
\end{equation}
Using Lemma \ref{lemma:lipschitz} and Condition NLin.L.1, it is easy to show
that there exists $L_{3}$ depending only on $k, K$ and $L_1$ such that 
\begin{align*}
E_{\eta}^{1/2}\left[ \left( \varphi_{a}(\langle \theta, \phi(Z)\rangle)
-a(Z) \right)^{4}\right]\leq L_{3} \left(\frac{s_{a}\log(p)}{m} \right).
\end{align*}
This together with \eqref{eq:nonlin_sab} implies that 
\begin{equation*}
E_{\eta}\left[ S_{ab} \left(\varphi_{a}(\langle \theta, \phi(Z)\rangle)
-a(Z) \right)^{2} \right] \leq L_{2}\left(\frac{s_{a}\log(p)}{m} \right),
\end{equation*}
where $L_{2}$ depends only on $k, K$ and $L_1$. Thus \eqref{eq:nonlinDR_CS_a}
holds. \eqref{eq:nonlinDR_CS_b} is proven analogously. Hence 
\begin{equation*}
\sqrt{m}\Gamma_{ab,m}=\sqrt{m} O_{P}\left(\sqrt{\frac{s_{a}\log(p)}{m} }
\right) O_{P}\left(\sqrt{\frac{s_{b}\log(p)}{m} } \right)=O_{P}\left(\sqrt{%
\frac{s_a s_b}{m}}\log(p) \right).
\end{equation*}
This finishes the proof the first part of the Theorem. The proof of the
second part follows by using the same arguments used in the proof of the
second part of Theorem \ref{theo:rate_double_lin}.

The third part can be proven using the same arguments used in the proof of
the third part of Theorem \ref{theo:rate_double_lin}, the only difference
being the proof that 
\begin{equation}
\mathbb{P}_{m}^{1/2}\left\{ S_{ab}^{2}\left( \widetilde{b}-b\right)
^{2}\left( \widetilde{a}-a\right) ^{2}\right\}=o_{P}(1).
\label{eq:nonlin_crossvar}
\end{equation}
To bound this term, we use Cauchy-Schwartz to get 
\begin{equation*}
\mathbb{P}_{m}^{1/2}\left\{ S_{ab}^{2}\left( \widetilde{b}%
-b\right)^{2}\left( \widetilde{a}-a\right) ^{2}\right\}\leq \mathbb{P}%
_{m}^{1/4}\left\{S_{ab}^{2} \left( \widetilde{b}-b\right) ^{4}\right\} 
\mathbb{P}_{m}^{1/4}\left\{S_{ab}^{2} \left( \widetilde{a}%
-a\right)^{4}\right\}.
\end{equation*}
Using the same type of arguments used to prove \eqref{eq:nonlinDR_CS_a} and %
\eqref{eq:nonlinDR_CS_b} it is easy to show that 
\begin{equation*}
\mathbb{P}_{m}^{1/4}\left\{ S_{ab}^{2}\left( \widetilde{a}-a\right)
^{4}\right\} =O_{P}\left(\sqrt{\frac{s_{a}\log(p)}{m}} \right) \quad \text{%
and} \quad \mathbb{P}_{m}^{1/4}\left\{S_{ab}^{2} \left( \widetilde{b}%
-b\right)^{4}\right\} =O_{P}\left(\sqrt{\frac{s_{b}\log(p)}{m}} \right).
\end{equation*}
Thus \eqref{eq:nonlin_crossvar} holds and the third part of the Theorem is
proven.
\end{proof}

\subsubsection{Model double robustness for the estimator $\widehat{\protect%
\chi }_{nonlin}$}

\label{sec:proof_dr_model_nonlin}

\begin{proof}
\lbrack Proof of Theorem \ref{theo:model_DR_Nonlin}] Let $n_{k}$ be the size
of sample $\mathcal{D}_{nk}, k=1,2,3$. To simplify the proof, we assume that 
$n_k=n/3, k=1,2,3$. Now 
\begin{equation*}
\sqrt{n}\left\{ \widehat{\chi }_{nonlin}-\chi \left( \eta \right) \right\} =%
\sqrt{n}\left\{ \frac{1}{3}\sum_{k=1}^{3}\left[ \widehat{\chi }%
_{lin}^{\left( k\right) }-\chi \left( \eta \right) \right] \right\}
\end{equation*}%
so, since $n_k=n/3$, to show \eqref{eq:Nonlin_model_DR_expansion} it
suffices to show that for $k=1,2,3$ 
\begin{equation}
\sqrt{n_k}\left[ \widehat{\chi }_{nonlin}^{\left( k\right) }-\chi \left(
\eta \right) \right] =\mathbb{G}_{nk}\left[ \Upsilon \left( a,b\right) %
\right] +O_{p}\left( \sqrt{\frac{s_{a}s_{b}}{n_k}}\log (p)\right)
+o_{p}\left( 1\right).  \label{eq:model_nonlin_goal}
\end{equation}

Applying Theorem \ref{theo:rate_main} with $w\equiv 1$, $\widehat{\beta}%
\equiv\beta^{\ast}\equiv 0$, we have that for all $k\in\left\lbrace 1,2,3
\right\rbrace$ 
\begin{align*}
\Vert \widehat{\theta}^{0}_{a,(k)}-\theta^{\ast}_{a}\Vert_{2}=O_{P}\left(%
\sqrt{\frac{s_{a}\log(p)}{n_{k}}} \right) \quad \text{and} \quad \Vert 
\widehat{\theta}^{0}_{b,(k)}-\theta^{0\ast}_{b}\Vert_{2}=O_{P}\left(\sqrt{%
\frac{s_{b}\log(p)}{n_{k}}} \right).
\end{align*}
Fix any $k\in\left\lbrace 1,2,3\right\rbrace$ and $l\in\left\lbrace 1,2
\right\rbrace$. Then applying Theorem \ref{theo:rate_main} with $%
w=\varphi_{b}^{\prime}$, $\widehat{\beta}=\widehat{\theta}%
^{0}_{b,(j_{l}(k))} $ and $\beta^{\ast}=\theta^{0\ast}_{b}$ we get 
\begin{align*}
\Vert \widehat{\theta}_{a,(k),
j_{l}(k)}-\theta^{\ast}_{a}\Vert_{2}=O_{P}\left(\sqrt{\frac{s_{a}\log(p)}{%
n_{k}}} \right) \quad \text{and} \quad \Vert \widehat{\theta}_{a,(k),
j_{l}(k)}-\theta^{\ast}_{a}\Vert_{1}=O_{P}\left(s_{a}\sqrt{\frac{\log(p)}{%
n_{k}}} \right).
\end{align*}
Applying Theorem \ref{theo:rate_main} with $w=\varphi_{a}^{\prime}$, $%
\widehat{\beta}=\widehat{\theta}^{0}_{a,(j_{l}(k))}$ and $%
\beta^{\ast}=\theta^{\ast}_{a}$ we get 
\begin{align*}
\Vert \widehat{\theta}_{b,(k),
j_{l}(k)}-\theta^{1\ast}_{b}\Vert_{2}=O_{P}\left(\sqrt{\frac{%
\max(s_{b},s_{a})\log(p)}{n_{k}}} \right).
\end{align*}

Take $k=1$, we will show that \eqref{eq:model_nonlin_goal} holds. The proof
for $k=2,3$ is entirely analogous. To simplify the notation let $m=n/3$, let 
$\mathcal{D}_{m}$ denote $\mathcal{D}_{n1}$ and $\mathcal{D}_{m}^{c}$ denote 
$\mathcal{D}_{n}\setminus \mathcal{D}_{n1}$. Let 
\begin{equation*}
\widetilde{\theta}_{a}=\frac{\widetilde{\theta }_{a,(2),3} + \widetilde{%
\theta }_{a,(3),2}}{2} \quad \text{and} \quad \widetilde{\theta}_{b}=\frac{%
\widetilde{\theta }_{b,(2),3} + \widetilde{\theta }_{b,(3),2}}{2}.
\end{equation*}
Let 
\begin{equation*}
\widetilde{a}\left( Z\right) \equiv \varphi_{a}\left(\left\langle \widetilde{%
\theta }_{a},\phi \left( Z\right) \right\rangle\right) \quad \text{and}
\quad \widetilde{b}\left( Z\right) \equiv \varphi_{b}\left(\left\langle 
\widetilde{\theta }_{b},\phi \left( Z\right) \right\rangle\right)
\end{equation*}
denote the estimators $\widehat{a}_{\left( \overline{1}\right) }\left(
Z\right) $ and $\widehat{b}_{\left( \overline{1}\right) }\left( Z\right) $
computed using data $\mathcal{D}_{m}^{c}$.

Arguing like in the proof of Theorem \ref{theo:rate_double_nonlin}, it is
easy to show that 
\begin{equation}
\left\Vert \widetilde{a}(Z)-a(Z)\right\Vert _{L_{2}\left( P_{\eta }\right)
}=O_{P}\left(\sqrt{\frac{s_{a}\log(p)}{m}} \right)=o_{P}(1),
\label{eq:nonlin_model_atil}
\end{equation}
and 
\begin{equation}
\left\Vert \widetilde{b}(Z)-b^{1}(Z)\right\Vert _{L_{2}\left( P_{\eta
}\right) }=O_{P}\left(\sqrt{\frac{\max(s_{b},s_{a})\log(p)}{m}}
\right)=o_{P}(1).  \label{eq:nonlin_model_btil}
\end{equation}

Now,%
\begin{equation*}
\widehat{\chi }^{\left( 1\right) }_{nonlin}-\chi \left( \eta \right)
=N_{m}^{\ast }+\Gamma _{a,m}^{\ast }+\Gamma _{b,m}^{\ast }+\Gamma
_{ab,m}^{\ast }
\end{equation*}%
where 
\begin{align*}
& N_{m}^{\ast }\equiv \mathbb{P}_{m}\left[ \Upsilon \left( a,b^{1}\right)
-\chi \left( \eta \right) \right] , \\
& \Gamma _{a,m}^{\ast }\equiv \mathbb{P}_{m}\left[ S_{ab}\left( \widetilde{a}
-a\right) \left( Z\right) b^{1}\left( Z\right) +m_{a}(O,\widetilde{a}
)-m_{a}(O,a)\right] , \\
& \Gamma _{b,m}^{\ast }\equiv \mathbb{P}_{m}\left[ S_{ab}\left( \widetilde{b}%
-b^{1}\right) \left( Z\right) a\left( Z\right) +m_{b}(O,\widetilde{b}
)-m_{b}(O,b^{1})\right] , \\
& \Gamma _{ab,m}^{\ast }\equiv \mathbb{P}_{m}\left[ S_{ab}\left( \widetilde{a%
}-a\right) \left( Z\right) \left( \widetilde{b}-b^{1}\right) \left( Z\right) %
\right] .
\end{align*}%
Following the arguments in the proof of Theorem \ref{theo:rate_double_nonlin}
it can be shown that 
\begin{equation*}
\sqrt{m}\Gamma _{ab,m}^{\ast }=O_{P}\left( \sqrt{\frac{s_{a}\max(s_{b},s_{a})%
}{m}}\log (p)\right)
\end{equation*}%
and 
\begin{equation*}
\sqrt{m}\Gamma _{b,m}^{\ast }=o_{p}\left( 1\right).
\end{equation*}%
So, to prove \eqref{eq:model_nonlin_goal} it suffices to show that 
\begin{equation}
\sqrt{m}\Gamma _{a,m}^{\ast }=o_{P}\left(1\right).  \label{eq:Op_II}
\end{equation}

Recall that 
\begin{equation*}
r(Z)\equiv a(Z)- \varphi_{a}\left(\langle \theta^{\ast}_{a}, \phi(Z)\rangle
\right).
\end{equation*}
Then 
\begin{equation*}
\sqrt{m} \mathbb{P}_{m}\left[ S_{ab}b^{1}\left( \varphi_{a}\left(\langle 
\widetilde{\theta}_{a}, \phi(Z)\rangle \right)-a(Z)\right)
+m_{a}(O,\varphi_{a}\left(\langle \widetilde{\theta}_{a}, \phi\rangle
\right)-a) \right]
\end{equation*}
is equal to 
\begin{align*}
&\sqrt{m} \mathbb{P}_{m}\left[ S_{ab}b^{1}\left( \varphi_{a}\left(\langle 
\widetilde{\theta}_{a}, \phi(Z)\rangle \right)-\varphi_{a}\left(\langle
\theta^{\ast}_{a}, \phi(Z)\rangle \right)\right)
+m_{a}(O,\varphi_{a}\left(\langle \widetilde{\theta}_{a}, \phi\rangle
\right)-\varphi_{a}\left(\langle \theta^{\ast}_{a}, \phi\rangle \right)) %
\right]- \\
& \sqrt{m} \mathbb{P}_{m}\left[ S_{ab}b^{1} r(Z) +m_{a}(O,r) \right].
\end{align*}
By Conditions M.W, NLin.E.W.1 and NLin.L.W.1, Lemma \ref{lemma:bound_m_term}
implies 
\begin{equation*}
\left\vert \mathbb{P}_{m}\left[ S_{ab}b^{1}(Z)r(Z)+m_{a}(O,r)\right]
\right\vert =o_{P}(m^{-1/2}).
\end{equation*}
Thus, to prove \eqref{eq:Op_II}, it suffices to show that 
\begin{align*}
&\left\vert \sqrt{m} \mathbb{P}_{m}\left[ S_{ab}b^{1}\left(
\varphi_{a}\left(\langle \widetilde{\theta}_{a}, \phi(Z)\rangle
\right)-\varphi_{a}\left(\langle \theta^{\ast}_{a}, \phi(Z)\rangle
\right)\right) +m_{a}(O,\varphi_{a}\left(\langle \widetilde{\theta}_{a},
\phi\rangle \right)-\varphi_{a}\left(\langle \theta^{\ast}_{a}, \phi\rangle
\right)) \right]\right\vert \\
&=O_{P}\left( s_{a}\sqrt{\frac{\log(p)}{m}}\right),
\end{align*}
since by Condition NLin.W.1. 
\begin{equation*}
s_{a}\sqrt{\frac{\log(p)}{m}}=o(1).
\end{equation*}

Fix $\varepsilon>0$ and take $L_1 >0$ to be chosen later. Let 
\begin{align*}
A=\left\lbrace \Vert \widetilde{\theta}_{a} - \theta^{\ast}_{a}
\Vert_{1}\leq L_{1} s_{a}\sqrt{\frac{\log(p)}{m}}, \quad \Vert \widetilde{%
\theta}_{a} - \theta^{\ast}_{a} \Vert_{2}\leq L_{1} \sqrt{\frac{s_{a}\log(p)%
}{{m}}}\right\rbrace
\end{align*}
and 
\begin{equation*}
\mathcal{W}=\sqrt{m} \mathbb{P}_{m}\left[ S_{ab}b^{1}\left(
\varphi_{a}\left(\langle \widetilde{\theta}_{a}, \phi(Z)\rangle
\right)-\varphi_{a}\left(\langle \theta^{\ast}_{a}, \phi(Z)\rangle
\right)\right) +m_{a}(O,\varphi_{a}\left(\langle \widetilde{\theta}_{a},
\phi\rangle \right)-\varphi_{a}\left(\langle \theta^{\ast}_{a}, \phi\rangle
\right)) \right].
\end{equation*}
Then 
\begin{align*}
&P_{\eta}\left( \vert \mathcal{W}\vert \geq L_{0} \frac{s_{a}\log(p)}{\sqrt{m%
}}\right) = E_{\eta}\left\lbrace P\left( \vert \mathcal{W} \vert \geq L_0 
\frac{s_{a}\log(p)}{\sqrt{m}} \mid \widetilde{\theta}_{a}\right)
\right\rbrace= \\
& E_{\eta}\left\lbrace P\left( \vert \mathcal{W} \vert \geq L_0 \frac{%
s_{a}\log(p)}{\sqrt{m}} \mid \widetilde{\theta}_{a} \right) \mid
A\right\rbrace P_{\eta}(A) + \\
& E_{\eta}\left\lbrace P\left( \vert \mathcal{W} \vert \geq L_0 \frac{%
s_{a}\log(p)}{\sqrt{m}} \mid \widetilde{\theta}_{a}\right) \mid
A^{c}\right\rbrace P_{\eta}(A^{c}) \leq \\
& E_{\eta}\left\lbrace P\left( \vert \mathcal{W} \vert \geq L_0 \frac{%
s_{a}\log(p)}{\sqrt{m}} \mid \widetilde{\theta}_{a}\right) \mid
A\right\rbrace + P_{\eta}(A^{c}).
\end{align*}
By Theorem \ref{theo:rate_main} we can choose $L_1$ such that for all
sufficiently large $n$, $P_{\eta}(A^{c})< \varepsilon/2$. We will show that
we can choose $L_0$ such that the first term in the last display is smaller
than $\varepsilon/2$ too. Since $\widetilde{\theta}_{a}$ is independent of
the data in sample $\mathcal{D}_{m}$, by Markov's inequality it suffices to
show that there exists $L_2>0$ depending only on $L_1$, $k$ and $K$ such
that 
\begin{equation*}
\sqrt{m}E_{\eta} \left\lbrace \left\vert \mathbb{P}_{m}\left[
S_{ab}b^{1}\left( \varphi_{a}\left(\langle \theta, \phi(Z)\rangle \right)-
\varphi_{a}\left(\langle \theta^{\ast}_{a}, \phi(Z)\rangle \right) \right)
+m_{a}(O,\varphi_{a}\left(\langle \theta, \phi\rangle \right)-
\varphi_{a}\left(\langle \theta^{\ast}_{a}, \phi\rangle \right) ) \right]
\right\vert \right\rbrace
\end{equation*}
is bounded by 
\begin{equation*}
L_{2}\left( \frac{s_{a}\log(p)}{\sqrt{m}} \right)
\end{equation*}
for all $\theta$ such that 
\begin{equation*}
\Vert \theta- \theta^{\ast}_{a} \Vert_{1}\leq L_{1} s_{a}\sqrt{\frac{\log(p)%
}{m}} \quad \text{and} \quad \Vert \theta- \theta^{\ast}_{a} \Vert_{2}\leq
L_{1} \sqrt{\frac{s_{a}\log(p)}{m}}.
\end{equation*}

Take $m$ large enough such that 
\begin{equation*}
s_{a}\sqrt{\frac{\log(p)}{m}} \leq 1 \quad \text{ and } \quad \frac{%
s_{a}\log(p)}{{m}} \leq 1.
\end{equation*}
Take $\theta$ such that 
\begin{equation*}
\Vert \theta- \theta^{\ast}_{a} \Vert_{1}\leq L_{1} s_{a}\sqrt{\frac{\log(p)%
}{m}} \quad \text{and} \quad \Vert \theta- \theta^{\ast}_{a} \Vert_{2}\leq
L_{1} \sqrt{\frac{s_{a}\log(p)}{m}}.
\end{equation*}
Using a first order Taylor expansion, we get that 
\begin{equation*}
\sqrt{m} \mathbb{P}_{m}\left[ S_{ab}b^{1}\left( \varphi_{a}\left(\langle
\theta, \phi(Z)\rangle \right)- \varphi_{a}\left(\langle \theta^{\ast}_{a},
\phi(Z)\rangle \right) \right) +m_{a}(O,\varphi_{a}\left(\langle \theta,
\phi\rangle \right)- \varphi_{a}\left(\langle \theta^{\ast}_{a}, \phi\rangle
\right) ) \right]
\end{equation*}
is equal to 
\begin{align}
& \sqrt{m} \mathbb{P}_{m}\left[ S_{ab}b^{1} \varphi^{\prime}_{a}\left(
\langle \theta^{\ast}_{a}, \phi(Z)\rangle\right) \langle
\theta-\theta^{\ast}_{a}, \phi(Z)\rangle +m_{a}(O,
\varphi^{\prime}_{a}\left( \langle \theta^{\ast}_{a}, \phi\rangle\right)
\langle \theta-\theta^{\ast}_{a}, \phi\rangle) \right] +  \notag \\
& \frac{\sqrt{m}}{2} \mathbb{P}_{m}\left[ S_{ab}b^{1}
\varphi^{\prime\prime}_{a}\left( \langle \theta_{a}^{\dagger},
\phi(Z)\rangle\right) \langle \theta-\theta^{\ast}_{a}, \phi(Z)\rangle^{2}
+m_{a}(O, \varphi^{\prime\prime}_{a}\left( \langle \theta^{\dagger}_{a},
\phi\rangle\right) \langle \theta-\theta^{\ast}_{a}, \phi\rangle^{2}) \right]%
,  \label{eq:mod_double_taylor}
\end{align}
where $\Vert\theta^{\dagger}_{a}-\theta^{\ast}_{a}\Vert_{2}\leq
\Vert\theta-\theta^{\ast}_{a}\Vert_{2}$. Using the linearity of $%
m_a(O,\cdot) $ and Holder's inequality 
\begin{align}
&\left\vert \sqrt{m} \mathbb{P}_{m}\left[ S_{ab}b^{1}
\varphi^{\prime}_{a}\left( \langle \theta^{\ast}_{a}, \phi(Z)\rangle\right)
\langle \theta-\theta^{\ast}_{a}, \phi(Z)\rangle +m_{a}(O,
\varphi^{\prime}_{a}\left( \langle \theta^{\ast}_{a}, \phi\rangle\right)
\langle \theta-\theta^{\ast}_{a}, \phi\rangle) \right] \right\vert =  \notag
\\
& \sqrt{m}\left\vert \langle \theta-\theta^{\ast}_{a}, \mathbb{P}_{m}\left[
S_{ab}b^{1} \phi(Z)\varphi^{\prime}_{a}\left( \langle \theta^{\ast}_{a},
\phi(Z)\rangle\right) +m_{a}(O, \varphi^{\prime}_{a}\left( \langle
\theta^{\ast}_{a}, \phi\rangle\right) \phi) \right] \rangle\right\vert \leq 
\notag \\
& \sqrt{m} \Vert \theta-\theta^{\ast}_{a} \Vert_{1} \Vert \mathbb{P}_{m}%
\left[ S_{ab}b^{1} \phi(Z)\varphi^{\prime}_{a}\left( \langle
\theta^{\ast}_{a}, \phi(Z)\rangle\right) +m_{a}(O,
\varphi^{\prime}_{a}\left( \langle \theta^{\ast}_{a}, \phi\rangle\right)
\phi) \right] \Vert_{\infty}.  \label{eq:nonlin_mod_holder}
\end{align}
By assumption 
\begin{equation}
\Vert \theta - \theta^{\ast}_{a} \Vert_{1}\leq L_1 s_{a} \sqrt{\frac{ \log(p)%
}{m}}.  \label{eq:bound_l1_nonlin_mod}
\end{equation}
On the other hand, 
\begin{equation*}
E_{\eta} \left\lbrace S_{ab}b^{1} \phi(Z)\varphi^{\prime}_{a}\left( \langle
\theta^{\ast}_{a}, \phi(Z)\rangle\right) +m_{a}(O,
\varphi^{\prime}_{a}\left( \langle \theta^{\ast}_{a}, \phi\rangle\right)
\phi) \right\rbrace = 0,
\end{equation*}
because $b^{1}(Z)=\varphi_{b}\left(\langle \theta^{1}_{b} ,
\phi(Z)\rangle\right)$ and $\theta^{1}_{b}$ is defined by 
\begin{equation*}
\theta^{1}_{b}\in \arg\min\limits_{\theta\in\mathbb{R}^{p}}
E_{\eta}\left\lbrace S_{ab}\psi_{b}\left(\langle \theta, \phi(Z) \rangle
\right)\varphi_{a}^{\prime}\left(\langle \theta^{\ast}_{a}, \phi(Z) \rangle
\right) + \langle \theta, m_{a}(O, \varphi_{a}^{\prime}\left(\langle
\theta^{\ast}_{a}, \phi \rangle \right) \phi\rangle \right\rbrace.
\end{equation*}
Then, using Conditions NLin.L.W.2, NLin.L.W.4 and arguments similar to those
used in the proof of Lemma \ref{lemma:lambda_choice} in Appendix A 
\begin{equation}
E_{\eta} \left\lbrace \Vert \mathbb{P}_{m}\left[ S_{ab}b^{1}
\phi(Z)\varphi^{\prime}_{a}\left( \langle \theta^{\ast}_{a},
\phi(Z)\rangle\right) +m_{a}(O, \varphi^{\prime}_{a}\left( \langle
\theta^{\ast}_{a}, \phi\rangle\right) \phi) \right] \Vert_{\infty}\right%
\rbrace\leq L_{3}\sqrt{\frac{\log(p)}{m}},  \label{eq:bound_inf_taylor_mod}
\end{equation}
where $L_3$ depends only on $K$. Hence by \eqref{eq:nonlin_mod_holder}, %
\eqref{eq:bound_l1_nonlin_mod}, \eqref{eq:bound_inf_taylor_mod} 
\begin{align}
&E_{\eta}\left\lbrace \left\vert \sqrt{m} \mathbb{P}_{m}\left[ S_{ab}b^{1}
\varphi^{\prime}_{a}\left( \langle \theta^{\ast}_{a}, \phi(Z)\rangle\right)
\langle \theta-\theta^{\ast}_{a}, \phi(Z)\rangle +m_{a}(O,
\varphi^{\prime}_{a}\left( \langle \theta^{\ast}_{a}, \phi\rangle\right)
\langle \theta-\theta^{\ast}_{a}, \phi\rangle) \right] \right\vert\right%
\rbrace\leq  \notag \\
& L_{1}L_{3} \frac{s_{a} \log(p)}{\sqrt{m}}.  \label{bound_nemir_term}
\end{align}

Now, to bound the expectation of the absolute value of%
\eqref{eq:mod_double_taylor}, using Condition M.W 
\begin{align}
& E_{\eta}\left\lbrace \left\vert \mathbb{P}_{m}\left[ S_{ab}b^{1}
\varphi^{\prime\prime}_{a}\left( \langle \theta_{a}^{\dagger},
\phi(Z)\rangle\right) \langle \theta-\theta^{\ast}_{a}, \phi(Z)\rangle^{2}
+m_{a}(O, \varphi^{\prime\prime}_{a}\left( \langle \theta^{\dagger}_{a},
\phi\rangle\right) \langle \theta-\theta^{\ast}_{a}, \phi\rangle^{2}) \right]
\right\vert\right\rbrace \leq  \notag \\
& E_{\eta}\left\lbrace \left\vert S_{ab}b^{1}
\varphi^{\prime\prime}_{a}\left( \langle \theta_{a}^{\dagger},
\phi(Z)\rangle\right) \langle \theta-\theta^{\ast}_{a},
\phi(Z)\rangle^{2}\right\vert\right\rbrace +E_{\eta}\left\lbrace \left\vert
m_{a}(O, \varphi^{\prime\prime}_{a}\left( \langle \theta^{\dagger}_{a},
\phi\rangle\right) \langle \theta-\theta^{\ast}_{a}, \phi\rangle^{2})
\right\vert\right\rbrace\leq  \notag \\
& E_{\eta}\left\lbrace \left\vert S_{ab}b^{1}
\varphi^{\prime\prime}_{a}\left( \langle \theta_{a}^{\dagger},
\phi(Z)\rangle\right) \langle \theta-\theta^{\ast}_{a},
\phi(Z)\rangle^{2}\right\vert\right\rbrace +E_{\eta}\left\lbrace
m^{\ddagger}_{a}(O, \left\vert \varphi^{\prime\prime}_{a}\left( \langle
\theta^{\dagger}_{a}, \phi\rangle\right) \langle \theta-\theta^{\ast}_{a},
\phi\rangle^{2}\right\vert) \right\rbrace=  \notag \\
& E_{\eta}\left\lbrace \left\vert S_{ab}b^{1}
\varphi^{\prime\prime}_{a}\left( \langle \theta_{a}^{\dagger},
\phi(Z)\rangle\right) \langle \theta-\theta^{\ast}_{a},
\phi(Z)\rangle^{2}\right\vert\right\rbrace +E_{\eta}\left\lbrace \mathcal{R}%
^{\ddagger}_{a} \left\vert \varphi^{\prime\prime}_{a}\left( \langle
\theta^{\dagger}_{a}, \phi\rangle\right) \langle \theta-\theta^{\ast}_{a},
\phi\rangle^{2}\right\vert \right\rbrace.  \notag
\end{align}

By Cauchy Schwartz 
\begin{equation*}
E_{\eta}\left[ \left\vert S_{ab}b^{1} \varphi^{\prime\prime}_{a}\left(
\langle \theta_{a}^{\dagger}, \phi(Z)\rangle\right) \langle
\theta-\theta^{\ast}_{a}, \phi(Z)\rangle^{2}\right\vert\right] \leq
E_{\eta}^{1/2}\left[(S_{ab}b^{1})^{2}\right] E_{\eta}^{1/4}\left[ \left(
\varphi^{\prime\prime}_{a}\left( \langle \theta_{a}^{\dagger},
\phi(Z)\rangle\right) \right)^{4} \right] E_{\eta}^{1/4}\left[ \langle
\theta-\theta^{\ast}_{a}, \phi(Z)\rangle^{8}\right].
\end{equation*}
By Condition NLin.E.W.1 
\begin{equation*}
E_{\eta}^{1/2}\left[(S_{ab}b^{1})^{2}\right] \leq K^{1/2}.
\end{equation*}
Using Condition NLin.Link.4, by Lemma \ref{lemma:bound_moment_lipschitz} 
\begin{equation*}
E_{\eta}^{1/4}\left\lbrace \left\vert
\varphi_{a}^{\prime\prime}\left(\langle \theta_{a}^{\dagger},\phi(Z) \rangle
\right) \right\vert^{4} \right\rbrace \leq B_{2}(\vert
\varphi_{a}^{\prime\prime}(0)\vert, \Vert \theta_{a}^{\dagger}\Vert_{2}, 4,
k, K),
\end{equation*}
where $B_2$ is a function that is increasing in $\Vert
\theta_{a}^{\dagger}\Vert_{2}$. Now by Condition NLin.L.W.1, $\Vert
\theta^{\ast}_{a}\Vert_{2}\leq K$, and since $\sqrt{s_{a}\log(p)/m}\leq 1$ 
\begin{equation*}
\Vert \theta_{a}^{\dagger}\Vert_{2} \leq \Vert
\theta^{\ast}_{a}-\theta_{a}^{\dagger}\Vert_{2} + \Vert
\theta^{\ast}_{a}\Vert_{2} \leq \Vert\theta^{\ast}_{a}-\theta\Vert_{2} +
\Vert \theta^{\ast}_{a}\Vert_{2} \leq L_{1} \sqrt{\frac{s_a\log(p)}{m}} +
K\leq L_1 +K.
\end{equation*}
This implies 
\begin{equation}
E_{\eta}^{1/4}\left\lbrace \left\vert
\varphi_{a}^{\prime\prime}\left(\langle \widetilde{\theta}_{a},\phi(Z)
\rangle \right) \right\vert^{4} \right\rbrace \leq B_{2}(\vert
\varphi_{a}^{\prime\prime}(0)\vert, L_1 +K, 4, k, K).  \notag
\end{equation}
On the other hand, by Condition NLin.L.6 
\begin{equation*}
E_{\eta}^{1/4}\left[ \langle \theta-\theta^{\ast}_{a}, \phi(Z)\rangle^{8}%
\right]\leq 8K^{2} \Vert \theta-\theta^{\ast}_{a}\Vert_{2}^{2}\leq 8 K^{2}
L_{1}^{2}\frac{s_{a}\log(p)}{m}.
\end{equation*}
Thus 
\begin{equation*}
E_{\eta}\left[ \left\vert S_{ab}b^{1} \varphi^{\prime\prime}_{a}\left(
\langle \theta_{a}^{\dagger}, \phi(Z)\rangle\right) \langle
\theta-\theta^{\ast}_{a}, \phi(Z)\rangle^{2}\right\vert\right] \leq K^{1/2}
B_{2}(\vert \varphi_{a}^{\prime\prime}(0)\vert, L_1 +K, 4, k, K)8 K^{2}
L_{1}^{2}\frac{s_{a}\log(p)}{m}
\end{equation*}
Similarly 
\begin{equation*}
E_{\eta}\left[ \mathcal{R}_{a}^{\ddagger}\left\vert
\varphi^{\prime\prime}_{a}\left( \langle \theta_{a}^{\dagger},
\phi(Z)\rangle\right) \langle \theta-\theta^{\ast}_{a},
\phi(Z)\rangle^{2}\right\vert\right] \leq K^{1/2} B_{2}(\vert
\varphi_{a}^{\prime\prime}(0)\vert, L_1 +K, 4, k, K)8 K^{2} L_{1}^{2}\frac{%
s_{a}\log(p)}{m}.
\end{equation*}
Hence, there exists a constant $L_4$, depending only on $k,K$ and $L_1$,
such that the expectation of the absolute value of %
\eqref{eq:mod_double_taylor} is bounded by 
\begin{equation*}
L_{4}\frac{s_{a}\log(p)}{\sqrt{m}}.
\end{equation*}
This together with \eqref{bound_nemir_term} implies that there exists $L_2>0$
depending only on $L_1$, $k$ and $K$ such that 
\begin{equation*}
\sqrt{m}E_{\eta}\left\lbrace \left\vert \mathbb{P}_{m}\left[
S_{ab}b^{1}\left( \varphi_{a}\left(\langle \theta, \phi(Z)\rangle \right)-
\varphi_{a}\left(\langle \theta^{\ast}_{a}, \phi(Z)\rangle \right) \right)
+m_{a}(O,\varphi_{a}\left(\langle \theta, \phi\rangle \right)-
\varphi_{a}\left(\langle \theta^{\ast}_{a}, \phi\rangle \right) ) \right]
\right\vert \right\rbrace
\end{equation*}
is bounded by 
\begin{equation*}
L_{2}\left( \frac{s_{a}\log(p)}{\sqrt{m}} \right),
\end{equation*}
which finishes the proof of the first part of the theorem.

Note that 
\begin{equation*}
E_{\eta}\left[\Upsilon(a, b^{1}) - \chi(\eta) \right]=0
\end{equation*}
by \eqref{bias1}. The second and third part are then
proven using arguments similar to those used in the proofs of Theorems \ref%
{theo:rate_double_lin} and \ref{theo:rate_double_nonlin}.
\end{proof}

\section{Appendix C: Auxiliary technical results}

\label{sec:app_tech}

This appendix contains the proofs of various technical results that are
needed in the proofs appearing in the previous sections.

\begin{lemma}
\label{lemma:mgf_subgauss} Let $U$ be a random variable satisfying $\Vert U
\Vert_{\psi_2} <\infty$ and let $t>0$. Then 
\begin{equation*}
E\left(\exp(t \vert U \vert) \right)\leq B_{0}(\Vert U \Vert_{\psi_2}, t),
\end{equation*}
where $B_{0}(r_1,r_2)$ is a function defined on $\mathbb{R}_{+}^{2}$ that is
increasing in $r_1$ and in $r_2$.
\end{lemma}

\begin{proof}
\lbrack Proof of Lemma \ref{lemma:mgf_subgauss}] 
\begin{align*}
E\left(\exp(t \vert U \vert) \right)= \sum\limits_{k=0}^{\infty}\frac{%
E(\vert tU \vert^{k})}{k!}\leq \sum\limits_{k=0}^{\infty}\frac{\Vert
tU\Vert_{\psi_2}^{k} k^{k/2}}{k!}\leq \sum\limits_{k=0}^{\infty}\frac{\Vert
tU\Vert_{\psi_2}^{k} e^{k} k^{k/2}}{k^{k}}=\sum\limits_{k=0}^{\infty}\left(%
\frac{t\Vert U\Vert_{\psi_2}e}{k^{1/2}} \right)^{k}<\infty,
\end{align*}
where in the second inequality we used the definition of $\Vert U
\Vert_{\psi_2}$ and in the third inequality we used the bound $k!\geq
k^{k}/e^k$. Defining 
\begin{equation*}
B_{0}(r_1, r_2)=\sum\limits_{k=0}^{\infty}\left(\frac{r_1 r_2 e}{k^{1/2}}
\right)^{k},
\end{equation*}
the result follows.
\end{proof}

\begin{lemma}
\label{lemma:higher_isotropy} Assume $W$ is a random vector that satisfies 
\begin{equation*}
k \leq \lambda_{min}\left(E\left\lbrace W W^{\top} \right\rbrace \right)
\end{equation*}
and 
\begin{equation*}
\sup_{\Vert \Delta\Vert_{2}=1} \Vert \langle \Delta,W \rangle
\Vert_{\psi_{2}}\leq K.
\end{equation*}
Then for all $\Delta\in\mathbb{R}^{p}$ and $l\in\mathbb{N}$ 
\begin{equation*}
E^{1/l}\left( \vert \langle W,\Delta \rangle \vert^{l} \right) \leq \frac{K}{%
k} \sqrt{l} E^{1/2}\left(\langle W,\Delta \rangle^{2} \right).
\end{equation*}
\end{lemma}

\begin{proof}
\lbrack Proof of Lemma \ref{lemma:higher_isotropy}] Take $\Delta $ with $%
\Vert \Delta \Vert _{2}=1$. 
\begin{equation*}
E^{1/l}\left( |\langle W,\Delta \rangle |^{l}\right) \leq \Vert \langle
W,\Delta \rangle \Vert _{\psi _{2}}\sqrt{l}\leq K\sqrt{l}\leq \frac{K}{k}%
\sqrt{l}E^{1/2}\left( \langle W,\Delta \rangle ^{2}\right) .
\end{equation*}%
Now for an arbitrary $\Delta $ the results follows from applying the
preceding inequality to $\Delta /\Vert \Delta \Vert _{2}$.
\end{proof}

\begin{lemma}
\label{lemma:lipschitz} Assume $W$ is a random vector that satisfies 
\begin{equation*}
k \leq \lambda_{min}\left(E\left\lbrace W W^{\top} \right\rbrace \right)
\end{equation*}
and 
\begin{equation*}
\sup_{\Vert \Delta\Vert_{2}=1} \Vert \langle \Delta,W \rangle
\Vert_{\psi_{2}}\leq K.
\end{equation*}
Let $f:\mathbb{R}\to\mathbb{R}$ be a function satisfying 
\begin{equation*}
\vert f(u) -f(v)\vert \leq K\exp\left\lbrace K \left( \vert u\vert+\vert v
\vert \right) \right\rbrace \vert u -v\vert \quad \text{for all } u,v \in%
\mathbb{R}.
\end{equation*}
Then for all $l\in\mathbb{N}$, $\theta,\Delta \in\mathbb{R}^{p}$ we have 
\begin{equation*}
E^{1/l} \left\lbrace \left\vert f(\langle \theta+\Delta, W\rangle)-f(\langle
\theta, W\rangle)\right\vert^{l} \right\rbrace\leq B_{1}(\Vert
\theta\Vert_{2}, \Vert \Delta \Vert_{2},l, k, K) E^{1/2}\left(\langle
W,\Delta \rangle^{2} \right),
\end{equation*}
where $B_{1}(r_1,r_2,r_3,r_4,r_5) $ is a function defined in $\mathbb{R}%
_{+}^{5}$ that is increasing in $r_1$ and $r_2$.
\end{lemma}

\begin{proof}
\lbrack Proof of Lemma \ref{lemma:lipschitz}] Using the assumption on $f$
and Cauchy-Schwartz we get 
\begin{align*}
&E^{1/l} \left\lbrace \left\vert f(\langle \theta+\Delta,
W\rangle)-f(\langle \theta, W\rangle)\right\vert^{l} \right\rbrace \leq
KE^{1/l}\left\lbrace \exp\left\lbrace l K \left( \vert\langle \theta+\Delta,
W\rangle \vert+\vert \langle \theta, W\rangle \vert \right) \right\rbrace
\vert \langle \Delta, W\rangle\vert^{l} \right\rbrace \\
& \leq K E^{1/(2l)}\left\lbrace \exp\left\lbrace 2l K \left( \vert\langle
\theta+\Delta, W\rangle \vert+\vert \langle \theta, W\rangle \vert \right)
\right\rbrace \right\rbrace E^{1/(2l)} \left(\langle \Delta, W\rangle^{2l}
\right).
\end{align*}
Using Lemma \ref{lemma:higher_isotropy} we get 
\begin{equation*}
E^{1/(2l)} \left(\langle \Delta, W\rangle^{2l} \right) \leq \frac{K}{k} 
\sqrt{2l} E^{1/2}\left(\langle W,\Delta \rangle^{2} \right).
\end{equation*}
By the triangle inequality 
\begin{equation*}
\Vert \vert\langle \theta+\Delta, W\rangle \vert+\vert \langle \theta,
W\rangle \vert \Vert_{\psi_2} \leq K \left(\Vert \theta + \Delta \Vert_{2}
+\Vert \theta\Vert_{2} \right) \leq K \left( \Vert\Delta \Vert_{2} + 2 \Vert
\theta\Vert_{2} \right).
\end{equation*}
Hence by Lemma \ref{lemma:mgf_subgauss} 
\begin{equation*}
E^{1/(2l)}\left\lbrace \exp\left\lbrace 2l K \left( \vert\langle
\theta+\Delta, W\rangle \vert+\vert \langle \theta, W\rangle \vert \right)
\right\rbrace \right\rbrace \leq B_{0}^{1/(2l)}( K \left( \Vert\Delta
\Vert_{2} + 2 \Vert \theta\Vert_{2} \right), 2 l K).
\end{equation*}
We have shown that 
\begin{equation*}
E^{1/l} \left\lbrace \left\vert f(\langle \theta+\Delta, W\rangle)-f(\langle
\theta, W\rangle)\right\vert^{l} \right\rbrace \leq K B_{0}^{1/(2l)}( K
\left( \Vert\Delta \Vert_{2} + 2 \Vert \theta\Vert_{2} \right), 2 l K)\frac{K%
}{k} \sqrt{2l} E^{1/2}\left(\langle W,\Delta \rangle^{2} \right).
\end{equation*}
Defining 
\begin{equation*}
B_{1}(r_{1}, r_{2}, r_{3}, r_{4}, r_{5})=r_{5} B_{0}^{1/(2r_{3})}(r_{5}
\left( r_{2} + 2 r_{1}\right), 2 r_{3} r_{5})\frac{r_{5}}{r_{4}} \sqrt{2r_{3}%
},
\end{equation*}
the result follows.
\end{proof}

\begin{lemma}
\label{lemma:bound_moment_lipschitz} Assume $W$ is a random vector that
satisfies 
\begin{equation*}
k \leq \lambda_{min}\left(E\left\lbrace W W^{\top} \right\rbrace \right)
\end{equation*}
and 
\begin{equation*}
\sup_{\Vert \Delta\Vert_{2}=1} \Vert \langle \Delta,W \rangle
\Vert_{\psi_{2}}\leq K.
\end{equation*}
Let $f:\mathbb{R}\to\mathbb{R}$ be a function satisfying 
\begin{equation*}
\vert f(u) -f(v)\vert \leq K\exp\left\lbrace K \left( \vert u\vert+\vert v
\vert \right) \right\rbrace \vert u -v\vert \quad \text{for all } u,v \in%
\mathbb{R}.
\end{equation*}
Then for all $l\in\mathbb{N}$, $\theta \in\mathbb{R}^{p}$ 
\begin{equation*}
E^{1/l}\left\lbrace \vert f(\langle \theta, W \rangle)
\vert^{l}\right\rbrace\leq B_{2}(\vert f(0)\vert, \Vert \theta \Vert_{2},l,
k, K),
\end{equation*}
where $B_{2}(r_1,r_2,r_3,r_4,r_5)$ is a function that is increasing in $r_2$.
\end{lemma}

\begin{proof}
\lbrack Proof of Lemma \ref{lemma:bound_moment_lipschitz}] Using the
triangle inequality and Lemma \ref{lemma:lipschitz} 
\begin{align*}
E^{1/l}\left\lbrace \vert f(\langle \theta, W \rangle)
\vert^{l}\right\rbrace &\leq \vert f(0)\vert+E^{1/l}\left\lbrace \vert
f(\langle \theta, W \rangle) -f(0)\vert^{l}\right\rbrace \\
& \leq \vert f(0)\vert + B_{1}(0, \Vert \theta \Vert_{2},l, k, K)
E^{1/2}\left(\langle W,\theta \rangle^{2} \right).
\end{align*}
Moreover, 
\begin{equation*}
\vert f(0)\vert + B_{1}(0, \Vert \theta \Vert_{2},l, k, K)
E^{1/2}\left(\langle W,\theta \rangle^{2} \right) \leq f(0)\vert + B_{1}(0,
\Vert \theta \Vert_{2},l, k, K)\sqrt{2}K\Vert \theta \Vert_{2}.
\end{equation*}
The result now follows easily.
\end{proof}

\begin{lemma}
\label{lemma:squaredsubgaus} Let $\Delta\in\mathbb{R}^{p}$ be a fixed vector
and let $W$ be a random vector in $\mathbb{R}^{p}$ satisfying $\Vert \langle
W, \Delta \rangle \Vert_{\psi_2}\leq C_{0}$. Then 
\begin{equation*}
\Vert \langle W, \Delta \rangle^{2} - E\left( \langle W, \Delta
\rangle^{2}\right)\Vert_{\psi_1} \leq 4 C_{0}^{2}.
\end{equation*}
\end{lemma}

\begin{proof}
\lbrack Proof of Lemma \ref{lemma:squaredsubgaus}] Using the definition of $%
\Vert \cdot\Vert_{\psi_2}$ 
\begin{equation*}
\frac{\left(E\left( \langle W, \Delta \rangle^{2}\right)\right)^{1/2}}{\sqrt{%
2}}\leq \Vert \langle W, \Delta \rangle\Vert_{\psi_2}\leq C_{0}
\end{equation*}
and hence 
\begin{equation*}
E\left( \langle W, \Delta \rangle^{2}\right) \leq 2 C_{0}^{2}.
\end{equation*}
On the other hand for all $l\in\mathbb{N}$ 
\begin{equation*}
\frac{\left(E\left(\vert\langle W,
\Delta\rangle\vert^{2l}\right)\right)^{1/l}}{l}=2 \left\lbrace \frac{%
\left(E\left(\vert\langle W, \Delta\rangle\vert^{2l}\right)\right)^{1/2l}}{%
\sqrt{2l}} \right\rbrace^{2}\leq 2 \Vert \langle W,
\Delta\rangle\Vert_{\psi_2}^{2}\leq 2 C_{0}^{2}.
\end{equation*}
Hence 
\begin{equation*}
\Vert \langle W, \Delta \rangle^{2} \Vert_{\psi_1} =\sup\limits_{l\in\mathbb{%
N}} \frac{\left(E\left(\vert\langle W,
\Delta\rangle\vert^{2l}\right)\right)^{1/l}}{l} \leq 2C_{0}^{2}.
\end{equation*}
The result now follows from 
\begin{equation*}
\Vert \langle W, \Delta \rangle^{2} - E\left( \langle W, \Delta
\rangle^{2}\right)\Vert_{\psi_1} \leq \Vert \langle W, \Delta \rangle^{2}
\Vert_{\psi_1} + \Vert E\left( \langle W, \Delta
\rangle^{2}\right)\Vert_{\psi_1}=\Vert \langle W, \Delta \rangle^{2}
\Vert_{\psi_1} + E\left( \langle W, \Delta \rangle^{2}\right).
\end{equation*}
\end{proof}

The following lemma is a straightforward application of Proposition 5.16
from \cite{compressed-book}.

\begin{lemma}
\label{lemma:concent_cov_fixed} Let $\Delta\in\mathbb{R}^{p}$ be a fixed
vector and let $W_{1},\dots,W_{n}$ be i.i.d. random vectors in $\mathbb{R}%
^{p}$ such that $\Vert \langle W_{1}, \Delta \rangle \Vert_{\psi_2}\leq
C_{0} $. Then there exists a universal constant $C>0$ such that 
\begin{equation*}
P\left( \vert \mathbb{P}_{n}\left( \langle W, \Delta \rangle^{2}\right) -
E\left( \langle W, \Delta \rangle^{2}\right) \vert\geq t \right) \leq
2\exp\left\lbrace -C\min\left(\frac{nt^{2}}{16 C_{0}^{4}}, \frac{nt}{4
C_{0}^{2}}\right) \right\rbrace.
\end{equation*}
\end{lemma}

\begin{proof}
\lbrack Proof of Lemma \ref{lemma:concent_cov_fixed}] By Lemma \ref%
{lemma:squaredsubgaus}, the random variables 
\begin{equation*}
\langle W_{i}, \Delta \rangle^{2} - E\left( \langle W_{i}, \Delta
\rangle^{2}\right)
\end{equation*}
have sub-Exponential norm bounded by $4C_{0}^{2}$. Moreover, they have zero
mean. The result now follows from Proposition 5.16 from \cite%
{compressed-book}. 
\end{proof}

The following lemma is a straightforward adaptation of Lemma 15 in \cite%
{Loh-smooth}, the main difference being that we do not require the random
variables involved to have zero mean.

\begin{lemma}
\label{lemma:concent_cov_sparse} Let $W_{1},\dots,W_{n}$ be i.i.d. random
vectors in $\mathbb{R}^{p}$ such that $\Vert \langle W_{1}, \Delta \rangle
\Vert_{\psi_2}\leq C_{0}$ for all $\Delta$ with $\Vert \Delta \Vert_{2}=1$.
Then there exists a fixed universal constant $C>0$ such that for all $l\geq
1 $ 
\begin{equation*}
P\left( \sup\limits_{\Delta \in \mathbb{K}(2l)}\vert \mathbb{P}_{n}\left(
\langle W, \Delta \rangle^{2}\right) - E\left( \langle W, \Delta
\rangle^{2}\right) \vert\geq t \right) \leq 2\exp\left\lbrace -C\min\left(%
\frac{nt^{2}}{16 C_{0}^{4}}, \frac{nt}{4 C_{0}^{2}}\right) + 2 l
\log(p)\right\rbrace.
\end{equation*}
\end{lemma}

\begin{proof}
\lbrack Proof of Lemma \ref{lemma:concent_cov_sparse}] The proof is
identical to that of Lemma 15 from \cite{Loh-smooth}, replacing the use of
their Lemma 14, with our Lemma \ref{lemma:concent_cov_fixed}.
\end{proof}

Finally, Condition R below is needed to prove Proposition \ref{prop:DR
functional}. See \cite{globalclass} for the proof.

\begin{condition}[Condition R]
\label{cond:R1}There exists a dense set $H_{a}$ of $L_{2}\left( P_{\eta
,Z}\right) $ such that  $H_{a}\cap 
\mathcal{A}\not=\varnothing$, and for each $\eta $ and each $h\in H_{a},$ there
exists $\varepsilon \left( \eta ,h\right) >0$ such that $a+th$ $\in \mathcal{%
A}$ if $\left\vert t\right\vert <\varepsilon \left( \eta ,h\right) $ where $%
a\left( Z\right) \equiv a\left( Z;\eta \right) .$  The same holds replacing $a$ with $b$ and $%
\mathcal{A}$ with $\mathcal{B}.$ Furthermore $E_{\eta }\left\{ \left\vert
S_{ab}b\left( Z\right) h\left( Z\right) \right\vert \right\} <\infty $ for $%
h\in H_{a}$ and $E_{\eta }\left\{ \left\vert S_{ab}a\left( Z\right) h\left(
Z\right) \right\vert \right\} <\infty $ for $h\in H_{b}.$  Moreover for all $\eta$, $E_{\eta}\left[ \left\vert S_{ab}a^{\prime} b^{\prime} \right\vert \right] < \infty$ for all $a^{\prime}\in \mathcal{A}$ and $b^{\prime}\in \mathcal{B}$.
\end{condition}

\section{Appendix D: Simulation studies}

\label{sec:sim}

In this Appendix we report the results of four simulation experiments
conducted to study the performance of our estimators of $\chi \left( \eta
\right) =E_{\eta }\left[ a\left( Z\right) b\left( Z\right) \right] $ where
the data are $n=1000$ i.i.d. copies of $O=\left( Y,D,Z\right) $ and $a\left(
Z\right) =E_{\eta }\left( Y|Z\right) $ and $b\left( L\right) =E_{\eta
}\left( D|Z\right) .$ All simulations were carried out in \texttt{R}, using
the package \texttt{glmnet} and using the tuning parameter that minimizes
the ten-fold cross-validated mean square prediction error.

All experiments consisted of 500 replicates. In each replicate we generated $%
n=1000$ iid copies of $O.$ $Z$ was a $p\times 1$ multivariate normal vector $%
Z$ with 0 mean and a Toeplitz covariance matrix $\Sigma
_{ij}=0.1^{\left\vert i-j\right\vert }$ where $p=200.$ Also, in all
experiments, $E_{\eta }\left( Y|Z\right) =\varphi _{a}\left( \left\langle
\theta _{a},Z\right\rangle \right) $ and $E_{\eta }\left( D|Z\right)
=\varphi _{b}\left( \left\langle \theta _{b},Z\right\rangle \right) $ where $%
\theta _{a,j}=c_{a}j^{-\alpha _{a}}\left( -1\right) ^{j}$ and $\theta
_{b,j}=c_{b}j^{-\alpha _{b}}\left( -1\right) ^{j+1},$ $j=1,\dots ,p,$ with $%
\varphi _{a},\varphi _{b},c_{a},c_{b}$ depending on the experiment and, in
all experiments, $\left( \alpha _{a},\alpha _{b}\right) $ varying over the
set $\left\{ 0.8,1.5,3,5\right\} \times \left\{ 0.8,1.5,3,5\right\} $.
Notice that this is the setup in Example \ref{ex:parametric_weakly_sparse}.
When both working models are correctly specified, regardless of the links
used in the working models, our results establish that, under regularity
conditions, our estimators of $\chi \left( \eta \right) $ should be
asymptotically mean zero normal provided $\alpha _{a}^{-1}+\alpha
_{b}^{-1}<2.$ Among the values of $\left( \alpha _{a},\alpha _{b}\right) $
used$,$ only the choice $\left( \alpha _{a},\alpha _{b}\right) =\left(
0.8,0.8\right) $ fails this condition. The choices $\left( \alpha
_{a},\alpha _{b}\right) =\left( 1.5,0.8\right) $ and $\left( \alpha
_{a},\alpha _{b}\right) =\left( 0.8,1.5\right) $ are borderline. In
addition, we chose the constants $c_{a}$ and $c_{b}$ so as to set the signal
to noise ratio defined as $SNR_{\eta }\left( \cdot \right) =Var_{\eta }\left[
E_{\eta }\left( \cdot |Z\right) \right] /E_{\eta }\left[ Var_{\eta }\left(
\cdot |Z\right) \right] $ equal to 2. \ Besides the data generating process
the experiments differed in the estimators used. In experiments 1 and 2 we
computed estimators using algorithm \ref{Algo both ALS}, i.e. in which both
working models used links equal to the identity. In experiments 3 and 4 we
computed estimators using algorithm \ref{Algo both GALS}, in which both
working models used links equal to the exponential function. Using these
estimators and the estimators of their asymptotic variance as given in
Theorems \ref{theo:rate_double_lin} and \ref{theo:rate_double_nonlin}, we
computed Normal theory Wald based confidence intervals.

In addition, for comparison, we also computed three estimators that did not
use sample splitting or cross-fitting, called "naive\_a", "naive\_b" and
naive\_ab" in the pictures displaying the results. The estimator "naive\_a"
is of the form $\mathbb{P}_{n}\left[ D\widehat{a}_{N}\left( Z\right) \right]
,$ the estimator "naive\_b" is of the form $\mathbb{P}_{n}\left[ Y\widehat{b}%
_{N}\left( Z\right) \right] $ and the estimator "naive\_ab" is of the form $%
\mathbb{P}_{n}\left[ \widehat{a}_{N}\left( Z\right) \widehat{b}_{N}\left(
Z\right) \right] ,$ where $\widehat{a}_{N}$ and $\widehat{b}_{N}\,\ $are
specific estimators of $a$ and $b$ which vary with the experiment.
Specifically, $\widehat{c}_{N}\left( Z\right) =\varphi _{c}\left\langle 
\widehat{\theta }_{c,N},Z\right\rangle ,\,\ $for $c=a$ and $b,$ where $%
\widehat{\theta }_{c,N}=\arg \min_{\theta }\mathbb{P}_{n}\left[ Q_{c}\left(
\theta ,\phi =id,w=1\right) \right] +\lambda \left\Vert \theta \right\Vert
_{1}$ with $Q_{c}$ defined in equation \eqref{eq:lossDef} and $\lambda $ is
chosen by ten-fold cross-validation. In experiments 1 and 2, $\varphi
_{c}\left( u\right) =u$ and the link function used in $Q_{c}$ is also the
identity. In experiments 3 and 4, $\varphi _{c}\left( u\right) =\exp \left(
u\right) $ and the link function used in $Q_{c}$ is also the exponential
function.

For the naive estimators we computed the following ad-hoc estimators of
their variance. We estimated the variance of the "naive\_a" estimator with
similarly and 
\begin{align*}
& \widehat{Var}_{naive\_a}=n^{-1}\left[ \mathbb{P}_{n}\left[ \left\{ D%
\widehat{a}_{N}\left( Z\right) \right\} ^{2}\right] -\left\{ \mathbb{P}_{n}%
\left[ D\widehat{a}_{N}\left( Z\right) \right] \right\} ^{2}\right], \\
& \widehat{Var}_{naive\_b}=n^{-1}\left[ \mathbb{P}_{n}\left[ \left\{ Y%
\widehat{b}_{N}\left( Z\right) \right\} ^{2}\right] -\left\{ \mathbb{P}_{n}%
\left[ Y\widehat{b}_{N}\left( Z\right) \right] \right\} ^{2}\right], \\
& \widehat{Var}_{naive\_ab}=n^{-1}\left[ \mathbb{P}_{n}\left[ \left\{ 
\widehat{a}_{N}\left( Z\right) \widehat{b}_{N}\left( Z\right) \right\} ^{2}%
\right] -\left\{ \mathbb{P}_{n}\left[ \widehat{a}_{N}\left( Z\right) 
\widehat{b}_{N}\left( Z\right) \right] \right\} ^{2}\right].
\end{align*}
We used these variance estimators to compute Wald confidence intervals of
the form Estimator 
\begin{equation*}
\pm 1.96\sqrt{\text{Estimated Variance}}.
\end{equation*}
Note that these ad-hoc variance estimators and confidence intervals are not,
in principle, supported by theory since the naive estimators of $\chi \left(
\eta \right) $ may not even be $\sqrt{n}-$consistent.

For each experiment we provide three figures, the first figure reports the
absolute value of the estimated bias over the 500 replications, i.e. the
absolute value of the mean of the 500 estimators minus the true parameter
value, the second figure reports the standard deviation of the estimators of 
$\chi \left( \eta \right) $ over the 500 replications, the third figure
reports the estimated coverage probability of the nominal 95\% confidence
intervals.

The goal of the first experiment was to study the performance of the
estimator in Algorithm \ref{Algo both ALS}, i.e. when the links of both
working models are the identity, and under a scenario in which both models
are correctly specified. We used $c_{a}=c_{b}=1,$ and we generated 
\begin{align*}
& Y=\langle \theta _{a},Z\rangle +U_{a}, \\
& D=\langle \theta _{b},Z\rangle +U_{b},
\end{align*}%
where $(U_{a},U_{b})$ is a bivariate normal random vector, independent of $Z$%
, and with $Var(U_{a})=\theta _{a}^{T}\Sigma \theta _{a}/2$, $%
Var(U_{b})=\theta _{b}^{T}\Sigma \theta _{b}/2$ and $corr(U_{a},U_{b})=0.1$.

\begin{figure}[tbp]
\centering
\includegraphics[scale=0.7]{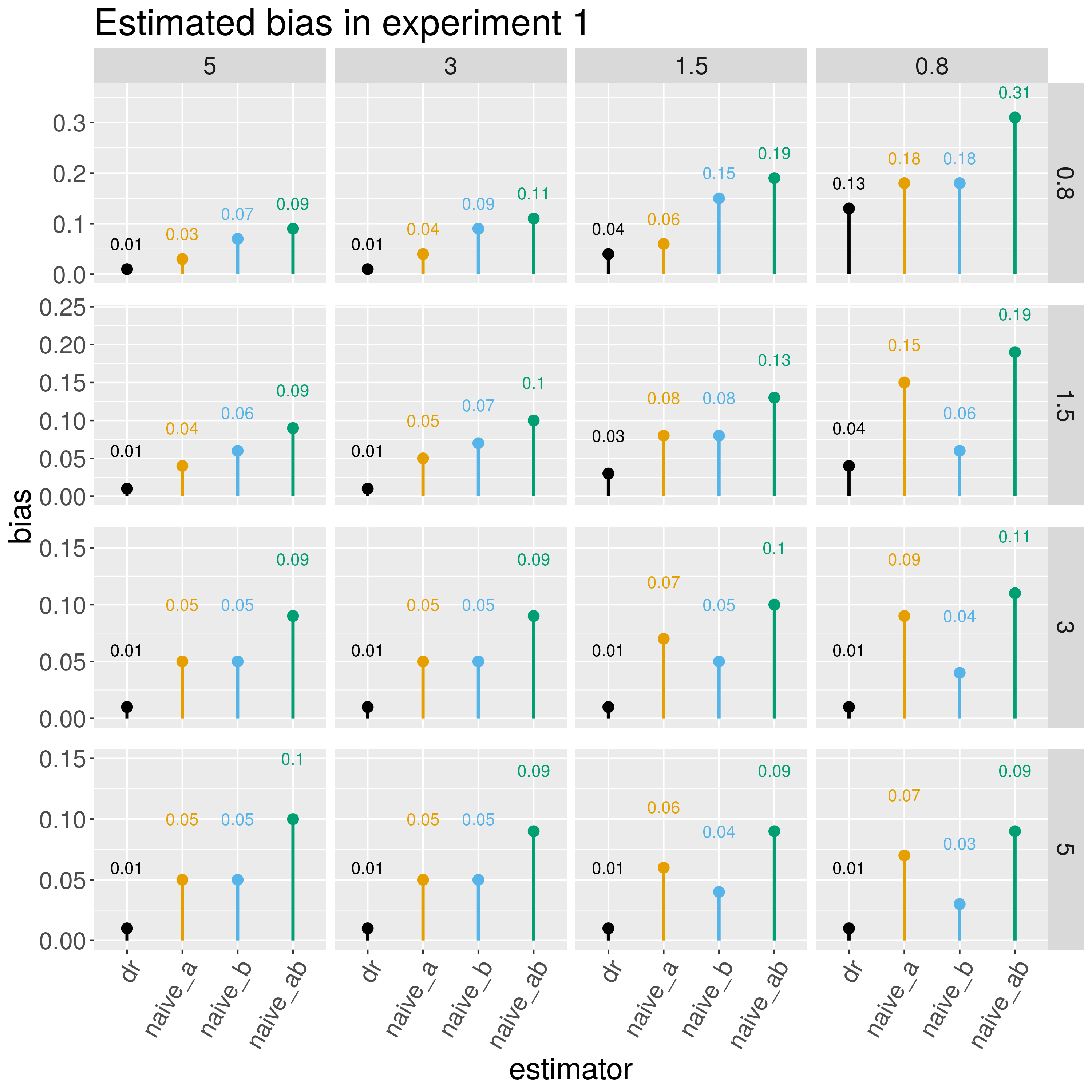}
\caption{Estimated bias in experiment 1. Rows correspond to different values
of $\protect\alpha_{a}$ and columns to different values of $\protect\alpha%
_{b}$.}
\label{fig:exp1_bias}
\end{figure}

\begin{figure}[tbp]
\centering
\includegraphics[scale=0.7]{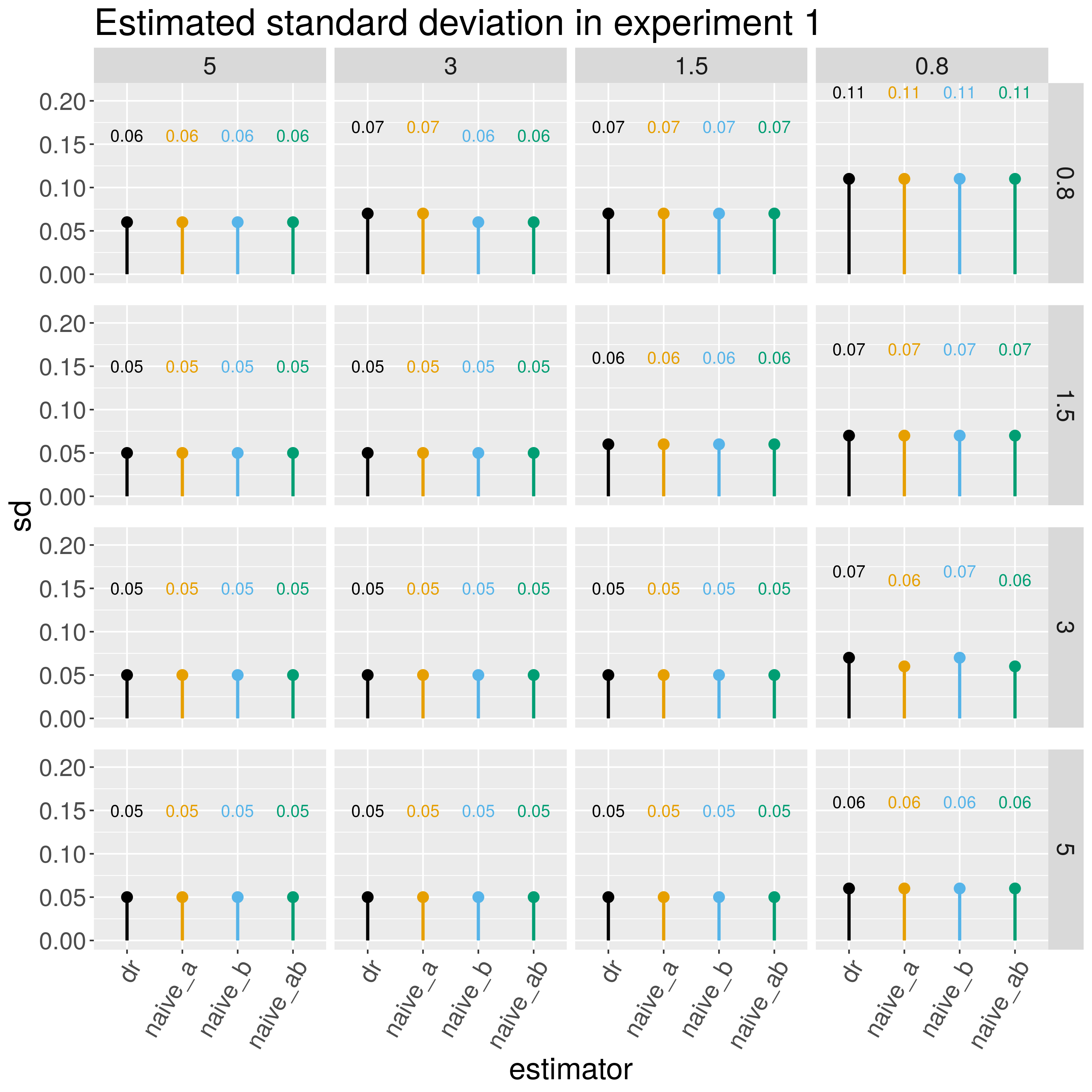}
\caption{Estimated standard errors in experiment 1. Rows correspond to
different values of $\protect\alpha_{a}$ and columns to different values of $%
\protect\alpha_{b}$.}
\label{fig:exp1_sd}
\end{figure}

\begin{figure}[tbp]
\centering
\includegraphics[scale=0.7]{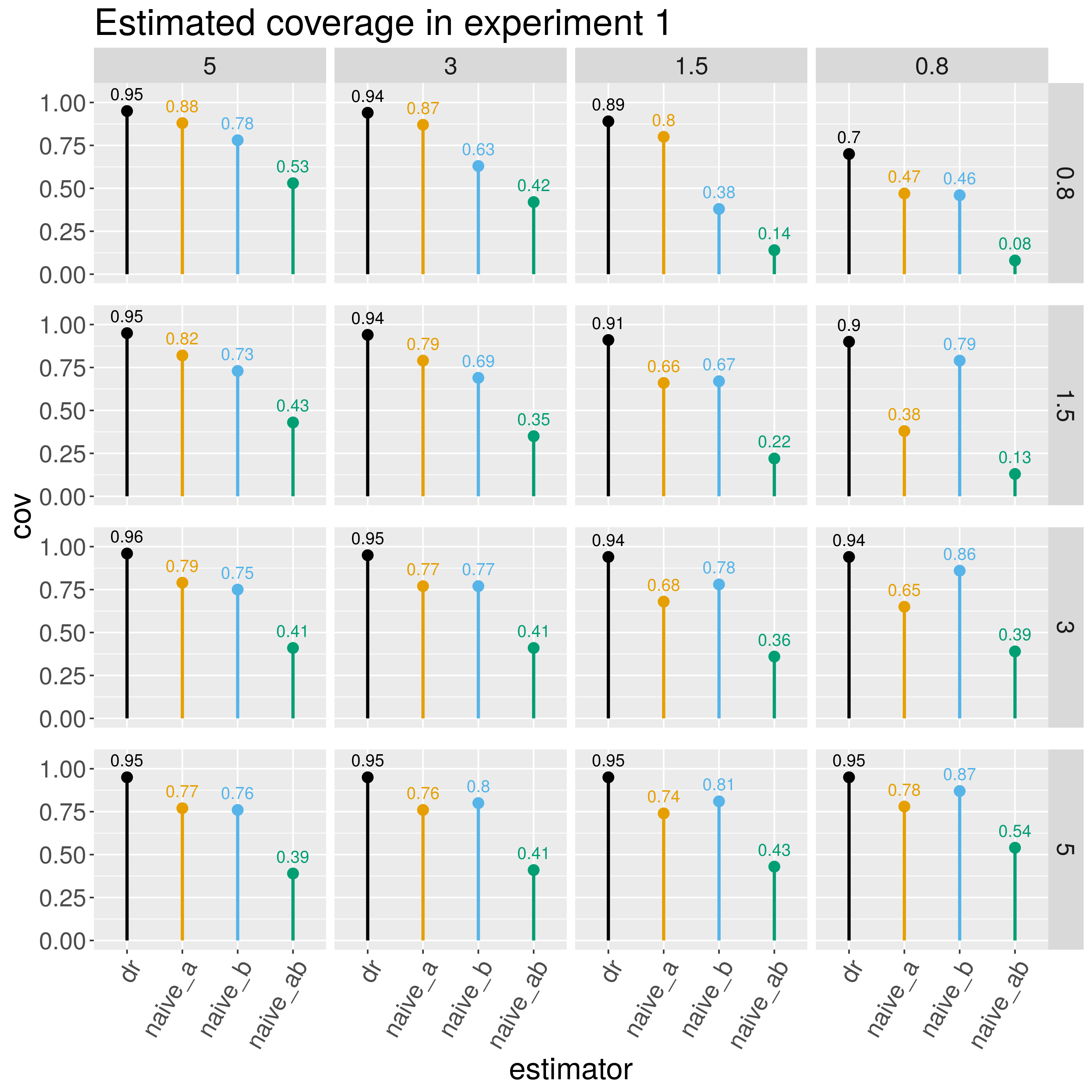}
\caption{Estimated coverage at the 95\% nominal level in experiment 1. Rows
correspond to different values of $\protect\alpha_{a}$ and columns to
different values of $\protect\alpha_{b}$.}
\label{fig:exp1_cov}
\end{figure}

\FloatBarrier

The goal of the second experiment was to study the performance of the
estimator in Algorithm \ref{Algo both ALS}, i.e. when the links of both
working models are the identity, and under a scenario in which the model for 
$a(Z)$ is correctly specified but the model for $b(Z)$ is incorrectly
specified. We used $c_{a}=1$, $c_{b}=(\log (3)\theta _{b}^{T}\Sigma \theta
_{b})^{-1/2}$ and we generated 
\begin{align*}
& Y=\langle \theta _{a},Z\rangle +U_{a}, \\
& D|Z\sim Pois(\nu \left( Z\right) )\text{ with }\nu \left( Z\right) =\exp
(c_{b}\langle \theta _{b},Z\rangle ),
\end{align*}%
where $U_{a}$ is normal with $Var(U_{a})=\theta _{a}^{T}\Sigma \theta _{a}/2$%
.

\begin{figure}[tbp]
\centering
\includegraphics[scale=0.7]{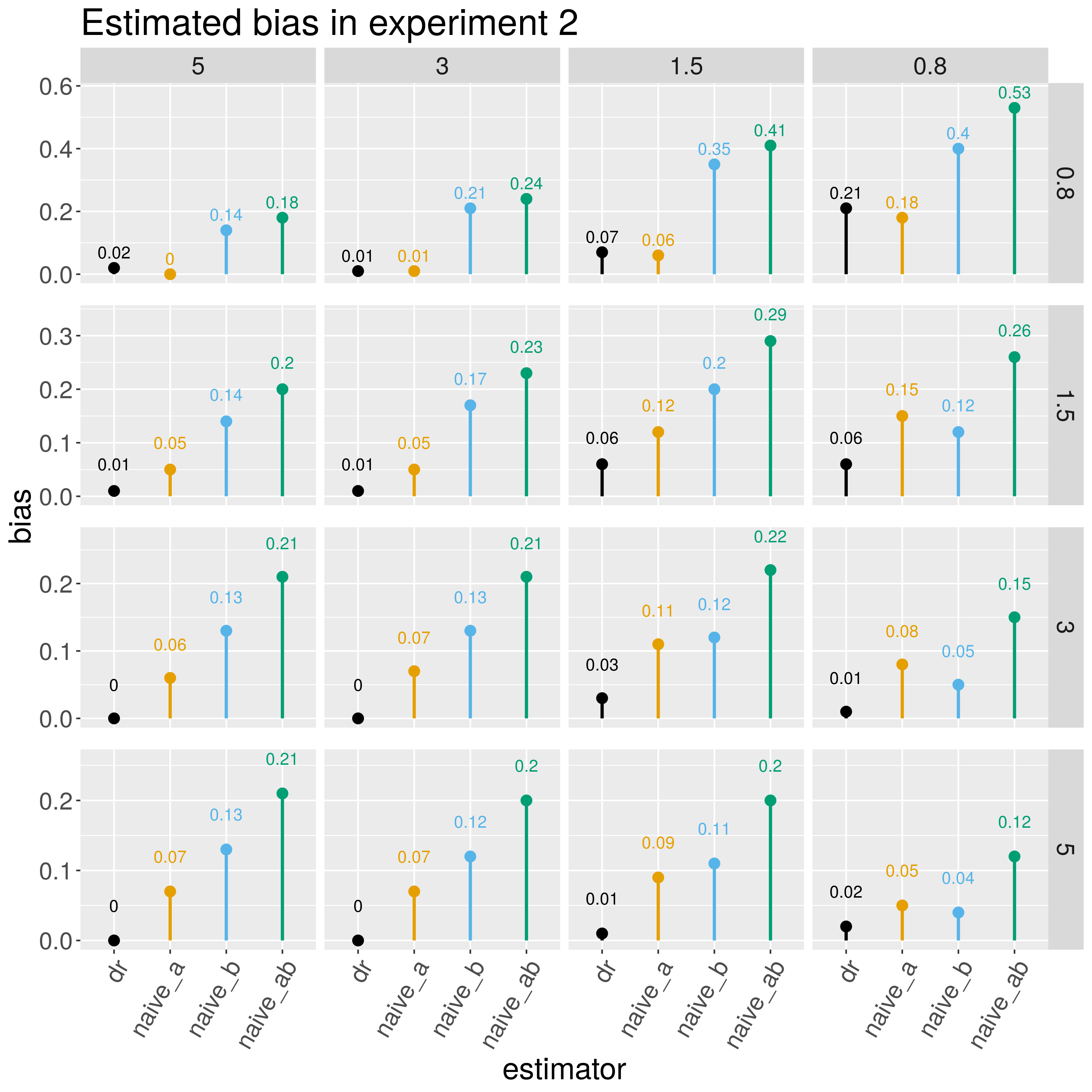}
\caption{Estimated bias in experiment 2. Rows correspond to different values
of $\protect\alpha_{a}$ and columns to different values of $\protect\alpha%
_{b}$.}
\label{fig:exp2_bias}
\end{figure}

\begin{figure}[tbp]
\centering
\includegraphics[scale=0.7]{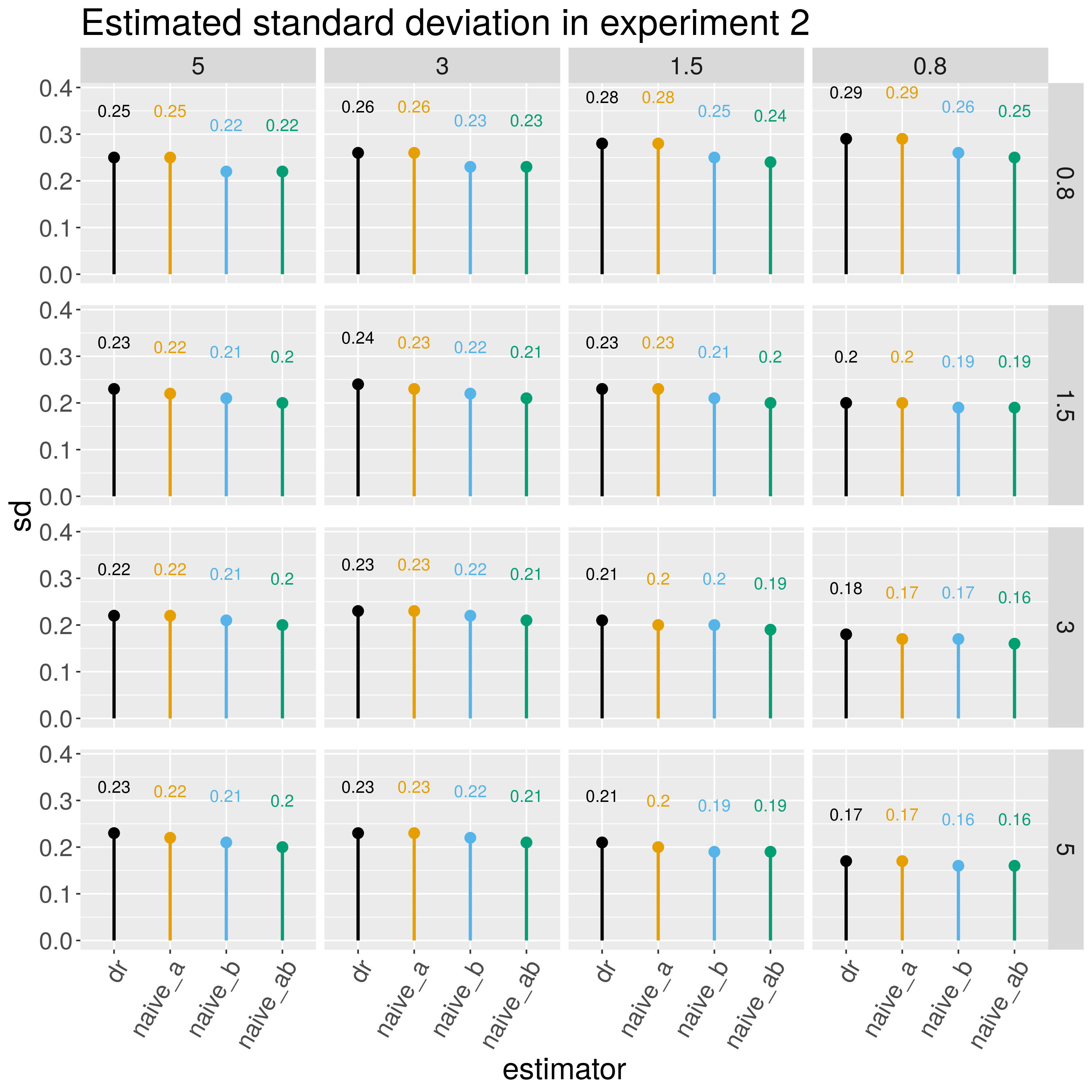}
\caption{Estimated standard errors in experiment 2. Rows correspond to
different values of $\protect\alpha_{a}$ and columns to different values of $%
\protect\alpha_{b}$.}
\label{fig:exp2_sd}
\end{figure}

\begin{figure}[tbp]
\centering
\includegraphics[scale=0.7]{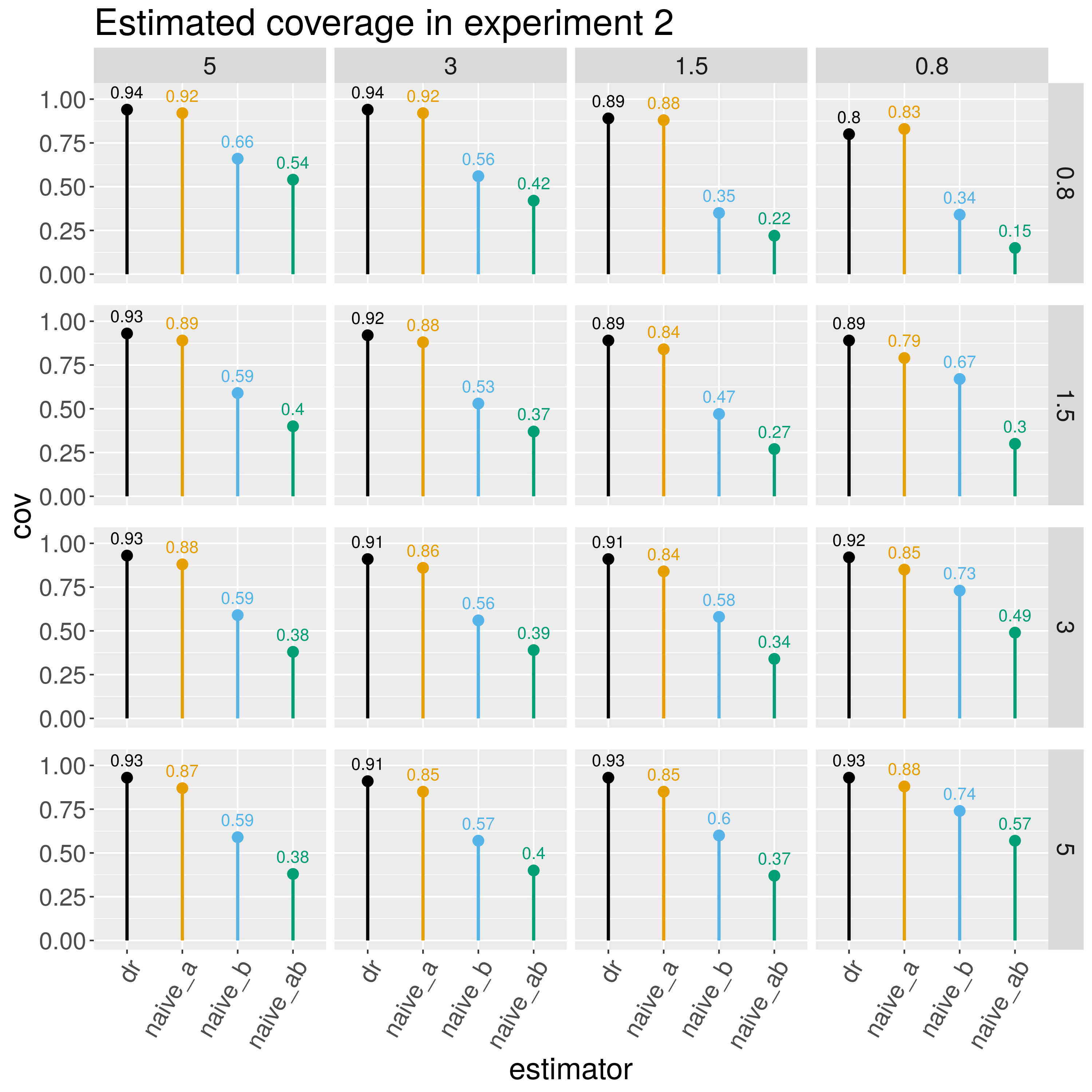}
\caption{Estimated coverage at the 95\% nominal level in experiment 2. Rows
correspond to different values of $\protect\alpha_{a}$ and columns to
different values of $\protect\alpha_{b}$.}
\label{fig:exp2_cov}
\end{figure}

\FloatBarrier

The goal of the third experiment was to study the performance of the
estimator in Algorithm \ref{Algo both GALS}, when the links of both working
models are the exponential function, and under a scenario in which both
models are correctly specified. We used $c_{a}=(\log (3)\theta
_{a}^{T}\Sigma \theta _{a})^{-1/2},$ $c_{b}=(\log (3)\theta _{b}^{T}\Sigma
\theta _{b})^{-1/2}$ and we generated $Y$ and $D$ independent given $Z,$
with 
\begin{align*}
& Y|Z\sim Pois(\nu _{a}\left( Z\right) )\text{ with }\nu _{a}\left( Z\right)
=\exp (c_{a}\langle \theta _{a},Z\rangle ) \\
& D|Z\sim Pois(\nu _{b}\left( Z\right) )\text{ with }\nu _{b}\left( Z\right)
=\exp (c_{b}\langle \theta _{b},Z\rangle ).
\end{align*}

\begin{figure}[tbp]
\centering
\includegraphics[scale=0.7]{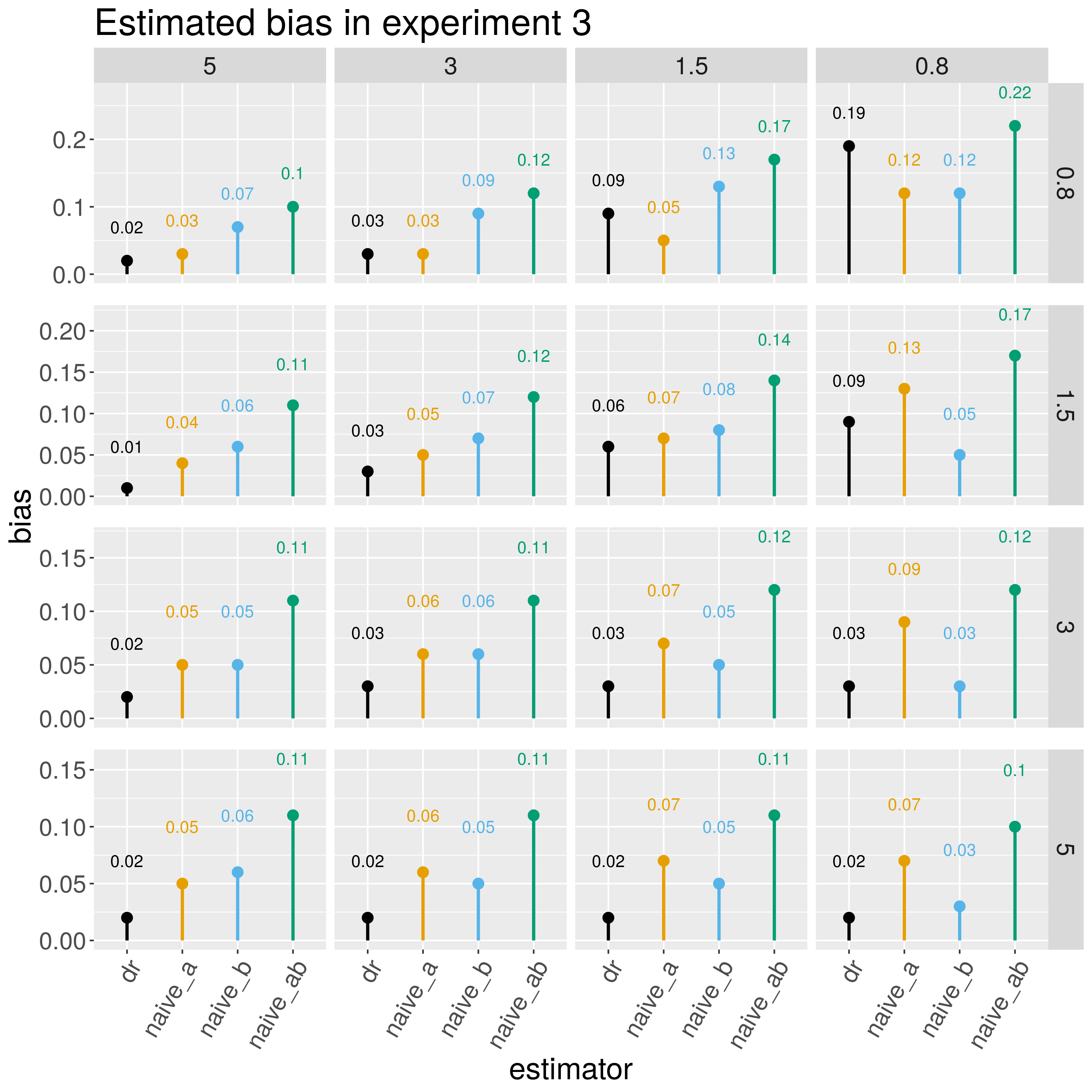}
\caption{Estimated bias in experiment 3. Rows correspond to different values
of $\protect\alpha_{a}$ and columns to different values of $\protect\alpha%
_{b}$.}
\label{fig:exp3_bias}
\end{figure}

\begin{figure}[tbp]
\centering
\includegraphics[scale=0.7]{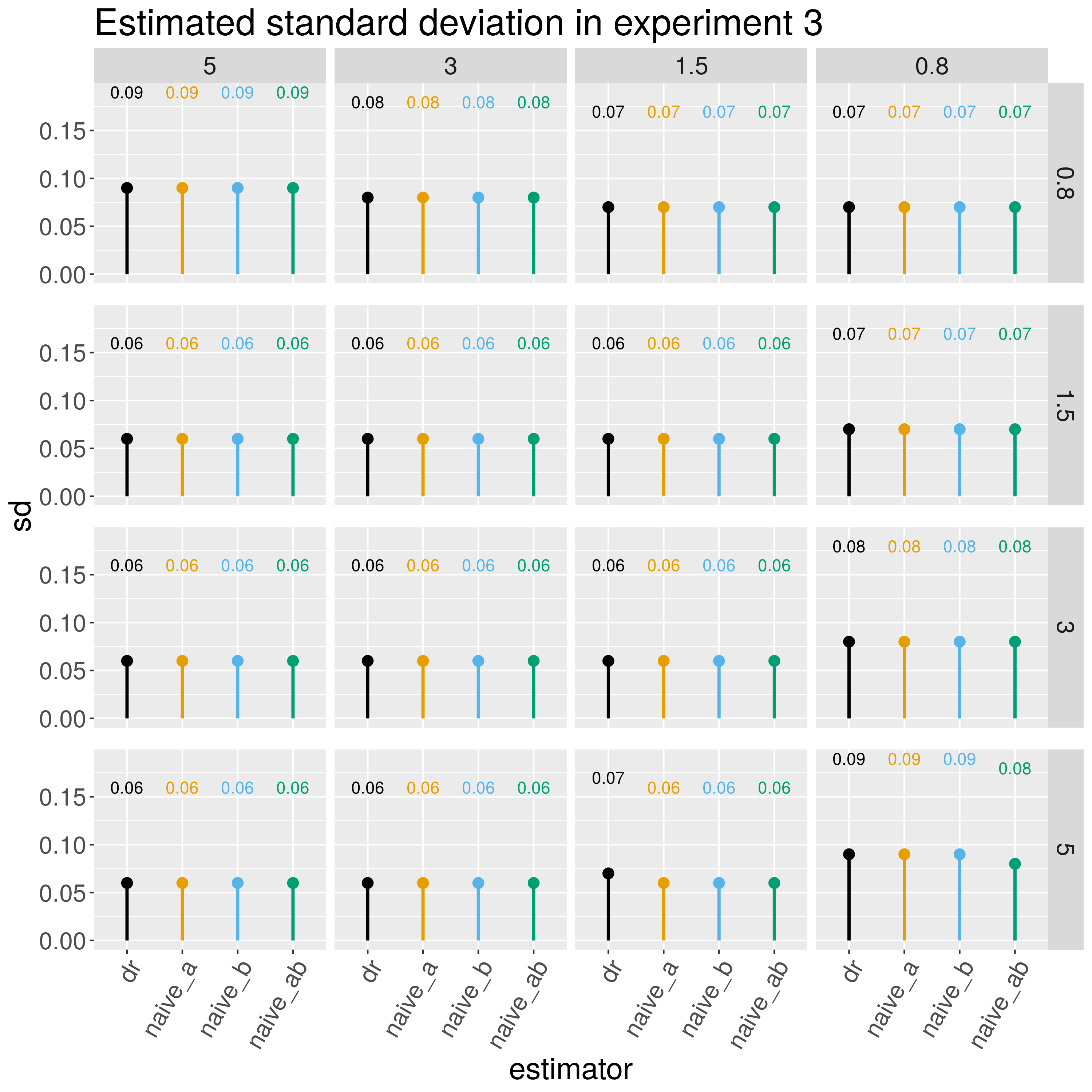}
\caption{Estimated standard errors in experiment 3. Rows correspond to
different values of $\protect\alpha_{a}$ and columns to different values of $%
\protect\alpha_{b}$.}
\label{fig:exp3_sd}
\end{figure}

\begin{figure}[tbp]
\centering
\includegraphics[scale=0.7]{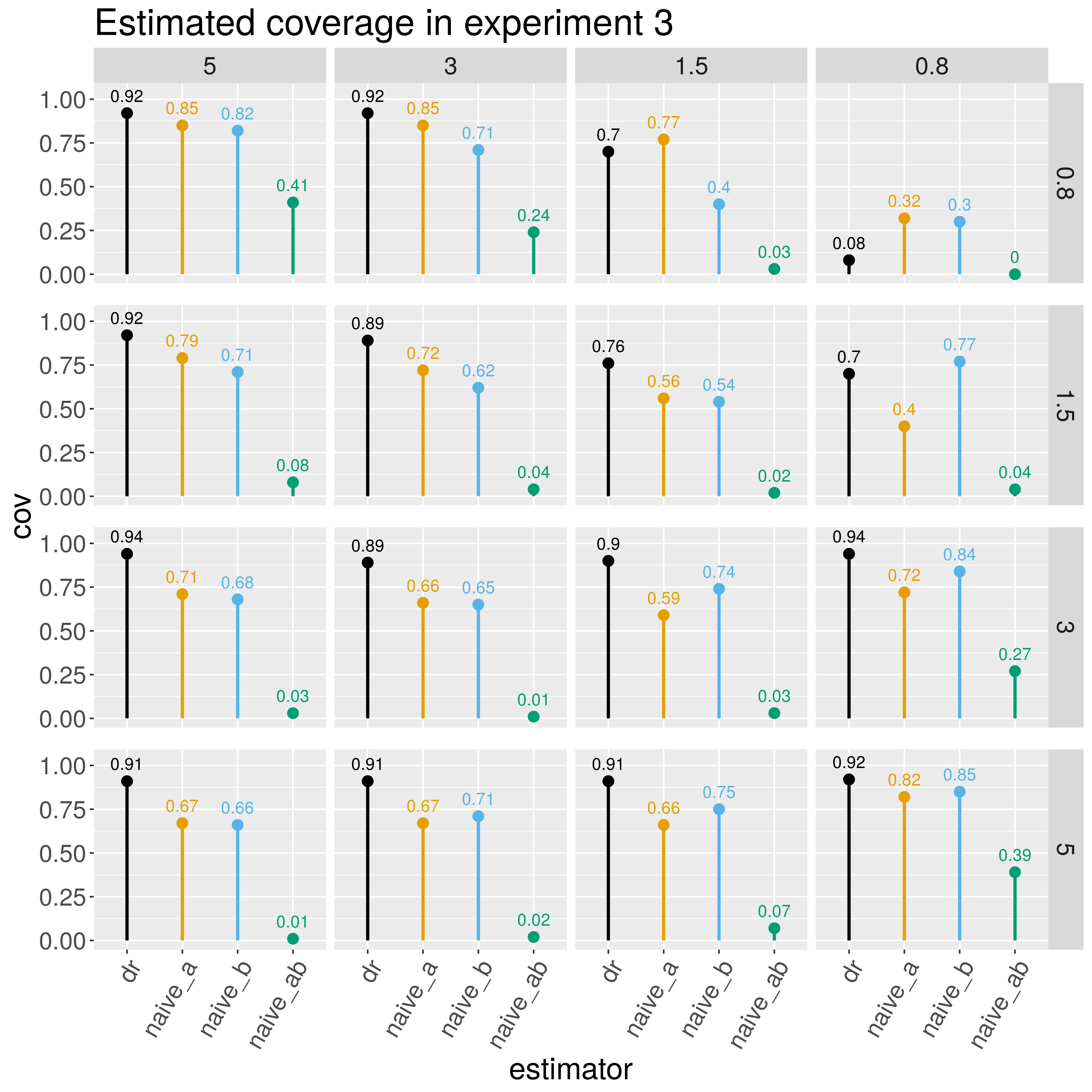}
\caption{Estimated coverage at the 95\% nominal level in experiment 3. Rows
correspond to different values of $\protect\alpha_{a}$ and columns to
different values of $\protect\alpha_{b}$.}
\label{fig:exp3_cov}
\end{figure}

\FloatBarrier

The goal of the fourth experiment was to study the performance of the
estimator in Algorithm \ref{Algo both GALS}, when the links of both working
models are the exponential function, and under a scenario in which the model
for $a(Z)$ is correctly specified but the model for $b(Z)$ is incorrectly
specified. We used $c_{a}=(\log (3)\theta _{a}^{T}\Sigma \theta
_{a})^{-1/2}, $ $c_{b}=(\theta _{b}^{T}\Sigma \theta _{b})^{-1/2}$ and we
generated $Y$ and $D$ independent given $Z,$ with 
\begin{align*}
& Y|Z\sim Pois(\nu _{a}\left( Z\right) )\text{ with }\nu _{a}\left( Z\right)
=\exp (c_{a}\langle \theta _{a},Z\rangle ) \\
& D|Z\sim Pois(\nu _{b}\left( Z\right) )\text{ with }\nu _{b}\left( Z\right)
=(c_{b}\langle \theta _{b},Z\rangle )^{2}.
\end{align*}

We chose this specific generative process for the law of $D|Z$ so as to
roughly satisfy the regularity conditions on the $\ell_{1}$ regularized
estimator of $b\left( Z\right) $ of step 2 under an incorrect model. This is
because results in \cite{glm-misp} imply that for large $n,$ this estimator
should be close to $\exp \langle \theta _{b}^{L},Z\rangle $ where $\theta
_{b}^{L}=Kc_{b}\theta _{b}$ for some constant $K.$ Thus, the coefficients of 
$\theta _{b}^{L}$ exhibit the same geometric decay as $\theta _{b}.$

\begin{figure}[tbp]
\centering
\includegraphics[scale=0.7]{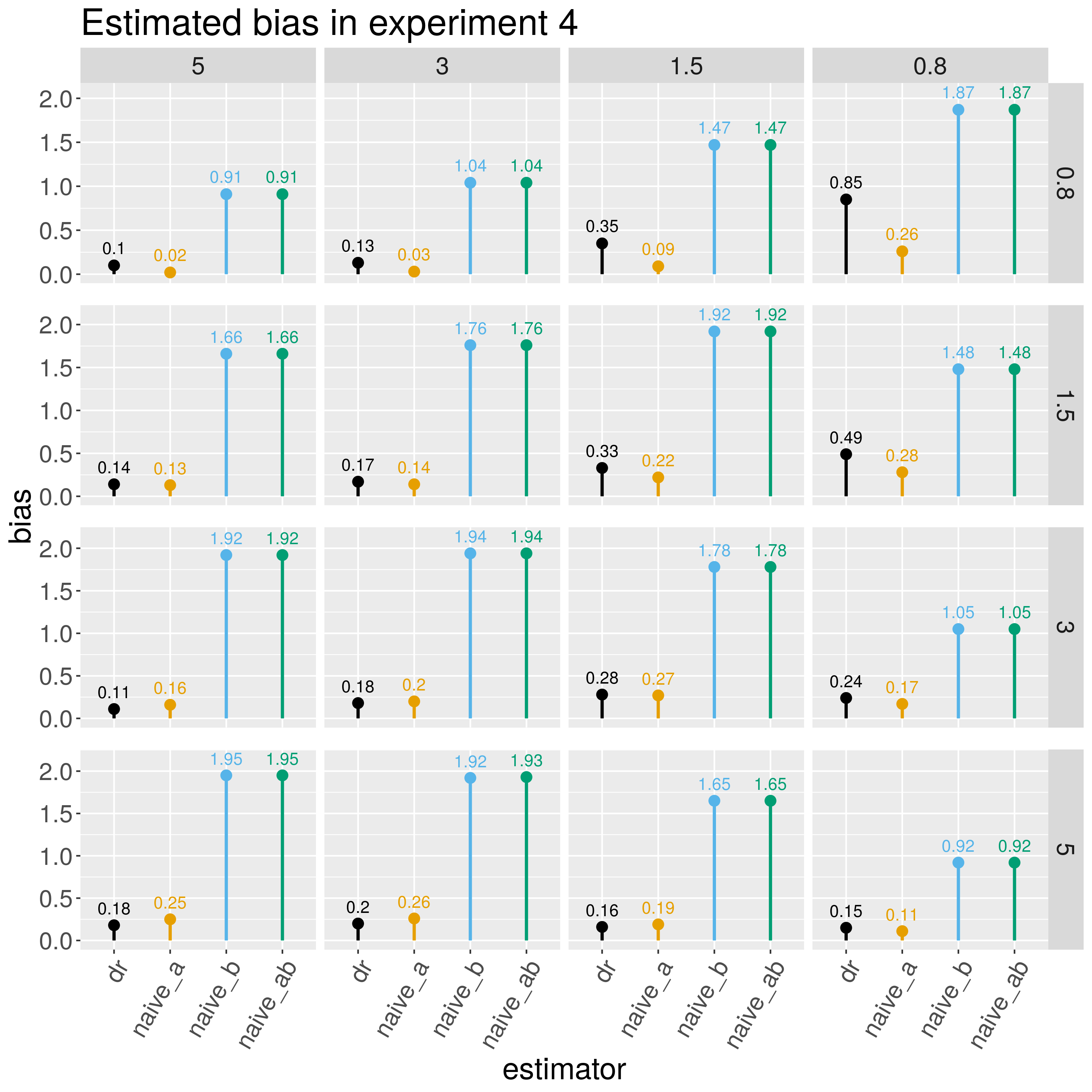}
\caption{Estimated bias in experiment 4. Rows correspond to different values
of $\protect\alpha_{a}$ and columns to different values of $\protect\alpha%
_{b}$.}
\label{fig:exp4_bias}
\end{figure}

\begin{figure}[tbp]
\centering
\includegraphics[scale=0.7]{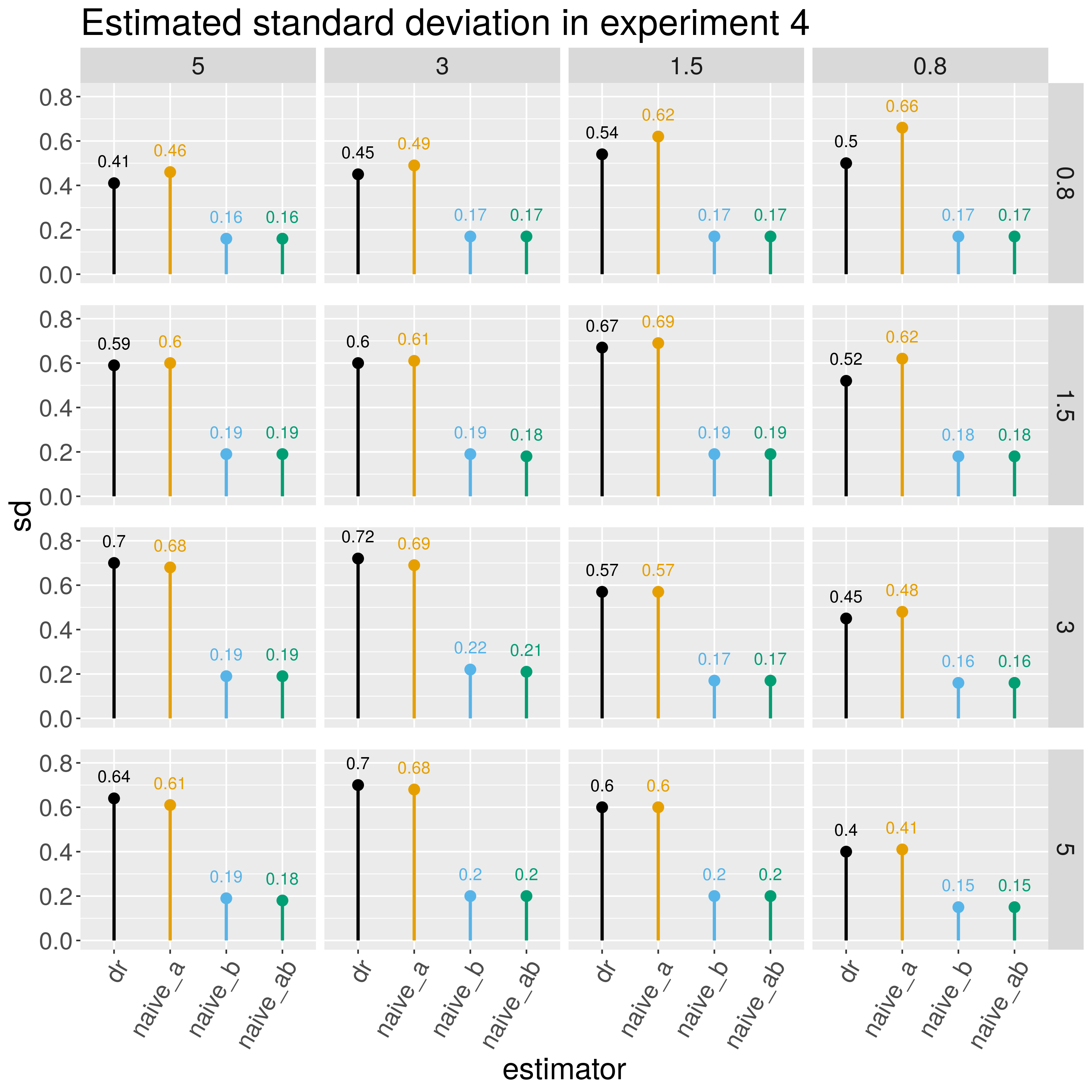}
\caption{Estimated standard errors in experiment 4. Rows correspond to
different values of $\protect\alpha_{a}$ and columns to different values of $%
\protect\alpha_{b}$.}
\label{fig:exp4_sd}
\end{figure}

\begin{figure}[tbp]
\centering
\includegraphics[scale=0.7]{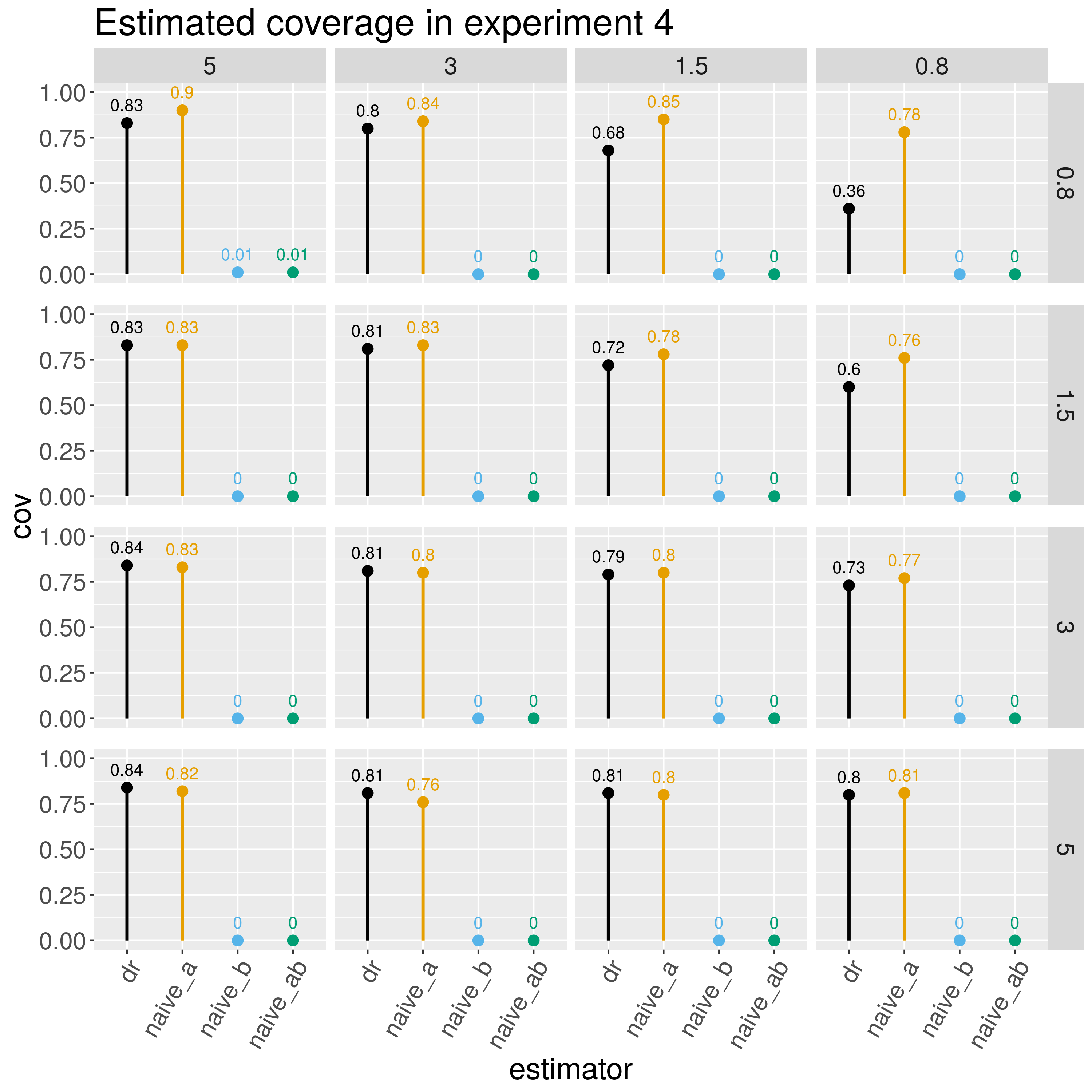}
\caption{Estimated coverage at the 95\% nominal level in experiment 4. Rows
correspond to different values of $\protect\alpha_{a}$ and columns to
different values of $\protect\alpha_{b}$.}
\label{fig:exp4_cov}
\end{figure}

\FloatBarrier

Figures \ref{fig:exp1_bias}-\ref{fig:exp1_cov} and \ref{fig:exp3_bias}-\ref%
{fig:exp3_cov} show that in experiments 1 and 3, i.e. in the cases in which
the working models for $a$ and $b$ are correctly specified, the bias of the
doubly robust estimator tends to diminish with increasing values of $\alpha
_{a}$ and $\alpha _{b}$. Moreover, the bias of the doubly robust estimator
is, for nearly all choices of $\left( \alpha _{a},\alpha _{b}\right) $ is
small, relative to its standard error, and much lower than the bias of the
naive estimators. The exceptions are the cases 
\begin{equation}
\left( \alpha _{a},\alpha _{b}\right) \in \left\{ \left( 0.8,0.8\right)
,\left( 0.8,1.5\right) ,\left( 1.5,0.8\right) \right\}
\label{eq:alpha_cases}
\end{equation}%
and, in experiment 3, also the case $\left( \alpha _{a},\alpha _{b}\right)
=\left( 1.5,1.5\right) .$ In these cases the bias of the DR estimator is of
the same magnitude or larger than that of its standard error. Recall that
for $\left( \alpha _{a},\alpha _{b}\right) =\left( 0.8,0.8\right) $ the key
condition $\alpha _{a}^{-1}+\alpha _{b}^{-1}<2$ does not hold and therefore
the convergence of our estimators at rate $\sqrt{n}$ is not supported by
theory. On the other hand, for the cases $\left( \alpha _{a},\alpha
_{b}\right) =\left( 1.5,0.8\right) $ and $\left( \alpha _{a},\alpha
_{b}\right) =\left( 0.8,1.5\right) ,\alpha _{a}^{-1}+\alpha _{b}^{-1}=1.92,$
and this might reflect the fact that for this borderline case the asymptotic
normal approximation is poor. A similar phenomenon might be operating in the
case $\left( \alpha _{a},\alpha _{b}\right) =\left( 1.5,1.5\right) $ for
experiment 3. The coverage probability of the Wald confidence intervals
centered at the DR estimators is very close to the nominal 95\%, once again,
for all choices of $\left( \alpha _{a},\alpha _{b}\right) ,$ except, as
expected, in the aforementioned cases. In contrast, the coverage of the Wald
confidence intervals centered at the naive estimators are, for most cases,
very poor.

Turn now to experiment 2, i.e. when both working models use linear links,
the working model for $a$ is correct but the working model for $b$ is
incorrect. Figure \ref{fig:exp2_bias} shows that in that experiment, the
bias of the doubly robust estimator tends to diminish with increasing values
of $\alpha _{a}$ and $\alpha _{b}.$ We note that Theorem \ref%
{theo:model_DR_Nonlin} does not guarantee asymptotic normality of the DR
estimator for the case $\alpha _{a}=0.8$ because the key condition of ultra
sparsity for the true and correctly modeled function $a$ does not hold. So,
the low biases observed for the row corresponding to $\alpha _{a}=0.8\,\ $%
and $\alpha _{b}>0.8$ are not supported by our theoretical results. Also, in
experiment 2, Figure \ref{fig:exp2_sd} shows that the bias of the DR
estimator is much smaller than its standard error except for the cases $%
\left( \ref{eq:alpha_cases}\right) $ and the case $\left( \alpha _{a},\alpha
_{b}\right) =\left( 1.5,1.5\right) .$ For the cases $\left( \alpha
_{a},\alpha _{b}\right) =\left( 1.5,0.8\right) $ and $\left( \alpha
_{a},\alpha _{b}\right) =\left( 1.5,1.5\right) $ the poor behaviour of the
DR estimator might reflect the fact that the asymptotic theory might be a
poor approximation to the finite sample behaviour. The coverages of the
confidence intervals centered at the DR estimator are close to the nominal
95\% value except for the aforementioned cases.

Finally turn to experiment 4, i.e. when both working models use exponential
links, the working model for $a$ is correct but the working model for $b$ is
incorrect. Unlike experiment 2, the DR estimator exhibits a somewhat poor
behavior, even for the cases in which $\left( \alpha _{a},\alpha _{b}\right) 
$ satisfy $\alpha _{a}^{-1}+\alpha _{b}^{-1}<2$ and $\alpha _{a}>0.8,$ i.e.
cases in which the assumptions of Theorem \ref{theo:model_DR_Nonlin} are
satisfied. For instance, when $\left( \alpha _{a},\alpha _{b}\right) =\left(
3,1.5\right) $ the bias of the DR estimator is 0.28 which is half its Monte
Carlo standard error of 0.57. We suspect that this poor behaviour is, once
again, due to the poor approximation of the asymptotic distribution to its
finite sample counterpart. This suspicion is supported by the behaviour of
the DR estimator when the number of covariates is $p=100$ instead of 200 and 
$n$ stays equal to 1000. Figures \ref{fig:exp5_bias}-\ref{fig:exp5_cov}
report the results of this latter experiment, again based on 500
replications.

\begin{figure}[tbp]
\centering
\includegraphics[scale=0.7]{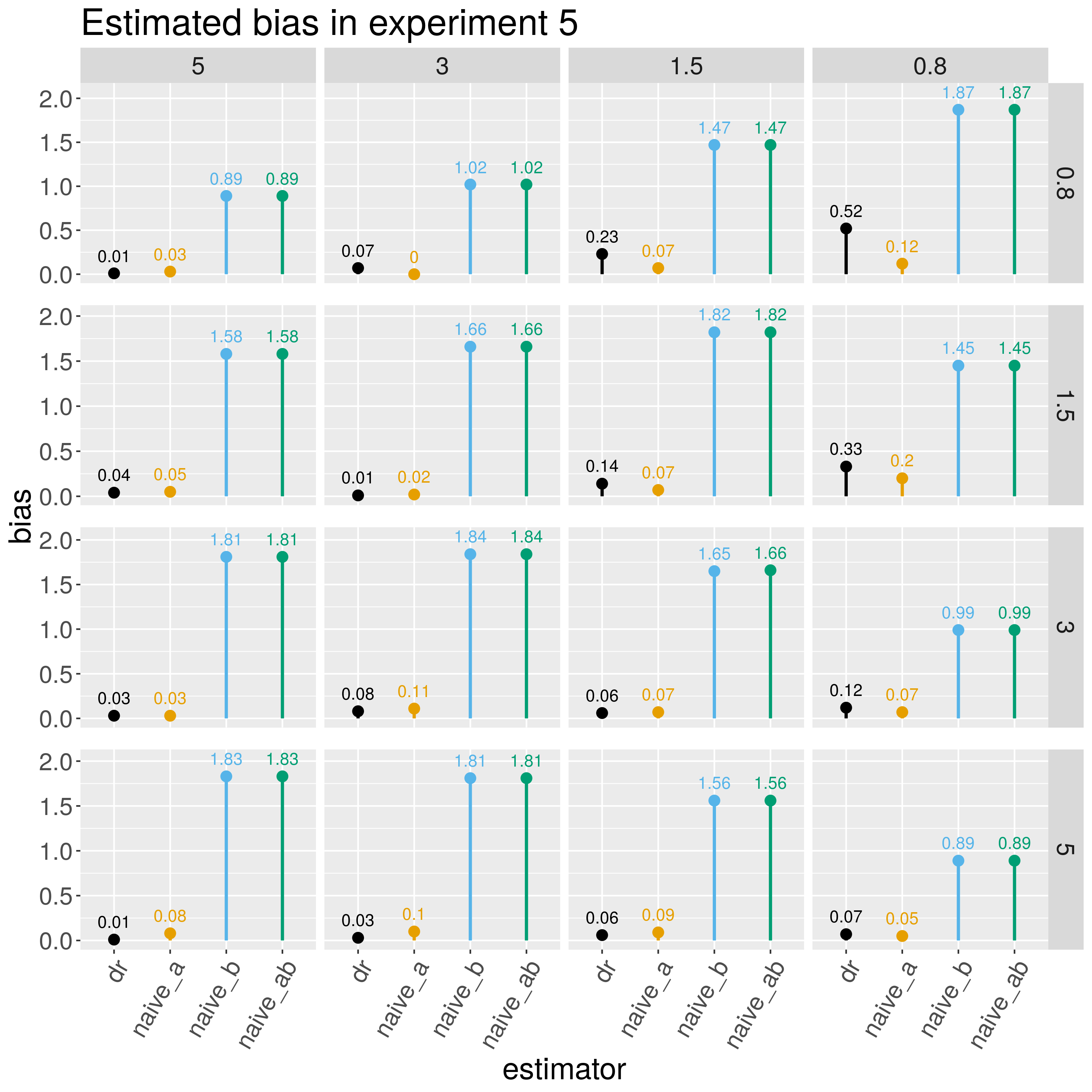}
\caption{Estimated bias in experiment 5. Rows correspond to different values
of $\protect\alpha_{a}$ and columns to different values of $\protect\alpha%
_{b}$.}
\label{fig:exp5_bias}
\end{figure}

\begin{figure}[tbp]
\centering
\includegraphics[scale=0.7]{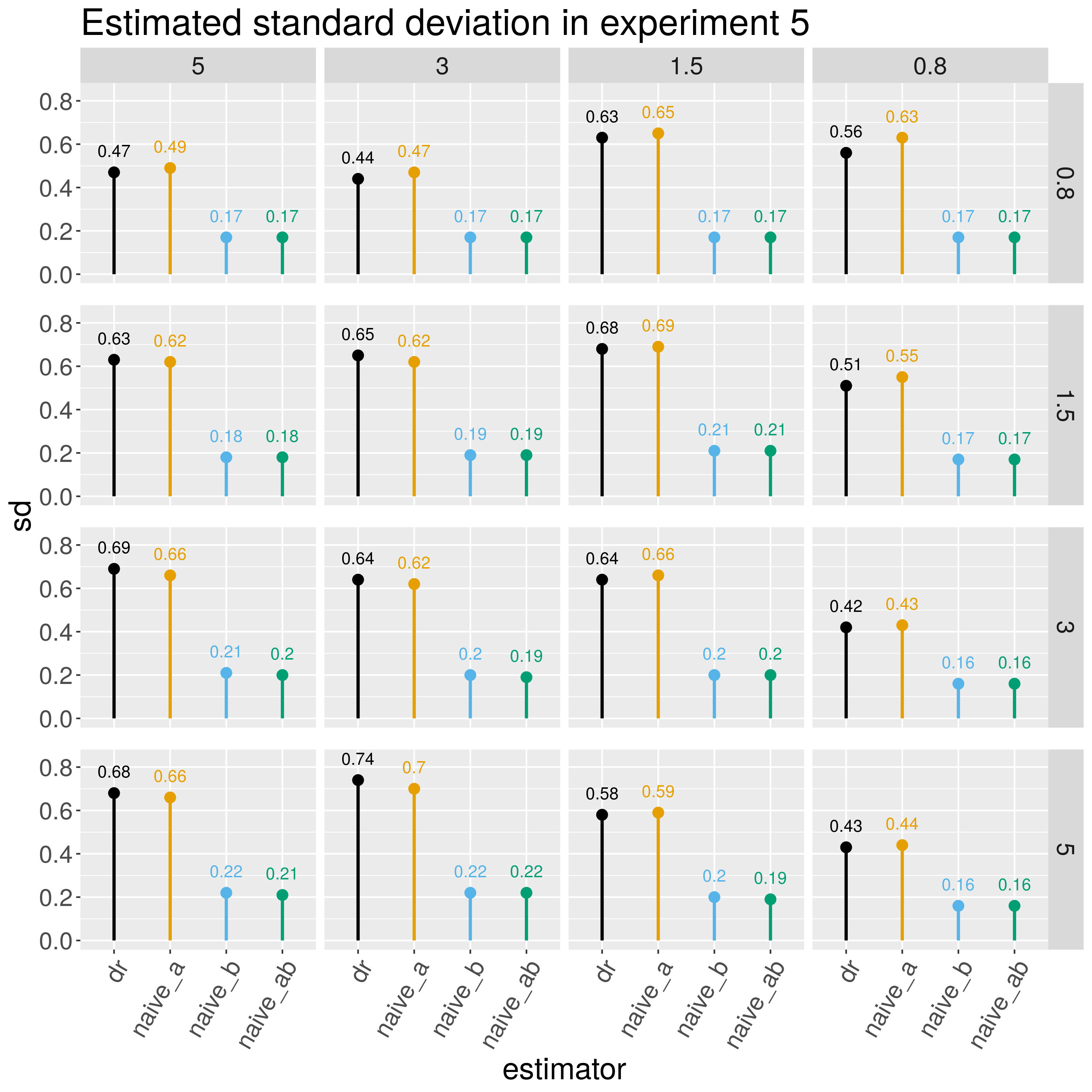}
\caption{Estimated standard errors in experiment 5. Rows correspond to
different values of $\protect\alpha_{a}$ and columns to different values of $%
\protect\alpha_{b}$.}
\label{fig:exp5_sd}
\end{figure}

\begin{figure}[tbp]
\centering
\includegraphics[scale=0.7]{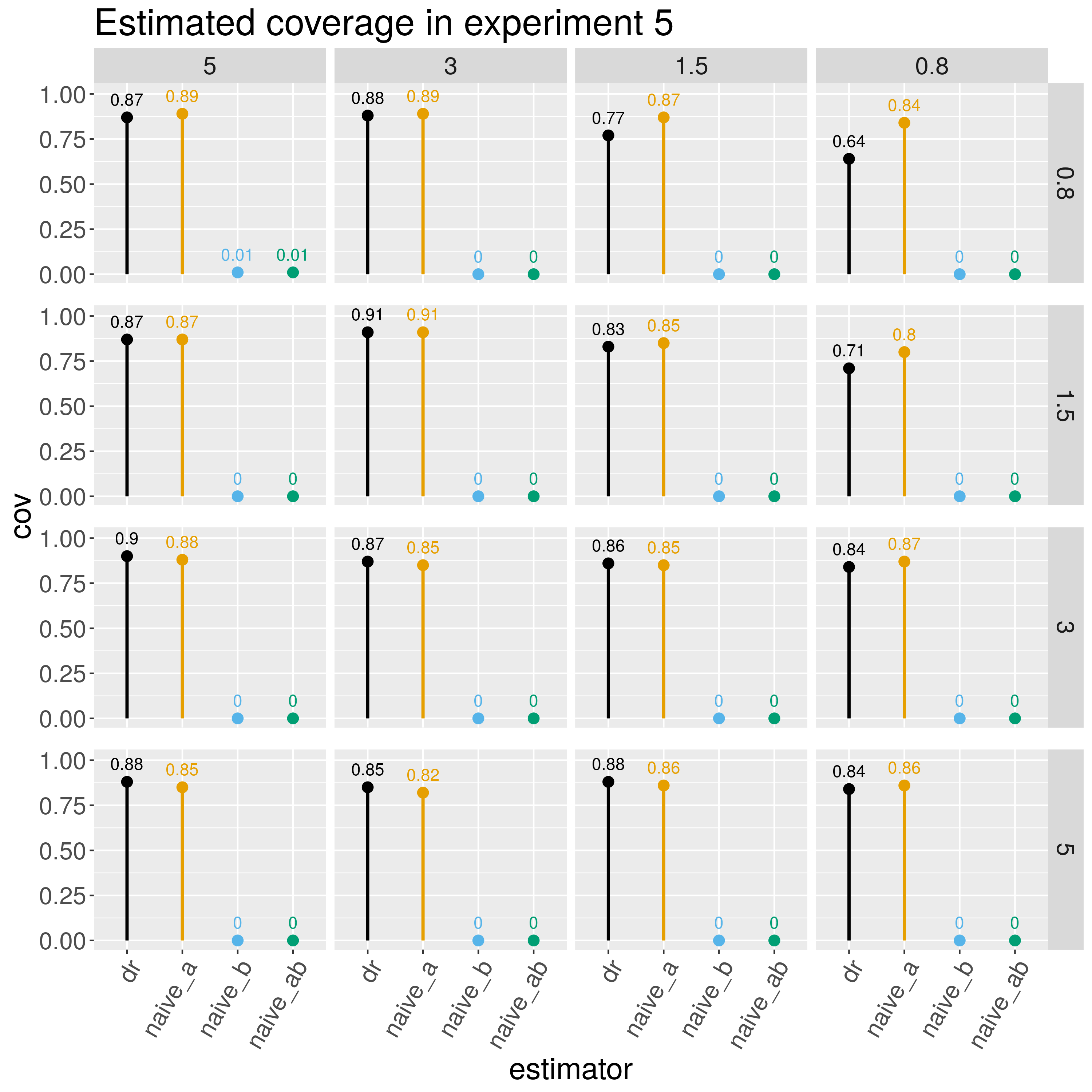}
\caption{Estimated coverage at the 95\% nominal level in experiment 5. Rows
correspond to different values of $\protect\alpha_{a}$ and columns to
different values of $\protect\alpha_{b}$.}
\label{fig:exp5_cov}
\end{figure}

\FloatBarrier

Observe that for this experiment the bias of the DR estimator is orders of
magnitude smaller than its standard error. Another concern is the poor
coverage probability of the nominal 95\% Wald confidence intervals centered
at the DR estimator, this problem being more predominant when $p=200$ but
also present when $p=100.$ When $p=200$, part of this poor coverage can be
attributed in part to the bias of the DR estimator, and also in part to the
fact that the standard error estimator underestimates the true standard
error. For instance, for a bias that is equal to half the standard error,
one would expect the nominal 95\% Wald confidence interval to cover with
probability roughly equal to 0.92. Yet, when $\left( \alpha _{a},\alpha
_{b}\right) =\left( 3,1.5\right) $, the bias is half the standard error but
the coverage probability is equal to 0.79. A possible explanation for this
excess undercoverage is the fact that the mean and median of the estimated
standard errors over the 500 replications are equal to 0.51 and 0.45 whereas
the Monte Carlo standard error is 0.57. Curiously, in experiment 4, the bias
of the "naive\_a" estimator $\mathbb{P}_{n}\left[ D\widehat{a}_{N}\left(
Z\right) \right] $ is roughly of the same magnitude as the bias of the DR
estimator, or even smaller for the scenarios $\left( \ref{eq:alpha_cases}%
\right) $ and $\left( \alpha _{a},\alpha _{b}\right) =\left( 1.5,1.5\right)
. $ This peculiar behavior of the "naive\_a"~estimator was not exhibited in
the experiments in which both working models were correctly specified
(experiments 1 and 3), so we suspect that it might be due to the specific
data generating process that we used in experiment 4. Further investigation
of this peculiarity is warranted.

\bibliographystyle{apalike}
\bibliography{non-linear}

\end{document}